\documentclass[11pt]{amsart}

\makeatletter
\renewcommand\part{%
   \if@noskipsec \leavevmode \fi
   \par
   \addvspace{4ex}%
   \@afterindentfalse
   \secdef\@part\@spart}

\def\@part[#1]#2{%
    \ifnum \c@secnumdepth >\m@ne
      \refstepcounter{part}%
      \addcontentsline{toc}{part}{\thepart\hspace{1em}#1}%
    \else
      \addcontentsline{toc}{part}{#1}%
    \fi
    {\parindent \z@ \raggedright
     \interlinepenalty \@M
     \normalfont
     \ifnum \c@secnumdepth >\m@ne
       \Large\bfseries \partname\nobreakspace\thepart
       \par\nobreak
     \fi
     \huge \bfseries #2%
     \par}%
    \nobreak
    \vskip 3ex
    \@afterheading}
\def\@spart#1{%
    {\parindent \z@ \raggedright
     \interlinepenalty \@M
     \normalfont
     \huge \bfseries #1\par}%
     \nobreak
     \vskip 3ex
     \@afterheading}
\makeatother

\usepackage{framed}
\usepackage[english]{babel}
\usepackage[utf8]{inputenc}  
\usepackage[T1]{fontenc}
\usepackage{enumitem}
\usepackage{xcolor}
\usepackage{tabularx}

\usepackage{amssymb,latexsym}
\usepackage{amsmath}
\usepackage{wrapfig}
\usepackage{tikz-cd}
\tikzcdset{scale cd/.style={every label/.append style={scale=#1},
    cells={nodes={scale=#1}}}}
\usepackage{subcaption}
\usepackage{graphicx}
\graphicspath{{./dessins/}}
\usepackage{mathrsfs}
\usepackage{bm}
\usepackage{amscd}
\usepackage[hypertexnames=false]{hyperref}
\usepackage{bbm}
\usepackage{tikz}
\usepackage[export]{adjustbox}

\usepackage{graphicx}
\graphicspath{ {./dessinsinkscape/}{./croquismain/} {./photoshop/} }

\hypersetup{
    colorlinks=false}

\usepackage[left=2.5cm,right=2.5cm,top=4cm,bottom=3cm,headsep=1cm]{geometry}

\theoremstyle{definition}
\newtheorem{definition}{Definition}

\theoremstyle{definition}
\newtheorem{definition-proposition}{Definition-Proposition}

\theoremstyle{theorem}
\newtheorem{theorem}{Theorem}

\theoremstyle{theorem}
\newtheorem{lemma}{Lemma}

\theoremstyle{theorem}
\newtheorem{corollary}{Corollary}

\theoremstyle{theorem}
\newtheorem{proposition}{Proposition}

\theoremstyle{remark}

\numberwithin{equation}{section}

\begin{document}

\title{Higher algebra of \Ainf\ and \ombas -algebras in Morse theory I}
\author{Thibaut Mazuir}


\newcommand{\End}{\ensuremath{\mathrm{End}}}
\newcommand{\Hom}{\ensuremath{\mathrm{Hom}}}
\newcommand{\N}{\ensuremath{\mathbb{N}}}
\newcommand{\R}{\ensuremath{\mathbb{R}}}
\newcommand{\Z}{\ensuremath{\mathbb{Z}}}
\newcommand{\Sp}{\ensuremath{\mathbb{S}}}
\newcommand{\Lac}{\ensuremath{\mathrm{L}}}
\newcommand{\ide}{\ensuremath{\mathrm{id}}}
\newcommand{\calA}{\ensuremath{\mathcal{A}}}
\newcommand{\Ainf}{\ensuremath{A_\infty}}
\newcommand{\ombas}{\ensuremath{\Omega B As}}
\newcommand{\infmor}{\ensuremath{\Ainf - \mathrm{Morph}}}
\newcommand{\ombasmor}{\ensuremath{\ombas - \mathrm{Morph}}}
\newcommand{\inprod}{$\infty$-inner product}
\newcommand{\infcat}{$\infty$-category}
\newcommand{\degr}{\ensuremath{\mathrm{deg}}}
\newcommand{\CM}{\ensuremath{\mathrm{CM}}}

\newcommand{\rouge}[1]{\textcolor{red}{#1}}
\newcommand{\bleu}[1]{\textcolor{blue}{#1}}
\newcommand{\verde}[1]{\textcolor{green}{#1}}
\newcommand{\violet}[1]{\textcolor{violet}{#1}}
\newcommand{\jaune}[1]{\textcolor{yellow}{#1}}

\newcommand*{\inserim}[1]{%
  \raisebox{-.3\baselineskip}{%
    \includegraphics[
      width=0.08\textwidth
    ]{#1}%
  }%
}

\newcommand*{\inserimi}[1]{%
  \raisebox{-.3\baselineskip}{%
    \includegraphics[
      width=0.15\textwidth
    ]{#1}%
  }%
}



\newcommand{\arbreop}[1]{
\begin{tikzpicture}[scale=#1,baseline=(pt_base.base)]

\coordinate (pt_base) at (0,0) ;

\draw (-2,1.25) node[above]{\tiny 1} ;
\draw (-1,1.25) node[above]{\tiny 2} ;
\draw (1,1.25) node[above]{};
\draw (2,1.25) node[above]{\tiny $n$} ;

\draw (-2,1.25) -- (0,0) ;
\draw (-1,1.25) -- (0,0) ;
\draw (1,1.25) -- (0,0) ;
\draw (2,1.25) -- (0,0) ;
\draw (0,-1.25) -- (0,0) ;

\draw [densely dotted] (-0.5,0.9) -- (0.5,0.9) ;

\end{tikzpicture}}


\newcommand{\arbreopdeux}{
\begin{tikzpicture}[scale=0.15,baseline=(pt_base.base)]

\coordinate (pt_base) at (0,-0.5) ;

\draw (-1.25,1.25) -- (0,0) ;
\draw (1.25,1.25) -- (0,0) ;
\draw (0,-1.25) -- (0,0) ;

\end{tikzpicture}
}


\newcommand{\arbreoptrois}{
\begin{tikzpicture}[scale=0.15,baseline=(pt_base.base)]

\coordinate (pt_base) at (0,-0.5) ;

\draw (0,1.25) -- (0,0) ;
\draw (-1.25,1.25) -- (0,0) ;
\draw (1.25,1.25) -- (0,0) ;
\draw (0,-1.25) -- (0,0) ;

\end{tikzpicture}
}


\newcommand{\arbreopquatre}{
\begin{tikzpicture}[scale=0.15,baseline=(pt_base.base)]

\coordinate (pt_base) at (0,-0.5) ;

\draw (-1.5,1.25) -- (0,0) ;
\draw (1.5,1.25) -- (0,0) ;
\draw (-0.5,1.25) -- (0,0) ;
\draw (0.5,1.25) -- (0,0) ;
\draw (0,-1.25) -- (0,0) ;

\end{tikzpicture}
}


\newcommand{\eqainf}{
\begin{tikzpicture}[scale=0.2,baseline=(pt_base.base)]

\coordinate (pt_base) at (0,1.25) ;

\draw (-4,1.25) node[above]{\tiny 1} ;
\draw (-3,1.25) node[above]{};
\draw (4,1.25) node[above]{\tiny $k$} ;
\draw (3,1.25) node[above]{};
\draw (0,1.75) node[right]{\tiny $i$} ;

\draw (-2,4) node[above]{\tiny 1} ;
\draw (-1,4) node[above]{};
\draw (1,4) node[above]{} ;
\draw (2,4) node[above]{\tiny $h$};

\draw (-4,1.25) -- (0,0) ;
\draw (-3,1.25) -- (0,0) ;
\draw (4,1.25) -- (0,0) ;
\draw (3,1.25) -- (0,0) ;
\draw (0,1.25) -- (0,0) ;
\draw(0,-1.25) -- (0,0) ; 

\draw (-2,4) -- (0,2.75) ;
\draw (-1,4) -- (0,2.75) ;
\draw (1,4) -- (0,2.75) ;
\draw (2,4) -- (0,2.75) ;
\draw (0,1.5) -- (0,2.75) ;

\draw [densely dotted] (-0.5,3.65) -- (0.5,3.65) ;
\draw [densely dotted] (-0.5,0.9) -- (-1.5,0.9) ;
\draw [densely dotted] (0.5,0.9) -- (1.5,0.9) ;

\end{tikzpicture}}


\newcommand{\arbreoprouge}[1]{
\begin{tikzpicture}[scale=#1,baseline=(pt_base.base)]

\coordinate (pt_base) at (0,0) ;

\draw (-2,1.25) node[above]{\tiny 1} ;
\draw (-1,1.25) node[above]{\tiny 2} ;
\draw (1,1.25) node[above]{};
\draw (2,1.25) node[above]{\tiny $n$} ;

\draw[red] (-2,1.25) -- (0,0) ;
\draw[red] (-1,1.25) -- (0,0) ;
\draw[red] (1,1.25) -- (0,0) ;
\draw[red] (2,1.25) -- (0,0) ;
\draw[red] (0,-1.25) -- (0,0) ;

\draw[red] [densely dotted] (-0.5,0.9) -- (0.5,0.9) ;

\end{tikzpicture}}

\newcommand{\arbreopbleu}[1]{
\begin{tikzpicture}[scale=#1,baseline=(pt_base.base)]

\coordinate (pt_base) at (0,0) ;

\draw (-2,1.25) node[above]{\tiny 1} ;
\draw (-1,1.25) node[above]{\tiny 2} ;
\draw (1,1.25) node[above]{};
\draw (2,1.25) node[above]{\tiny $n$} ;

\draw[blue] (-2,1.25) -- (0,0) ;
\draw[blue] (-1,1.25) -- (0,0) ;
\draw[blue] (1,1.25) -- (0,0) ;
\draw[blue] (2,1.25) -- (0,0) ;
\draw[blue] (0,-1.25) -- (0,0) ;

\draw[blue] [densely dotted] (-0.5,0.9) -- (0.5,0.9) ;

\end{tikzpicture}}


\newcommand{\arbreopmorph}[1]{
\begin{tikzpicture}[scale=#1,baseline=(pt_base.base)]

\coordinate (pt_base) at (0,0) ;

\draw (-2,1.25) node[above]{\tiny 1} ;
\draw (-1,1.25) node[above]{\tiny 2} ;
\draw (1,1.25) node[above]{};
\draw (2,1.25) node[above]{\tiny $n$} ;

\draw[blue] (-2,1.25) -- (-0.625,0.5) ;
\draw[blue] (-1,1.25) -- (-0.3125,0.5) ;
\draw[blue] (1,1.25) -- (0.3125,0.5) ;
\draw[blue] (2,1.25) -- (0.625,0.5) ;

\draw[red] (-0.625,0.5) -- (0,0) ;
\draw[red] (-0.3125,0.5) -- (0,0) ;
\draw[red] (0.625,0.5) -- (0,0) ;
\draw[red] (0.3125,0.5) -- (0,0) ;
\draw[red] (0,-1.25) -- (0,0) ;

\draw[blue] [densely dotted] (-0.5,0.9) -- (0.5,0.9) ;

\draw  (-2,0.5) -- (2,0.5) ;

\end{tikzpicture}}


\newcommand{\arbreopunmorph}{
\begin{tikzpicture}[scale=0.15,baseline=(pt_base.base)]

\coordinate (pt_base) at (0,-0.5) ;

\draw[blue] (0,1.25) -- (0,0.5) ;
\draw[red] (0,-1.25) -- (0,0.5) ;

\draw (-1.25,0.5) -- (1.25,0.5) ;

\end{tikzpicture}
}

\newcommand{\arbreopunmorpha}{
\begin{tikzpicture}[scale=0.08,baseline=(pt_base.base)]

\coordinate (pt_base) at (0,-0.5) ;

\draw[blue] (0,1.25) -- (0,0.5) ;
\draw[red] (0,-1.25) -- (0,0.5) ;

\draw (-1.25,0.5) -- (1.25,0.5) ;

\end{tikzpicture}
}

\newcommand{\arbreopunmorphfg}{
\begin{tikzpicture}[scale=0.08,baseline=(pt_base.base)]

\coordinate (pt_base) at (0,-0.5) ;

\draw[blue] (0,1.25) -- (0,0.5) ;
\draw[red] (0,-1.25) -- (0,0.5) ;

\draw (-1.25,0.5) -- (1.25,0.5) ;

\end{tikzpicture}
}

\newcommand{\arbreopunmorphff}{
\begin{tikzpicture}[scale=0.08,baseline=(pt_base.base)]

\coordinate (pt_base) at (0,-0.5) ;

\draw[blue] (0,1.25) -- (0,0.5) ;
\draw[blue] (0,-1.25) -- (0,0.5) ;

\draw (-1.25,0) -- (1.25,0) ;

\end{tikzpicture}
}

\newcommand{\arbreopunmorphgf}{
\begin{tikzpicture}[scale=0.08,baseline=(pt_base.base)]

\coordinate (pt_base) at (0,-0.5) ;

\draw[red] (0,1.25) -- (0,0.5) ;
\draw[blue] (0,-1.25) -- (0,0.5) ;

\draw (-1.25,-0.5) -- (1.25,-0.5) ;

\end{tikzpicture}
}

\newcommand{\quilteddiskexample}{
\begin{tikzpicture}[scale = 0.5]

\draw[fill,blue!50] (0,0) circle (3) ;
\draw (0,0) circle (3) ;

\draw (45 : 3) node[scale=0.1]{\pointbullet};
\draw (90 : 3) node[scale=0.1]{\pointbullet};
\draw (135 : 3) node[scale=0.1]{\pointbullet};

\draw (45 : 3) node[above=3pt,scale=0.7]{$z_3$};
\draw (90 : 3) node[above=3pt,scale=0.7]{$z_2$};
\draw (135 : 3) node[above=3pt,scale=0.7]{$z_1$};
\draw (270 : 3) node[below=3pt,scale=0.7]{$z_0$};

\draw[fill,red!50] (0,-1) circle (2) ;
\draw (0,-1) circle (2) ;

\draw (270 : 3) node[scale = 0.1]{\pointbullet};

\end{tikzpicture}
}


\newcommand{\arbreopdeuxmorph}{
\begin{tikzpicture}[scale=0.15,baseline=(pt_base.base)]

\coordinate (pt_base) at (0,-0.5) ;

\draw[blue] (-1.25,1.25) -- (-0.5,0.5) ;
\draw[blue] (1.25,1.25) -- (0.5,0.5) ;

\draw[red] (-0.5,0.5) -- (0,0) ;
\draw[red] (0.5,0.5) -- (0,0) ;
\draw[red] (0,-1.25) -- (0,0) ;

\draw (-1.25,0.5) -- (1.25,0.5) ;

\end{tikzpicture}
}


\newcommand{\arbreoptroismorph}{
\begin{tikzpicture}[scale=0.15,baseline=(pt_base.base)]

\coordinate (pt_base) at (0,-0.5) ;

\draw[blue] (0,1.25) -- (0,0.5) ;
\draw[blue] (-1.25,1.25) -- (-0.5,0.5) ;
\draw[blue] (1.25,1.25) -- (0.5,0.5) ;

\draw[red] (0,0.5) -- (0,0) ;
\draw[red] (-0.5,0.5) -- (0,0) ;
\draw[red] (0.5,0.5) -- (0,0) ;
\draw[red] (0,-1.25) -- (0,0) ;

\draw (-1.25,0.5) -- (1.25,0.5) ;

\end{tikzpicture}
}


\newcommand{\arbreopquatremorph}{
\begin{tikzpicture}[scale=0.15,baseline=(pt_base.base)]

\coordinate (pt_base) at (0,-0.5) ;

\draw[blue] (-1.5,1.25) -- (-0.6,0.5) ;
\draw[blue] (1.5,1.25) -- (0.6,0.5) ;
\draw[blue] (-0.5,1.25) -- (-0.2,0.5) ;
\draw[blue] (0.5,1.25) -- (0.2,0.5) ;

\draw[red] (-0.6,0.5) -- (0,0) ;
\draw[red] (0.6,0.5) -- (0,0) ;
\draw[red] (0.2,0.5) -- (0,0) ;
\draw[red] (-0.2,0.5) -- (0,0) ;
\draw[red] (0,-1.25) -- (0,0) ;

\draw (-1.5,0.5) -- (1.5,0.5) ;

\end{tikzpicture}
}


\newcommand{\arbreopdeuxcol}[1]{
\begin{tikzpicture}[scale=0.15,baseline=(pt_base.base)]

\coordinate (pt_base) at (0,-0.5) ;

\draw[#1] (-1.25,1.25) -- (0,0) ;
\draw[#1] (1.25,1.25) -- (0,0) ;
\draw[#1] (0,-1.25) -- (0,0) ;

\end{tikzpicture}
}


\newcommand{\arbreoptroiscol}[1]{
\begin{tikzpicture}[scale=0.15,baseline=(pt_base.base)]

\coordinate (pt_base) at (0,-0.5) ;

\draw[#1] (0,1.25) -- (0,0) ;
\draw[#1] (-1.25,1.25) -- (0,0) ;
\draw[#1] (1.25,1.25) -- (0,0) ;
\draw[#1] (0,-1.25) -- (0,0) ;

\end{tikzpicture}
}


\newcommand{\arbreopquatrecol}[1]{
\begin{tikzpicture}[scale=0.15,baseline=(pt_base.base)]

\coordinate (pt_base) at (0,-0.5) ;

\draw[#1] (-1.5,1.25) -- (0,0) ;
\draw[#1] (1.5,1.25) -- (0,0) ;
\draw[#1] (-0.5,1.25) -- (0,0) ;
\draw[#1] (0.5,1.25) -- (0,0) ;
\draw[#1] (0,-1.25) -- (0,0) ;

\end{tikzpicture}
}


\newcommand{\eqainfmorphun}{
\begin{tikzpicture}[scale=0.2,baseline=(pt_base.base)]

\coordinate (pt_base) at (0,1.25) ;


\draw (-4,1.25) node[above]{\tiny 1} ;
\draw (-3,1.25) node[above]{};
\draw (4,1.25) node[above]{\tiny $k$} ;
\draw (3,1.25) node[above]{};
\draw (0,1.75) node[right]{\tiny $i$} ;

\draw (-2,4) node[above]{\tiny 1} ;
\draw (-1,4) node[above]{};
\draw (1,4) node[above]{} ;
\draw (2,4) node[above]{\tiny $h$};


\draw[blue] (-4,1.25) -- (-1.6,0.5) ;
\draw[blue] (-3,1.25) -- (-1.2,0.5) ;
\draw[blue] (4,1.25) -- (1.6,0.5) ;
\draw[blue] (3,1.25) -- (1.2,0.5) ;
\draw[blue] (0,1.25) -- (0,0.5) ;

\draw[red] (-1.6,0.5) -- (0,0) ;
\draw[red] (-1.2,0.5) -- (0,0) ;
\draw[red] (1.6,0.5) -- (0,0) ;
\draw[red] (1.2,0.5) -- (0,0) ;
\draw[red] (0,0.5) -- (0,0) ;
\draw[red] (0,-1.25) -- (0,0) ; 


\draw[blue] (-2,4) -- (0,2.75) ;
\draw[blue] (-1,4) -- (0,2.75) ;
\draw[blue] (1,4) -- (0,2.75) ;
\draw[blue] (2,4) -- (0,2.75) ;
\draw[blue] (0,1.5) -- (0,2.75) ;


\draw[blue] [densely dotted] (-0.5,3.65) -- (0.5,3.65) ;
\draw[blue] [densely dotted] (-0.5,0.9) -- (-1.5,0.9) ;
\draw[blue] [densely dotted] (0.5,0.9) -- (1.5,0.9) ;

\draw (-4,0.5) -- (4,0.5) ;

\end{tikzpicture}}


\newcommand{\eqainfmorphdeux}{
\begin{tikzpicture}[scale=0.2,baseline=(pt_base.base)]

\coordinate (pt_base) at (0,1.25) ;


\draw (2.5,4) node[above]{\tiny 1} ;
\draw (5.5,4) node[above]{\tiny $i_s$};
\draw (-2.5,4) node[above]{\tiny $i_1$} ;
\draw (-5.5,4) node[above]{\tiny 1};


\draw[red] (-4,1.25) -- (0,0) ;
\draw[red] (-2,1.25) -- (0,0) ;
\draw[red] (4,1.25) -- (0,0) ;
\draw[red] (2,1.25) -- (0,0) ;
\draw[red] (0,-1.25) -- (0,0) ; 


\draw[blue] (2.5,4) -- (3.4,3.25) ;
\draw[blue] (3.5,4) -- (3.8,3.25) ;
\draw[blue] (4.5,4) -- (4.2,3.25) ;
\draw[blue] (5.5,4) -- (4.6,3.25) ;

\draw[red] (3.4,3.25) -- (4,2.75) ;
\draw[red] (3.8,3.25) -- (4,2.75) ;
\draw[red] (4.2,3.25) -- (4,2.75) ;
\draw[red] (4.6,3.25) -- (4,2.75) ;
\draw[red] (4,1.5) -- (4,2.75) ;

\draw (2.5,3.25) -- (5.5,3.25) ;


\draw[blue] (-2.5,4) -- (-3.4,3.25) ;
\draw[blue] (-3.5,4) -- (-3.8,3.25) ;
\draw[blue] (-4.5,4) -- (-4.2,3.25) ;
\draw[blue] (-5.5,4) -- (-4.6,3.25) ;

\draw[red] (-3.4,3.25) -- (-4,2.75) ;
\draw[red] (-3.8,3.25) -- (-4,2.75) ;
\draw[red] (-4.2,3.25) -- (-4,2.75) ;
\draw[red] (-4.6,3.25) -- (-4,2.75) ;
\draw[red] (-4,1.5) -- (-4,2.75) ;

\draw (-2.5,3.25) -- (-5.5,3.25) ;


\draw [densely dotted,red] (-1,0.9) -- (1,0.9) ;
\draw [densely dotted] (2,2.75) -- (-2,2.75) ; 
\draw [densely dotted,blue] (3.8,3.7) -- (4.2,3.7) ;
\draw [densely dotted,blue] (-3.8,3.7) -- (-4.2,3.7) ;

\end{tikzpicture}}



\newcommand{\associaedreun}{
\begin{tikzpicture}[ baseline = (pt_base.base)]
\coordinate (pt_base) at (0,0) ;
\node[scale = 0.1] at (0,0) {\pointbullet} ;
\node[above left] at (0,0) {\arbreopdeux} ;
\end{tikzpicture}
}


\newcommand{\arbreopdeuxun}{
\begin{tikzpicture}

\begin{scope}
\draw (-1.25,1.25) -- (0,0) ;
\draw (1.25,1.25) -- (0,0) ;
\draw (0,-1.25) -- (0,0) ;
\end{scope}

\begin{scope}[shift = {(-1.25,3)}]
\draw (-1.25,1.25) -- (0,0) ;
\draw (1.25,1.25) -- (0,0) ;
\draw (0,-1.25) -- (0,0) ;
\end{scope}

\end{tikzpicture}
}

\newcommand{\arbreopdeuxdeux}{
\begin{tikzpicture}
\begin{tikzpicture}

\begin{scope}
\draw (-1.25,1.25) -- (0,0) ;
\draw (1.25,1.25) -- (0,0) ;
\draw (0,-1.25) -- (0,0) ;
\end{scope}

\begin{scope}[shift = {(1.25,3)}]
\draw (-1.25,1.25) -- (0,0) ;
\draw (1.25,1.25) -- (0,0) ;
\draw (0,-1.25) -- (0,0) ;
\end{scope}

\end{tikzpicture}
\end{tikzpicture}
}

\newcommand{\associaedredeux}{
\begin{tikzpicture}[ baseline = (pt_base.base)]
\coordinate (pt_base) at (0,0) ;
\node[scale = 0.1] (A) at (-1,0) {\pointbullet} ;
\node[scale = 0.1] (B) at (1,0) {\pointbullet} ;
\node[above left = 2 pt,scale = 0.1] at (A) {\arbreopdeuxun} ;
\node[above right = 2 pt,scale = 0.1] at (B) {\arbreopdeuxdeux} ;
\draw (A.center) -- (B.center) ;
\node[above] at (0,0) {\arbreoptrois} ;
\end{tikzpicture}
}


\newcommand{\arbreoptdAB}{
\begin{tikzpicture}
\begin{scope}
\draw (0,1.25) -- (0,0) ;
\draw (-1.25,1.25) -- (0,0) ;
\draw (1.25,1.25) -- (0,0) ;
\draw (0,-1.25) -- (0,0) ;
\end{scope}

\begin{scope}[shift = {(1.25,3)}]
\draw (-1.25,1.25) -- (0,0) ;
\draw (1.25,1.25) -- (0,0) ;
\draw (0,-1.25) -- (0,0) ;
\end{scope}
\end{tikzpicture}
}

\newcommand{\arbreoptdDE}{
\begin{tikzpicture}
\begin{scope}
\draw (0,1.25) -- (0,0) ;
\draw (-1.25,1.25) -- (0,0) ;
\draw (1.25,1.25) -- (0,0) ;
\draw (0,-1.25) -- (0,0) ;
\end{scope}

\begin{scope}[shift = {(0,3)}]
\draw (-1.25,1.25) -- (0,0) ;
\draw (1.25,1.25) -- (0,0) ;
\draw (0,-1.25) -- (0,0) ;
\end{scope}
\end{tikzpicture}
}

\newcommand{\arbreoptdBC}{
\begin{tikzpicture}
\begin{scope}
\draw (0,1.25) -- (0,0) ;
\draw (-1.25,1.25) -- (0,0) ;
\draw (1.25,1.25) -- (0,0) ;
\draw (0,-1.25) -- (0,0) ;
\end{scope}

\begin{scope}[shift = {(-1.25,3)}]
\draw (-1.25,1.25) -- (0,0) ;
\draw (1.25,1.25) -- (0,0) ;
\draw (0,-1.25) -- (0,0) ;
\end{scope}
\end{tikzpicture}
}

\newcommand{\arbreoptdEA}{
\begin{tikzpicture}

\begin{scope}
\draw (-1.25,1.25) -- (0,0) ;
\draw (1.25,1.25) -- (0,0) ;
\draw (0,-1.25) -- (0,0) ;
\end{scope}

\begin{scope}[shift = {(1.25,3)}]
\draw (0,1.25) -- (0,0) ;
\draw (-1.25,1.25) -- (0,0) ;
\draw (1.25,1.25) -- (0,0) ;
\draw (0,-1.25) -- (0,0) ;
\end{scope}

\end{tikzpicture}
}

\newcommand{\arbreoptdCD}{
\begin{tikzpicture}
\begin{scope}
\draw (-1.25,1.25) -- (0,0) ;
\draw (1.25,1.25) -- (0,0) ;
\draw (0,-1.25) -- (0,0) ;
\end{scope}

\begin{scope}[shift = {(-1.25,3)}]
\draw (0,1.25) -- (0,0) ;
\draw (-1.25,1.25) -- (0,0) ;
\draw (1.25,1.25) -- (0,0) ;
\draw (0,-1.25) -- (0,0) ;
\end{scope}
\end{tikzpicture}
}

\newcommand{\arbreopdddC}{
\begin{tikzpicture}

\begin{scope}
\draw (-1.25,1.25) -- (0,0) ;
\draw (1.25,1.25) -- (0,0) ;
\draw (0,-1.25) -- (0,0) ;
\end{scope}

\begin{scope}[shift = {(-1.25,3)}]
\draw (-1.25,1.25) -- (0,0) ;
\draw (1.25,1.25) -- (0,0) ;
\draw (0,-1.25) -- (0,0) ;
\end{scope}

\begin{scope}[shift = {(-2.5,6)}]
\draw (-1.25,1.25) -- (0,0) ;
\draw (1.25,1.25) -- (0,0) ;
\draw (0,-1.25) -- (0,0) ;
\end{scope}
\end{tikzpicture}
}

\newcommand{\arbreopdddB}{
\begin{tikzpicture}

\begin{scope}
\draw (-1.25,1.25) -- (0,0) ;
\draw (1.25,1.25) -- (0,0) ;
\draw (0,-1.25) -- (0,0) ;
\end{scope}

\begin{scope}[shift = {(-1.4,3)}]
\draw (-1.25,1.25) -- (0,0) ;
\draw (1.25,1.25) -- (0,0) ;
\draw (0,-1.25) -- (0,0) ;
\end{scope}

\begin{scope}[shift = {(1.4,3)}]
\draw (-1.25,1.25) -- (0,0) ;
\draw (1.25,1.25) -- (0,0) ;
\draw (0,-1.25) -- (0,0) ;
\end{scope}
\end{tikzpicture}
}

\newcommand{\arbreopdddA}{
\begin{tikzpicture}

\begin{scope}
\draw (-1.25,1.25) -- (0,0) ;
\draw (1.25,1.25) -- (0,0) ;
\draw (0,-1.25) -- (0,0) ;
\end{scope}

\begin{scope}[shift = {(1.25,3)}]
\draw (-1.25,1.25) -- (0,0) ;
\draw (1.25,1.25) -- (0,0) ;
\draw (0,-1.25) -- (0,0) ;
\end{scope}

\begin{scope}[shift = {(2.5,6)}]
\draw (-1.25,1.25) -- (0,0) ;
\draw (1.25,1.25) -- (0,0) ;
\draw (0,-1.25) -- (0,0) ;
\end{scope}
\end{tikzpicture}
}

\newcommand{\arbreopdddE}{
\begin{tikzpicture}

\begin{scope}
\draw (-1.25,1.25) -- (0,0) ;
\draw (1.25,1.25) -- (0,0) ;
\draw (0,-1.25) -- (0,0) ;
\end{scope}

\begin{scope}[shift = {(1.25,3)}]
\draw (-1.25,1.25) -- (0,0) ;
\draw (1.25,1.25) -- (0,0) ;
\draw (0,-1.25) -- (0,0) ;
\end{scope}

\begin{scope}[shift = {(0,6)}]
\draw (-1.25,1.25) -- (0,0) ;
\draw (1.25,1.25) -- (0,0) ;
\draw (0,-1.25) -- (0,0) ;
\end{scope}
\end{tikzpicture}
}

\newcommand{\arbreopdddD}{
\begin{tikzpicture}

\begin{scope}
\draw (-1.25,1.25) -- (0,0) ;
\draw (1.25,1.25) -- (0,0) ;
\draw (0,-1.25) -- (0,0) ;
\end{scope}

\begin{scope}[shift = {(-1.25,3)}]
\draw (-1.25,1.25) -- (0,0) ;
\draw (1.25,1.25) -- (0,0) ;
\draw (0,-1.25) -- (0,0) ;
\end{scope}

\begin{scope}[shift = {(0,6)}]
\draw (-1.25,1.25) -- (0,0) ;
\draw (1.25,1.25) -- (0,0) ;
\draw (0,-1.25) -- (0,0) ;
\end{scope}
\end{tikzpicture}
}

\newcommand{\associaedretrois}{
\begin{tikzpicture}[scale = 0.5 , baseline = (pt_base.base)]

\coordinate (pt_base) at (-1,0) ; 

\node[scale = 0.1] (A) at (0,3) {} ;
\node[scale=0.1,above = 5pt] at (A) {\arbreopdddA} ;
\node[scale = 0.1] (B) at (-2,0) {} ;
\node[scale = 0.1,left = 5 pt] at (B) {\arbreopdddB} ;
\node[scale = 0.1] (C) at (0,-3) {} ;
\node[scale = 0.1 , below left = 5 pt] at (C) {\arbreopdddC} ;
\node[scale = 0.1] (D) at (4,-1.5) {} ;
\node[scale = 0.1,below right = 7 pt] at (D) {\arbreopdddD} ;
\node[scale = 0.1] (E) at (4,1.5) {} ;
\node[scale = 0.1,above right = 7 pt] at (E) {\arbreopdddE} ;

\node at (1,0) {\arbreopquatre} ;

\draw[fill,opacity = 0.1] (A.center) -- (B.center) -- (C.center) -- (D.center) -- (E.center) -- (A.center) ;

\draw (A.center) -- (B.center) -- (C.center) -- (D.center) -- (E.center) -- (A.center) ;

\node[scale=0.1] at (A) {\pointbullet} ;
\node[scale=0.1] at (B) {\pointbullet} ;
\node[scale=0.1] at (C) {\pointbullet} ;
\node[scale=0.1] at (D) {\pointbullet} ;
\node[scale=0.1] at (E) {\pointbullet} ;

\node[scale = 0.1,above left = 3 pt] at ($(A)!0.5!(B)$) {\arbreoptdAB} ;
\node[scale = 0.1,below left = 3 pt] at ($(B)!0.5!(C)$) {\arbreoptdBC} ;
\node[scale = 0.1,below right = 6 pt] at ($(C)!0.5!(D)$) {\arbreoptdCD} ;
\node[scale = 0.1,right = 2 pt] at ($(D)!0.5!(E)$) {\arbreoptdDE} ;
\node[scale = 0.1,above right = 6 pt] at ($(E)!0.5!(A)$) {\arbreoptdEA} ;

\end{tikzpicture}
}


\newcommand{\multiplaedreun}{
\begin{tikzpicture}[ baseline = (pt_base.base)]
\coordinate (pt_base) at (0,0) ;
\node[scale = 0.1] at (0,0) {\pointbullet} ;
\node[above left] at (0,0) {\arbreopunmorph} ;
\end{tikzpicture}
}


\newcommand{\arbreopdeuxunmorph}{
\begin{tikzpicture}

\begin{scope}[red]
\draw (-1.25,1.25) -- (0,0) ;
\draw (1.25,1.25) -- (0,0) ;
\draw (0,-1.25) -- (0,0) ;
\end{scope}

\begin{scope}[shift = {(-1.25,3)}]

\draw[blue] (0,1.25) -- (0,0.5) ;
\draw[red] (0,-1.25) -- (0,0.5) ;

\draw (-1.25,0.5) -- (0.75,0.5) ;

\end{scope}

\begin{scope}[shift = {(1.25,3)}]

\draw[blue] (0,1.25) -- (0,0.5) ;
\draw[red] (0,-1.25) -- (0,0.5) ;

\draw (-0.75,0.5) -- (1.25,0.5) ;

\end{scope}

\end{tikzpicture}
}

\newcommand{\arbreopdeuxdeuxmorph}{
\begin{tikzpicture}

\begin{scope}
\draw[blue] (0,1.25) -- (0,0.5) ;
\draw[red] (0,-1.25) -- (0,0.5) ;

\draw (-1.25,0.5) -- (1.25,0.5) ;
\end{scope}

\begin{scope}[blue, shift = {(0,3)}]
\draw (-1.25,1.25) -- (0,0) ;
\draw (1.25,1.25) -- (0,0) ;
\draw (0,-1.25) -- (0,0) ;
\end{scope}

\end{tikzpicture}
}

\newcommand{\multiplaedredeux}{
\begin{tikzpicture}[ baseline = (pt_base.base)]
\coordinate (pt_base) at (0,0) ;
\node[scale = 0.1] (A) at (-1,0) {\pointbullet} ;
\node[scale = 0.1] (B) at (1,0) {\pointbullet} ;
\node[above left = 2 pt,scale = 0.1] at (A) {\arbreopdeuxunmorph} ;
\node[above right = 2 pt,scale = 0.1] at (B) {\arbreopdeuxdeuxmorph} ;
\draw (A.center) -- (B.center) ;
\node[above] at (0,0) {\arbreopdeuxmorph} ;
\end{tikzpicture}
}


\newcommand{\multiplaedretroisA}{
\begin{tikzpicture}

\begin{scope}
\draw[blue] (0,1.25) -- (0,0.5) ;
\draw[red] (0,-1.25) -- (0,0.5) ;

\draw (-1.25,0.5) -- (1.25,0.5) ;
\end{scope}

\begin{scope}[shift = {(0,3)},blue]
\draw (-1.25,1.25) -- (0,0) ;
\draw (1.25,1.25) -- (0,0) ;
\draw (0,-1.25) -- (0,0) ;
\end{scope}

\begin{scope}[shift = {(1.25,6)},blue]
\draw (-1.25,1.25) -- (0,0) ;
\draw (1.25,1.25) -- (0,0) ;
\draw (0,-1.25) -- (0,0) ;
\end{scope}
\end{tikzpicture}
}

\newcommand{\multiplaedretroisB}{
\begin{tikzpicture}

\begin{scope}[red]
\draw (-1.25,1.25) -- (0,0) ;
\draw (1.25,1.25) -- (0,0) ;
\draw (0,-1.25) -- (0,0) ;
\end{scope}

\begin{scope}[shift = {(-1.25,3)}]
\draw[blue] (0,1.25) -- (0,0.5) ;
\draw[red] (0,-1.25) -- (0,0.5) ;

\draw (-1.25,0.5) -- (0.75,0.5) ;
\end{scope}

\begin{scope}[shift = {(1.25,3)}]
\draw[blue] (0,1.25) -- (0,0.5) ;
\draw[red] (0,-1.25) -- (0,0.5) ;

\draw (-0.75,0.5) -- (1.25,0.5) ;
\end{scope}

\begin{scope}[shift = {(1.25,6)},blue]
\draw (-1.25,1.25) -- (0,0) ;
\draw (1.25,1.25) -- (0,0) ;
\draw (0,-1.25) -- (0,0) ;
\end{scope}
\end{tikzpicture}
}

\newcommand{\multiplaedretroisC}{
\begin{tikzpicture}

\begin{scope}[red]
\draw (-1.25,1.25) -- (0,0) ;
\draw (1.25,1.25) -- (0,0) ;
\draw (0,-1.25) -- (0,0) ;
\end{scope}

\begin{scope}[shift = {(1.25,3)},red]
\draw (-1.25,1.25) -- (0,0) ;
\draw (1.25,1.25) -- (0,0) ;
\draw (0,-1.25) -- (0,0) ;
\end{scope}

\begin{scope}[shift = {(2.5,6)}]
\draw[blue] (0,1.25) -- (0,0.5) ;
\draw[red] (0,-1.25) -- (0,0.5) ;

\draw (-0.75,0.5) -- (1.25,0.5) ;
\end{scope}

\begin{scope}[shift = {(0,6)}]
\draw[blue] (0,1.25) -- (0,0.5) ;
\draw[red] (0,-1.25) -- (0,0.5) ;

\draw (-1.25,0.5) -- (0.75,0.5) ;
\end{scope}

\begin{scope}[shift = {(-1.25,3)}]
\draw[blue] (0,1.25) -- (0,0.5) ;
\draw[red] (0,-1.25) -- (0,0.5) ;

\draw (-1.25,0.5) -- (1.25,0.5) ;
\end{scope}

\end{tikzpicture}
}

\newcommand{\multiplaedretroisD}{
\begin{tikzpicture}

\begin{scope}[red]
\draw (-1.25,1.25) -- (0,0) ;
\draw (1.25,1.25) -- (0,0) ;
\draw (0,-1.25) -- (0,0) ;
\end{scope}

\begin{scope}[shift = {(-1.25,3)},red]
\draw (-1.25,1.25) -- (0,0) ;
\draw (1.25,1.25) -- (0,0) ;
\draw (0,-1.25) -- (0,0) ;
\end{scope}

\begin{scope}[shift = {(-2.5,6)}]
\draw[blue] (0,1.25) -- (0,0.5) ;
\draw[red] (0,-1.25) -- (0,0.5) ;

\draw (-1.25,0.5) -- (0.75,0.5) ;
\end{scope}

\begin{scope}[shift = {(0,6)}]
\draw[blue] (0,1.25) -- (0,0.5) ;
\draw[red] (0,-1.25) -- (0,0.5) ;

\draw (-0.75,0.5) -- (1.25,0.5) ;
\end{scope}

\begin{scope}[shift = {(1.25,3)}]
\draw[blue] (0,1.25) -- (0,0.5) ;
\draw[red] (0,-1.25) -- (0,0.5) ;

\draw (-1.25,0.5) -- (1.25,0.5) ;
\end{scope}

\end{tikzpicture}
}

\newcommand{\multiplaedretroisE}{
\begin{tikzpicture}

\begin{scope}[red]
\draw (-1.25,1.25) -- (0,0) ;
\draw (1.25,1.25) -- (0,0) ;
\draw (0,-1.25) -- (0,0) ;
\end{scope}

\begin{scope}[shift = {(-1.25,3)}]
\draw[blue] (0,1.25) -- (0,0.5) ;
\draw[red] (0,-1.25) -- (0,0.5) ;

\draw (-1.25,0.5) -- (0.75,0.5) ;
\end{scope}

\begin{scope}[shift = {(1.25,3)}]
\draw[blue] (0,1.25) -- (0,0.5) ;
\draw[red] (0,-1.25) -- (0,0.5) ;

\draw (-0.75,0.5) -- (1.25,0.5) ;
\end{scope}

\begin{scope}[shift = {(-1.25,6)},blue]
\draw (-1.25,1.25) -- (0,0) ;
\draw (1.25,1.25) -- (0,0) ;
\draw (0,-1.25) -- (0,0) ;
\end{scope}
\end{tikzpicture}
}

\newcommand{\multiplaedretroisF}{
\begin{tikzpicture}

\begin{scope}
\draw[blue] (0,1.25) -- (0,0.5) ;
\draw[red] (0,-1.25) -- (0,0.5) ;

\draw (-1.25,0.5) -- (1.25,0.5) ;
\end{scope}

\begin{scope}[shift = {(0,3)},blue]
\draw (-1.25,1.25) -- (0,0) ;
\draw (1.25,1.25) -- (0,0) ;
\draw (0,-1.25) -- (0,0) ;
\end{scope}

\begin{scope}[shift = {(-1.25,6)},blue]
\draw (-1.25,1.25) -- (0,0) ;
\draw (1.25,1.25) -- (0,0) ;
\draw (0,-1.25) -- (0,0) ;
\end{scope}
\end{tikzpicture}
}

\newcommand{\multiplaedretroisAB}{
\begin{tikzpicture}

\begin{scope}
\draw[blue] (-1.25,1.25) -- (-0.5,0.5) ;
\draw[blue] (1.25,1.25) -- (0.5,0.5) ;

\draw[red] (-0.5,0.5) -- (0,0) ;
\draw[red] (0.5,0.5) -- (0,0) ;
\draw[red] (0,-1.25) -- (0,0) ;

\draw (-1.25,0.5) -- (1.25,0.5) ;
\end{scope}

\begin{scope}[shift = {(1.25,3)},blue]
\draw (-1.25,1.25) -- (0,0) ;
\draw (1.25,1.25) -- (0,0) ;
\draw (0,-1.25) -- (0,0) ;
\end{scope}
\end{tikzpicture}
}

\newcommand{\multiplaedretroisBC}{
\begin{tikzpicture}

\begin{scope}[shift = {(1.25,3)}]
\draw[blue] (-1.25,1.25) -- (-0.5,0.5) ;
\draw[blue] (1.25,1.25) -- (0.5,0.5) ;

\draw[red] (-0.5,0.5) -- (0,0) ;
\draw[red] (0.5,0.5) -- (0,0) ;
\draw[red] (0,-1.25) -- (0,0) ;

\draw (-1.25,0.5) -- (1.25,0.5) ;
\end{scope}

\begin{scope}[red]
\draw (-1.25,1.25) -- (0,0) ;
\draw (1.25,1.25) -- (0,0) ;
\draw (0,-1.25) -- (0,0) ;
\end{scope}

\begin{scope}[shift = {(-1.25,3)}]
\draw[blue] (0,1.25) -- (0,0.5) ;
\draw[red] (0,-1.25) -- (0,0.5) ;

\draw (-1.25,0.5) -- (0.75,0.5) ;
\end{scope}
\end{tikzpicture}
}

\newcommand{\multiplaedretroisDE}{
\begin{tikzpicture}

\begin{scope}[shift = {(-1.25,3)}]
\draw[blue] (-1.25,1.25) -- (-0.5,0.5) ;
\draw[blue] (1.25,1.25) -- (0.5,0.5) ;

\draw[red] (-0.5,0.5) -- (0,0) ;
\draw[red] (0.5,0.5) -- (0,0) ;
\draw[red] (0,-1.25) -- (0,0) ;

\draw (-1.25,0.5) -- (1.25,0.5) ;
\end{scope}

\begin{scope}[red]
\draw (-1.25,1.25) -- (0,0) ;
\draw (1.25,1.25) -- (0,0) ;
\draw (0,-1.25) -- (0,0) ;
\end{scope}

\begin{scope}[shift = {(1.25,3)}]
\draw[blue] (0,1.25) -- (0,0.5) ;
\draw[red] (0,-1.25) -- (0,0.5) ;

\draw (-0.75,0.5) -- (1.25,0.5) ;
\end{scope}
\end{tikzpicture}
}

\newcommand{\multiplaedretroisEF}{
\begin{tikzpicture}

\begin{scope}
\draw[blue] (-1.25,1.25) -- (-0.5,0.5) ;
\draw[blue] (1.25,1.25) -- (0.5,0.5) ;

\draw[red] (-0.5,0.5) -- (0,0) ;
\draw[red] (0.5,0.5) -- (0,0) ;
\draw[red] (0,-1.25) -- (0,0) ;

\draw (-1.25,0.5) -- (1.25,0.5) ;
\end{scope}

\begin{scope}[shift = {(-1.25,3)},blue]
\draw (-1.25,1.25) -- (0,0) ;
\draw (1.25,1.25) -- (0,0) ;
\draw (0,-1.25) -- (0,0) ;
\end{scope}
\end{tikzpicture}
}

\newcommand{\multiplaedretroisCD}{
\begin{tikzpicture}

\begin{scope}[shift = {(-1.25,3)}]
\draw[blue] (0,1.25) -- (0,0.5) ;
\draw[red] (0,-1.25) -- (0,0.5) ;

\draw (-1.25,0.5) -- (0.5,0.5) ;
\end{scope}

\begin{scope}[shift = {(0,3)}]
\draw[blue] (0,1.25) -- (0,0.5) ;
\draw[red] (0,-1.25) -- (0,0.5) ;

\draw (-0.5,0.5) -- (0.5,0.5) ;
\end{scope}

\begin{scope}[shift = {(1.25,3)}]
\draw[blue] (0,1.25) -- (0,0.5) ;
\draw[red] (0,-1.25) -- (0,0.5) ;

\draw (-0.5,0.5) -- (1.25,0.5) ;
\end{scope}

\begin{scope}[red]
\draw (0,1.25) -- (0,0) ;
\draw (-1.25,1.25) -- (0,0) ;
\draw (1.25,1.25) -- (0,0) ;
\draw (0,-1.25) -- (0,0) ;
\end{scope}

\end{tikzpicture}
}

\newcommand{\multiplaedretroisFA}{
\begin{tikzpicture}

\begin{scope}
\draw[blue] (0,1.25) -- (0,0.5) ;
\draw[red] (0,-1.25) -- (0,0.5) ;

\draw (-1.25,0.5) -- (1.25,0.5) ;
\end{scope}

\begin{scope}[blue, shift = {(0,3)}]
\draw (0,1.25) -- (0,0) ;
\draw (-1.25,1.25) -- (0,0) ;
\draw (1.25,1.25) -- (0,0) ;
\draw (0,-1.25) -- (0,0) ;
\end{scope}

\end{tikzpicture}
}

\newcommand{\multiplaedretrois}{
\begin{tikzpicture}[scale = 1 , baseline = (pt_base.base)]

\coordinate (pt_base) at (-1,0) ; 

\node[scale = 0.1] (A) at (0.85,-1.4) {} ;
\node[scale=0.1,below = 5pt] at (A) {\multiplaedretroisA} ;
\node[scale = 0.1] (B) at (1.65,0) {} ;
\node[scale = 0.1,right = 5 pt] at (B) {\multiplaedretroisB} ;
\node[scale = 0.1] (C) at (0.85,1.35) {} ;
\node[scale = 0.1 , above = 5 pt] at (C) {\multiplaedretroisC} ;
\node[scale = 0.1] (D) at (-0.75,1.35) {} ;
\node[scale = 0.1,above = 5 pt] at (D) {\multiplaedretroisD} ;
\node[scale = 0.1] (E) at (-1.5,0) {} ;
\node[scale = 0.1,left = 5 pt] at (E) {\multiplaedretroisE} ;
\node[scale = 0.1] (F) at (-0.75,-1.4) {} ;
\node[scale = 0.1,below = 5 pt] at (F) {\multiplaedretroisF} ;

\node at (0,0) {\arbreoptroismorph} ;

\draw[fill,opacity = 0.1] (A.center) -- (B.center) -- (C.center) -- (D.center) -- (E.center) -- (F.center) -- (A.center) ;

\draw (A.center) -- (B.center) -- (C.center) -- (D.center) -- (E.center) -- (F.center) -- (A.center) ;

\node[scale=0.1] at (A) {\pointbullet} ;
\node[scale=0.1] at (B) {\pointbullet} ;
\node[scale=0.1] at (C) {\pointbullet} ;
\node[scale=0.1] at (D) {\pointbullet} ;
\node[scale=0.1] at (E) {\pointbullet} ;
\node[scale=0.1] at (F) {\pointbullet} ;

\node[scale = 0.1,below right = 5 pt] at ($(A)!0.5!(B)$) {\multiplaedretroisAB} ;
\node[scale = 0.1,above right = 5 pt] at ($(B)!0.5!(C)$) {\multiplaedretroisBC} ;
\node[scale = 0.1,above = 5 pt] at ($(C)!0.5!(D)$) {\multiplaedretroisCD} ;
\node[scale = 0.1,above left = 5 pt] at ($(D)!0.5!(E)$) {\multiplaedretroisDE} ;
\node[scale = 0.1,below left = 5 pt] at ($(E)!0.5!(F)$) {\multiplaedretroisEF} ;
\node[scale = 0.1,below = 5 pt] at ($(F)!0.5!(A)$) {\multiplaedretroisFA} ;

\end{tikzpicture}
}


\newcommand{\exampleribbontree}{
\begin{tikzpicture}[scale = 0.4]
\node (A) at (-1,1) {} ;
\node (B) at (-1,0) {} ;
\node (C) at (0,-1) {} ;
\node (D) at (1,0) {} ;
\node (E) at (0,1) {} ;
\node (F) at (2,1) {} ;
\node (G) at (0,-2) {} ;

\draw (A.center) -- (B.center) -- (C.center) -- (D.center) -- (E.center) ;
\draw (D.center) -- (F.center) ;
\draw (C.center) -- (G.center) ;

\end{tikzpicture}}

\newcommand{\examplemetricribbontree}{
\begin{tikzpicture}[scale = 0.4]
\node (A) at (-1,1) {} ;
\node (B) at (-1,0) {} ;
\node (C) at (0,-1) {} ;
\node (D) at (1,0) {} ;
\node (E) at (0,1) {} ;
\node (F) at (2,1) {} ;
\node (G) at (0,-2) {} ;

\draw (A.center) -- (B.center) -- (C.center) -- (D.center) -- (E.center) ;
\draw (D.center) -- (F.center) ;
\draw (C.center) -- (G.center) ;

\node[below left = 0.2 pt] at ($(B)!0.5!(C)$) {\tiny $l_1$} ;
\node[below right = 0.2 pt] at ($(C)!0.5!(D)$) {\tiny $l_2$} ;

\end{tikzpicture}}

\newcommand{\examplestablemetricribbontree}{
\begin{tikzpicture}[scale = 0.4]
\node (A) at (-1.5,1) {} ;
\node (B) at (0,0) {} ;
\node (C) at (1.5,1) {} ;
\node (D) at (0,-1) {} ;
\node (E) at (0.25,2) {} ;
\node (F) at (2.75,2) {} ;
\node (G) at (-0.25,2) {} ;
\node (H) at (-1.5,2) {} ;
\node (I) at (-2.75,2) {} ;

\draw (I.center) -- (A.center) -- (B.center) -- (C.center) -- (E.center) ;
\draw (F.center) -- (C.center) ;
\draw (H.center) -- (A.center) ;
\draw (G.center) -- (A.center) ;
\draw (B.center) -- (D.center) ;

\node[below left = 0.2 pt] at ($(A)!0.5!(B)$) {\tiny $l_1$} ;
\node[below right = 0.2 pt] at ($(C)!0.5!(B)$) {\tiny $l_2$} ;

\end{tikzpicture}
}

\newcommand{\examplebinarymetricribbontree}{
\begin{tikzpicture}[scale = 0.4]
\node (A) at (-1.5,1) {} ;
\node (B) at (0,0) {} ;
\node (C) at (1.5,1) {} ;
\node (D) at (0,-1) {} ;
\node (E) at (0.25,2) {} ;
\node (F) at (2.75,2) {} ;
\node (G) at (-0.25,2) {} ;
\node (I) at (-2.75,2) {} ;

\draw (I.center) -- (A.center) -- (B.center) -- (C.center) -- (E.center) ;
\draw (F.center) -- (C.center) ;
\draw (G.center) -- (A.center) ;
\draw (B.center) -- (D.center) ;

\node[below left = 0.2 pt] at ($(A)!0.5!(B)$) {\tiny $l_1$} ;
\node[below right = 0.2 pt] at ($(C)!0.5!(B)$) {\tiny $l_2$} ;

\end{tikzpicture}
}


\newcommand{\segm}{
\begin{tikzpicture}[very thick]
\draw (0,-1) -- (0,1) ;
\draw (-0.5,1) -- (0.5,1) ;
\draw (-0.5,-1) -- (0.5,-1) ;
\end{tikzpicture}}

\newcommand{\TtroisstratA}{
\begin{tikzpicture}
\node (K) at (-0.625,0.625) {} ;
\node (L) at (0,0) {} ;
\draw (-0.625,0.625) -- (0,1.25) ;
\draw (-1.25,1.25) -- (0,0) ;
\draw (1.25,1.25) -- (0,0) ;
\draw (0,-1.25) -- (0,0) ;

\node[below left = 0.2 pt] at ($(K)!0.5!(L)$) {\Huge $l$} ;
\end{tikzpicture}}

\newcommand{\TtroisstratB}{
\begin{tikzpicture}
\node (K) at (0.625,0.625) {} ;
\node (L) at (0,0) {} ;
\draw (0.625,0.625) -- (0,1.25) ;
\draw (-1.25,1.25) -- (0,0) ;
\draw (1.25,1.25) -- (0,0) ;
\draw (0,-1.25) -- (0,0) ;

\node[below right = 0.2 pt] at ($(K)!0.5!(L)$) {\Huge $l$} ;
\end{tikzpicture}}

\newcommand{\TtroisstratO}{
\begin{tikzpicture}
\draw (0,0) -- (0,1.25) ;
\draw (-1.25,1.25) -- (0,0) ;
\draw (1.25,1.25) -- (0,0) ;
\draw (0,-1.25) -- (0,0) ;
\end{tikzpicture}}

\newcommand{\Ttroisstrat}{
\begin{tikzpicture}[ scale = 1.75 , baseline = (pt_base.base)]
\coordinate (pt_base) at (0,0) ;
\node[scale = 0.1] (A) at (-1,0) {\pointbullet} ;
\node[scale = 0.1] (B) at (1,0) {\pointbullet} ;
\node[scale = 0.1] (O) at (0,0) {\pointbullet} ;
\draw (A.center) -- (B.center) ;

\node[below = 2 pt,scale = 0.4] at ($(A)!0.5!(O)$) {\TtroisstratA} ;
\node[below = 2 pt,scale = 0.4] at ($(O)!0.5!(B)$) {\TtroisstratB} ;
\node[above = 3 pt,scale = 0.25] at (O) {\TtroisstratO} ;
\node[left = 4 pt,scale = 0.175] at (A) {\arbreopdeuxun} ; 
\node[right = 3 pt,scale = 0.175] at (B) {\arbreopdeuxdeux} ; 
\end{tikzpicture}}


\newcommand{\TquatrestratC}{
\begin{tikzpicture}
\node (A) at (0,0) {} ;
\node (B) at (0,-1) {} ;
\node (C) at (-2,2) {} ;
\node (D) at (-1,1) {} ;
\node (E) at (0,2) {} ;
\node (F) at (-1.5,1.5) {} ;
\node (G) at (-1,2) {} ;
\node (H) at (2,2) {} ;

\draw (C.center) -- (F.center) -- (D.center) -- (A.center) -- (B.center) ;
\draw (A.center) -- (H.center) ;
\draw (D.center) -- (E.center) ;
\draw (F.center) -- (G.center) ;

\node[below left = 0.2 pt] at ($(D)!0.5!(A)$) {\Huge $l_1$} ;
\node[below left = 0.2 pt] at ($(D)!0.5!(F)$) {\Huge $l_2$} ;

\end{tikzpicture}
}

\newcommand{\TquatrestratA}{
\begin{tikzpicture}
\node (A) at (0,0) {} ;
\node (B) at (0,-1) {} ;
\node (C) at (-2,2) {} ;
\node (D) at (1,1) {} ;
\node (E) at (0,2) {} ;
\node (F) at (1.5,1.5) {} ;
\node (G) at (1,2) {} ;
\node (H) at (2,2) {} ;

\draw (H.center) -- (F.center) -- (D.center) -- (A.center) -- (B.center) ;
\draw (A.center) -- (C.center) ;
\draw (D.center) -- (E.center) ;
\draw (F.center) -- (G.center) ;

\node[below right = 0.2 pt] at ($(D)!0.5!(A)$) {\Huge $l_1$} ;
\node[below right = 0.2 pt] at ($(D)!0.5!(F)$) {\Huge $l_2$} ;

\end{tikzpicture}
}

\newcommand{\TquatrestratB}{
\begin{tikzpicture}
\node (A) at (-1.5,1) {} ;
\node (B) at (0,0) {} ;
\node (C) at (1.5,1) {} ;
\node (D) at (0,-1) {} ;
\node (E) at (0.25,2) {} ;
\node (F) at (2.75,2) {} ;
\node (G) at (-0.25,2) {} ;
\node (I) at (-2.75,2) {} ;

\draw (I.center) -- (A.center) -- (B.center) -- (C.center) -- (E.center) ;
\draw (F.center) -- (C.center) ;
\draw (G.center) -- (A.center) ;
\draw (B.center) -- (D.center) ;

\node[below left = 0.2 pt] at ($(A)!0.5!(B)$) {\Huge $l_1$} ;
\node[below right = 0.2 pt] at ($(C)!0.5!(B)$) {\Huge $l_2$} ;
\end{tikzpicture}
}

\newcommand{\TquatrestratE}{
\begin{tikzpicture}
\node (A) at (0,0) {} ;
\node (B) at (2,2) {} ;
\node (H) at (-2,2) {} ;
\node (C) at (0,2) {} ;
\node (E) at (0.5,1.5) {} ;
\node (F) at (1,2) {} ;
\node (G) at (0,-1) {} ;
\node (D) at (1,1) {} ;

\draw (C.center) -- (E.center) -- (D.center) -- (A.center) -- (G.center) ;
\draw (A.center) -- (H.center) ;
\draw (E.center) -- (F.center) ;
\draw (B.center) -- (D.center) ;

\node[below right = 0.2 pt] at ($(A)!0.5!(D)$) {\Huge $l_1$} ;
\node[ xshift = -15 pt,yshift = -7 pt] at ($(E)!0.5!(D)$) {\Huge $l_2$} ;
\end{tikzpicture}
}

\newcommand{\TquatrestratD}{
\begin{tikzpicture}
\node (A) at (0,0) {} ;
\node (B) at (0,-1) {} ;
\node (C) at (-2,2) {} ;
\node (D) at (-1,1) {} ;
\node (E) at (0,2) {} ;
\node (F) at (-0.5,1.5) {} ;
\node (G) at (-1,2) {} ;
\node (H) at (2,2) {} ;

\draw (E.center) -- (F.center) -- (D.center) -- (A.center) -- (B.center) ;
\draw (A.center) -- (H.center) ;
\draw (G.center) -- (F.center) ;
\draw (C.center) -- (D.center) ;

\node[below left = 0.2 pt] at ($(A)!0.5!(D)$) {\Huge $l_1$} ;
\node[ xshift = 15 pt,yshift = -9 pt] at ($(D)!0.5!(F)$) {\Huge $l_2$} ;
\end{tikzpicture}
}

\newcommand{\Tquatrestrat}{
\begin{tikzpicture}[scale = 0.8, baseline = (pt_base.base)]

\coordinate (pt_base) at (-1,0) ; 

\node[scale = 0.25] (Ap) at (0.6,1.6) {\TquatrestratA} ;
\node[scale = 0.25] (Bp) at (-0.7,-0.2) {\TquatrestratB} ;
\node[scale = 0.25] (Cp) at (0.5,-1.7) {\TquatrestratC} ;
\node[scale = 0.25] (Dp) at (2.7,-0.9) {\TquatrestratD} ;
\node[scale = 0.25] (Ep) at (2.8,0.9) {\TquatrestratE} ;

\node[scale = 0.1] (A) at (0,3) {} ;
\node[scale = 0.1] (B) at (-2,0) {} ;
\node[scale = 0.1] (C) at (0,-3) {} ;
\node[scale = 0.1] (D) at (4,-1.5) {} ;
\node[scale = 0.1] (E) at (4,1.5) {} ;
\node (O) at (1,0) {} ;

\draw[fill,opacity = 0.1] (A.center) -- (B.center) -- (C.center) -- (D.center) -- (E.center) -- (A.center) ;

\draw (A.center) -- (B.center) -- (C.center) -- (D.center) -- (E.center) -- (A.center) ;

\draw ($(A)!0.5!(B)$) -- (O.center) ;
\draw ($(B)!0.5!(C)$) -- (O.center) ;
\draw ($(C)!0.5!(D)$) -- (O.center) ;
\draw ($(D)!0.5!(E)$) -- (O.center) ;
\draw ($(E)!0.5!(A)$) -- (O.center) ;

\node[scale=0.1] at (A) {\pointbullet} ;
\node[scale=0.1] at (B) {\pointbullet} ;
\node[scale=0.1] at (C) {\pointbullet} ;
\node[scale=0.1] at (D) {\pointbullet} ;
\node[scale=0.1] at (E) {\pointbullet} ;

\end{tikzpicture}}

\newcommand{\TquatrestratCbis}{
\begin{tikzpicture}
\node (A) at (0,0) {} ;
\node (B) at (0,-1) {} ;
\node (C) at (-2,2) {} ;
\node (D) at (-1,1) {} ;
\node (E) at (0,2) {} ;
\node (F) at (-1.5,1.5) {} ;
\node (G) at (-1,2) {} ;
\node (H) at (2,2) {} ;

\draw (C.center) -- (F.center) -- (D.center) -- (A.center) -- (B.center) ;
\draw (A.center) -- (H.center) ;
\draw (D.center) -- (E.center) ;
\draw (F.center) -- (G.center) ;
\end{tikzpicture}
}

\newcommand{\TquatrestratAbis}{
\begin{tikzpicture}
\node (A) at (0,0) {} ;
\node (B) at (0,-1) {} ;
\node (C) at (-2,2) {} ;
\node (D) at (1,1) {} ;
\node (E) at (0,2) {} ;
\node (F) at (1.5,1.5) {} ;
\node (G) at (1,2) {} ;
\node (H) at (2,2) {} ;

\draw (H.center) -- (F.center) -- (D.center) -- (A.center) -- (B.center) ;
\draw (A.center) -- (C.center) ;
\draw (D.center) -- (E.center) ;
\draw (F.center) -- (G.center) ;
\end{tikzpicture}
}

\newcommand{\TquatrestratBbis}{
\begin{tikzpicture}
\node (A) at (-1.5,1) {} ;
\node (B) at (0,0) {} ;
\node (C) at (1.5,1) {} ;
\node (D) at (0,-1) {} ;
\node (E) at (0.25,2) {} ;
\node (F) at (2.75,2) {} ;
\node (G) at (-0.25,2) {} ;
\node (I) at (-2.75,2) {} ;

\draw (I.center) -- (A.center) -- (B.center) -- (C.center) -- (E.center) ;
\draw (F.center) -- (C.center) ;
\draw (G.center) -- (A.center) ;
\draw (B.center) -- (D.center) ;
\end{tikzpicture}
}

\newcommand{\TquatrestratEbis}{
\begin{tikzpicture}
\node (A) at (0,0) {} ;
\node (B) at (2,2) {} ;
\node (H) at (-2,2) {} ;
\node (C) at (0,2) {} ;
\node (E) at (0.5,1.5) {} ;
\node (F) at (1,2) {} ;
\node (G) at (0,-1) {} ;
\node (D) at (1,1) {} ;

\draw (C.center) -- (E.center) -- (D.center) -- (A.center) -- (G.center) ;
\draw (A.center) -- (H.center) ;
\draw (E.center) -- (F.center) ;
\draw (B.center) -- (D.center) ;
\end{tikzpicture}
}

\newcommand{\TquatrestratDbis}{
\begin{tikzpicture}
\node (A) at (0,0) {} ;
\node (B) at (0,-1) {} ;
\node (C) at (-2,2) {} ;
\node (D) at (-1,1) {} ;
\node (E) at (0,2) {} ;
\node (F) at (-0.5,1.5) {} ;
\node (G) at (-1,2) {} ;
\node (H) at (2,2) {} ;

\draw (E.center) -- (F.center) -- (D.center) -- (A.center) -- (B.center) ;
\draw (A.center) -- (H.center) ;
\draw (G.center) -- (F.center) ;
\draw (C.center) -- (D.center) ;
\end{tikzpicture}
}

\newcommand{\Tquatrestratbis}{
\begin{tikzpicture}[scale = 0.6, baseline = (pt_base.base)]

\coordinate (pt_base) at (-1,0) ; 

\node[scale = 0.2] (Ap) at (0.6,1.6) {\TquatrestratAbis} ;
\node[scale = 0.2] (Bp) at (-0.7,-0.2) {\TquatrestratBbis} ;
\node[scale = 0.2] (Cp) at (0.5,-1.7) {\TquatrestratCbis} ;
\node[scale = 0.2] (Dp) at (2.7,-0.9) {\TquatrestratDbis} ;
\node[scale = 0.2] (Ep) at (2.8,0.9) {\TquatrestratEbis} ;

\node[scale = 0.1] (A) at (0,3) {} ;
\node[scale = 0.1] (B) at (-2,0) {} ;
\node[scale = 0.1] (C) at (0,-3) {} ;
\node[scale = 0.1] (D) at (4,-1.5) {} ;
\node[scale = 0.1] (E) at (4,1.5) {} ;
\node (O) at (1,0) {} ;

\draw[fill,opacity = 0.1] (A.center) -- (B.center) -- (C.center) -- (D.center) -- (E.center) -- (A.center) ;

\draw (A.center) -- (B.center) -- (C.center) -- (D.center) -- (E.center) -- (A.center) ;

\draw ($(A)!0.5!(B)$) -- (O.center) ;
\draw ($(B)!0.5!(C)$) -- (O.center) ;
\draw ($(C)!0.5!(D)$) -- (O.center) ;
\draw ($(D)!0.5!(E)$) -- (O.center) ;
\draw ($(E)!0.5!(A)$) -- (O.center) ;

\node[scale=0.1] at (A) {\pointbullet} ;
\node[scale=0.1] at (B) {\pointbullet} ;
\node[scale=0.1] at (C) {\pointbullet} ;
\node[scale=0.1] at (D) {\pointbullet} ;
\node[scale=0.1] at (E) {\pointbullet} ;

\end{tikzpicture}}


\newcommand{\exampletcribbontreeun}{
\begin{tikzpicture}
\node (A) at (0,0) {} ;
\node (B) at (0,-1) {} ;
\node (C) at (1.25,1.25) {} ;
\node (D) at (-1.25,1.25) {} ;
\node (E) at (0,1.25) {} ;
\node (F) at (-0.625,0.625) {} ;
\node (G) at (-1.25,0.3) {} ;
\node (H) at (1.25,0.3) {} ;
\node (I) at (1.5,0.3) {}  ;
\node (J) at (1.5,0) {} ;

\node (K) at (-0.3,0.3) {} ;
\node (L) at (0.3,0.3) {} ;

\draw[blue] (D.center) -- (F.center) -- (K.center) ;
\draw[red] (K.center) -- (A.center) ;
\draw[blue] (F.center) -- (E.center) ;
\draw[red] (A.center) -- (B.center) ;
\draw[red] (A.center) -- (L.center) ;
\draw[blue] (L.center) -- (C.center) ;
\draw[blue] (F.center) -- (E.center) ;
\draw (G.center) -- (H.center) ;
\draw[densely dotted,>=stealth,<->] (I.center) -- (J.center) ;

\node[right] at ($(I)!0.5!(J)$) {\tiny $\lambda$} ;
\node at ($(A)!0.5!(G)$) {\tiny $l$} ;

\end{tikzpicture}}

\newcommand{\exampletcribbontreedeux}{
\begin{tikzpicture}
\node (A) at (0,0) {} ;
\node (B) at (0,-1) {} ;
\node (C) at (1.25,1.25) {} ;
\node (D) at (-1.25,1.25) {} ;
\node (E) at (0,1.25) {} ;
\node (F) at (-0.625,0.625) {} ;

\node (K) at (-0.625,0.3) {} ;
\node (L) at (1.25,0.3) {} ;

\draw[blue] (D.center) -- (F.center) -- (K.center) ;
\draw[red] (K.center) -- (A.center) ;
\draw[blue] (F.center) -- (E.center) ;
\draw[red] (A.center) -- (B.center) ;
\draw[red] (A.center) -- (L.center) ;
\draw[blue] (L.center) -- (C.center) ;
\draw[blue] (F.center) -- (E.center) ;

\node[left] at ($(F)!0.5!(K)$) {\tiny $l_2$} ;
\node[below left] at ($(K)!0.5!(A)$) {\tiny $l_1$} ;
\node[below right] at ($(A)!0.5!(L)$) {\tiny $l_3$} ;

\end{tikzpicture}}


\newcommand{\twocolbordun}{
\begin{tikzpicture}[scale = 0.7]
\begin{scope}
\node (A) at (0,0) {} ;
\node (B) at (0,-1) {} ;
\node (C) at (1.25,1.25) {} ;
\node (D) at (-1.25,1.25) {} ;
\node (E) at (0,1.25) {} ;
\node (F) at (-0.625,0.625) {} ;

\node (K) at (-0.625,0.3) {} ;
\node (L) at (1.25,0.3) {} ;

\draw[blue] (D.center) -- (F.center) -- (K.center) ;
\draw[red] (K.center) -- (A.center) ;
\draw[blue] (F.center) -- (E.center) ;
\draw[red] (A.center) -- (B.center) ;
\draw[red] (A.center) -- (L.center) ;
\draw[blue] (L.center) -- (C.center) ;
\draw[blue] (F.center) -- (E.center) ;

\node[left] at ($(F)!0.5!(K)$) {\tiny $l_2$} ;
\node[below left] at ($(K)!0.5!(A)$) {\tiny $l_1$} ;
\node[below right] at ($(A)!0.5!(L)$) {\tiny $l_3$} ;
\end{scope}

\begin{scope}[scale = 0.3,xshift = 10 cm]
\draw[->=latex] (-4,0) node{} -- (4,0) node{} node[midway,below = 2pt,text width=2 cm, align = center,scale = 0.7]{\small $l_2 \longrightarrow + \infty$} ;
\end{scope}

\begin{scope}[xshift = 6cm]
\node (A) at (0,0) {} ;
\node (B) at (0,-1) {} ;
\node (C) at (1.25,1.25) {} ;
\node (F) at (-0.625,0.625) {} ;

\node (K) at (-0.625,0.3) {} ;
\node (L) at (1.25,0.3) {} ;

\draw[blue] (K.center) -- (F.center) ;
\draw[red] (K.center) -- (A.center) ;
\draw[red] (A.center) -- (B.center) ;
\draw[red] (A.center) -- (L.center) ;
\draw[blue] (L.center) -- (C.center) ;

\node[below left] at ($(K)!0.5!(A)$) {\tiny $l_1$} ;
\node[below right] at ($(A)!0.5!(L)$) {\tiny $l_3$} ;
\end{scope}

\begin{scope}[xshift = 6 cm , yshift = 0.5 cm]
\node (D) at (-1.25,1.25) {} ;
\node (E) at (0,1.25) {} ;
\node (F) at (-0.625,0.625) {} ;

\node (K) at (-0.625,0.3) {} ;
\draw[blue] (D.center) -- (F.center) -- (K.center) ;
\draw[blue] (F.center) -- (E.center) ;
\end{scope}

\end{tikzpicture}}

\newcommand{\twocolborddeux}{
\begin{tikzpicture}[scale = 0.7]
\begin{scope}
\node (A) at (0,0) {} ;
\node (B) at (0,-1) {} ;
\node (C) at (1.25,1.25) {} ;
\node (D) at (-1.25,1.25) {} ;
\node (E) at (0,1.25) {} ;
\node (F) at (-0.625,0.625) {} ;

\node (K) at (-0.625,0.3) {} ;
\node (L) at (1.25,0.3) {} ;

\draw[blue] (D.center) -- (F.center) -- (K.center) ;
\draw[red] (K.center) -- (A.center) ;
\draw[blue] (F.center) -- (E.center) ;
\draw[red] (A.center) -- (B.center) ;
\draw[red] (A.center) -- (L.center) ;
\draw[blue] (L.center) -- (C.center) ;
\draw[blue] (F.center) -- (E.center) ;

\node[left] at ($(F)!0.5!(K)$) {\tiny $l_2$} ;
\node[below left] at ($(K)!0.5!(A)$) {\tiny $l_1$} ;
\node[below right] at ($(A)!0.5!(L)$) {\tiny $l_3$} ;
\end{scope}

\begin{scope}[scale = 0.3,xshift = 12 cm]
\draw[->=latex] (-5,0) node{} -- (5,0) node{} node[midway,below = 2pt,text width=6 cm, align = center,scale = 0.7]{\small $l_1 = l_3 \longrightarrow + \infty$} ;
\end{scope}

\begin{scope}[xshift = 7cm]
\node (A) at (0,0) {} ;
\node (B) at (0,-1) {} ;

\node (K) at (-0.625,0.3) {} ;
\node (L) at (1.25,0.3) {} ;

\node (I) at (-0.625,0.6) {} ;
\node (J) at (1.25,0.6) {} ;

\draw[red] (I.center) -- (K.center) ;
\draw[red] (J.center) -- (L.center) ;
\draw[red] (A.center) -- (L.center) ;
\draw[red] (A.center) -- (K.center) ;
\draw[red] (A.center) -- (B.center) ;

\end{scope}

\begin{scope}[xshift = 7cm , yshift = 0.9cm]

\node (D) at (-1.25,1.25) {} ;
\node (E) at (0,1.25) {} ;
\node (F) at (-0.625,0.625) {} ;
\node (K) at (-0.625,0.3) {} ;
\node (I) at (-0.625,0) {} ;

\node[left] at ($(F)!0.5!(K)$) {\tiny $l_2$} ;

\draw[blue] (D.center) -- (F.center) ;
\draw[blue] (F.center) -- (E.center) ;
\draw[blue] (F.center) -- (K.center) ;
\draw[red] (I.center) -- (K.center) ;

\end{scope}

\begin{scope}[xshift = 7cm , yshift = 0.9cm]

\node (C) at (1.25,1.25) {} ;
\node (L) at (1.25,0.3) {} ;
\node (J) at (1.25,0) {} ;

\draw[blue] (C.center) -- (L.center) ;
\draw[red] (J.center) -- (L.center) ;
\end{scope}

\end{tikzpicture}}

\newcommand{\CTdeuxstratA}{
\begin{tikzpicture}

\draw[blue] (-1.25,1.25) -- (-0.5,0.5) ;
\draw[blue] (1.25,1.25) -- (0.5,0.5) ;

\draw[red] (-0.5,0.5) -- (0,0) ;
\draw[red] (0.5,0.5) -- (0,0) ;
\draw[red] (0,-1.25) -- (0,0) ;

\draw (-1.25,0.5) -- (1.25,0.5) ;

\node (K) at (1.5,0.5) {} ;
\node (J) at (1.5,0) {} ; 
\draw[densely dotted,>=stealth,<->] (K.center) -- (J.center) ;
\node[right] at ($(K)!0.5!(J)$) {$\lambda$} ;

\end{tikzpicture}
}

\newcommand{\CTdeuxstratB}{
\begin{tikzpicture}

\draw[blue] (-1.25,1.25) -- (0,0) ;
\draw[blue] (1.25,1.25) -- (0,0) ;

\draw[blue] (0,-0.5) -- (0,0) ;
\draw[red] (0,-1.25) -- (0,-0.5) ;

\draw (-1.25,-0.5) -- (1.25,-0.5) ;

\node (K) at (1.5,-0.5) {} ;
\node (J) at (1.5,0) {} ; 
\draw[densely dotted,>=stealth,<->] (K.center) -- (J.center) ;
\node[right] at ($(K)!0.5!(J)$) {$\lambda$} ;

\end{tikzpicture}
}

\newcommand{\CTdeuxstratO}{
\begin{tikzpicture}

\draw[blue] (-1.25,1.25) -- (0,0) ;
\draw[blue] (1.25,1.25) -- (0,0) ;

\draw[red] (0,-1.25) -- (0,0) ;

\draw (-1.25,0) -- (1.25,0) ;

\end{tikzpicture}
}

\newcommand{\CTdeuxstrat}{
\begin{tikzpicture}[ scale = 1.75 , baseline = (pt_base.base)]
\coordinate (pt_base) at (0,0) ;
\node[scale = 0.1] (A) at (-1,0) {\pointbullet} ;
\node[scale = 0.1] (B) at (1,0) {\pointbullet} ;
\node[scale = 0.1] (O) at (0,0) {\pointbullet} ;
\draw (A.center) -- (B.center) ;

\node[below = 2 pt,scale = 0.4] at ($(A)!0.5!(O)$) {\CTdeuxstratA} ;
\node[below = 2 pt,scale = 0.4] at ($(O)!0.5!(B)$) {\CTdeuxstratB} ;
\node[above = 3 pt,scale = 0.3] at (O) {\CTdeuxstratO} ;
\node[left = 3 pt, scale = 0.2] at (A) {\arbreopdeuxunmorph} ;
\node[right = 3 pt, scale = 0.2] at (B) {\arbreopdeuxdeuxmorph} ;
\end{tikzpicture}}

\newcommand{\CTtroisstratA}{
\begin{tikzpicture}
\node (A) at (0,0) {} ;
\node (B) at (0,-1) {} ;
\node (C) at (1.25,1.25) {} ;
\node (D) at (-1.25,1.25) {} ;
\node (E) at (0,1.25) {} ;
\node (F) at (0.625,0.625) {} ;
\node (G) at (0,-0.3) {} ;
\node (H) at (-1.25,-0.3) {} ;
\node (I) at (1.25,-0.3) {}  ;
\node (J) at (1.5,-0.3) {} ;
\node (K) at (1.5,0) {} ;

\draw[blue] (D.center) -- (A.center) ;
\draw[blue] (A.center) -- (G.center) ;
\draw[blue] (A.center) -- (F.center) ;
\draw[blue] (F.center) -- (E.center) ;
\draw[blue] (C.center) -- (F.center) ;

\draw[red] (G.center) -- (B.center) ;

\draw (H.center) -- (I.center) ;

\draw[densely dotted,>=stealth,<->] (K.center) -- (J.center) ;

\node[right] at ($(K)!0.5!(J)$) {$\lambda$} ;
\node[below right] at ($(A)!0.5!(F)$) {$l$} ;
\end{tikzpicture}}

\newcommand{\CTtroisstratB}{
\begin{tikzpicture}
\node (A) at (0,0) {} ;
\node (B) at (0,-1) {} ;
\node (C) at (1.25,1.25) {} ;
\node (D) at (-1.25,1.25) {} ;
\node (E) at (0,1.25) {} ;
\node (F) at (0.625,0.625) {} ;
\node (G) at (-1.25,0.3) {} ;
\node (H) at (1.25,0.3) {} ;
\node (I) at (1.5,0.3) {}  ;
\node (J) at (1.5,0) {} ;

\node (K) at (-0.3,0.3) {} ;
\node (L) at (0.3,0.3) {} ;

\draw[blue] (D.center) -- (K.center) ;
\draw[blue] (L.center) -- (F.center) ;
\draw[blue] (F.center) -- (E.center) ;
\draw[blue] (F.center) -- (C.center) ;
\draw[red] (K.center) -- (A.center) ;
\draw[red] (L.center) -- (A.center) ;
\draw[red] (A.center) -- (B.center) ;
\draw (G.center) -- (H.center) ;
\draw[densely dotted,>=stealth,<->] (I.center) -- (J.center) ;

\node[right] at ($(I)!0.5!(J)$) { $\lambda$} ;
\node[below right] at ($(A)!0.5!(F)$) { $l$} ;
\end{tikzpicture}}

\newcommand{\CTtroisstratC}{
\begin{tikzpicture}
\node (A) at (0,0) {} ;
\node (B) at (0,-1) {} ;
\node (C) at (1.25,1.25) {} ;
\node (D) at (-1.25,1.25) {} ;
\node (E) at (0,1.25) {} ;
\node (F) at (0.625,0.625) {} ;
\node (G) at (-0.9,0.9) {} ;
\node (H) at (0.35,0.9) {} ;
\node (I) at (0.9,0.9) {}  ;
\node (J) at (-1.25,0.9) {} ;
\node (K) at (1.25,0.9) {} ;
\node (L) at (1.5,0) {} ;
\node (M) at (1.5,0.9) {} ;

\draw[blue] (D.center) -- (G.center) ;
\draw[blue] (E.center) -- (H.center) ;
\draw[blue] (C.center) -- (I.center) ;

\draw[red] (G.center) -- (A.center) ;
\draw[red] (H.center) -- (F.center) ;
\draw[red] (F.center) -- (A.center) ;
\draw[red] (A.center) -- (I.center) ;
\draw[red] (A.center) -- (B.center) ;

\draw (J.center) -- (K.center) ;

\draw[densely dotted,>=stealth,<->] (L.center) -- (M.center) ;

\node[right] at ($(L)!0.5!(M)$) { $\lambda$} ;
\node[below right] at ($(A)!0.5!(F)$) { $l$} ;
\end{tikzpicture}}

\newcommand{\CTtroisstratD}{
\begin{tikzpicture}
\node (A) at (0,0) {} ;
\node (B) at (0,-1) {} ;
\node (C) at (1.25,1.25) {} ;
\node (D) at (-1.25,1.25) {} ;
\node (E) at (0,1.25) {} ;
\node (F) at (-0.625,0.625) {} ;
\node (G) at (-0.9,0.9) {} ;
\node (H) at (-0.35,0.9) {} ;
\node (I) at (0.9,0.9) {}  ;
\node (J) at (-1.25,0.9) {} ;
\node (K) at (1.25,0.9) {} ;
\node (L) at (1.5,0) {} ;
\node (M) at (1.5,0.9) {} ;

\draw[blue] (D.center) -- (G.center) ;
\draw[blue] (E.center) -- (H.center) ;
\draw[blue] (C.center) -- (I.center) ;

\draw[red] (G.center) -- (F.center) ;
\draw[red] (H.center) -- (F.center) ;
\draw[red] (F.center) -- (A.center) ;
\draw[red] (A.center) -- (I.center) ;
\draw[red] (A.center) -- (B.center) ;

\draw (J.center) -- (K.center) ;

\draw[densely dotted,>=stealth,<->] (L.center) -- (M.center) ;

\node[right] at ($(L)!0.5!(M)$) { $\lambda$} ;
\node[below left] at ($(A)!0.5!(F)$) { $l$} ;
\end{tikzpicture}}

\newcommand{\CTtroisstratE}{
\begin{tikzpicture}
\node (A) at (0,0) {} ;
\node (B) at (0,-1) {} ;
\node (C) at (1.25,1.25) {} ;
\node (D) at (-1.25,1.25) {} ;
\node (E) at (0,1.25) {} ;
\node (F) at (-0.625,0.625) {} ;
\node (G) at (-1.25,0.3) {} ;
\node (H) at (1.25,0.3) {} ;
\node (I) at (1.5,0.3) {}  ;
\node (J) at (1.5,0) {} ;

\node (K) at (-0.3,0.3) {} ;
\node (L) at (0.3,0.3) {} ;

\draw[blue] (D.center) -- (F.center) -- (K.center) ;
\draw[red] (K.center) -- (A.center) ;
\draw[blue] (F.center) -- (E.center) ;
\draw[red] (A.center) -- (B.center) ;
\draw[red] (A.center) -- (L.center) ;
\draw[blue] (L.center) -- (C.center) ;
\draw[blue] (F.center) -- (E.center) ;
\draw (G.center) -- (H.center) ;
\draw[densely dotted,>=stealth,<->] (I.center) -- (J.center) ;

\node[right] at ($(I)!0.5!(J)$) { $\lambda$} ;
\node[below left] at ($(A)!0.5!(F)$) { $l$} ;
\end{tikzpicture}}

\newcommand{\CTtroisstratF}{
\begin{tikzpicture}
\node (A) at (0,0) {} ;
\node (B) at (0,-1) {} ;
\node (C) at (1.25,1.25) {} ;
\node (D) at (-1.25,1.25) {} ;
\node (E) at (0,1.25) {} ;
\node (F) at (-0.625,0.625) {} ;
\node (G) at (0,-0.3) {} ;
\node (H) at (-1.25,-0.3) {} ;
\node (I) at (1.25,-0.3) {}  ;
\node (J) at (1.5,-0.3) {} ;
\node (K) at (1.5,0) {} ;

\draw[blue] (D.center) -- (A.center) ;
\draw[blue] (A.center) -- (G.center) ;
\draw[blue] (A.center) -- (F.center) ;
\draw[blue] (F.center) -- (E.center) ;
\draw[blue] (C.center) -- (A.center) ;

\draw[red] (G.center) -- (B.center) ;

\draw (H.center) -- (I.center) ;

\draw[densely dotted,>=stealth,<->] (K.center) -- (J.center) ;

\node[right] at ($(K)!0.5!(J)$) { $\lambda$} ;
\node[below left] at ($(A)!0.5!(F)$) { $l$} ;
\end{tikzpicture}}

\newcommand{\CTtroisstrat}{
\begin{tikzpicture}[scale = 2 , baseline = (pt_base.base)]

\coordinate (pt_base) at (-1,0) ;

\node[scale = 0.4] (Ap) at (0.52,-0.82) {\CTtroisstratA} ;
\node[scale = 0.4] (Bp) at (1.1,0) {\CTtroisstratB} ;
\node[scale = 0.4] (Cp) at (0.5,0.8) {\CTtroisstratC} ;
\node[scale = 0.4] (Dp) at (-0.4,0.8) {\CTtroisstratD} ;
\node[scale = 0.4] (Ep) at (-0.9,0) {\CTtroisstratE} ; 
\node[scale = 0.4] (Fp) at (-0.4,-0.8) {\CTtroisstratF} ; 

\node[scale = 0.1] (A) at (0.85,-1.4) {} ;
\node[scale = 0.1] (B) at (1.65,0) {} ;
\node[scale = 0.1] (C) at (0.85,1.35) {} ;
\node[scale = 0.1] (D) at (-0.75,1.35) {} ;
\node[scale = 0.1] (E) at (-1.5,0) {} ;
\node[scale = 0.1] (F) at (-0.75,-1.4) {} ;

\node (O) at (0,0) {} ;

\draw[fill,opacity = 0.1] (A.center) -- (B.center) -- (C.center) -- (D.center) -- (E.center) -- (F.center) -- (A.center) ;

\draw (A.center) -- (B.center) -- (C.center) -- (D.center) -- (E.center) -- (F.center) -- (A.center) ;

\node[scale=0.1] at (A) {\pointbullet} ;
\node[scale=0.1] at (B) {\pointbullet} ;
\node[scale=0.1] at (C) {\pointbullet} ;
\node[scale=0.1] at (D) {\pointbullet} ;
\node[scale=0.1] at (E) {\pointbullet} ;
\node[scale=0.1] at (F) {\pointbullet} ;

\draw ($(A)!0.5!(B)$) --(O.center) ;
\draw ($(B)!0.5!(C)$) --(O.center) ;
\draw ($(C)!0.5!(D)$) --(O.center) ;
\draw ($(D)!0.5!(E)$) --(O.center) ;
\draw ($(E)!0.5!(F)$) --(O.center) ;
\draw ($(F)!0.5!(A)$) --(O.center) ;
\end{tikzpicture}
}

\newcommand{\CTtroisstratAbis}{
\begin{tikzpicture}
\node (A) at (0,0) {} ;
\node (B) at (0,-1) {} ;
\node (C) at (1.25,1.25) {} ;
\node (D) at (-1.25,1.25) {} ;
\node (E) at (0,1.25) {} ;
\node (F) at (0.625,0.625) {} ;
\node (G) at (0,-0.3) {} ;
\node (H) at (-1.25,-0.3) {} ;
\node (I) at (1.25,-0.3) {}  ;
\node (J) at (1.5,-0.3) {} ;
\node (K) at (1.5,0) {} ;

\draw[blue] (D.center) -- (A.center) ;
\draw[blue] (A.center) -- (G.center) ;
\draw[blue] (A.center) -- (F.center) ;
\draw[blue] (F.center) -- (E.center) ;
\draw[blue] (C.center) -- (F.center) ;

\draw[red] (G.center) -- (B.center) ;

\draw (H.center) -- (I.center) ;

\end{tikzpicture}}

\newcommand{\CTtroisstratBbis}{
\begin{tikzpicture}
\node (A) at (0,0) {} ;
\node (B) at (0,-1) {} ;
\node (C) at (1.25,1.25) {} ;
\node (D) at (-1.25,1.25) {} ;
\node (E) at (0,1.25) {} ;
\node (F) at (0.625,0.625) {} ;
\node (G) at (-1.25,0.3) {} ;
\node (H) at (1.25,0.3) {} ;
\node (I) at (1.5,0.3) {}  ;
\node (J) at (1.5,0) {} ;

\node (K) at (-0.3,0.3) {} ;
\node (L) at (0.3,0.3) {} ;

\draw[blue] (D.center) -- (K.center) ;
\draw[blue] (L.center) -- (F.center) ;
\draw[blue] (F.center) -- (E.center) ;
\draw[blue] (F.center) -- (C.center) ;
\draw[red] (K.center) -- (A.center) ;
\draw[red] (L.center) -- (A.center) ;
\draw[red] (A.center) -- (B.center) ;
\draw (G.center) -- (H.center) ;
\end{tikzpicture}}

\newcommand{\CTtroisstratCbis}{
\begin{tikzpicture}
\node (A) at (0,0) {} ;
\node (B) at (0,-1) {} ;
\node (C) at (1.25,1.25) {} ;
\node (D) at (-1.25,1.25) {} ;
\node (E) at (0,1.25) {} ;
\node (F) at (0.625,0.625) {} ;
\node (G) at (-0.9,0.9) {} ;
\node (H) at (0.35,0.9) {} ;
\node (I) at (0.9,0.9) {}  ;
\node (J) at (-1.25,0.9) {} ;
\node (K) at (1.25,0.9) {} ;
\node (L) at (1.5,0) {} ;
\node (M) at (1.5,0.9) {} ;

\draw[blue] (D.center) -- (G.center) ;
\draw[blue] (E.center) -- (H.center) ;
\draw[blue] (C.center) -- (I.center) ;

\draw[red] (G.center) -- (A.center) ;
\draw[red] (H.center) -- (F.center) ;
\draw[red] (F.center) -- (A.center) ;
\draw[red] (A.center) -- (I.center) ;
\draw[red] (A.center) -- (B.center) ;

\draw (J.center) -- (K.center) ;

\end{tikzpicture}}

\newcommand{\CTtroisstratDbis}{
\begin{tikzpicture}
\node (A) at (0,0) {} ;
\node (B) at (0,-1) {} ;
\node (C) at (1.25,1.25) {} ;
\node (D) at (-1.25,1.25) {} ;
\node (E) at (0,1.25) {} ;
\node (F) at (-0.625,0.625) {} ;
\node (G) at (-0.9,0.9) {} ;
\node (H) at (-0.35,0.9) {} ;
\node (I) at (0.9,0.9) {}  ;
\node (J) at (-1.25,0.9) {} ;
\node (K) at (1.25,0.9) {} ;
\node (L) at (1.5,0) {} ;
\node (M) at (1.5,0.9) {} ;

\draw[blue] (D.center) -- (G.center) ;
\draw[blue] (E.center) -- (H.center) ;
\draw[blue] (C.center) -- (I.center) ;

\draw[red] (G.center) -- (F.center) ;
\draw[red] (H.center) -- (F.center) ;
\draw[red] (F.center) -- (A.center) ;
\draw[red] (A.center) -- (I.center) ;
\draw[red] (A.center) -- (B.center) ;

\draw (J.center) -- (K.center) ;

\end{tikzpicture}}

\newcommand{\CTtroisstratEbis}{
\begin{tikzpicture}
\node (A) at (0,0) {} ;
\node (B) at (0,-1) {} ;
\node (C) at (1.25,1.25) {} ;
\node (D) at (-1.25,1.25) {} ;
\node (E) at (0,1.25) {} ;
\node (F) at (-0.625,0.625) {} ;
\node (G) at (-1.25,0.3) {} ;
\node (H) at (1.25,0.3) {} ;
\node (I) at (1.5,0.3) {}  ;
\node (J) at (1.5,0) {} ;

\node (K) at (-0.3,0.3) {} ;
\node (L) at (0.3,0.3) {} ;

\draw[blue] (D.center) -- (F.center) -- (K.center) ;
\draw[red] (K.center) -- (A.center) ;
\draw[blue] (F.center) -- (E.center) ;
\draw[red] (A.center) -- (B.center) ;
\draw[red] (A.center) -- (L.center) ;
\draw[blue] (L.center) -- (C.center) ;
\draw[blue] (F.center) -- (E.center) ;
\draw (G.center) -- (H.center) ;
\end{tikzpicture}}

\newcommand{\CTtroisstratFbis}{
\begin{tikzpicture}
\node (A) at (0,0) {} ;
\node (B) at (0,-1) {} ;
\node (C) at (1.25,1.25) {} ;
\node (D) at (-1.25,1.25) {} ;
\node (E) at (0,1.25) {} ;
\node (F) at (-0.625,0.625) {} ;
\node (G) at (0,-0.3) {} ;
\node (H) at (-1.25,-0.3) {} ;
\node (I) at (1.25,-0.3) {}  ;
\node (J) at (1.5,-0.3) {} ;
\node (K) at (1.5,0) {} ;

\draw[blue] (D.center) -- (A.center) ;
\draw[blue] (A.center) -- (G.center) ;
\draw[blue] (A.center) -- (F.center) ;
\draw[blue] (F.center) -- (E.center) ;
\draw[blue] (C.center) -- (A.center) ;

\draw[red] (G.center) -- (B.center) ;

\draw (H.center) -- (I.center) ;

\end{tikzpicture}}

\newcommand{\CTtroisstratbis}{
\begin{tikzpicture}[scale = 1.3 , baseline = (pt_base.base)]

\coordinate (pt_base) at (-1,0) ;

\node[scale = 0.34] (Ap) at (0.52,-0.82) {\CTtroisstratAbis} ;
\node[scale = 0.34] (Bp) at (1.1,0) {\CTtroisstratBbis} ;
\node[scale = 0.34] (Cp) at (0.5,0.8) {\CTtroisstratCbis} ;
\node[scale = 0.34] (Dp) at (-0.4,0.8) {\CTtroisstratDbis} ;
\node[scale = 0.34] (Ep) at (-0.9,0) {\CTtroisstratEbis} ; 
\node[scale = 0.34] (Fp) at (-0.4,-0.8) {\CTtroisstratFbis} ; 

\node[scale = 0.1] (A) at (0.85,-1.4) {} ;
\node[scale = 0.1] (B) at (1.65,0) {} ;
\node[scale = 0.1] (C) at (0.85,1.35) {} ;
\node[scale = 0.1] (D) at (-0.75,1.35) {} ;
\node[scale = 0.1] (E) at (-1.5,0) {} ;
\node[scale = 0.1] (F) at (-0.75,-1.4) {} ;

\node (O) at (0,0) {} ;

\draw[fill,opacity = 0.1] (A.center) -- (B.center) -- (C.center) -- (D.center) -- (E.center) -- (F.center) -- (A.center) ;

\draw (A.center) -- (B.center) -- (C.center) -- (D.center) -- (E.center) -- (F.center) -- (A.center) ;

\node[scale=0.1] at (A) {\pointbullet} ;
\node[scale=0.1] at (B) {\pointbullet} ;
\node[scale=0.1] at (C) {\pointbullet} ;
\node[scale=0.1] at (D) {\pointbullet} ;
\node[scale=0.1] at (E) {\pointbullet} ;
\node[scale=0.1] at (F) {\pointbullet} ;

\draw ($(A)!0.5!(B)$) --(O.center) ;
\draw ($(B)!0.5!(C)$) --(O.center) ;
\draw ($(C)!0.5!(D)$) --(O.center) ;
\draw ($(D)!0.5!(E)$) --(O.center) ;
\draw ($(E)!0.5!(F)$) --(O.center) ;
\draw ($(F)!0.5!(A)$) --(O.center) ;
\end{tikzpicture}
}

\newcommand{\examplestratumBC}{
\begin{tikzpicture}
\begin{scope}
\node (K) at (0.625,0.625) {} ;
\node (L) at (0,0) {} ;
\draw (0.625,0.625) -- (0,1.25) ;
\draw (-1.25,1.25) -- (0,0) ;
\draw (1.25,1.25) -- (0,0) ;
\draw (0,-1.25) -- (0,0) ;

\node[below right = 0.2 pt] at ($(K)!0.5!(L)$) {\Huge $l_2$} ;
\end{scope}

\begin{scope}[shift = {(-1.25,2.1)}]
\draw (-0.625,0.625) -- (0,0) ;
\draw (0.625,0.625) -- (0,0) ;
\draw (0,-0.625) -- (0,0) ;
\end{scope}
\end{tikzpicture}}

\newcommand{\examplestratumCD}{
\begin{tikzpicture}
\begin{scope}
\node (K) at (-0.625,0.625) {} ;
\node (L) at (0,0) {} ;
\draw (-0.625,0.625) -- (0,1.25) ;
\draw (-1.25,1.25) -- (0,0) ;
\draw (1.25,1.25) -- (0,0) ;
\draw (0,-1.25) -- (0,0) ;

\node[below left = 0.2 pt] at ($(K)!0.5!(L)$) {\Huge $l_1$} ;
\end{scope}

\begin{scope}[shift = {(1.25,2.1)}]
\draw (-0.625,0.625) -- (0,0) ;
\draw (0.625,0.625) -- (0,0) ;
\draw (0,-0.625) -- (0,0) ;
\end{scope}
\end{tikzpicture}}

\newcommand{\examplestratumAB}{
\begin{tikzpicture}
\node (A) at (-1.5,1) {} ;
\node (B) at (0,0) {} ;
\node (C) at (1.5,1) {} ;
\node (D) at (0,-1) {} ;
\node (E) at (1,2) {} ;
\node (F) at (2.75,2) {} ;
\node (G) at (-0.25,2) {} ;
\node (I) at (-2.75,2) {} ;

\draw (I.center) -- (A.center) -- (B.center) -- (C.center) ;
\draw (F.center) -- (C.center) ;
\draw (G.center) -- (A.center) ;
\draw (B.center) -- (D.center) ;
\draw (B.center) -- (E.center) ;

\node[below left = 0.2 pt] at ($(A)!0.5!(B)$) {\Huge $l_1$} ;
\end{tikzpicture}}

\newcommand{\examplestratumDA}{
\begin{tikzpicture}
\node (A) at (-1.5,1) {} ;
\node (B) at (0,0) {} ;
\node (C) at (1.5,1) {} ;
\node (D) at (0,-1) {} ;
\node (E) at (0.25,2) {} ;
\node (F) at (2.75,2) {} ;
\node (G) at (-1,2) {} ;
\node (I) at (-2.75,2) {} ;

\draw (I.center) -- (A.center) -- (B.center) -- (C.center) -- (E.center) ;
\draw (F.center) -- (C.center) ;
\draw (B.center) -- (D.center) ;
\draw (B.center) -- (G.center) ;

\node[below right = 0.2 pt] at ($(C)!0.5!(B)$) {\Huge $l_2$} ;
\end{tikzpicture}}

\newcommand{\disquebordlagrangienassoc}{
\begin{tikzpicture}[scale = 0.5]

\draw (0,0) circle (3) ; 

\draw (360/20 : 3) node[scale = 0.1]{\pointbullet};
\draw (3*360/20 : 3) node[scale = 0.1]{\pointbullet};
\draw (9*360/20 : 3) node[scale = 0.1]{\pointbullet};
\draw (7*360/20 : 3) node[scale = 0.1]{\pointbullet};
\draw (270 : 3) node[scale = 0.1]{\pointbullet};

\draw (360/20 : 3) node[right,scale=0.7]{$x_n$};
\draw (3*360/20 : 3) node[above=3pt,scale=0.7]{$x_{n-1}$};
\draw (9*360/20 : 3) node[left,scale=0.7]{$x_1$};
\draw (7*360/20 : 3) node[above=3pt,scale=0.7]{$x_2$};
\draw (270 : 3) node[below=3pt,scale=0.7]{$y$};

\draw (2*360/20 : 3) node[above right,scale=0.8]{$L_{n-1}$};
\draw (8*360/20 : 3) node[above left,scale=0.8]{$L_1$};
\draw (11.5*360/20 : 3) node[below left,scale=0.8]{$L_0$};
\draw (18.5*360/20 : 3) node[below right,scale=0.8]{$L_n$};

\draw[densely dotted] (0,3.3) arc (90:108:3.3) ;
\draw[densely dotted] (0,3.3) arc (90:72:3.3) ;

\node at (0,0) {$M$} ;

\end{tikzpicture}
}

\newcommand{\examplestratum}{
\begin{tikzpicture}[scale = 1.5]
\node[scale = 0.4] (A) at (-1,-1) {} ;
\node (B) at (1,-1) {} ;
\node[scale = 0.4] (C) at (1,1) {} ;
\node (D) at (-1,1) {} ;
\node[scale = 0.3] (AB) at ($(A)!0.5!(B)$) {} ;
\node[scale = 0.3] (BC) at ($(B)!0.5!(C)$) {} ;
\node[scale = 0.3] (CD) at ($(C)!0.5!(D)$) {} ;
\node[scale = 0.3] (DA) at ($(D)!0.5!(A)$) {} ;
\node[scale = 0.3] (AC) at ($(A)!0.5!(C)$) {} ;

\draw[dotted,fill,opacity = 0.1] (A.center) -- (B.center) -- (C.center)  -- (D.center)  -- (A.center) ;
\draw (A.center) -- (B.center) ;
\draw[dotted] (B.center) -- (C.center) ;
\draw[dotted] (C.center) -- (D.center) ;
\draw (D.center) -- (A.center) ;

\node[below] at (B) {\tiny $l_1$} ;
\node[left] at (D) {\tiny $l_2$} ;

\node[scale = 0.3 , below = 8 pt] at (AB) {\examplestratumAB} ;
\node[scale = 0.3 , right = 8 pt] at (BC) {\examplestratumBC} ;
\node[scale = 0.3 , above = 8 pt] at (CD) {\examplestratumCD} ;
\node[scale = 0.25 , left = 8 pt ] at (DA) {\examplestratumDA} ;
\node[scale = 0.25] at (AC) {\TquatrestratB} ;
\end{tikzpicture}}

\newcommand{\examplediffcobarbarA}{\begin{tikzpicture}[scale = 0.15 ,baseline=(pt_base.base)]

\coordinate (pt_base) at (0,-0.5) ;
\node (A) at (-1.5,1) {} ;
\node (B) at (0,0) {} ;
\node (C) at (1.5,1) {} ;
\node (D) at (0,-1) {} ;
\node (E) at (0.25,2) {} ;
\node (F) at (2.75,2) {} ;
\node (G) at (-0.25,2) {} ;
\node (I) at (-2.75,2) {} ;

\draw (I.center) -- (A.center) -- (B.center) -- (C.center) -- (E.center) ;
\draw (F.center) -- (C.center) ;
\draw (G.center) -- (A.center) ;
\draw (B.center) -- (D.center) ;
\end{tikzpicture}}

\newcommand{\examplediffcobarbarB}{
\begin{tikzpicture}[scale = 0.15 ,baseline=(pt_base.base)]

\coordinate (pt_base) at (0,-0.5) ;
\begin{scope}
\node (K) at (0.625,0.625) {} ;
\node (L) at (0,0) {} ;
\draw (0.625,0.625) -- (0,1.25) ;
\draw (-1.25,1.25) -- (0,0) ;
\draw (1.25,1.25) -- (0,0) ;
\draw (0,-1.25) -- (0,0) ;
\node (O) at (0,-1.25) {} ;
\end{scope}

\begin{scope}[shift = {(-1.25,2.1)}]
\draw (-0.625,0.625) -- (0,0) ;
\draw (0.625,0.625) -- (0,0) ;
\draw (0,-0.625) -- (0,0) ;
\end{scope}
\end{tikzpicture}}

\newcommand{\examplediffcobarbarC}{
\begin{tikzpicture}[scale = 0.15,baseline=(pt_base.base)]

\coordinate (pt_base) at (0,-0.5) ;
\begin{scope}
\node (K) at (-0.625,0.625) {} ;
\node (L) at (0,0) {} ;
\draw (-0.625,0.625) -- (0,1.25) ;
\draw (-1.25,1.25) -- (0,0) ;
\draw (1.25,1.25) -- (0,0) ;
\draw (0,-1.25) -- (0,0) ;
\node (O) at (0,-1.25) {} ;
\end{scope}

\begin{scope}[shift = {(1.25,2.1)}]
\draw (-0.625,0.625) -- (0,0) ;
\draw (0.625,0.625) -- (0,0) ;
\draw (0,-0.625) -- (0,0) ;
\end{scope}
\end{tikzpicture}}

\newcommand{\examplediffcobarbarD}{
\begin{tikzpicture}[scale = 0.15 ,baseline=(pt_base.base)]

\coordinate (pt_base) at (0,-0.5) ;
\node (A) at (-1.5,1) {} ;
\node (B) at (0,0) {} ;
\node (C) at (1.5,1) {} ;
\node (D) at (0,-1) {} ;
\node (E) at (1,2) {} ;
\node (F) at (2.75,2) {} ;
\node (G) at (-0.25,2) {} ;
\node (I) at (-2.75,2) {} ;

\draw (I.center) -- (A.center) -- (B.center) -- (C.center) ;
\draw (F.center) -- (C.center) ;
\draw (G.center) -- (A.center) ;
\draw (B.center) -- (D.center) ;
\draw (B.center) -- (E.center) ;
\end{tikzpicture}}

\newcommand{\examplediffcobarbarE}{
\begin{tikzpicture}[scale = 0.15 ,baseline=(pt_base.base)]

\coordinate (pt_base) at (0,-0.5) ;
\node (A) at (-1.5,1) {} ;
\node (B) at (0,0) {} ;
\node (C) at (1.5,1) {} ;
\node (D) at (0,-1) {} ;
\node (E) at (0.25,2) {} ;
\node (F) at (2.75,2) {} ;
\node (G) at (-1,2) {} ;
\node (I) at (-2.75,2) {} ;

\draw (I.center) -- (A.center) -- (B.center) -- (C.center) -- (E.center) ;
\draw (F.center) -- (C.center) ;
\draw (B.center) -- (D.center) ;
\draw (B.center) -- (G.center) ;

\end{tikzpicture}}

\newcommand{\premiertermecobarbarA}{
\begin{tikzpicture}[scale=0.15,baseline=(pt_base.base)]

\coordinate (pt_base) at (0,-0.5) ;

\draw (-1.25,1.25) -- (0,0) ;
\draw (1.25,1.25) -- (0,0) ;
\draw (0,-1.25) -- (0,0) ;

\end{tikzpicture}
}

\newcommand{\premiertermecobarbarAprime}{
\begin{tikzpicture}[scale=0.07,baseline=(pt_base.base)]

\coordinate (pt_base) at (0,-0.5) ;

\draw (-1.25,1.25) -- (0,0) ;
\draw (1.25,1.25) -- (0,0) ;
\draw (0,-1.25) -- (0,0) ;

\end{tikzpicture}
}

\newcommand{\premiertermecobarbarB}{
\begin{tikzpicture}[scale = 0.15,baseline = (O.base)]
\node (K) at (-0.625,0.625) {} ;
\node (L) at (0,0) {} ;
\draw (-0.625,0.625) -- (0,1.25) ;
\draw (-1.25,1.25) -- (0,0) ;
\draw (1.25,1.25) -- (0,0) ;
\draw (0,-1.25) -- (0,0) ;
\node (O) at (0,-0.5) {} ;
\end{tikzpicture}
}

\newcommand{\premiertermecobarbarC}{
\begin{tikzpicture}[scale = 0.15,baseline = (O.base)]
\node (K) at (0.625,0.625) {} ;
\node (L) at (0,0) {} ;
\draw (0.625,0.625) -- (0,1.25) ;
\draw (-1.25,1.25) -- (0,0) ;
\draw (1.25,1.25) -- (0,0) ;
\draw (0,-1.25) -- (0,0) ;
\node (O) at (0,-0.5) {} ;
\end{tikzpicture}
}

\newcommand{\premiertermecobarbarD}{
\begin{tikzpicture}[scale=0.15,baseline=(pt_base.base)]

\coordinate (pt_base) at (0,-0.5) ;

\draw (0,1.25) -- (0,0) ;
\draw (-1.25,1.25) -- (0,0) ;
\draw (1.25,1.25) -- (0,0) ;
\draw (0,-1.25) -- (0,0) ;

\end{tikzpicture}
}

\newcommand{\arbrebicoloreA}{
\begin{tikzpicture}[scale = 0.15 ,baseline=(pt_base.base)]

\coordinate (pt_base) at (0,-0.5) ;
\node (A) at (-1.5,1) {} ;
\node (B) at (0,0) {} ;
\node (C) at (1.5,1) {} ;
\node (D) at (0,-1) {} ;
\node (E) at (0.25,2) {} ;
\node (F) at (2.75,2) {} ;
\node (G) at (-0.25,2) {} ;
\node (I) at (-2.75,2) {} ;
\node (O) at (0.9,0.6) {} ;
\node (P) at (-0.9,0.6) {} ;

\node (K) at (-2.75,0.6) {} ;
\node (L) at (2.75,0.6) {} ;

\node (M) at (3,0.6) {} ;
\node (N) at (3,0) {} ;

\draw[blue] (F.center) -- (C.center) ;
\draw[blue] (E.center) -- (C.center) ;
\draw[blue] (O.center) -- (C.center) ;
\draw[blue] (A.center) -- (I.center) ;
\draw[blue] (A.center) -- (G.center) ;
\draw[blue] (A.center) -- (P.center) ;
\draw[red] (B.center) -- (P.center) ;
\draw[red] (B.center) -- (O.center) ;
\draw[red] (B.center) -- (D.center) ;

\draw (K.center) -- (L.center) ;

\end{tikzpicture}
}

\newcommand{\arbrebicoloreB}{
\begin{tikzpicture}[scale = 0.15 ,baseline=(pt_base.base)]

\coordinate (pt_base) at (0,-0.5) ;
\node (A) at (0,0) {} ;
\node (B) at (0,-1) {} ;
\node (C) at (1.25,1.25) {} ;
\node (D) at (-1.25,1.25) {} ;
\node (E) at (0,1.25) {} ;
\node (F) at (-0.625,0.625) {} ;
\node (G) at (-0.9,0.9) {} ;
\node (H) at (-0.35,0.9) {} ;
\node (I) at (0.9,0.9) {}  ;
\node (J) at (-1.25,0.9) {} ;
\node (K) at (1.25,0.9) {} ;
\node (L) at (1.5,0) {} ;
\node (M) at (1.5,0.9) {} ;

\draw[blue] (D.center) -- (G.center) ;
\draw[blue] (E.center) -- (H.center) ;
\draw[blue] (C.center) -- (I.center) ;

\draw[red] (G.center) -- (F.center) ;
\draw[red] (H.center) -- (F.center) ;
\draw[red] (F.center) -- (A.center) ;
\draw[red] (A.center) -- (I.center) ;
\draw[red] (A.center) -- (B.center) ;

\draw (J.center) -- (K.center) ;

\end{tikzpicture}}

\newcommand{\arbrebicoloreC}{
\begin{tikzpicture}[scale = 0.15 ,baseline=(pt_base.base)]

\coordinate (pt_base) at (0,-0.5) ;
\node (A) at (-1.5,1) {} ;
\node (B) at (0,0) {} ;
\node (C) at (1.5,1) {} ;
\node (D) at (0,-1) {} ;
\node (E) at (0.25,2) {} ;
\node (F) at (2.75,2) {} ;
\node (G) at (-0.25,2) {} ;
\node (I) at (-2.75,2) {} ;
\node (O) at (0.9,0.6) {} ;
\node (P) at (-0.9,0.6) {} ;

\node (K) at (-2.75,0.6) {} ;
\node (L) at (2.75,0.6) {} ;

\node (M) at (3,0.6) {} ;
\node (N) at (3,0) {} ;

\draw (F.center) -- (C.center) ;
\draw (E.center) -- (C.center) ;
\draw (O.center) -- (C.center) ;
\draw (A.center) -- (I.center) ;
\draw (A.center) -- (G.center) ;
\draw (A.center) -- (P.center) ;
\draw (B.center) -- (P.center) ;
\draw (B.center) -- (O.center) ;
\draw (B.center) -- (D.center) ;

\end{tikzpicture}
}

\newcommand{\arbrebicoloreD}{
\begin{tikzpicture}[scale = 0.15 ,baseline=(pt_base.base)]

\coordinate (pt_base) at (0,-0.5) ;
\node (A) at (-1.5,1) {} ;
\node (B) at (0,0) {} ;
\node (C) at (0.9,0.6) {} ;
\node (D) at (0,-1) {} ;
\node (E) at (0.25,2) {} ;
\node (F) at (2.75,2) {} ;
\node (G) at (-0.25,2) {} ;
\node (I) at (-2.75,2) {} ;
\node (O) at (0.9,0.6) {} ;
\node (P) at (-0.9,0.6) {} ;

\node (K) at (-2.75,0.6) {} ;
\node (L) at (2.75,0.6) {} ;

\node (M) at (3,0.6) {} ;
\node (N) at (3,0) {} ;

\draw[blue] (F.center) -- (C.center) ;
\draw[blue] (E.center) -- (C.center) ;
\draw[blue] (O.center) -- (C.center) ;
\draw[blue] (A.center) -- (I.center) ;
\draw[blue] (A.center) -- (G.center) ;
\draw[blue] (A.center) -- (P.center) ;
\draw[red] (B.center) -- (P.center) ;
\draw[red] (B.center) -- (O.center) ;
\draw[red] (B.center) -- (D.center) ;

\draw (K.center) -- (L.center) ;

\end{tikzpicture}
}

\newcommand{\arbrebicoloreE}{
\begin{tikzpicture}[scale = 0.15 ,baseline=(pt_base.base)]

\coordinate (pt_base) at (0,-0.5) ;
\node (A) at (-0.9,0.6) {} ;
\node (B) at (0,0) {} ;
\node (C) at (1.5,1) {} ;
\node (D) at (0,-1) {} ;
\node (E) at (0.25,2) {} ;
\node (F) at (2.75,2) {} ;
\node (G) at (-0.25,2) {} ;
\node (I) at (-2.75,2) {} ;
\node (O) at (0.9,0.6) {} ;
\node (P) at (-0.9,0.6) {} ;

\node (K) at (-2.75,0.6) {} ;
\node (L) at (2.75,0.6) {} ;

\node (M) at (3,0.6) {} ;
\node (N) at (3,0) {} ;

\draw[blue] (F.center) -- (C.center) ;
\draw[blue] (E.center) -- (C.center) ;
\draw[blue] (O.center) -- (C.center) ;
\draw[blue] (A.center) -- (I.center) ;
\draw[blue] (A.center) -- (G.center) ;
\draw[blue] (A.center) -- (P.center) ;
\draw[red] (B.center) -- (P.center) ;
\draw[red] (B.center) -- (O.center) ;
\draw[red] (B.center) -- (D.center) ;

\draw (K.center) -- (L.center) ;

\end{tikzpicture}
}

\newcommand{\arbrebicoloreF}{
\begin{tikzpicture}[scale = 0.15 ,baseline=(pt_base.base)]

\coordinate (pt_base) at (0,-0.5) ;
\node (A) at (-1.5,1) {} ;
\node (B) at (0,0) {} ;
\node (C) at (1.5,1) {} ;
\node (D) at (0,-1) {} ;
\node (E) at (0.25,2) {} ;
\node (F) at (2.75,2) {} ;
\node (G) at (-0.25,2) {} ;
\node (I) at (-2.75,2) {} ;

\node (K) at (-2.75,0) {} ;
\node (L) at (2.75,0) {} ;

\draw[blue] (F.center) -- (C.center) ;
\draw[blue] (E.center) -- (C.center) ;
\draw[blue] (B.center) -- (C.center) ;
\draw[blue] (A.center) -- (I.center) ;
\draw[blue] (A.center) -- (G.center) ;
\draw[blue] (A.center) -- (B.center) ;
\draw[red] (B.center) -- (D.center) ;

\draw (K.center) -- (L.center) ;

\end{tikzpicture}
}

\newcommand{\arbrebicoloreI}{
\begin{tikzpicture}[scale = 0.18,baseline=(pt_base.base)]

\coordinate (pt_base) at (0,-0.5) ;

\begin{scope}
\node (K) at (-0.625,0.625) {} ;
\node (L) at (0,0) {} ;
\draw[blue] (-0.625,0.625) -- (0,1.25) ;
\draw[blue] (-1.25,1.25) -- (-0.3,0.3) ;
\draw[blue] (1.25,1.25) -- (0.3,0.3) ;
\draw[red] (0,0) -- (-0.3,0.3) ;
\draw[red] (0,0) -- (0.3,0.3) ;
\draw[red] (0,-1.25) -- (0,0) ;

\draw (-1.25,0.3) -- (1.25,0.3) ;
\end{scope}

\begin{scope}[shift = {(1.25,2.1)},blue]
\draw (-0.625,0.625) -- (0,0) ;
\draw (0.625,0.625) -- (0,0) ;
\draw (0,-0.625) -- (0,0) ;
\end{scope}
\end{tikzpicture}
}

\newcommand{\arbrebicoloreJ}{
\begin{tikzpicture}[scale = 0.18,baseline=(pt_base.base)]

\coordinate (pt_base) at (0,-0.5) ;

\begin{scope}
\node (K) at (-0.625,0.625) {} ;
\node (L) at (0,0) {} ;
\draw[blue] (0.625,0.625) -- (0,1.25) ;
\draw[blue] (-1.25,1.25) -- (-0.3,0.3) ;
\draw[blue] (1.25,1.25) -- (0.3,0.3) ;
\draw[red] (0,0) -- (-0.3,0.3) ;
\draw[red] (0,0) -- (0.3,0.3) ;
\draw[red] (0,-1.25) -- (0,0) ;

\draw (-1.25,0.3) -- (1.25,0.3) ;
\end{scope}

\begin{scope}[shift = {(-1.25,2.1)},blue]
\draw (-0.625,0.625) -- (0,0) ;
\draw (0.625,0.625) -- (0,0) ;
\draw (0,-0.625) -- (0,0) ;
\end{scope}
\end{tikzpicture}
}

\newcommand{\arbrebicoloreK}{
\begin{tikzpicture}[scale = 0.1,baseline=(pt_base.base)]

\coordinate (pt_base) at (0,-0.5) ;

\begin{scope}[red]
\draw (-1.6,1.25) -- (0,0) ;
\draw (1.6,1.25) -- (0,0) ;
\draw (0,-1.25) -- (0,0) ;
\end{scope}

\begin{scope}[shift = {(-1.6,2.8)}]
\draw[blue] (-1.25,1.25) -- (0,0) ;
\draw[blue] (1.25,1.25) -- (0,0) ;
\draw[blue] (0,-0.625) -- (0,0) ;
\draw[red] (0,-0.625) -- (0,-1.25) ;
\draw (-1.25,-0.625) -- (1.25,-0.625) ;
\end{scope}

\begin{scope}[shift = {(1.6,2.8)}]
\draw[blue] (-1.25,1.25) -- (0,0) ;
\draw[blue] (1.25,1.25) -- (0,0) ;
\draw[blue] (0,-0.625) -- (0,0) ;
\draw[red] (0,-0.625) -- (0,-1.25) ;
\draw (-1.25,-0.625) -- (1.25,-0.625) ;
\end{scope}
\end{tikzpicture}
}

\newcommand{\arbrebicoloreL}{
\begin{tikzpicture}[scale=0.15,baseline=(pt_base.base)]

\coordinate (pt_base) at (0,-0.5) ;

\draw[blue] (-1.25,1.25) -- (-0.5,0.5) ;
\draw[blue] (1.25,1.25) -- (0.5,0.5) ;

\draw[red] (-0.5,0.5) -- (0,0) ;
\draw[red] (0.5,0.5) -- (0,0) ;
\draw[red] (0,-1.25) -- (0,0) ;

\draw (-1.25,0.5) -- (1.25,0.5) ;

\end{tikzpicture}
}

\newcommand{\arbrebicoloreM}{
\begin{tikzpicture}[scale=0.15,baseline=(pt_base.base)]

\coordinate (pt_base) at (0,-0.5) ;

\draw[blue] (-1.25,1.25) -- (0,0) ;
\draw[blue] (1.25,1.25) -- (0,0) ;

\draw[blue] (0,-0.5) -- (0,0) ;
\draw[red] (0,-0.5) -- (0,-1.25) ;

\draw (-1.25,-0.5) -- (1.25,-0.5) ;

\end{tikzpicture}
}

\newcommand{\arbrebicoloreN}{
\begin{tikzpicture}[scale=0.15,baseline=(pt_base.base)]

\coordinate (pt_base) at (0,-0.5) ;

\draw[blue] (-1.25,1.25) -- (0,0) ;
\draw[blue] (1.25,1.25) -- (0,0) ;

\draw[red] (0,-1.25) -- (0,0) ;

\draw (-1.25,0) -- (1.25,0) ;

\end{tikzpicture}
}

\newcommand{\examplestratumCTBC}{
\begin{tikzpicture}

\begin{scope}[red]
\draw (-0.625,0.625) -- (0,0) ;
\draw (0.625,0.625) -- (0,0) ;
\draw (0,-0.625) -- (0,0) ;
\end{scope}

\begin{scope}[shift = {(-0.625,1.5)}]
\draw[blue] (-0.625,0.625) -- (-0.3,0.3) ;
\draw[blue] (0.625,0.625) -- (0.3,0.3) ;
\draw[red] (-0.3,0.3) -- (0,0) ;
\draw[red] (0.3,0.3) -- (0,0) ;
\draw[red] (0,-0.625) -- (0,0) ;

\draw (-0.625,0.3) -- (0.625,0.3) ;

\draw[densely dotted,>=stealth,<->] (-0.7,0.3) -- (-0.7,0) node[midway,left] {$\lambda$} ;
\end{scope}

\begin{scope}[shift = {(0.625,1.5)}]
\draw[red] (0,-0.625) -- (0,0.3) ;
\draw[blue] (0,0.625) -- (0,0.3) ;

\draw (-0.3,0.3) -- (0.3,0.3) ;
\end{scope}
\end{tikzpicture}}

\newcommand{\examplestratumCTCD}{
\begin{tikzpicture}
\begin{scope}
\node (K) at (-0.625,0.625) {} ;
\node (L) at (0,0) {} ;
\draw[red] (-0.625,0.625) -- (0,1.25) ;
\draw[red] (-1.25,1.25) -- (0,0) ;
\draw[red] (1.25,1.25) -- (0,0) ;
\draw[red] (0,-1.25) -- (0,0) ;

\node[below left = 0.2 pt] at ($(K)!0.5!(L)$) { $l$} ;
\end{scope}

\begin{scope}[shift = {(0,1.9)}]
\draw[blue] (0,0.5) -- (0,0) ;
\draw[red] (0,-0.5) -- (0,0) ;

\draw (-0.6,0) -- (0.6,0) ;
\end{scope}

\begin{scope}[shift = {(1.25,1.9)}]
\draw[blue] (0,0.5) -- (0,0) ;
\draw[red] (0,-0.5) -- (0,0) ;

\draw (-0.3,0) -- (0.6,0) ;
\end{scope}

\begin{scope}[shift = {(-1.25,1.9)}]
\draw[blue] (0,0.5) -- (0,0) ;
\draw[red] (0,-0.5) -- (0,0) ;

\draw (-0.6,0) -- (0.3,0) ;
\end{scope}
\end{tikzpicture}}

\newcommand{\examplestratumCTAB}{
\begin{tikzpicture}
\node (A) at (0,0) {} ;
\node (B) at (0,-1) {} ;
\node (C) at (1.25,1.25) {} ;
\node (D) at (-1.25,1.25) {} ;
\node (E) at (0,1.25) {} ;
\node (F) at (-0.625,0.625) {} ;
\node (G) at (-0.625,0.625) {} ;
\node (H) at (-0.625,0.625) {} ;
\node (I) at (0.625,0.625) {}  ;
\node (J) at (-1.25,0.625) {} ;
\node (K) at (1.25,0.625) {} ;
\node (L) at (1.5,0) {} ;
\node (M) at (1.5,0.625) {} ;

\draw[blue] (D.center) -- (G.center) ;
\draw[blue] (E.center) -- (H.center) ;
\draw[blue] (C.center) -- (I.center) ;

\draw[red] (G.center) -- (F.center) ;
\draw[red] (H.center) -- (F.center) ;
\draw[red] (F.center) -- (A.center) ;
\draw[red] (A.center) -- (I.center) ;
\draw[red] (A.center) -- (B.center) ;

\draw (J.center) -- (K.center) ;

\draw[densely dotted,>=stealth,<->] (L.center) -- (M.center) ;

\node[right] at ($(L)!0.5!(M)$) { $\lambda = - l$} ;
\node[below left] at ($(A)!0.5!(F)$) { $l$} ;
\end{tikzpicture}}

\newcommand{\examplestratumCTABbis}{
\begin{tikzpicture}[scale=0.18,baseline=(pt_base.base)]

\coordinate (pt_base) at (0,-0.5) ;

\node (A) at (0,0) {} ;
\node (B) at (0,-1) {} ;
\node (C) at (1.25,1.25) {} ;
\node (D) at (-1.25,1.25) {} ;
\node (E) at (0,1.25) {} ;
\node (F) at (-0.625,0.625) {} ;
\node (G) at (-0.625,0.625) {} ;
\node (H) at (-0.625,0.625) {} ;
\node (I) at (0.625,0.625) {}  ;
\node (J) at (-1.25,0.625) {} ;
\node (K) at (1.25,0.625) {} ;
\node (L) at (1.5,0) {} ;
\node (M) at (1.5,0.625) {} ;

\draw[blue] (D.center) -- (G.center) ;
\draw[blue] (E.center) -- (H.center) ;
\draw[blue] (C.center) -- (I.center) ;

\draw[red] (G.center) -- (F.center) ;
\draw[red] (H.center) -- (F.center) ;
\draw[red] (F.center) -- (A.center) ;
\draw[red] (A.center) -- (I.center) ;
\draw[red] (A.center) -- (B.center) ;

\draw (J.center) -- (K.center) ;

\end{tikzpicture}}

\newcommand{\examplestratumCTDA}{
\begin{tikzpicture}
\node (A) at (0,0) {} ;
\node (B) at (0,-1) {} ;
\node (C) at (1.25,1.25) {} ;
\node (D) at (-1.25,1.25) {} ;
\node (E) at (0,1.25) {} ;
\node (F) at (-0.625,0.625) {} ;
\node (G) at (-0.625,0.625) {} ;
\node (H) at (-0.625,0.625) {} ;
\node (I) at (0.625,0.625) {}  ;
\node (J) at (-1.25,0.625) {} ;
\node (K) at (1.25,0.625) {} ;
\node (L) at (1.5,0) {} ;
\node (M) at (1.5,0.625) {} ;

\node (N) at (0,0.625) {} ;

\draw[blue] (D.center) -- (G.center) ;
\draw[blue] (E.center) -- (N.center) ;
\draw[blue] (C.center) -- (I.center) ;

\draw[red] (G.center) -- (F.center) ;
\draw[red] (N.center) -- (A.center) ;
\draw[red] (F.center) -- (A.center) ;
\draw[red] (A.center) -- (I.center) ;
\draw[red] (A.center) -- (B.center) ;

\draw (J.center) -- (K.center) ;

\draw[densely dotted,>=stealth,<->] (L.center) -- (M.center) ;

\node[right] at ($(L)!0.5!(M)$) { $\lambda$} ;
\end{tikzpicture}}

\newcommand{\examplestratumCT}{
\begin{tikzpicture}[scale = 1.5]
\node (A) at (0,0) {} ;
\node (B) at (-1.5,1.5) {} ;
\node (C) at (-2,1) {} ;
\node (D) at (-2,0) {} ;
\node (E) at (0,2.3) {} ;
\node (F) at (0.3,0) {} ;
\node (AB) at ($(A)!0.5!(B)$) {} ;
\node (BC) at ($(B)!0.5!(C)$) {} ;
\node (CD) at ($(C)!0.5!(D)$) {} ;
\node (DA) at ($(D)!0.5!(A)$) {} ;
\node (AC) at ($(A)!0.6!(C)$) {} ;

\draw[dotted,fill,opacity = 0.1] (A.center) -- (B.center) -- (C.center) -- (D.center) -- (A.center) ;
\draw (A.center) -- (B.center) ;
\draw[dotted] (B.center) -- (C.center) ;
\draw[dotted] (C.center) -- (D.center) ;
\draw[>=stealth,->] (A.center) -- (E.center) ;
\draw[>=stealth,->] (D.center) -- (F.center) ;

\node[above] at (E) {\tiny $l$} ;
\node[right] at (F) {\tiny $\lambda$} ;

\node[xshift = 0.5 cm , yshift = 0.5 cm , scale = 0.4] at (AB) {\examplestratumCTAB} ;
\node[xshift = -0.5 cm, yshift = 0.5 cm , scale = 0.4] at (BC) {\examplestratumCTBC} ;
\node[xshift = -0.7 cm , scale = 0.3] at (CD) {\examplestratumCTCD} ;
\node[yshift = -0.7 cm , scale = 0.4] at (DA) {\examplestratumCTDA} ;
\node[yshift = -0.1 cm , scale = 0.4] at (AC) {\CTtroisstratD} ;
\end{tikzpicture}}


\usetikzlibrary{positioning}
\usetikzlibrary{calc}


\newcommand{\midarrow}{
\begin{tikzpicture}

\draw (-2,1) -- (0,0) ;
\draw (-2,-1) -- (0,0) ;
\end{tikzpicture}}

\newcommand{\pointbullet}{
\begin{tikzpicture}[very thick]
\draw[fill,black] (0,0) circle (0.5) ;
\end{tikzpicture}}

\newcommand{\cross}{
\begin{tikzpicture}[very thick]
\draw (0.5,0.5) -- (-0.5,-0.5) ;
\draw (-0.5,0.5) -- (0.5,-0.5) ;
\end{tikzpicture}}

\newcommand{\perturb}{
\begin{tikzpicture}[very thick]
\draw (-1,0) -- (1,1) ;
\draw (-1,1) -- (1,-1) ;
\end{tikzpicture}}

\newcommand{\arbredegradientexemple}{
\begin{tikzpicture}[scale=0.8, every node/.style={scale=0.8}]

\begin{scope}

\draw (-2,2) node[above=2 pt]{\small $x_1$} node[above,scale = 0.08]{\cross} -- (-1,1) node{}
node[near end,left = 3 pt]{\tiny $- \nabla f$} node[midway,sloped,allow upside down,scale=0.1]{\midarrow} ;
\draw (-1,1) node{} -- (0,0) node{}
node[near end,left = 3 pt]{\tiny $- \nabla f$} node[midway,sloped,allow upside down,scale=0.1]{\midarrow} ;
\draw (2,2) node[above=2pt]{\small $x_3$} node[above,scale = 0.08]{\cross} -- (0,0) node{}
node[near end,right]{\tiny $- \nabla f$} node[midway,sloped,allow upside down,scale=0.1]{\midarrow} ;
\draw (0,0) node{} -- (0,-1) node[below right  = 2 pt]{\small $y$} node[below ,scale = 0.08]{\cross}
node[near end,left]{\tiny $- \nabla f$} node[midway,sloped,allow upside down,scale=0.1]{\midarrow} ;
\draw (0,2) node[above=2 pt]{\small $x_2$} node[above,scale = 0.08]{\cross} -- (-1,1) node{}
node[near end,right]{\tiny $- \nabla f$} node[midway,sloped,allow upside down,scale=0.1]{\midarrow} ;
\end{scope}

\begin{scope}[xshift = 5 cm]
\draw[->=latex] (-3,0) node{} -- (3,0) node{} node[midway,below = 2pt,text width=6 cm, align = center]{\small Perturbing the gradient vector field around each vertex of the tree} ;
\end{scope}

\begin{scope}[xshift = 10 cm]
\draw (-2,2) node[above=2 pt]{\small $x_1$} node[above,scale = 0.08]{\cross} -- (-1,1) node{}
node[near end,left = 3 pt]{\tiny $- \nabla f$} node[midway,sloped,allow upside down,scale=0.1]{\midarrow} ;
\draw (-1,1) node{} -- (0,0) node{}
node[near end,left = 3 pt]{\tiny $- \nabla f$} node[midway,sloped,allow upside down,scale=0.1]{\midarrow} ;
\draw (2,2) node[above=2pt]{\small $x_3$} node[above,scale = 0.08]{\cross} -- (0,0) node{}
node[near end,right]{\tiny $- \nabla f$} node[midway,sloped,allow upside down,scale=0.1]{\midarrow} ;
\draw (0,0) node{} -- (0,-1) node[below right  = 2 pt]{\small $y$} node[below ,scale = 0.08]{\cross}
node[near end,left]{\tiny $- \nabla f$} node[midway,sloped,allow upside down,scale=0.1]{\midarrow} ;
\draw (0,2) node[above=2 pt]{\small $x_2$} node[above,scale = 0.08]{\cross} -- (-1,1) node{}
node[near end,right]{\tiny $- \nabla f$} node[midway,sloped,allow upside down,scale=0.1]{\midarrow} ;

\node (O) at (0,0) {};
\node (A) at (-1,1) {} ;
\node (a) at ($(A)!0.7!(O)$)  {};
\node (B) at (1,1) {} ;
\node (b) at ($(B)!0.7!(O)$)  {};
\node (C) at (0,-1) {} ;
\node (c) at ($(C)!0.7!(O)$)  {};

\draw [line width = 1.5 pt,red] (a.center) -- (O.center) node[below right,red] {\tiny $ - \nabla f + \mathbb{X}$} ;
\draw [line width = 1.5 pt,red] (b.center) -- (O.center) ;
\draw [line width = 1.5 pt,red] (c.center) -- (O.center) ;

\node (D) at (-2,2) {};
\node (E) at (0,2) {} ;
\node (d) at ($(D)!0.7!(A)$)  {};
\node (e) at ($(E)!0.7!(A)$)  {};
\node (f) at ($(O)!0.7!(A)$)  {};

\draw [line width = 1.5 pt,red] (d.center) -- (A.center) node[below left,red] {\tiny $ - \nabla f + \mathbb{X}$} ;
\draw [line width = 1.5 pt,red] (e.center) -- (A.center) ;
\draw [line width = 1.5 pt,red] (f.center) -- (A.center) ;
\end{scope}

\end{tikzpicture}
}

\newcommand{\trajmorse}{
\begin{tikzpicture}[scale = 0.5, baseline = (pt_base.base)]

\coordinate (pt_base) at (0,-0.2) ;

\node (A) at (-0.25,1) {\small y} ;
\node (B) at (0.25,-1) {\small x} ;
\node[scale = 0.05,yshift = -100 pt] at (A) {\cross} ;
\node[scale = 0.05,yshift = 80 pt] at (B) {\cross} ;

\draw (A) to[out=-90,in=90] (B) node[midway,xshift = -1 pt,rotate = -45,scale=0.1] {\midarrow} ;
\end{tikzpicture}}

\newcommand{\perturbedgradienttreecorrespondence}{
\begin{tikzpicture}[scale = 1.5]
\node (A) at (0,0) {} ;
\node (B) at (0,-1) {} ;
\node (C) at (-2,2) {} ;
\node (D) at (-1,1) {} ;
\node (E) at (0,2) {} ;
\node (F) at (-1.5,1.5) {} ;
\node (G) at (-1,2) {} ;
\node (H) at (2,2) {} ;

\draw (C.center) -- (F.center) -- (D.center) -- (A.center) -- (B.center) ;
\draw (A.center) -- (H.center) ;
\draw (D.center) -- (E.center) ;
\draw (F.center) -- (G.center) ;

\node[below left = 10pt] at ($(D)!0.5!(A)$) { $l_1$} ;
\node[below left = 10pt] at ($(D)!0.5!(F)$) { $l_2$} ;
\node[fill,white] at ($(A)!0.5!(B)$) { $\mathbf{e_0}$} ;
\node[draw] at ($(A)!0.5!(B)$) { $\mathbf{e_0}$} ;
\node[fill,white] at ($(C)!0.5!(F)$) {$\mathbf{e_1}$} ;
\node[draw] at ($(C)!0.5!(F)$) {$\mathbf{e_1}$} ;
\node[fill,white] at ($(G)!0.5!(F)$) {$\mathbf{e_2}$} ;
\node[draw] at ($(G)!0.5!(F)$) {$\mathbf{e_2}$} ;
\node[fill,white] at ($(D)!0.5!(E)$) {$\mathbf{e_3}$} ;
\node[draw] at ($(D)!0.5!(E)$) {$\mathbf{e_3}$} ;
\node[fill,white] at ($(A)!0.5!(H)$) {$\mathbf{e_4}$} ;
\node[draw] at ($(A)!0.5!(H)$) {$\mathbf{e_4}$} ;
\node[fill,white] at ($(D)!0.5!(F)$) {$\mathbf{f}$} ;
\node[draw] at ($(D)!0.5!(F)$) {$\mathbf{f}$} ;
\node[fill,white] at ($(D)!0.5!(A)$) {$\mathbf{g}$} ;
\node[draw] at ($(D)!0.5!(A)$) {$\mathbf{g}$} ;

\end{tikzpicture}
}

\newcommand{\examplebreakingMorseribbontreeun}{
\begin{tikzpicture}[scale = 0.5]
\begin{scope}

\draw (-2,2) node[above=2 pt]{\small $x_1$} node[above,scale = 0.08]{\cross} -- (-1,1) node[midway,sloped,allow upside down,scale=0.1]{\midarrow} ;
\draw (-1,1) node{} -- (0,0) node[midway,sloped,allow upside down,scale=0.1]{\midarrow} ;
\draw (2,2) node[above=2pt]{\small $x_3$} node[above,scale = 0.08]{\cross} -- (0,0) node[midway,sloped,allow upside down,scale=0.1]{\midarrow} ;
\draw (0,0) node{} -- (0,-1) node[below right  = 2 pt]{\small $y$} node[below ,scale = 0.08]{\cross} node[midway,sloped,allow upside down,scale=0.1]{\midarrow} ;
\draw (0,2) node[right=2 pt]{\small $z$} node[above,scale = 0.08]{\cross} -- (-1,1) node[midway,sloped,allow upside down,scale=0.1]{\midarrow} ;
\end{scope}

\begin{scope}[yshift = 2.2 cm]
\draw (0,1.5) node[ above=2 pt]{\small $x_2$} node[above,scale = 0.08]{\cross} -- (0,0) node[midway,sloped,allow upside down,scale=0.1]{\midarrow} ;
\end{scope}

\end{tikzpicture}
}

\newcommand{\examplebreakingMorseribbontreedeux}{
\begin{tikzpicture}[scale = 0.5]
\begin{scope}
\draw (-1.25,1.25) node[right=2 pt]{\small $z$} node[above,scale = 0.08]{\cross} -- (0,0) node[midway,sloped,allow upside down,scale=0.1]{\midarrow} ;
\draw (1.25,1.25) node[right=2 pt]{\small $x_3$} node[above,scale = 0.08]{\cross} -- (0,0) node[midway,sloped,allow upside down,scale=0.1]{\midarrow} ;
\draw (0,0) -- (0,-1.25) node[right=2 pt]{\small $y$} node[above,scale = 0.08]{\cross} node[midway,sloped,allow upside down,scale=0.1]{\midarrow} ;
\end{scope}

\begin{scope}[xshift = -1.25 cm , yshift = 2.7 cm]
\draw (-1.25,1.25) node[right=2 pt]{\small $x_1$} node[above,scale = 0.08]{\cross} -- (0,0) node[midway,sloped,allow upside down,scale=0.1]{\midarrow} ;
\draw (1.25,1.25) node[right=2 pt]{\small $x_2$} node[above,scale = 0.08]{\cross} -- (0,0) node[midway,sloped,allow upside down,scale=0.1]{\midarrow} ;
\draw (0,0) -- (0,-1.25) node[midway,sloped,allow upside down,scale=0.1]{\midarrow} ;
\end{scope}
\end{tikzpicture}}

\newcommand{\exampleperturbationbreakun}{
\begin{tikzpicture}[scale = 0.8]

\begin{scope}[yshift = -0.3cm]
\node (A) at (-2.75,2) {} ;
\node (B) at (-1.5,2) {} ;
\node (C) at (-0.25,2) {} ;
\node (D) at (0.25,2) {} ;
\node (E) at (2.75,2) {} ;
\node (G) at (1.5,1) {} ;
\node (H) at (-1.5,1) {} ;
\node (F) at ($(H)!0.5!(C)$) {} ;
\node (I) at (0,0) {} ;
\node (J) at (0,-1) {} ;

\node at (-2.75,0) {$t$} ;

\draw (A.center) -- (H.center) ;
\draw (B.center) -- (F.center) ;
\draw (C.center) -- (F.center) ;
\draw (F.center) -- (H.center) ;
\draw (D.center) -- (G.center) ;
\draw (E.center) -- (G.center) ;
\draw (H.center) -- (I.center) ;
\draw (G.center) -- (I.center) ;
\draw (I.center) -- (J.center) ;

\node[fill,white,scale=0.5] at ($(F)!0.5!(H)$) { $\mathbf{f_1}$} ;
\node[draw,scale=0.5] at ($(F)!0.5!(H)$) { $\mathbf{f_1}$} ;

\node[fill,white,scale=0.5] at ($(I)!0.5!(H)$) { $\mathbf{f_2}$} ;
\node[draw,scale=0.5] at ($(I)!0.5!(H)$) { $\mathbf{f_2}$} ;

\node[fill,white,scale=0.5] at ($(G)!0.5!(I)$) { $\mathbf{f_3}$} ;
\node[draw,scale=0.5] at ($(G)!0.5!(I)$) { $\mathbf{f_3}$} ;

\node (i) at ($(G)!0.8!(I)$) {} ;
\node (g) at ($(I)!0.8!(G)$) {} ;

\draw [line width = 1.5 pt,red] (i.center) -- (I.center) ;
\draw [line width = 1.5 pt,red] (g.center) -- (G.center) ;

\node[red,xshift = 20 pt, yshift = -17 pt] at ($(G)!0.5!(I)$) { $\mathbb{X}_{t,f_3}$} ;

\end{scope}

\begin{scope}[scale = 0.3,xshift = 15 cm]
\draw[->=latex] (-4,0) node{} -- (4,0) node{} node[midway,below = 2pt,text width=2 cm, align = center]{\small $l_{f_2} \longrightarrow + \infty$} ;
\end{scope}

\begin{scope}[xshift = 10cm,yshift = -0.7 cm]
\node (A) at (-1.5,1) {} ;
\node (B) at (0,1) {} ;
\node (C) at (1.5,1) {} ;
\node (E) at (0,0) {} ;
\node (D) at ($(E)!0.5!(C)$) {} ;
\node (F) at (0,-1) {} ;

\node at (-1.5,0) {$t_1$} ;

\draw (A.center) -- (E.center) ;
\draw (B.center) -- (D.center) ;
\draw (C.center) -- (E.center) ;
\draw (F.center) -- (E.center) ;

\node (e) at ($(D)!0.8!(E)$) {} ;
\node (d) at ($(E)!0.8!(D)$) {} ;

\draw [line width = 1.5 pt,red] (e.center) -- (E.center) ;
\draw [line width = 1.5 pt,red] (d.center) -- (D.center) ;

\node[fill,white,scale=0.5] at ($(D)!0.5!(E)$) { $\mathbf{f_3}$} ;
\node[draw,scale=0.5] at ($(D)!0.5!(E)$) { $\mathbf{f_3}$} ;

\node[red,xshift = 20 pt, yshift = -17 pt] at ($(D)!0.5!(E)$) { $\mathbb{X}_{t,f_3}^{\infty}$} ;
\end{scope}

\begin{scope}[xshift = 8.5cm,yshift = 1.4 cm]
\node (A) at (-1.5,1) {} ;
\node (B) at (0,1) {} ;
\node (C) at (1.5,1) {} ;
\node (E) at (0,0) {} ;
\node (D) at ($(E)!0.5!(C)$) {} ;
\node (F) at (0,-1) {} ;

\node at (-1.5,0) {$t_2$} ;

\draw (A.center) -- (E.center) ;
\draw (B.center) -- (D.center) ;
\draw (C.center) -- (E.center) ;
\draw (F.center) -- (E.center) ;

\node[fill,white,scale=0.5] at ($(D)!0.5!(E)$) { $\mathbf{f_1}$} ;
\node[draw,scale=0.5] at ($(D)!0.5!(E)$) { $\mathbf{f_1}$} ;
\end{scope}

\end{tikzpicture}}

\newcommand{\exampleperturbationbreakdeux}{
\begin{tikzpicture}[scale = 0.8]

\begin{scope}[yshift = -0.3 cm]
\node (A) at (-2.75,2) {} ;
\node (B) at (-1.5,2) {} ;
\node (C) at (-0.25,2) {} ;
\node (D) at (0.25,2) {} ;
\node (E) at (2.75,2) {} ;
\node (G) at (1.5,1) {} ;
\node (H) at (-1.5,1) {} ;
\node (F) at ($(H)!0.5!(C)$) {} ;
\node (I) at (0,0) {} ;
\node (J) at (0,-1) {} ;

\node at (2.25,0) {$t$} ;

\draw (A.center) -- (H.center) ;
\draw (B.center) -- (F.center) ;
\draw (C.center) -- (F.center) ;
\draw (F.center) -- (H.center) ;
\draw (D.center) -- (G.center) ;
\draw (E.center) -- (G.center) ;
\draw (H.center) -- (I.center) ;
\draw (G.center) -- (I.center) ;
\draw (I.center) -- (J.center) ;

\node[fill,white,scale=0.5] at ($(F)!0.5!(H)$) { $\mathbf{f_1}$} ;
\node[draw,scale=0.5] at ($(F)!0.5!(H)$) { $\mathbf{f_1}$} ;

\node[fill,white,scale=0.5] at ($(I)!0.5!(H)$) { $\mathbf{f_2}$} ;
\node[draw,scale=0.5] at ($(I)!0.5!(H)$) { $\mathbf{f_2}$} ;

\node[fill,white,scale=0.5] at ($(G)!0.5!(I)$) { $\mathbf{f_3}$} ;
\node[draw,scale=0.5] at ($(G)!0.5!(I)$) { $\mathbf{f_3}$} ;

\node (h) at ($(I)!0.8!(H)$) {} ;
\node (i) at ($(H)!0.8!(I)$) {} ;

\draw [line width = 1.5 pt,red] (i.center) -- (I.center) ;
\draw [line width = 1.5 pt,red] (h.center) -- (H.center) ;

\node[red,xshift = -20 pt, yshift = -17 pt] at ($(H)!0.5!(I)$) { $\mathbb{X}_{t,f_2}$} ;

\end{scope}

\begin{scope}[scale = 0.3, xshift = 15 cm]
\draw[->=latex] (-4,0) node{} -- (4,0) node{} node[midway,below = 2pt,text width=2 cm, align = center]{\small $l_{f_2} \longrightarrow + \infty$} ;
\end{scope}

\begin{scope}[xshift = 10cm,yshift = -0.7 cm]
\node (A) at (-1.5,1) {} ;
\node (B) at (0,1) {} ;
\node (C) at (1.5,1) {} ;
\node (E) at (0,0) {} ;
\node (D) at ($(E)!0.5!(C)$) {} ;
\node (F) at (0,-1) {} ;

\node at (1.5,0) {$t_1$} ;

\draw (A.center) -- (E.center) ;
\draw (B.center) -- (D.center) ;
\draw (C.center) -- (E.center) ;
\draw (F.center) -- (E.center) ;

\node (a) at ($(A)!0.8!(E)$) {} ;

\draw [line width = 1.5 pt,red] (a.center) -- (E.center) ;

\node[fill,white,scale=0.5] at ($(D)!0.5!(E)$) { $\mathbf{f_3}$} ;
\node[draw,scale=0.5] at ($(D)!0.5!(E)$) { $\mathbf{f_3}$} ;

\node[red,below left] at ($(a)!0.5!(E)$) { $\mathbb{X}_{t,f_2}^{-}$} ;
\end{scope}

\begin{scope}[xshift = 8.5cm,yshift = 1.4 cm]
\node (A) at (-1.5,1) {} ;
\node (B) at (0,1) {} ;
\node (C) at (1.5,1) {} ;
\node (E) at (0,0) {} ;
\node (D) at ($(E)!0.5!(C)$) {} ;
\node (F) at (0,-1) {} ;

\node at (1.5,0) {$t_2$} ;

\draw (A.center) -- (E.center) ;
\draw (B.center) -- (D.center) ;
\draw (C.center) -- (E.center) ;
\draw (F.center) -- (E.center) ;

\node (f) at ($(F)!0.8!(E)$) {} ;

\draw [line width = 1.5 pt,red] (f.center) -- (E.center) ;

\node[fill,white,scale=0.5] at ($(D)!0.5!(E)$) { $\mathbf{f_1}$} ;
\node[draw,scale=0.5] at ($(D)!0.5!(E)$) { $\mathbf{f_1}$} ;

\node[red, left] at ($(f)!0.5!(E)$) { $\mathbb{X}_{t,f_2}^{+}$} ;
\end{scope}

\end{tikzpicture}}


\newcommand{\arbredegradientbicoloreexemple}{
\begin{tikzpicture}[scale = 1.7]

\node (A) at (-2,2) {};
\node (B) at (-0.5,2) {} ;
\node (C) at (0.5,2)  {};
\node (D) at (2,2) {} ;
\node (E) at (-1,1)  {};
\node (F) at (1,1) {} ;
\node (G) at (-0.5,0.5)  {};
\node (H) at (0.5,0.5) {};
\node (I) at (0,0) {} ;
\node (J) at (0,-1)  {};
\node (K) at (-2,0.5) {} ;
\node (L) at (2,0.5)  {};
\node (a) at ($(A)!0.7!(E)$) {} ;
\node (b) at ($(B)!0.7!(E)$)  {};
\node (c) at ($(C)!0.7!(F)$) {};
\node (d) at ($(D)!0.7!(F)$) {} ;
\node (e) at ($(G)!0.7!(E)$)  {};
\node (f) at ($(H)!0.7!(F)$) {} ;
\node (g) at ($(E)!0.7!(G)$)  {};
\node (gg) at ($(I)!0.7!(G)$) {} ;
\node (h) at ($(F)!0.7!(H)$)  {};
\node (hh) at ($(I)!0.7!(H)$) {};
\node (i) at ($(G)!0.7!(I)$) {} ;
\node (ii) at ($(H)!0.7!(I)$)  {};
\node (iii) at ($(J)!0.7!(I)$) {} ;

\draw (A.center) -- (a.center) node[near end,left = 3 pt]{\tiny $- \nabla f$} node[midway,sloped,allow upside down,scale=0.1]{\midarrow} ;
\draw (B.center) -- (b.center) node[near end,right = 3 pt]{\tiny $- \nabla f$} node[midway,sloped,allow upside down,scale=0.1]{\midarrow} ;
\draw (C.center) -- (c.center) node[near end,left = 3 pt]{\tiny $- \nabla f$} node[midway,sloped,allow upside down,scale=0.1]{\midarrow} ;
\draw (D.center) -- (d.center) node[near end,right = 3 pt]{\tiny $- \nabla f$} node[midway,sloped,allow upside down,scale=0.1]{\midarrow} ;

\draw [line width = 1.5 pt,blue] (a.center) -- (E.center) node[left,blue] {\tiny $ - \nabla f + \mathbb{X}^f$} ;
\draw [line width = 1.5 pt,blue] (b.center) -- (E.center) ;
\draw [line width = 1.5 pt,blue] (c.center) -- (F.center) node[right,blue] {\tiny $ - \nabla f + \mathbb{X}^f$} ;
\draw [line width = 1.5 pt,blue] (d.center) -- (F.center)  ;

\draw [line width = 1.5 pt,blue] (e.center) -- (E.center)  ;
\draw [line width = 1.5 pt,blue] (f.center) -- (F.center)  ;

\draw (e.center) -- (g.center) node[near end,left = 5 pt]{\tiny $- \nabla f$} node[midway,sloped,allow upside down,scale=0.1]{\midarrow} ;
\draw (f.center) -- (h.center) node[near end,right = 3 pt]{\tiny $- \nabla f$} node[midway,sloped,allow upside down,scale=0.1]{\midarrow} ;

\draw [line width = 1.5 pt,green!40!gray] (g.center) -- (G.center) ;
\draw [line width = 1.5 pt,green!40!gray] (h.center) -- (H.center)  ;

\node[green!40!gray] at (0,0.65) {\tiny $ - \nabla f + \mathbb{Y}$} ;

\draw [line width = 1.5 pt,green!40!gray] (hh.center) -- (H.center) ;
\draw [line width = 1.5 pt,green!40!gray] (gg.center) -- (G.center) ;

\node[green!40!gray] at (0,0.4) {\tiny $ - \nabla g + \mathbb{Y}$} ;

\draw (gg.center) -- (i.center) node[near end,left = 3 pt]{\tiny $- \nabla g$} node[midway,sloped,allow upside down,scale=0.1]{\midarrow} ;
\draw (hh.center) -- (ii.center) node[near end,right = 3 pt]{\tiny $- \nabla g$} node[midway,sloped,allow upside down,scale=0.1]{\midarrow} ;

\draw [line width = 1.5 pt,red] (i.center) -- (I.center) node[below left,red] {\tiny $ - \nabla g + \mathbb{X}^g$} ;
\draw [line width = 1.5 pt,red] (ii.center) -- (I.center) ;
\draw [line width = 1.5 pt,red] (iii.center) -- (I.center) ;

\draw (iii.center) -- (J.center) node[near end,left = 3 pt]{\tiny $- \nabla g$} node[midway,sloped,allow upside down,scale=0.1]{\midarrow} ;

\draw (K) -- (L) ;

\node[above=2 pt] at (A) {\small $x_1$}  ; 
\node[above,scale = 0.1] at (A) {\cross}  ;

\node[above=2 pt] at (B) {\small $x_2$}  ; 
\node[above,scale = 0.1] at (B) {\cross}  ;

\node[above=2 pt] at (C) {\small $x_3$}  ; 
\node[above,scale = 0.1] at (C) {\cross}  ;

\node[above=2 pt] at (D) {\small $x_4$}  ; 
\node[above,scale = 0.1] at (D) {\cross}  ;

\node[below=2 pt] at (J) {\small $y$}  ; 
\node[below,scale = 0.1] at (J) {\cross}  ;

\end{tikzpicture}}

\newcommand{\exampleperturbationCTbreakun}{\begin{tikzpicture}

\begin{scope}[yshift = -0.3cm]

\node (A) at (-1.25,1.25) {} ;
\node (B) at (0,1.25) {} ;
\node (C) at (1.25,1.25) {} ;
\node (D) at (-0.625,0.625) {} ;
\node (E) at (0,0) {} ;
\node (F) at (0,-1) {} ;

\node (G) at (-0.3,0.3) {} ;
\node (H) at (0.3,0.3) {} ;

\node (I) at (-1.25,0.3) {} ;
\node (J) at (1.25,0.3) {} ;

\node at (-1.5,-1) {$t_g$} ;

\draw (A.center) -- (D.center) ;
\draw (B.center) -- (D.center) ;
\draw (D.center) -- (G.center) ;
\draw (G.center) -- (E.center) ;
\draw (C.center) -- (H.center) ;
\draw (H.center) -- (E.center) ;
\draw (E.center) -- (F.center) ;

\draw (I.center) -- (J.center) ;

\node (d) at ($(A)!0.7!(D)$) {} ;

\draw [line width = 1.5 pt,green!40!gray] (d.center) -- (D.center) ;

\node[green!40!gray,xshift = -20 pt, yshift = -10 pt] at ($(A)!0.5!(D)$) { $\mathbb{Y}_{t_g,e_c}$} ;

\end{scope}

\begin{scope}[scale = 0.3,xshift = 12 cm]
\draw[->=latex] (-4,0) node{} -- (4,0) node{} node[midway,below = 2pt,align = center]{$\lim \textcolor{green!40!gray}{\mathbb{Y}_{t_g,e_c}} = \textcolor{blue}{\mathbb{X}^f_{t^2,e_c}}$} ;
\end{scope}

\begin{scope}[xshift = 7cm,yshift = -0.7 cm,scale = 0.5]

\begin{scope}
\node (A) at (-1.25,1.25) {} ;
\node (B) at (1.25,1.25) {} ;
\node (C) at (0,0) {} ;
\node (D) at (0,-1) {} ;
\node (E) at (-1.25,0.5) {};
\node (F) at (1.25,0.5) {};

\draw (A.center) -- (C.center) ;
\draw (B.center) -- (C.center) ;
\draw (D.center) -- (C.center) ;

\draw (E.center) -- (F.center) ;

\node at (-1.5,-1) {$t^1_g$} ;
\end{scope}

\begin{scope}[shift = {(-1.25,2.6)}]
\node (A) at (-1.25,1.25) {} ;
\node (B) at (1.25,1.25) {} ;
\node (C) at (0,0) {} ;
\node (D) at (0,-1) {} ;

\node (c) at ($(A)!0.7!(C)$) {} ;

\draw (A.center) -- (C.center) ;
\draw (B.center) -- (C.center) ;
\draw (D.center) -- (C.center) ;

\node at (-1.5,-0.5) {$t^2$} ;

\draw [line width = 1.5 pt,blue] (c.center) -- (C.center) ;

\end{scope}

\end{scope}

\end{tikzpicture}}

\newcommand{\exampleperturbationCTbreakdeux}{\begin{tikzpicture}

\begin{scope}[yshift = -0.3cm]

\node (A) at (-1.25,1.25) {} ;
\node (B) at (0,1.25) {} ;
\node (C) at (1.25,1.25) {} ;
\node (D) at (-0.625,0.625) {} ;
\node (E) at (0,0) {} ;
\node (F) at (0,-1) {} ;

\node (G) at (-0.3,0.3) {} ;
\node (H) at (0.3,0.3) {} ;

\node (I) at (-1.25,0.3) {} ;
\node (J) at (1.25,0.3) {} ;

\node at (-1.5,-1) {$t_g$} ;

\draw (A.center) -- (D.center) ;
\draw (B.center) -- (D.center) ;
\draw (D.center) -- (G.center) ;
\draw (G.center) -- (E.center) ;
\draw (C.center) -- (H.center) ;
\draw (H.center) -- (E.center) ;
\draw (E.center) -- (F.center) ;

\node (g) at ($(E)!0.7!(G)$) {} ;
\node (e) at ($(G)!0.7!(E)$) {} ;

\draw [line width = 1.5 pt,green!40!gray] (g.center) -- (G.center) ;
\draw [line width = 1.5 pt,green!40!gray] (e.center) -- (E.center) ;

\node[green!40!gray,xshift = -20 pt, yshift = -17 pt] at ($(E)!0.5!(G)$) { $\mathbb{Y}_{t_g,e_c}$} ;

\draw (I.center) -- (J.center) ;

\end{scope}

\begin{scope}[scale = 0.3,xshift = 12 cm]
\draw[->=latex] (-4,0) node{} -- (4,0) node{} node[midway,below = 2pt, align = center]{$\lim \textcolor{green!40!gray}{\mathbb{Y}_{t_g,e_c}} = \textcolor{green!40!gray}{\mathbb{Y}_{t^1_g,e_c}}$} ;
\end{scope}

\begin{scope}[xshift = 7cm,yshift = -0.7 cm,scale = 0.5]

\begin{scope}
\node (A) at (-1.25,1.25) {} ;
\node (B) at (1.25,1.25) {} ;
\node (C) at (0,0) {} ;
\node (D) at (0,-1) {} ;
\node (E) at (-1.25,0.5) {};
\node (F) at (1.25,0.5) {};
\node (G) at (-0.5,0.5) {} ;

\draw (A.center) -- (C.center) ;
\draw (B.center) -- (C.center) ;
\draw (D.center) -- (C.center) ;

\node (g) at ($(C)!0.7!(G)$) {} ;
\node (c) at ($(G)!0.7!(C)$) {} ;

\draw [line width = 1.5 pt,green!40!gray] (c.center) -- (C.center) ;
\draw [line width = 1.5 pt,green!40!gray] (g.center) -- (G.center) ;

\draw (E.center) -- (F.center) ;

\node at (-1.5,-1) {$t^1_g$} ;
\end{scope}

\begin{scope}[shift = {(-1.25,2.6)}]
\node (A) at (-1.25,1.25) {} ;
\node (B) at (1.25,1.25) {} ;
\node (C) at (0,0) {} ;
\node (D) at (0,-1) {} ;

\draw (A.center) -- (C.center) ;
\draw (B.center) -- (C.center) ;
\draw (D.center) -- (C.center) ;

\node at (-1.5,-0.5) {$t^2$} ;

\end{scope}

\end{scope}

\end{tikzpicture}}

\newcommand{\exampleperturbationCTbreaktrois}{\begin{tikzpicture}

\begin{scope}[yshift = -0.3cm]

\node (A) at (-1.25,1.25) {} ;
\node (B) at (0,1.25) {} ;
\node (C) at (1.25,1.25) {} ;
\node (D) at (-0.625,0.625) {} ;
\node (E) at (0,0) {} ;
\node (F) at (0,-1) {} ;

\node (G) at (-0.3,0.3) {} ;
\node (H) at (0.3,0.3) {} ;

\node (I) at (-1.25,0.3) {} ;
\node (J) at (1.25,0.3) {} ;

\node at (-1.5,-1) {$t_g$} ;

\draw (A.center) -- (D.center) ;
\draw (B.center) -- (D.center) ;
\draw (D.center) -- (G.center) ;
\draw (G.center) -- (E.center) ;
\draw (C.center) -- (H.center) ;
\draw (H.center) -- (E.center) ;
\draw (E.center) -- (F.center) ;

\node (g) at ($(D)!0.7!(G)$) {} ;
\node (d) at ($(G)!0.7!(D)$) {} ;

\draw [line width = 1.5 pt,green!40!gray] (g.center) -- (G.center) ;
\draw [line width = 1.5 pt,green!40!gray] (d.center) -- (D.center) ;

\node[green!40!gray,xshift = -20 pt, yshift = -17 pt] at ($(D)!0.5!(G)$) { $\mathbb{Y}_{t_g,e_c}$} ;

\draw (I.center) -- (J.center) ;

\end{scope}

\begin{scope}[scale = 0.3,xshift = 12 cm]
\draw[->=latex] (-4,0) node{} -- (4,0) node{} node[midway,below = 20pt, align = center]{$\lim _{t^1_g}\textcolor{green!40!gray}{\mathbb{Y}_{t_g,e_c}} = \textcolor{green!40!gray}{\mathbb{Y}_{t^1_g,e_c}}$} node{} node[midway,below = 2pt, align = center]{$\lim _{t^2}\textcolor{green!40!gray}{\mathbb{Y}_{t_g,e_c}} = \textcolor{blue}{\mathbb{X}^f_{t^2,e_c}}$} ;
\end{scope}

\begin{scope}[xshift = 7cm,yshift = -0.7 cm,scale = 0.5]

\begin{scope}
\node (A) at (-1.25,1.25) {} ;
\node (B) at (1.25,1.25) {} ;
\node (C) at (0,0) {} ;
\node (D) at (0,-1) {} ;
\node (E) at (-1.25,0.5) {};
\node (F) at (1.25,0.5) {};
\node (G) at (-0.5,0.5) {} ;

\draw (A.center) -- (C.center) ;
\draw (B.center) -- (C.center) ;
\draw (D.center) -- (C.center) ;

\node (g) at ($(A)!0.7!(G)$) {} ;

\draw [line width = 1.5 pt,green!40!gray] (g.center) -- (G.center) ;

\draw (E.center) -- (F.center) ;

\node at (-1.5,-1) {$t^1_g$} ;
\end{scope}

\begin{scope}[shift = {(-1.25,2.6)}]
\node (A) at (-1.25,1.25) {} ;
\node (B) at (1.25,1.25) {} ;
\node (C) at (0,0) {} ;
\node (D) at (0,-1) {} ;

\draw (A.center) -- (C.center) ;
\draw (B.center) -- (C.center) ;
\draw (D.center) -- (C.center) ;

\node (c) at ($(D)!0.7!(C)$) {} ;

\draw [line width = 1.5 pt,blue] (c.center) -- (C.center) ;

\node at (-1.5,-0.5) {$t^2$} ;

\end{scope}

\end{scope}

\end{tikzpicture}}

\newcommand{\exampleperturbationCTbreakquatre}{\begin{tikzpicture}

\begin{scope}[yshift = -0.3cm]

\node (A) at (-1.25,1.25) {} ;
\node (B) at (0,1.25) {} ;
\node (C) at (1.25,1.25) {} ;
\node (D) at (-0.625,0.625) {} ;
\node (E) at (0,0) {} ;
\node (F) at (0,-1) {} ;

\node (G) at (-0.3,0.3) {} ;
\node (H) at (0.3,0.3) {} ;

\node (I) at (-1.25,0.3) {} ;
\node (J) at (1.25,0.3) {} ;

\node at (-1.5,-1) {$t_g$} ;

\draw (A.center) -- (D.center) ;
\draw (B.center) -- (D.center) ;
\draw (D.center) -- (G.center) ;
\draw (G.center) -- (E.center) ;
\draw (C.center) -- (H.center) ;
\draw (H.center) -- (E.center) ;
\draw (E.center) -- (F.center) ;

\node (d) at ($(A)!0.7!(D)$) {} ;

\draw [line width = 1.5 pt,green!40!gray] (d.center) -- (D.center) ;

\node[green!40!gray,xshift = -20 pt, yshift = -17 pt] at ($(D)!0.5!(A)$) { $\mathbb{Y}_{t_g,e_c}$} ;

\draw (I.center) -- (J.center) ;

\end{scope}

\begin{scope}[scale = 0.3,xshift = 12 cm]
\draw[->=latex] (-4,0) node{} -- (4,0) node{} node[midway,below = 2pt, align = center]{$\lim \textcolor{green!40!gray}{\mathbb{Y}_{t_g,e_c}} = \textcolor{green!40!gray}{\mathbb{Y}_{t^1_g,e_c}}$} ;
\end{scope}

\begin{scope}[xshift = 7cm,yshift = -0.7 cm,scale = 0.5]

\begin{scope}
\node (A) at (-1.25,1.25) {} ;
\node (B) at (1.25,1.25) {} ;
\node (C) at (0,0) {} ;
\node (D) at (0,-1) {} ;

\draw (A.center) -- (C.center) ;
\draw (B.center) -- (C.center) ;
\draw (D.center) -- (C.center) ;

\node[below left] at (C) {$t^1$} ;

\end{scope}

\begin{scope}[shift = {(-1.25,2.6)}]
\node (A) at (-1.25,1.25) {} ;
\node (B) at (1.25,1.25) {} ;
\node (C) at (0,0) {} ;
\node (D) at (0,-1) {} ;
\node (E) at (-1.25,-0.5) {};
\node (F) at (1.25,-0.5) {};
\node (G) at (-0.5,-0.5) {} ;

\draw (A.center) -- (C.center) ;
\draw (B.center) -- (C.center) ;
\draw (D.center) -- (C.center) ;

\node (c) at ($(A)!0.7!(C)$) {} ;

\draw [line width = 1.5 pt,green!40!gray] (c.center) -- (C.center) ;

\draw (E.center) -- (F.center) ;

\node[below left] at (E) {$t^1_g$} ;

\end{scope}

\begin{scope}[shift = {(1.25,2.6)}]
\node (A) at (0,1.25) {} ;
\node (B) at (0,-1) {} ;
\node (C) at (-0.6,-0.5) {};
\node (D) at (0.6,-0.5) {};
\node (E) at (0,-0.5) {} ;

\draw (A.center) -- (B.center) ;

\draw (C.center) -- (D.center) ;

\node[below right] at (D) {$t^2_g$} ;

\end{scope}

\end{scope}

\end{tikzpicture}}

\newcommand{\exampleperturbationCTbreakcinq}{\begin{tikzpicture}

\begin{scope}[yshift = -0.3cm]

\node (A) at (-1.25,1.25) {} ;
\node (B) at (0,1.25) {} ;
\node (C) at (1.25,1.25) {} ;
\node (D) at (-0.625,0.625) {} ;
\node (E) at (0,0) {} ;
\node (F) at (0,-1) {} ;

\node (G) at (-0.3,0.3) {} ;
\node (H) at (0.3,0.3) {} ;

\node (I) at (-1.25,0.3) {} ;
\node (J) at (1.25,0.3) {} ;

\node at (-1.5,-1) {$t_g$} ;

\draw (A.center) -- (D.center) ;
\draw (B.center) -- (D.center) ;
\draw (D.center) -- (G.center) ;
\draw (G.center) -- (E.center) ;
\draw (C.center) -- (H.center) ;
\draw (H.center) -- (E.center) ;
\draw (E.center) -- (F.center) ;

\node (e) at ($(F)!0.7!(E)$) {} ;

\draw [line width = 1.5 pt,green!40!gray] (e.center) -- (E.center) ;

\node[green!40!gray,xshift = -20 pt] at ($(F)!0.5!(E)$) { $\mathbb{Y}_{t_g,e_c}$} ;

\draw (I.center) -- (J.center) ;

\end{scope}

\begin{scope}[scale = 0.3,xshift = 12 cm]
\draw[->=latex] (-4,0) node{} -- (4,0) node{} node[midway,below = 2pt, align = center]{$\lim \textcolor{green!40!gray}{\mathbb{Y}_{t_g,e_c}} = \textcolor{red}{\mathbb{X}^g_{t^0,e_c}}$} ;
\end{scope}

\begin{scope}[xshift = 7cm,yshift = -0.7 cm,scale = 0.5]

\begin{scope}
\node (A) at (-1.25,1.25) {} ;
\node (B) at (1.25,1.25) {} ;
\node (C) at (0,0) {} ;
\node (D) at (0,-1) {} ;

\draw (A.center) -- (C.center) ;
\draw (B.center) -- (C.center) ;
\draw (D.center) -- (C.center) ;

\node[below left] at (C) {$t^0$} ;

\node (c) at ($(D)!0.7!(C)$) {} ;

\draw [line width = 1.5 pt,red] (c.center) -- (C.center) ;

\end{scope}

\begin{scope}[shift = {(-1.25,2.6)}]
\node (A) at (-1.25,1.25) {} ;
\node (B) at (1.25,1.25) {} ;
\node (C) at (0,0) {} ;
\node (D) at (0,-1) {} ;
\node (E) at (-1.25,-0.5) {};
\node (F) at (1.25,-0.5) {};
\node (G) at (-0.5,-0.5) {} ;

\draw (A.center) -- (C.center) ;
\draw (B.center) -- (C.center) ;
\draw (D.center) -- (C.center) ;

\draw (E.center) -- (F.center) ;

\node[below left] at (E) {$t^1_g$} ;

\end{scope}

\begin{scope}[shift = {(1.25,2.6)}]
\node (A) at (0,1.25) {} ;
\node (B) at (0,-1) {} ;
\node (C) at (-0.6,-0.5) {};
\node (D) at (0.6,-0.5) {};
\node (E) at (0,-0.5) {} ;

\draw (A.center) -- (B.center) ;

\draw (C.center) -- (D.center) ;

\node[below right] at (D) {$t^2_g$} ;

\end{scope}

\end{scope}

\end{tikzpicture}}

\newcommand{\exampleperturbationCTbreaksix}{\begin{tikzpicture}

\begin{scope}[yshift = -0.3cm]

\node (A) at (-1.25,1.25) {} ;
\node (B) at (0,1.25) {} ;
\node (C) at (1.25,1.25) {} ;
\node (D) at (-0.625,0.625) {} ;
\node (E) at (0,0) {} ;
\node (F) at (0,-1) {} ;

\node (G) at (-0.3,0.3) {} ;
\node (H) at (0.3,0.3) {} ;

\node (I) at (-1.25,0.3) {} ;
\node (J) at (1.25,0.3) {} ;

\node at (-1.5,-1) {$t_g$} ;

\draw (A.center) -- (D.center) ;
\draw (B.center) -- (D.center) ;
\draw (D.center) -- (G.center) ;
\draw (G.center) -- (E.center) ;
\draw (C.center) -- (H.center) ;
\draw (H.center) -- (E.center) ;
\draw (E.center) -- (F.center) ;

\node (h) at ($(E)!0.7!(H)$) {} ;
\node (e) at ($(H)!0.7!(E)$) {} ;

\draw [line width = 1.5 pt,green!40!gray] (e.center) -- (E.center) ;
\draw [line width = 1.5 pt,green!40!gray] (h.center) -- (H.center) ;

\node[green!40!gray,xshift = 20 pt, yshift = -10 pt] at ($(H)!0.5!(E)$) { $\mathbb{Y}_{t_g,e_c}$} ;

\draw (I.center) -- (J.center) ;

\end{scope}

\begin{scope}[scale = 0.3,xshift = 12 cm]
\draw[->=latex] (-4,0) node{} -- (4,0) node{} node[midway,below = 20pt, align = center]{$\lim_{t^1} \textcolor{green!40!gray}{\mathbb{Y}_{t_g,e_c}} = \textcolor{red}{\mathbb{X}^g_{t^0,e_c}}$} node[midway,below = 2pt, align = center]{$\lim_{t^2_g} \textcolor{green!40!gray}{\mathbb{Y}_{t_g,e_c}} = \textcolor{green!40!gray}{\mathbb{Y}_{t^2_g,e_c}}$} ;
\end{scope}

\begin{scope}[xshift = 7cm,yshift = -0.7 cm,scale = 0.5]

\begin{scope}
\node (A) at (-1.25,1.25) {} ;
\node (B) at (1.25,1.25) {} ;
\node (C) at (0,0) {} ;
\node (D) at (0,-1) {} ;

\draw (A.center) -- (C.center) ;
\draw (B.center) -- (C.center) ;
\draw (D.center) -- (C.center) ;

\node[below left] at (C) {$t^0$} ;

\node (c) at ($(B)!0.7!(C)$) {} ;

\draw [line width = 1.5 pt,red] (c.center) -- (C.center) ;

\end{scope}

\begin{scope}[shift = {(-1.25,2.6)}]
\node (A) at (-1.25,1.25) {} ;
\node (B) at (1.25,1.25) {} ;
\node (C) at (0,0) {} ;
\node (D) at (0,-1) {} ;
\node (E) at (-1.25,-0.5) {};
\node (F) at (1.25,-0.5) {};
\node (G) at (-0.5,-0.5) {} ;

\draw (A.center) -- (C.center) ;
\draw (B.center) -- (C.center) ;
\draw (D.center) -- (C.center) ;

\draw (E.center) -- (F.center) ;

\node[below left] at (E) {$t^1_g$} ;

\end{scope}

\begin{scope}[shift = {(1.25,2.6)}]
\node (A) at (0,1.25) {} ;
\node (B) at (0,-1) {} ;
\node (C) at (-0.6,-0.5) {};
\node (D) at (0.6,-0.5) {};
\node (E) at (0,-0.5) {} ;

\draw (A.center) -- (B.center) ;

\node (e) at ($(B)!0.3!(E)$) {} ;

\draw [line width = 1.5 pt,green!40!gray] (e.center) -- (E.center) ;

\draw (C.center) -- (D.center) ;

\node[below right] at (D) {$t^2_g$} ;

\end{scope}

\end{scope}

\end{tikzpicture}}

\newcommand{\hmodspaceAun}{
\begin{tikzpicture}

\draw (0,-1.25) -- (0,1.25) ;
\node[xshift=2pt,scale = 0.1] (A) at (0,0) {\pointbullet} ;

\end{tikzpicture}
}

\newcommand{\hmodspaceAdeux}{
\begin{tikzpicture}

\begin{scope}
\draw (0,-1.25) -- (0,1.25) ;
\node[xshift=-2pt,scale = 0.1] (A) at (0,0) {\pointbullet} ;
\end{scope}

\begin{scope}[shift = {(0,3)}]
\draw (0,-1.25) -- (0,1.25) ;
\node[xshift=-2pt,scale = 0.1] (A) at (0,0) {\pointbullet} ;
\end{scope}

\end{tikzpicture}
}

\newcommand{\hmodspaceAtrois}{
\begin{tikzpicture}

\draw (0,-2) -- (0,2) ;
\node[scale = 0.1] (A) at (0,0.75) {\pointbullet} ;
\node[scale = 0.1] (A) at (0,-0.75) {\pointbullet} ;
\node[xshift = -15pt,yshift = -7 pt] at (0,0) {\huge $l$} ;

\end{tikzpicture}
}

\newcommand{\hmodspaceA}{
\begin{tikzpicture}[ baseline = (pt_base.base),scale = 2]
\coordinate (pt_base) at (0,0) ;
\node[scale = 0.1] (A) at (-1,0) {\pointbullet} ;
\node[scale = 0.1] (B) at (1,0) {\pointbullet} ;
\node[above left = 2 pt,scale = 0.3] at (A) {\hmodspaceAun} ;
\node[above right = 2 pt,scale = 0.3] at (B) {\hmodspaceAdeux} ;
\node[scale = 0.1] (A) at (-1,0) {\pointbullet} ;
\node[scale = 0.1] (B) at (1,0) {\pointbullet} ;
\draw (A.center) -- (B.center) ;
\node[above,scale=0.4] at (0,0) {\hmodspaceAtrois} ;
\end{tikzpicture}
}

\newcommand{\hmodspaceBun}{
\begin{tikzpicture}

\draw[blue] (0,-1.25) -- (0,1.25) ;
\draw (-0.5,0) -- (0.5,0) ;

\end{tikzpicture}
}

\newcommand{\hmodspaceBdeux}{
\begin{tikzpicture}

\begin{scope}[shift = {(0,3)}]
\draw[red] (0,-1.25) -- (0,0) ;
\draw[blue] (0,1.25) -- (0,0) ;
\draw (-0.5,0) -- (0.5,0) ;
\end{scope}

\begin{scope}
\draw[red] (0,1.25) -- (0,0) ;
\draw[blue] (0,-1.25) -- (0,0) ;
\draw (-0.5,0) -- (0.5,0) ;
\end{scope}

\end{tikzpicture}
}

\newcommand{\hmodspaceBtrois}{
\begin{tikzpicture}

\draw[blue] (0,-2) -- (0,-0.75) ;
\draw[blue] (0,2) -- (0,0.75) ;
\draw (-0.5,-0.75) -- (0.5,-0.75) ;
\draw (-0.5,0.75) -- (0.5,0.75) ;
\draw[red] (0,-0.75) -- (0,0.75) ;

\node (K) at (0.7,-0.75) {} ;
\node (J) at (0.7,0.75) {} ;
\draw[densely dotted,>=stealth,<->] (K.center) -- (J) ;
\node[right] at ($(K)!0.5!(J)$) {\huge $l$} ;

\end{tikzpicture}
}

\newcommand{\ophomotopie}{
\begin{tikzpicture}[scale=0.12, baseline = (pt_base.base)]

\coordinate (pt_base) at (0,-0.6) ;

\draw[blue] (0,-2) -- (0,-0.75) ;
\draw[blue] (0,2) -- (0,0.75) ;
\draw (-0.5,-0.75) -- (0.5,-0.75) ;
\draw (-0.5,0.75) -- (0.5,0.75) ;
\draw[red] (0,-0.75) -- (0,0.75) ;

\end{tikzpicture}
}

\newcommand{\hmodspaceB}{
\begin{tikzpicture}[ baseline = (pt_base.base),scale = 2]
\coordinate (pt_base) at (0,0) ;
\node[scale = 0.1] (A) at (-1,0) {\pointbullet} ;
\node[scale = 0.1] (B) at (1,0) {\pointbullet} ;
\node[above left = 2 pt,scale = 0.3] at (A) {\hmodspaceBun} ;
\node[above right = 2 pt,scale = 0.3] at (B) {\hmodspaceBdeux} ;
\node[scale = 0.1] (A) at (-1,0) {\pointbullet} ;
\node[scale = 0.1] (B) at (1,0) {\pointbullet} ;
\draw (A.center) -- (B.center) ;
\node[above,scale=0.4] at (0,0) {\hmodspaceBtrois} ;
\end{tikzpicture}
}

\newcommand{\associaedrelodaydeux}{
\begin{tikzpicture}[scale = 0.7]
\draw[>=stealth,->] (0,0) -- (4,0) ;
\draw[>=stealth,->] (0,0) -- (0,4) ;
\draw[opacity=0.1] (0,3) -- (3,0) ;
\draw[opacity = 0.3] (1,2)--(2,1) ;
\node[scale = 0.05,opacity=0.7] at (1,2) {\pointbullet} ;
\node[scale = 0.05,,opacity=0.7] at (2,1) {\pointbullet} ;
\end{tikzpicture}
}

\newcommand{\associaedrelodaytrois}{
\begin{tikzpicture}[scale = 0.4]
\draw[>=stealth,->] (0,0,0) -- (7,0,0) ;
\draw[>=stealth,->] (0,0,0) -- (0,7,0) ;
\draw[>=stealth,->] (0,0,0) -- (0,0,7) ;
\draw[dotted] (6,0,0) -- (0,6,0) -- (0,0,6) -- cycle ;
\draw[fill,opacity=0.1] (6,0,0) -- (0,6,0) -- (0,0,6) -- cycle ;
\draw[fill,opacity = 0.3] (1,2,3)--(1,4,1)--(3,2,1)--(3,1,2)--(2,1,3)--cycle; 
\draw (1,2,3)--(1,4,1)--(3,2,1)--(3,1,2)--(2,1,3)--cycle; 
\end{tikzpicture}
}

\newcommand{\multiplaedreforceylodaytrois}{
\begin{tikzpicture}[scale = 0.7]
\draw[>=stealth,->] (0,0) -- (0,5) ;
\draw[>=stealth,->] (0,0) -- (5,0) ;
\draw[fill,opacity=0.1] (2,4) -- (1,4) -- (1,2) -- (2,1) -- (4,1) -- (4,2) -- cycle ;
\draw (2,4) -- (1,4) -- (1,2) -- (2,1) -- (4,1) -- (4,2) -- cycle ;
\end{tikzpicture}
}

\newcommand{\multiplaedreforceylodayquatre}{
\begin{tikzpicture}[scale = 0.4]
\draw[>=stealth,->,opacity=0.5] (0,0,0) -- (0,0,8) ;
\draw[>=stealth,->,opacity=0.5] (0,0,0) -- (0,8,0) ;
\draw[>=stealth,->,opacity=0.5] (0,0,0) -- (8,0,0) ;

\draw (2,8,2) -- (1,8,2) -- (1,8,1) -- (1,4,1) ;

\draw (1,8,1) -- (2,8,1) -- (2,8,2) ;

\draw (6,4,2) -- (6,4,1) -- (6,2,1) -- (3,2,1) ; 

\draw (2,4,6) -- (1,4,6) -- (1,2,6) -- (1,2,3) ;

\draw (6,2,4) -- (6,1,4) -- (6,1,2) -- (3,1,2) ;

\draw (4,2,6) -- (4,1,6) -- (2,1,6) -- (2,1,3) ;

\draw (6,4,2) -- (2,8,2) -- (2,4,6) -- (4,2,6) -- (6,2,4) -- cycle ;

\draw (1,8,2) -- (1,4,6) ;

\draw (2,8,1) -- (6,4,1) ;

\draw (3,2,1) -- (1,4,1) -- (1,2,3) -- (2,1,3) -- (3,1,2) -- cycle ;

\draw (1,2,6) -- (2,1,6) ;

\draw (6,1,4) -- (4,1,6) ;

\draw (6,1,2) -- (6,2,1) ; 
\end{tikzpicture}
}

\newcommand{\coveringrelationB}{
\begin{tikzpicture}[ baseline = (pt_base.base),scale = 0.5]
\coordinate (pt_base) at (0,-0.25) ;
\draw (-0.625,0.625) -- (0,1.25) node[above]{\small $t_2$}  ;
\draw (-1.25,1.25) node[above]{\small $t_1$}  -- (0,0) ;
\draw (1.25,1.25) node[above]{\small $t_3$}  -- (0,0) ;
\draw (0,-1.25) node[below]{\small $t_4$}  -- (0,0) ;
\end{tikzpicture}}

\newcommand{\coveringrelationA}{
\begin{tikzpicture}[ baseline = (pt_base.base),scale = 0.5]
\coordinate (pt_base) at (0,-0.25) ;
\draw (0.625,0.625) -- (0,1.25) node[above]{\small $t_2$}   ;
\draw (-1.25,1.25) node[above]{\small $t_1$}  -- (0,0) ;
\draw (1.25,1.25) node[above]{\small $t_3$}  -- (0,0) ;
\draw (0,-1.25) node[below]{\small $t_4$}  -- (0,0) ;
\end{tikzpicture}}

\newcommand{\coveringrelationBgaugedun}{
\begin{tikzpicture}[ baseline = (pt_base.base),scale = 0.5]
\coordinate (pt_base) at (0,-0.25) ;
\draw[blue] (-1.25,1.25) node[above,blue]{\small $t_1$}  -- (-0.5,0.5) ;
\draw[blue] (1.25,1.25) node[above,blue]{\small $t_2$}  -- (0.5,0.5) ;
\draw[red] (-0.5,0.5) -- (0,0) ; 
\draw[red] (0.5,0.5) -- (0,0) ; 
\draw (-1,0.5) -- (1 , 0.5) ; 
\draw[red] (0,-1.25) node[below,red]{\small $t_3$}  -- (0,0) ;
\end{tikzpicture}}

\newcommand{\coveringrelationAgaugedun}{
\begin{tikzpicture}[ baseline = (pt_base.base),scale = 0.5]
\coordinate (pt_base) at (0,-0.25) ;
\draw[blue] (-1.25,1.25) node[above,blue]{\small $t_1$}  -- (0,0) ;
\draw[blue] (1.25,1.25) node[above,blue]{\small $t_2$}  -- (0,0) ;
\draw[red] (0,-1.25) node[below,red]{\small $t_3$}  -- (0,-0.5) ;
\draw[blue] (0,-0.5) -- (0,0) ;
\draw (-0.5 , -0.5) -- (0.5 , -0.5) ; 
\end{tikzpicture}}

\newcommand{\coveringrelationBgaugeddeux}{
\begin{tikzpicture}[ baseline = (pt_base.base),scale = 0.5]
\coordinate (pt_base) at (0,-0.25) ;
\draw[red] (-0.625,0.625) -- (0,1.25) node[above,black]{\small $t_g^2$}  ;
\draw[red] (-1.25,1.25) node[above,black]{\small $t_g^1$}  -- (0,0) ;
\draw[red] (1.25,1.25) node[above,black]{\small $t_g^3$}  -- (0,0) ;
\draw[red] (0,-1.25) node[below,red]{\small $t$}  -- (0,0) ;
\end{tikzpicture}}

\newcommand{\coveringrelationAgaugeddeux}{
\begin{tikzpicture}[ baseline = (pt_base.base),scale = 0.5]
\coordinate (pt_base) at (0,-0.25) ;
\draw[red] (0.625,0.625) -- (0,1.25) node[above,black]{\small $t_g^2$}   ;
\draw[red] (-1.25,1.25) node[above,black]{\small $t_g^1$}  -- (0,0) ;
\draw[red] (1.25,1.25) node[above,black]{\small $t_g^3$}  -- (0,0) ;
\draw[red] (0,-1.25) node[below,red]{\small $t$}  -- (0,0) ;
\end{tikzpicture}}

\newcommand{\coveringrelationBgaugedtrois}{
\begin{tikzpicture}[ baseline = (pt_base.base),scale = 0.5]
\coordinate (pt_base) at (0,-0.25) ;
\draw[blue] (-0.625,0.625) -- (0,1.25) node[above,blue]{\small $t_2$}  ;
\draw[blue] (-1.25,1.25) node[above,blue]{\small $t_1$}  -- (0,0) ;
\draw[blue] (1.25,1.25) node[above,blue]{\small $t_3$}  -- (0,0) ;
\draw[blue] (0,-1.25) node[below,black]{\small $t_g$}  -- (0,0) ;
\end{tikzpicture}}

\newcommand{\coveringrelationAgaugedtrois}{
\begin{tikzpicture}[ baseline = (pt_base.base),scale = 0.5]
\coordinate (pt_base) at (0,-0.25) ;
\draw[blue] (0.625,0.625) -- (0,1.25) node[above,blue]{\small $t_2$}   ;
\draw[blue] (-1.25,1.25) node[above,blue]{\small $t_1$}  -- (0,0) ;
\draw[blue] (1.25,1.25) node[above,blue]{\small $t_3$}  -- (0,0) ;
\draw[blue] (0,-1.25) node[below,black]{\small $t_g$}  -- (0,0) ;
\end{tikzpicture}}

\newcommand{\coveringrelationorientationsB}{
\begin{tikzpicture}[ baseline = (pt_base.base),scale = 0.5]
\coordinate (pt_base) at (0,-0.25) ;

\node (K) at (-0.625,0.625) {} ;
\node (L) at (0,0) {} ;

\draw (-0.625,0.625) -- (0,1.25) node[above]{\small $t_2$}  ;
\draw (-1.25,1.25) node[above]{\small $t_1$}  -- (0,0) ;
\draw (1.25,1.25) node[above]{\small $t_3$}  -- (0,0) ;
\draw (0,-1.25) node[below]{\small $t_4$}  -- (0,0) ;

\node[below left = 0.2 pt] at ($(K)!0.5!(L)$) {\Small $e$} ;
\end{tikzpicture}}

\newcommand{\coveringrelationorientationsA}{
\begin{tikzpicture}[ baseline = (pt_base.base),scale = 0.5]
\coordinate (pt_base) at (0,-0.25) ;

\node (K) at (0.625,0.625) {} ;
\node (L) at (0,0) {} ;

\draw (0.625,0.625) -- (0,1.25) node[above]{\small $t_2$}   ;
\draw (-1.25,1.25) node[above]{\small $t_1$}  -- (0,0) ;
\draw (1.25,1.25) node[above]{\small $t_3$}  -- (0,0) ;
\draw (0,-1.25) node[below]{\small $t_4$}  -- (0,0) ;

\node[below right = 0.2 pt] at ($(K)!0.5!(L)$) {\Small $e$} ;
\end{tikzpicture}}

\newcommand{\TamariquatreC}{
\begin{tikzpicture}
\node (A) at (0,0) {} ;
\node (B) at (0,-1) {} ;
\node (C) at (-2,2) {} ;
\node (D) at (-1,1) {} ;
\node (E) at (0,2) {} ;
\node (F) at (-1.5,1.5) {} ;
\node (G) at (-1,2) {} ;
\node (H) at (2,2) {} ;

\draw (C.center) -- (F.center) -- (D.center) -- (A.center) -- (B.center) ;
\draw (A.center) -- (H.center) ;
\draw (D.center) -- (E.center) ;
\draw (F.center) -- (G.center) ;
\end{tikzpicture}
}

\newcommand{\TamariquatreA}{
\begin{tikzpicture}
\node (A) at (0,0) {} ;
\node (B) at (0,-1) {} ;
\node (C) at (-2,2) {} ;
\node (D) at (1,1) {} ;
\node (E) at (0,2) {} ;
\node (F) at (1.5,1.5) {} ;
\node (G) at (1,2) {} ;
\node (H) at (2,2) {} ;

\draw (H.center) -- (F.center) -- (D.center) -- (A.center) -- (B.center) ;
\draw (A.center) -- (C.center) ;
\draw (D.center) -- (E.center) ;
\draw (F.center) -- (G.center) ;
\end{tikzpicture}
}

\newcommand{\TamariquatreB}{
\begin{tikzpicture}
\node (A) at (-1.5,1) {} ;
\node (B) at (0,0) {} ;
\node (C) at (1.5,1) {} ;
\node (D) at (0,-1) {} ;
\node (E) at (0.25,2) {} ;
\node (F) at (2.75,2) {} ;
\node (G) at (-0.25,2) {} ;
\node (I) at (-2.75,2) {} ;

\draw (I.center) -- (A.center) -- (B.center) -- (C.center) -- (E.center) ;
\draw (F.center) -- (C.center) ;
\draw (G.center) -- (A.center) ;
\draw (B.center) -- (D.center) ;
\end{tikzpicture}
}

\newcommand{\TamariquatreE}{
\begin{tikzpicture}
\node (A) at (0,0) {} ;
\node (B) at (2,2) {} ;
\node (H) at (-2,2) {} ;
\node (C) at (0,2) {} ;
\node (E) at (0.5,1.5) {} ;
\node (F) at (1,2) {} ;
\node (G) at (0,-1) {} ;
\node (D) at (1,1) {} ;

\draw (C.center) -- (E.center) -- (D.center) -- (A.center) -- (G.center) ;
\draw (A.center) -- (H.center) ;
\draw (E.center) -- (F.center) ;
\draw (B.center) -- (D.center) ;
\end{tikzpicture}
}

\newcommand{\TamariquatreD}{
\begin{tikzpicture}
\node (A) at (0,0) {} ;
\node (B) at (0,-1) {} ;
\node (C) at (-2,2) {} ;
\node (D) at (-1,1) {} ;
\node (E) at (0,2) {} ;
\node (F) at (-0.5,1.5) {} ;
\node (G) at (-1,2) {} ;
\node (H) at (2,2) {} ;

\draw (E.center) -- (F.center) -- (D.center) -- (A.center) -- (B.center) ;
\draw (A.center) -- (H.center) ;
\draw (G.center) -- (F.center) ;
\draw (C.center) -- (D.center) ;
\end{tikzpicture}
}

\newcommand{\Tamariquatre}{
\begin{tikzpicture}[scale = 0.6, baseline = (pt_base.base)]

\coordinate (pt_base) at (-1,0) ; 

\node[scale = 0.2] (a) at (0,3) {\TamariquatreA} ;
\node[scale = 0.2] (b) at (-2,0) {\TamariquatreB} ;
\node[scale = 0.2] (c) at (0,-3) {\TamariquatreC} ;
\node[scale = 0.2] (d) at (4,-1.5) {\TamariquatreD} ;
\node[scale = 0.2] (e) at (4,1.5) {\TamariquatreE} ;

\draw (a) -- (b) -- (c) -- (d) -- (e) -- (a) ;

\end{tikzpicture}}

\newcommand{\TamariorientationsquatreC}{
\begin{tikzpicture}[scale = 0.3]
\node (A) at (0,0) {} ;
\node (B) at (0,-1) {} ;
\node (C) at (-2,2) {} ;
\node (D) at (-1,1) {} ;
\node (E) at (0,2) {} ;
\node (F) at (-1.5,1.5) {} ;
\node (G) at (-1,2) {} ;
\node (H) at (2,2) {} ;

\draw (C.center) -- (F.center) -- (D.center) -- (A.center) -- (B.center) ;
\draw (A.center) -- (H.center) ;
\draw (D.center) -- (E.center) ;
\draw (F.center) -- (G.center) ;

\node[below left = 0.2 pt] at ($(D)!0.5!(A)$) {\tiny $e_1$} ;
\node[below left = 0.2 pt] at ($(D)!0.5!(F)$) {\tiny $e_2$} ;

\end{tikzpicture}
}

\newcommand{\TamariorientationsquatreCprime}{
\begin{tikzpicture}[scale = 0.3]
\node (A) at (0,0) {} ;
\node (B) at (0,-1) {} ;
\node (C) at (-2,2) {} ;
\node (D) at (-1,1) {} ;
\node (E) at (0,2) {} ;
\node (F) at (-1.5,1.5) {} ;
\node (G) at (-1,2) {} ;
\node (H) at (2,2) {} ;

\draw (C.center) -- (F.center) -- (D.center) -- (A.center) -- (B.center) ;
\draw (A.center) -- (H.center) ;
\draw (D.center) -- (E.center) ;
\draw (F.center) -- (G.center) ;

\node[below left = 0.2 pt] at ($(D)!0.5!(A)$) {\tiny $e_2$} ;
\node[below left = 0.2 pt] at ($(D)!0.5!(F)$) {\tiny $e_1$} ;

\end{tikzpicture}
}

\newcommand{\TamariorientationsquatreA}{
\begin{tikzpicture}[scale = 0.3]
\node (A) at (0,0) {} ;
\node (B) at (0,-1) {} ;
\node (C) at (-2,2) {} ;
\node (D) at (1,1) {} ;
\node (E) at (0,2) {} ;
\node (F) at (1.5,1.5) {} ;
\node (G) at (1,2) {} ;
\node (H) at (2,2) {} ;

\draw (H.center) -- (F.center) -- (D.center) -- (A.center) -- (B.center) ;
\draw (A.center) -- (C.center) ;
\draw (D.center) -- (E.center) ;
\draw (F.center) -- (G.center) ;

\node[below right = 0.2 pt] at ($(D)!0.5!(A)$) {\tiny $e_1$} ;
\node[below right = 0.2 pt] at ($(D)!0.5!(F)$) {\tiny $e_2$} ;

\end{tikzpicture}
}

\newcommand{\TamariorientationsquatreB}{
\begin{tikzpicture}[scale = 0.3]
\node (A) at (-1.5,1) {} ;
\node (B) at (0,0) {} ;
\node (C) at (1.5,1) {} ;
\node (D) at (0,-1) {} ;
\node (E) at (0.25,2) {} ;
\node (F) at (2.75,2) {} ;
\node (G) at (-0.25,2) {} ;
\node (I) at (-2.75,2) {} ;

\draw (I.center) -- (A.center) -- (B.center) -- (C.center) -- (E.center) ;
\draw (F.center) -- (C.center) ;
\draw (G.center) -- (A.center) ;
\draw (B.center) -- (D.center) ;

\node[below left = 0.2 pt] at ($(A)!0.5!(B)$) {\tiny $e_1$} ;
\node[below right = 0.2 pt] at ($(C)!0.5!(B)$) {\tiny $e_2$} ;
\end{tikzpicture}
}

\newcommand{\TamariorientationsquatreE}{
\begin{tikzpicture}[scale = 0.3]
\node (A) at (0,0) {} ;
\node (B) at (2,2) {} ;
\node (H) at (-2,2) {} ;
\node (C) at (0,2) {} ;
\node (E) at (0.5,1.5) {} ;
\node (F) at (1,2) {} ;
\node (G) at (0,-1) {} ;
\node (D) at (1,1) {} ;

\draw (C.center) -- (E.center) -- (D.center) -- (A.center) -- (G.center) ;
\draw (A.center) -- (H.center) ;
\draw (E.center) -- (F.center) ;
\draw (B.center) -- (D.center) ;

\node[below right = 0.2 pt] at ($(A)!0.5!(D)$) {\tiny $e_1$} ;
\node[ xshift = -7 pt,yshift = -3 pt] at ($(E)!0.5!(D)$) {\tiny $e_2$} ;
\end{tikzpicture}
}

\newcommand{\TamariorientationsquatreD}{
\begin{tikzpicture}[scale = 0.3]
\node (A) at (0,0) {} ;
\node (B) at (0,-1) {} ;
\node (C) at (-2,2) {} ;
\node (D) at (-1,1) {} ;
\node (E) at (0,2) {} ;
\node (F) at (-0.5,1.5) {} ;
\node (G) at (-1,2) {} ;
\node (H) at (2,2) {} ;

\draw (E.center) -- (F.center) -- (D.center) -- (A.center) -- (B.center) ;
\draw (A.center) -- (H.center) ;
\draw (G.center) -- (F.center) ;
\draw (C.center) -- (D.center) ;

\node[below left = 0.2 pt] at ($(A)!0.5!(D)$) {\tiny $e_1$} ;
\node[ xshift = 7 pt,yshift = -3 pt] at ($(D)!0.5!(F)$) {\tiny $e_2$} ;
\end{tikzpicture}
}

\newcommand{\Tamariorientationsquatre}{
\begin{tikzpicture}[scale = 0.6, baseline = (pt_base.base)]

\coordinate (pt_base) at (-1,0) ; 

\node[scale = 0.7] (a) at (0,3) {$\TamariorientationsquatreA \ \  e_1 \wedge e_2  $} ;
\node[scale = 0.7] (b) at (-2,0) {$ \TamariorientationsquatreB \ \ - e_1 \wedge e_2  $} ;
\node[scale = 0.7] (c) at (0,-4) {$ \TamariorientationsquatreCprime \ \ e_1 \wedge e_2 \ = \ \TamariorientationsquatreC \ \ - e_1 \wedge e_2  $} ;
\node[scale = 0.7] (d) at (6,-1.5) {$ \TamariorientationsquatreD \ \ e_1 \wedge e_2  $} ;
\node[scale = 0.7] (e) at (6,1.5) {$ \TamariorientationsquatreE \ \ - e_1 \wedge e_2  $} ;

\draw[>=stealth,->] (a) -- (b)  ;
\draw[>=stealth,->] (b) -- (c)  ;
\draw[>=stealth,->] (a) -- (e)  ;
\draw[>=stealth,->] (e) -- (d)  ;
\draw[>=stealth,->] (d) -- (c)  ;

\end{tikzpicture}}

\newcommand{\rightleaningcomb}{
\begin{tikzpicture}[scale = 0.2, baseline = (pt_base.base)]

\coordinate (pt_base) at (-1,0) ; 

\node (A) at (0,-2) {} ;
\node (B) at (0,0) {} ;
\node (C) at (-5,4) {} ;
\node (D) at (5,4) {} ;

\node (F) at (1.5,1.2) {} ;
\node (G) at (2.5,2) {} ;

\node (E) at (1,0.8) {} ;
\node (H) at (3,2.4) {} ;
\node (I) at (4,3.2) {} ;

\node (e) at (-3,4) {} ; 
\node (h) at (1,4) {} ;
\node (i) at (3,4) {} ;

\draw (A.center) -- (B.center) -- (C.center) ;
\draw (E.center) -- (e.center) ;
\draw (B.center) -- (F.center) ;
\draw[densely dotted] (F) -- (G) ;
\draw (h.center) -- (H.center) -- (I.center) -- (i.center) ;
\draw (G.center) -- (D.center) ;

\node[below right = 0.1 pt] at ($(B)!0.5!(E)$) {\tiny $e_1$} ;
\node[below right = 0.1 pt] at ($(H)!0.5!(I)$) {\tiny $e_{n-2}$} ;

\end{tikzpicture}}

\newcommand{\rightleaningcombgauged}{
\begin{tikzpicture}[scale = 0.2, baseline = (pt_base.base)]

\coordinate (pt_base) at (-1,0) ; 

\node (a) at (0,-2) {} ; 
\node (A) at (0,-1) {} ;
\node (B) at (0,0) {} ;
\node (C) at (-5,4) {} ;
\node (D) at (5,4) {} ;

\node (F) at (1.5,1.2) {} ;
\node (G) at (2.5,2) {} ;

\node (E) at (1,0.8) {} ;
\node (H) at (3,2.4) {} ;
\node (I) at (4,3.2) {} ;

\node (e) at (-3,4) {} ; 
\node (h) at (1,4) {} ;
\node (i) at (3,4) {} ;

\node (l) at (-1 , -1 ) {} ;
\node (L) at (1 , -1 ) {} ;

\draw[red] (a.center) -- (A.center) ;
\draw[blue] (A.center) -- (B.center) -- (C.center) ;
\draw[blue] (E.center) -- (e.center) ;
\draw[blue] (B.center) -- (F.center) ;
\draw[densely dotted] (F) -- (G) ;
\draw[blue] (h.center) -- (H.center) -- (I.center) -- (i.center) ;
\draw[blue] (G.center) -- (D.center) ;
\draw (l.center) -- (L.center) ;

\node[below right = 0.1 pt] at ($(B)!0.5!(E)$) {\tiny $e_1$} ;
\node[below right = 0.1 pt] at ($(H)!0.5!(I)$) {\tiny $e_{n-2}$} ;

\end{tikzpicture}}

\newcommand{\coveringrelationsgeomP}{
\begin{tikzpicture}
\node (A) at (0,0) {} ;
\node (B) at (0,-1) {} ;
\node (C) at (-2,2) {} ;
\node (D) at (1,1) {} ;
\node (E) at (0,2) {} ;
\node (F) at (1.5,1.5) {} ;
\node (G) at (1,2) {} ;
\node (H) at (2,2) {} ;

\draw (H.center) -- (F.center) -- (D.center) -- (A.center) -- (B.center) ;
\draw (A.center) -- (C.center) ;
\draw (D.center) -- (E.center) ;
\draw (F.center) -- (G.center) ;

\node[below right = 0.2 pt] at ($(D)!0.5!(A)$) {\Huge $l_{e_1}$} ;
\node[below right = 0.2 pt] at ($(D)!0.5!(F)$) {\Huge $l_{e_2}$} ;

\end{tikzpicture}
}

\newcommand{\coveringrelationsgeomO}{
\begin{tikzpicture}
\node (A) at (0,0) {} ;
\node (B) at (2,2) {} ;
\node (H) at (-2,2) {} ;
\node (C) at (0,2) {} ;
\node (E) at (0.5,1.5) {} ;
\node (F) at (1,2) {} ;
\node (G) at (0,-1) {} ;
\node (D) at (1,1) {} ;

\draw (C.center) -- (E.center) -- (D.center) -- (A.center) -- (G.center) ;
\draw (A.center) -- (H.center) ;
\draw (F.center) -- (D.center) ;
\draw (B.center) -- (D.center) ;

\node[below right = 0.2 pt] at ($(A)!0.5!(D)$) {\Huge $l_{e_1}$} ;
\end{tikzpicture}
}

\newcommand{\coveringrelationsgeomQ}{
\begin{tikzpicture}
\node (A) at (0,0) {} ;
\node (B) at (2,2) {} ;
\node (H) at (-2,2) {} ;
\node (C) at (0,2) {} ;
\node (E) at (0.5,1.5) {} ;
\node (F) at (1,2) {} ;
\node (G) at (0,-1) {} ;
\node (D) at (1,1) {} ;

\draw (C.center) -- (E.center) -- (D.center) -- (A.center) -- (G.center) ;
\draw (A.center) -- (H.center) ;
\draw (E.center) -- (F.center) ;
\draw (B.center) -- (D.center) ;

\node[below right = 0.2 pt] at ($(A)!0.5!(D)$) {\Huge $l_{f_1}$} ;
\node[ xshift = -18 pt,yshift = -7 pt] at ($(E)!0.5!(D)$) {\Huge $l_{f_2}$} ;
\end{tikzpicture}
}

\newcommand{\coveringrelationsgeom}{
\begin{tikzpicture}[scale = 0.4, baseline = (pt_base.base)]

\coordinate (pt_base) at (-1,0) ; 

\node[scale = 0.25] (O) at (0,0) {\coveringrelationsgeomO} ;
\node[scale = 0.3] (P) at (-6,1) {\coveringrelationsgeomP} ;
\node[scale = 0.3] (Q) at (6,1) {\coveringrelationsgeomQ} ;

\node (A) at (-10,-5) {} ;
\node (B) at (10,-5) {} ;
\node (C) at (10,5) {} ;
\node (D) at (-10,5) {} ;
\node (E) at (-10.5,-6) {} ;
\node (F) at (-1,-6) {} ;
\node (G) at (1,-6) {} ;
\node (H) at (10.5,-6) {} ;
\node (I) at (11,-4) {} ;
\node (J) at (11,5.5) {} ;
\node (K) at (-11,5.5) {} ;
\node (L) at (-11,-4) {} ;
\node (M) at (0,-5.25) {} ;
\node (N) at (0,5.25) {} ;
\node (R) at (-2,-3.5) {} ;
\node (S) at (-2,-2.5) {} ;
\node (T) at (-3,-3.5) {} ;
\node (U) at (3,-3.5) {} ;
\node (V) at (3,-2.5) {} ;
\node (W) at (2,-3.5) {} ;
\node (X) at (0,-1.5) {} ;
\node (Y) at (0,1.5) {} ;

\draw[fill,opacity=0.1] (A.center) -- (B.center) -- (C.center) -- (D.center) -- (A.center) ;

\draw[>=stealth,->] (F.center) -- (E.center) node[below left]{$l_{e_2}$}  ; 
\draw[>=stealth,->] (G.center) -- (H.center) node[below right]{$l_{f_2}$} ; 
\draw[>=stealth,->] (I.center) -- (J.center) node[above right]{$l_{f_1}$} ; 
\draw[>=stealth,->] (L.center) -- (K.center) node[above left]{$l_{e_1}$} ; 

\draw (M.center)  -- (X.center) ;
\draw (N.center) node[above]{\Tiny $l_{e_2} = l_{f_2} = 0$} -- (Y.center) ;

\draw[>=stealth,->] (R.center) -- (S.center) ; 
\draw[>=stealth,->] (R.center) -- (T.center) ; 
\draw[>=stealth,->] (U.center) -- (V.center) ; 
\draw[>=stealth,->] (U.center) -- (W.center) ; 

\node[below] at ($(R)!0.5!(T)$) {\tiny $v_{e_2}$} ;
\node[right] at ($(R)!0.5!(S)$) {\tiny $v_{e_1}$} ;
\node[below] at ($(U)!0.5!(W)$) {\tiny $v_{e_2} = - v_{f_2}$} ;
\node[right] at ($(U)!0.5!(V)$) {\tiny $v_{e_1} = v_{f_1}$} ;

\end{tikzpicture}}

\newcommand{\TamariquatreCT}{
\begin{tikzpicture}[scale = 0.6, baseline = (pt_base.base)]

\coordinate (pt_base) at (-1,0) ; 

\node[scale = 0.3] (a) at (0,3) {\TamariquatreCTA} ;
\node[scale = 0.3] (b) at (3,1) {\TamariquatreCTB} ;
\node[scale = 0.3] (c) at (3,-1) {\TamariquatreCTC} ;
\node[scale = 0.3] (f) at (-3,1) {\TamariquatreCTF} ;
\node[scale = 0.3] (e) at (-3,-1) {\TamariquatreCTE} ;
\node[scale = 0.3] (d) at (0,-3) {\TamariquatreCTD} ;

\draw (a) -- (b) -- (c) -- (d) -- (e) -- (f) -- (a) ;

\end{tikzpicture}}

\newcommand{\TamariquatreCTA}{
\begin{tikzpicture}
\node (A) at (0,0) {} ;
\node (B) at (0,-1) {} ;
\node (C) at (1.25,1.25) {} ;
\node (D) at (-1.25,1.25) {} ;
\node (E) at (0,1.25) {} ;
\node (F) at (0.625,0.625) {} ;
\node (G) at (0,-0.3) {} ;
\node (H) at (-1.25,-0.3) {} ;
\node (I) at (1.25,-0.3) {}  ;
\node (J) at (1.5,-0.3) {} ;
\node (K) at (1.5,0) {} ;

\draw[blue] (D.center) -- (A.center) ;
\draw[blue] (A.center) -- (G.center) ;
\draw[blue] (A.center) -- (F.center) ;
\draw[blue] (F.center) -- (E.center) ;
\draw[blue] (C.center) -- (F.center) ;

\draw[red] (G.center) -- (B.center) ;

\draw (H.center) -- (I.center) ;

\end{tikzpicture}}

\newcommand{\TamariquatreCTB}{
\begin{tikzpicture}
\node (A) at (0,0) {} ;
\node (B) at (0,-1) {} ;
\node (C) at (1.25,1.25) {} ;
\node (D) at (-1.25,1.25) {} ;
\node (E) at (0,1.25) {} ;
\node (F) at (0.625,0.625) {} ;
\node (G) at (-1.25,0.3) {} ;
\node (H) at (1.25,0.3) {} ;
\node (I) at (1.5,0.3) {}  ;
\node (J) at (1.5,0) {} ;

\node (K) at (-0.3,0.3) {} ;
\node (L) at (0.3,0.3) {} ;

\draw[blue] (D.center) -- (K.center) ;
\draw[blue] (L.center) -- (F.center) ;
\draw[blue] (F.center) -- (E.center) ;
\draw[blue] (F.center) -- (C.center) ;
\draw[red] (K.center) -- (A.center) ;
\draw[red] (L.center) -- (A.center) ;
\draw[red] (A.center) -- (B.center) ;
\draw (G.center) -- (H.center) ;
\end{tikzpicture}}

\newcommand{\TamariquatreCTC}{
\begin{tikzpicture}
\node (A) at (0,0) {} ;
\node (B) at (0,-1) {} ;
\node (C) at (1.25,1.25) {} ;
\node (D) at (-1.25,1.25) {} ;
\node (E) at (0,1.25) {} ;
\node (F) at (0.625,0.625) {} ;
\node (G) at (-0.9,0.9) {} ;
\node (H) at (0.35,0.9) {} ;
\node (I) at (0.9,0.9) {}  ;
\node (J) at (-1.25,0.9) {} ;
\node (K) at (1.25,0.9) {} ;
\node (L) at (1.5,0) {} ;
\node (M) at (1.5,0.9) {} ;

\draw[blue] (D.center) -- (G.center) ;
\draw[blue] (E.center) -- (H.center) ;
\draw[blue] (C.center) -- (I.center) ;

\draw[red] (G.center) -- (A.center) ;
\draw[red] (H.center) -- (F.center) ;
\draw[red] (F.center) -- (A.center) ;
\draw[red] (A.center) -- (I.center) ;
\draw[red] (A.center) -- (B.center) ;

\draw (J.center) -- (K.center) ;
\end{tikzpicture}}

\newcommand{\TamariquatreCTD}{
\begin{tikzpicture}
\node (A) at (0,0) {} ;
\node (B) at (0,-1) {} ;
\node (C) at (1.25,1.25) {} ;
\node (D) at (-1.25,1.25) {} ;
\node (E) at (0,1.25) {} ;
\node (F) at (-0.625,0.625) {} ;
\node (G) at (-0.9,0.9) {} ;
\node (H) at (-0.35,0.9) {} ;
\node (I) at (0.9,0.9) {}  ;
\node (J) at (-1.25,0.9) {} ;
\node (K) at (1.25,0.9) {} ;
\node (L) at (1.5,0) {} ;
\node (M) at (1.5,0.9) {} ;

\draw[blue] (D.center) -- (G.center) ;
\draw[blue] (E.center) -- (H.center) ;
\draw[blue] (C.center) -- (I.center) ;

\draw[red] (G.center) -- (F.center) ;
\draw[red] (H.center) -- (F.center) ;
\draw[red] (F.center) -- (A.center) ;
\draw[red] (A.center) -- (I.center) ;
\draw[red] (A.center) -- (B.center) ;

\draw (J.center) -- (K.center) ;
\end{tikzpicture}}

\newcommand{\TamariquatreCTE}{
\begin{tikzpicture}
\node (A) at (0,0) {} ;
\node (B) at (0,-1) {} ;
\node (C) at (1.25,1.25) {} ;
\node (D) at (-1.25,1.25) {} ;
\node (E) at (0,1.25) {} ;
\node (F) at (-0.625,0.625) {} ;
\node (G) at (-1.25,0.3) {} ;
\node (H) at (1.25,0.3) {} ;
\node (I) at (1.5,0.3) {}  ;
\node (J) at (1.5,0) {} ;

\node (K) at (-0.3,0.3) {} ;
\node (L) at (0.3,0.3) {} ;

\draw[blue] (D.center) -- (F.center) -- (K.center) ;
\draw[red] (K.center) -- (A.center) ;
\draw[blue] (F.center) -- (E.center) ;
\draw[red] (A.center) -- (B.center) ;
\draw[red] (A.center) -- (L.center) ;
\draw[blue] (L.center) -- (C.center) ;
\draw[blue] (F.center) -- (E.center) ;
\draw (G.center) -- (H.center) ;
\end{tikzpicture}}

\newcommand{\TamariquatreCTF}{
\begin{tikzpicture}
\node (A) at (0,0) {} ;
\node (B) at (0,-1) {} ;
\node (C) at (1.25,1.25) {} ;
\node (D) at (-1.25,1.25) {} ;
\node (E) at (0,1.25) {} ;
\node (F) at (-0.625,0.625) {} ;
\node (G) at (0,-0.3) {} ;
\node (H) at (-1.25,-0.3) {} ;
\node (I) at (1.25,-0.3) {}  ;
\node (J) at (1.5,-0.3) {} ;
\node (K) at (1.5,0) {} ;

\draw[blue] (D.center) -- (A.center) ;
\draw[blue] (A.center) -- (G.center) ;
\draw[blue] (A.center) -- (F.center) ;
\draw[blue] (F.center) -- (E.center) ;
\draw[blue] (C.center) -- (A.center) ;

\draw[red] (G.center) -- (B.center) ;

\draw (H.center) -- (I.center) ;
\end{tikzpicture}}

\newcommand{\TamariorientationsquatreCT}{
\begin{tikzpicture}[scale = 0.6, baseline = (pt_base.base)]

\coordinate (pt_base) at (-1,0) ; 

\node[scale = 0.7] (a) at (0,3) {$\TamariorientationsquatreCTA \ \  e  $} ;
\node[scale = 0.7] (b) at (3,1) {$ \TamariorientationsquatreCTB \ \ e  $} ;
\node[scale = 0.7] (c) at (3,-1) {$ \TamariorientationsquatreCTC \ \ e $} ;
\node[scale = 0.7] (d) at (0,-3) {$ \TamariorientationsquatreCTD \ \ - e  $} ;
\node[scale = 0.7] (e) at (-3,-1) {$ \TamariorientationsquatreCTE \ \ - e  $} ;
\node[scale = 0.7] (f) at (-3,1) {$ \TamariorientationsquatreCTF \ \ - e  $} ;

\draw[>=stealth,->] (a) -- (b)  ;
\draw[>=stealth,->] (b) -- (c)  ;
\draw[>=stealth,->] (c) -- (d)  ;
\draw[>=stealth,->] (a) -- (f)  ;
\draw[>=stealth,->] (f) -- (e)  ;
\draw[>=stealth,->] (e) -- (d)  ;

\end{tikzpicture}}

\newcommand{\TamariorientationsquatreCTA}{
\begin{tikzpicture}[scale = 0.4]
\node (A) at (0,0) {} ;
\node (B) at (0,-1) {} ;
\node (C) at (1.25,1.25) {} ;
\node (D) at (-1.25,1.25) {} ;
\node (E) at (0,1.25) {} ;
\node (F) at (0.625,0.625) {} ;
\node (G) at (0,-0.3) {} ;
\node (H) at (-1.25,-0.3) {} ;
\node (I) at (1.25,-0.3) {}  ;
\node (J) at (1.5,-0.3) {} ;
\node (K) at (1.5,0) {} ;

\draw[blue] (D.center) -- (A.center) ;
\draw[blue] (A.center) -- (G.center) ;
\draw[blue] (A.center) -- (F.center) ;
\draw[blue] (F.center) -- (E.center) ;
\draw[blue] (C.center) -- (F.center) ;

\draw[red] (G.center) -- (B.center) ;

\draw (H.center) -- (I.center) ;

\node[ xshift =  5 pt,yshift = -3 pt] at ($(A)!0.5!(F)$) {\tiny $e$} ;

\end{tikzpicture}}

\newcommand{\TamariorientationsquatreCTB}{
\begin{tikzpicture}[scale = 0.4]
\node (A) at (0,0) {} ;
\node (B) at (0,-1) {} ;
\node (C) at (1.25,1.25) {} ;
\node (D) at (-1.25,1.25) {} ;
\node (E) at (0,1.25) {} ;
\node (F) at (0.625,0.625) {} ;
\node (G) at (-1.25,0.3) {} ;
\node (H) at (1.25,0.3) {} ;
\node (I) at (1.5,0.3) {}  ;
\node (J) at (1.5,0) {} ;

\node (K) at (-0.3,0.3) {} ;
\node (L) at (0.3,0.3) {} ;

\draw[blue] (D.center) -- (K.center) ;
\draw[blue] (L.center) -- (F.center) ;
\draw[blue] (F.center) -- (E.center) ;
\draw[blue] (F.center) -- (C.center) ;
\draw[red] (K.center) -- (A.center) ;
\draw[red] (L.center) -- (A.center) ;
\draw[red] (A.center) -- (B.center) ;
\draw (G.center) -- (H.center) ;

\node[ xshift =  5 pt,yshift = -3 pt] at ($(A)!0.5!(F)$) {\tiny $e$} ;
\end{tikzpicture}}

\newcommand{\TamariorientationsquatreCTC}{
\begin{tikzpicture}[scale = 0.4]
\node (A) at (0,0) {} ;
\node (B) at (0,-1) {} ;
\node (C) at (1.25,1.25) {} ;
\node (D) at (-1.25,1.25) {} ;
\node (E) at (0,1.25) {} ;
\node (F) at (0.625,0.625) {} ;
\node (G) at (-0.9,0.9) {} ;
\node (H) at (0.35,0.9) {} ;
\node (I) at (0.9,0.9) {}  ;
\node (J) at (-1.25,0.9) {} ;
\node (K) at (1.25,0.9) {} ;
\node (L) at (1.5,0) {} ;
\node (M) at (1.5,0.9) {} ;

\draw[blue] (D.center) -- (G.center) ;
\draw[blue] (E.center) -- (H.center) ;
\draw[blue] (C.center) -- (I.center) ;

\draw[red] (G.center) -- (A.center) ;
\draw[red] (H.center) -- (F.center) ;
\draw[red] (F.center) -- (A.center) ;
\draw[red] (A.center) -- (I.center) ;
\draw[red] (A.center) -- (B.center) ;

\draw (J.center) -- (K.center) ;

\node[ xshift =  5 pt,yshift = -3 pt] at ($(A)!0.5!(F)$) {\tiny $e$} ;
\end{tikzpicture}}

\newcommand{\TamariorientationsquatreCTD}{
\begin{tikzpicture}[scale = 0.4]
\node (A) at (0,0) {} ;
\node (B) at (0,-1) {} ;
\node (C) at (1.25,1.25) {} ;
\node (D) at (-1.25,1.25) {} ;
\node (E) at (0,1.25) {} ;
\node (F) at (-0.625,0.625) {} ;
\node (G) at (-0.9,0.9) {} ;
\node (H) at (-0.35,0.9) {} ;
\node (I) at (0.9,0.9) {}  ;
\node (J) at (-1.25,0.9) {} ;
\node (K) at (1.25,0.9) {} ;
\node (L) at (1.5,0) {} ;
\node (M) at (1.5,0.9) {} ;

\draw[blue] (D.center) -- (G.center) ;
\draw[blue] (E.center) -- (H.center) ;
\draw[blue] (C.center) -- (I.center) ;

\draw[red] (G.center) -- (F.center) ;
\draw[red] (H.center) -- (F.center) ;
\draw[red] (F.center) -- (A.center) ;
\draw[red] (A.center) -- (I.center) ;
\draw[red] (A.center) -- (B.center) ;

\draw (J.center) -- (K.center) ;

\node[ xshift = - 5 pt,yshift = -3 pt] at ($(A)!0.5!(F)$) {\tiny $e$} ;
\end{tikzpicture}}

\newcommand{\TamariorientationsquatreCTE}{
\begin{tikzpicture}[scale = 0.4]
\node (A) at (0,0) {} ;
\node (B) at (0,-1) {} ;
\node (C) at (1.25,1.25) {} ;
\node (D) at (-1.25,1.25) {} ;
\node (E) at (0,1.25) {} ;
\node (F) at (-0.625,0.625) {} ;
\node (G) at (-1.25,0.3) {} ;
\node (H) at (1.25,0.3) {} ;
\node (I) at (1.5,0.3) {}  ;
\node (J) at (1.5,0) {} ;

\node (K) at (-0.3,0.3) {} ;
\node (L) at (0.3,0.3) {} ;

\draw[blue] (D.center) -- (F.center) -- (K.center) ;
\draw[red] (K.center) -- (A.center) ;
\draw[blue] (F.center) -- (E.center) ;
\draw[red] (A.center) -- (B.center) ;
\draw[red] (A.center) -- (L.center) ;
\draw[blue] (L.center) -- (C.center) ;
\draw[blue] (F.center) -- (E.center) ;
\draw (G.center) -- (H.center) ;

\node[ xshift = - 5 pt,yshift = -3 pt] at ($(A)!0.5!(F)$) {\tiny $e$} ;
\end{tikzpicture}}

\newcommand{\TamariorientationsquatreCTF}{
\begin{tikzpicture}[scale = 0.4]
\node (A) at (0,0) {} ;
\node (B) at (0,-1) {} ;
\node (C) at (1.25,1.25) {} ;
\node (D) at (-1.25,1.25) {} ;
\node (E) at (0,1.25) {} ;
\node (F) at (-0.625,0.625) {} ;
\node (G) at (0,-0.3) {} ;
\node (H) at (-1.25,-0.3) {} ;
\node (I) at (1.25,-0.3) {}  ;
\node (J) at (1.5,-0.3) {} ;
\node (K) at (1.5,0) {} ;

\draw[blue] (D.center) -- (A.center) ;
\draw[blue] (A.center) -- (G.center) ;
\draw[blue] (A.center) -- (F.center) ;
\draw[blue] (F.center) -- (E.center) ;
\draw[blue] (C.center) -- (A.center) ;

\draw[red] (G.center) -- (B.center) ;

\draw (H.center) -- (I.center) ;

\node[ xshift = - 5 pt,yshift = -3 pt] at ($(A)!0.5!(F)$) {\tiny $e$} ;
\end{tikzpicture}}

\newcommand{\localgaugeedgeA}{
\begin{tikzpicture}[scale=0.4,baseline=(pt_base.base)]

\coordinate (pt_base) at (0,0) ;

\draw[blue] (-2,1.25) -- (-0.625,0.5) ;
\draw[blue] (-1,1.25) -- (-0.3125,0.5) ;
\draw[blue] (1,1.25) -- (0.3125,0.5) ;
\draw[blue] (2,1.25) -- (0.625,0.5) ;

\draw[red] (-0.625,0.5) -- (0,0) ;
\draw[red] (-0.3125,0.5) -- (0,0) ;
\draw[red] (0.625,0.5) -- (0,0) ;
\draw[red] (0.3125,0.5) -- (0,0) ;
\draw[red] (0,-1.25) -- (0,0) ;

\node[scale = 0.1] at (0,0) {\pointbullet} ;
\node[below left] at (0,0) {\small $v$} ;

\draw[blue] [densely dotted] (-0.5,0.9) -- (0.5,0.9) ;

\draw  (-2,0.5) -- (2,0.5) ;

\end{tikzpicture}}

\newcommand{\localgaugeedgeB}{
\begin{tikzpicture}[scale=0.4,baseline=(pt_base.base)]

\coordinate (pt_base) at (0,0) ;

\draw[blue] (-2,1.25) -- (0,0) ;
\draw[blue] (-1,1.25) -- (0,0) ;
\draw[blue] (1,1.25) -- (0,0) ;
\draw[blue] (2,1.25) -- (0,0) ;
\draw[blue] (0,-0.5) -- (0,0) ;

\draw[red] (0,-1.25) -- (0,-0.5) ;

\node[scale = 0.1] at (0,0) {\pointbullet} ;
\node[right] at (0,0) {\small $v'$} ;

\draw[blue] [densely dotted] (-0.5,0.9) -- (0.5,0.9) ;

\draw  (-1,-0.5) -- (1,-0.5) ;

\end{tikzpicture}}

\newcommand{\localgaugeedgeC}{
\begin{tikzpicture}[scale=0.4,baseline=(pt_base.base)]

\coordinate (pt_base) at (0,0) ;

\draw[blue] (-2,1.25) -- (0,0) ;
\draw[blue] (-1,1.25) -- (0,0) ;
\draw[blue] (1,1.25) -- (0,0) ;
\draw[blue] (2,1.25) -- (0,0) ;
\draw[red] (0,-0.5) -- (0,0) ;

\draw[red] (0,-1.25) -- (0,-0.5) ;

\node[scale = 0.1] at (0,0) {\pointbullet} ;
\node[below right] at (0,0) {\small $v''$} ;

\draw[blue] [densely dotted] (-0.5,0.9) -- (0.5,0.9) ;

\draw  (-1.5,0) -- (1.5,0) ;

\end{tikzpicture}}

\newcommand{\localgaugeedgeCbis}{
\begin{tikzpicture}[scale=0.4,baseline=(pt_base.base)]

\coordinate (pt_base) at (0,0) ;

\draw[blue] (-2,1.25) -- (0,0) ;
\draw[blue] (-1,1.25) -- (0,0) ;
\draw[blue] (1,1.25) -- (0,0) ;
\draw[blue] (2,1.25) -- (0,0) ;
\draw[red] (0,-0.5) -- (0,0) ;

\draw[red] (0,-1.25) -- (0,-0.5) ;

\node[scale = 0.1] at (0,0) {\pointbullet} ;
\node[below right] at (0,0) {\small $v$} ;

\draw[blue] [densely dotted] (-0.5,0.9) -- (0.5,0.9) ;

\draw  (-1.5,0) -- (1.5,0) ;

\end{tikzpicture}}

\newcommand{\ordreunbrokengaugedtrees}{
\begin{tikzpicture}[scale=0.4,baseline=(pt_base.base)]
\begin{scope}

\coordinate (pt_base) at (0,0) ;

\draw[blue] (-2,1.25) node[above]{$t_{br}^{1,1}$} -- (-0.625,0.5) ;
\draw[blue] (-1,1.25) -- (-0.3125,0.5) ;
\draw[blue] (1,1.25) -- (0.3125,0.5) ;
\draw[blue] (2,1.25) node[above]{$t_{br}^{1,i_1}$} -- (0.625,0.5) ;

\draw[red] (-0.625,0.5) -- (0,0) ;
\draw[red] (-0.3125,0.5) -- (0,0) ;
\draw[red] (0.625,0.5) -- (0,0) ;
\draw[red] (0.3125,0.5) -- (0,0) ;
\draw[red] (0,-1.25) -- (0,0) ;

\draw[blue] [densely dotted] (-0.5,0.9) -- (0.5,0.9) ;

\draw[black]  (-2,0.5) -- (2,0.5) ;

\node[xshift = - 10 pt, yshift = -4 pt , black] at (0,0) {$t_g^1$} ;

\end{scope}

\draw[densely dotted,black] (3,0) -- (7,0) ; 

\begin{scope}[shift = {(10,0)}]

\coordinate (pt_base) at (0,0) ;

\draw[blue] (-2,1.25) node[above]{$t_{br}^{s,1}$} -- (-0.625,0.5) ;
\draw[blue] (-1,1.25) -- (-0.3125,0.5) ;
\draw[blue] (1,1.25) -- (0.3125,0.5) ;
\draw[blue] (2,1.25) node[above]{$t_{br}^{s,i_s}$} -- (0.625,0.5) ;

\draw[red] (-0.625,0.5) -- (0,0) ;
\draw[red] (-0.3125,0.5) -- (0,0) ;
\draw[red] (0.625,0.5) -- (0,0) ;
\draw[red] (0.3125,0.5) -- (0,0) ;
\draw[red] (0,-1.25) -- (0,0) ;

\draw[blue] [densely dotted] (-0.5,0.9) -- (0.5,0.9) ;

\draw[black]  (-2,0.5) -- (2,0.5) ;

\node[xshift = 10 pt, yshift = -4 pt , black] at (0,0) {$t_g^s$} ;

\end{scope}

\end{tikzpicture}}

\newcommand{\ordreunbrokengaugedtreeintersectionsB}{
\begin{tikzpicture}[scale=0.4,baseline=(pt_base.base)]
\begin{scope}

\coordinate (pt_base) at (0,0) ;

\draw[blue] (-2,1.25) -- (-0.625,0.5) ;
\draw[blue] (-1,1.25) -- (-0.3125,0.5) ;
\draw[blue] (1,1.25) -- (0.3125,0.5) ;
\draw[blue] (2,1.25) -- (0.625,0.5) ;

\draw[red] (-0.625,0.5) -- (0,0) ;
\draw[red] (-0.3125,0.5) -- (0,0) ;
\draw[red] (0.625,0.5) -- (0,0) ;
\draw[red] (0.3125,0.5) -- (0,0) ;
\draw[red] (0,-1.25) -- (0,0) ;

\node[scale = 0.1] at (0,0) {\pointbullet} ;
\node[below left] at (0,0) {\small $v_1$} ;

\draw[blue] [densely dotted] (-0.5,0.9) -- (0.5,0.9) ;

\draw  (-2,0.5) -- (2,0.5) ;

\end{scope}

\draw[densely dotted,black] (3,0) -- (7,0) ; 

\begin{scope}[shift = {(10,0)}]

\coordinate (pt_base) at (0,0) ;

\draw[blue] (-2,1.25) -- (-0.625,0.5) ;
\draw[blue] (-1,1.25) -- (-0.3125,0.5) ;
\draw[blue] (1,1.25) -- (0.3125,0.5) ;
\draw[blue] (2,1.25) -- (0.625,0.5) ;

\draw[red] (-0.625,0.5) -- (0,0) ;
\draw[red] (-0.3125,0.5) -- (0,0) ;
\draw[red] (0.625,0.5) -- (0,0) ;
\draw[red] (0.3125,0.5) -- (0,0) ;
\draw[red] (0,-1.25) -- (0,0) ;

\node[scale = 0.1] at (0,0) {\pointbullet} ;
\node[below left] at (0,0) {\small $v_j$} ;

\draw[blue] [densely dotted] (-0.5,0.9) -- (0.5,0.9) ;

\draw  (-2,0.5) -- (2,0.5) ;

\end{scope}

\end{tikzpicture}}

\newcommand{\ordreunbrokengaugedtreeintersectionsA}{
\begin{tikzpicture}[scale=0.4,baseline=(pt_base.base)]
\begin{scope}

\coordinate (pt_base) at (0,0) ;

\draw[blue] (-2,1.25) -- (0,0) ;
\draw[blue] (-1,1.25) -- (0,0) ;
\draw[blue] (1,1.25) -- (0,0) ;
\draw[blue] (2,1.25) -- (0,0) ;
\draw[red] (0,-0.5) -- (0,0) ;

\draw[red] (0,-1.25) -- (0,-0.5) ;

\node[scale = 0.1] at (0,0) {\pointbullet} ;
\node[below right] at (0,0) {\small $v_1$} ;

\draw[blue] [densely dotted] (-0.5,0.9) -- (0.5,0.9) ;

\draw  (-1.5,0) -- (1.5,0) ;

\end{scope}

\draw[densely dotted,black] (3,0) -- (7,0) ; 

\begin{scope}[shift = {(10,0)}]

\coordinate (pt_base) at (0,0) ;

\draw[blue] (-2,1.25) -- (0,0) ;
\draw[blue] (-1,1.25) -- (0,0) ;
\draw[blue] (1,1.25) -- (0,0) ;
\draw[blue] (2,1.25) -- (0,0) ;
\draw[red] (0,-0.5) -- (0,0) ;

\draw[red] (0,-1.25) -- (0,-0.5) ;

\node[scale = 0.1] at (0,0) {\pointbullet} ;
\node[below right] at (0,0) {\small $v_j$} ;

\draw[blue] [densely dotted] (-0.5,0.9) -- (0.5,0.9) ;

\draw  (-1.5,0) -- (1.5,0) ;

\end{scope}

\end{tikzpicture}}

\newcommand{\exampledecompositionpolyedrale}{
\begin{tikzpicture}[scale = 1.5]
\node (A) at (0,0) {} ;
\node (B) at (-1.7,1.7) {} ;
\node (C) at (0,2) {} ;
\node (D) at (-2.2,0) {} ;
\node (E) at (2,0) {} ;
\node (AB) at ($(A)!0.5!(B)$) {} ;
\node (BC) at ($(B)!0.5!(C)$) {} ;
\node (CD) at ($(C)!0.5!(D)$) {} ;
\node (DA) at ($(D)!0.5!(A)$) {} ;
\node (AC) at ($(A)!0.6!(C)$) {} ;

\draw (A.center) -- (B.center) ;
\draw[>=stealth,->] (D.center) -- (E.center) ;
\draw[>=stealth,->] (0,-0.2) -- (C.center) ;

\node[below] at (E) {\tiny $\lambda$} ;
\node[left] at (C) {\tiny $l$} ;

\node[scale = 0.5] at (1,1) {\exampledecompositionpolyedraleA} ;
\node[scale = 0.5 , xshift = - 10 pt, yshift = -7 pt] at (-0.5,1.5) {\exampledecompositionpolyedraleB} ;
\node[scale = 0.5 , xshift = 0 pt, yshift = 7 pt] at (-1.5,0.5) {\exampledecompositionpolyedraleC} ;
\end{tikzpicture}}

\newcommand{\exampledecompositionpolyedraleA}{
\begin{tikzpicture}
\node (A) at (0,0) {} ;
\node (B) at (0,-1) {} ;
\node (C) at (1.25,1.25) {} ;
\node (D) at (-1.25,1.25) {} ;
\node (E) at (0,1.25) {} ;
\node (F) at (0.625,0.625) {} ;
\node (G) at (0,-0.3) {} ;
\node (H) at (-1.25,-0.3) {} ;
\node (I) at (1.25,-0.3) {}  ;
\node (J) at (1.5,-0.3) {} ;
\node (K) at (1.5,0) {} ;

\draw[blue] (D.center) -- (A.center) ;
\draw[blue] (A.center) -- (G.center) ;
\draw[blue] (A.center) -- (F.center) ;
\draw[blue] (F.center) -- (E.center) ;
\draw[blue] (C.center) -- (F.center) ;

\draw[red] (G.center) -- (B.center) ;

\draw (H.center) -- (I.center) ;

\draw[densely dotted,>=stealth,<->] (K.center) -- (J.center) ;

\node[right] at ($(K)!0.5!(J)$) {$\lambda$} ;
\node[below right] at ($(A)!0.5!(F)$) {$l$} ;
\end{tikzpicture}}

\newcommand{\exampledecompositionpolyedraleB}{
\begin{tikzpicture}
\node (A) at (0,0) {} ;
\node (B) at (0,-1) {} ;
\node (C) at (1.25,1.25) {} ;
\node (D) at (-1.25,1.25) {} ;
\node (E) at (0,1.25) {} ;
\node (F) at (0.625,0.625) {} ;
\node (G) at (-1.25,0.3) {} ;
\node (H) at (1.25,0.3) {} ;
\node (I) at (1.5,0.3) {}  ;
\node (J) at (1.5,0) {} ;

\node (K) at (-0.3,0.3) {} ;
\node (L) at (0.3,0.3) {} ;

\draw[blue] (D.center) -- (K.center) ;
\draw[blue] (L.center) -- (F.center) ;
\draw[blue] (F.center) -- (E.center) ;
\draw[blue] (F.center) -- (C.center) ;
\draw[red] (K.center) -- (A.center) ;
\draw[red] (L.center) -- (A.center) ;
\draw[red] (A.center) -- (B.center) ;
\draw (G.center) -- (H.center) ;
\draw[densely dotted,>=stealth,<->] (I.center) -- (J.center) ;

\node[right] at ($(I)!0.5!(J)$) { $\lambda$} ;
\node[below right] at ($(A)!0.5!(F)$) { $l$} ;
\end{tikzpicture}}

\newcommand{\exampledecompositionpolyedraleC}{
\begin{tikzpicture}
\node (A) at (0,0) {} ;
\node (B) at (0,-1) {} ;
\node (C) at (1.25,1.25) {} ;
\node (D) at (-1.25,1.25) {} ;
\node (E) at (0,1.25) {} ;
\node (F) at (0.625,0.625) {} ;
\node (G) at (-0.9,0.9) {} ;
\node (H) at (0.35,0.9) {} ;
\node (I) at (0.9,0.9) {}  ;
\node (J) at (-1.25,0.9) {} ;
\node (K) at (1.25,0.9) {} ;
\node (L) at (1.5,0) {} ;
\node (M) at (1.5,0.9) {} ;

\draw[blue] (D.center) -- (G.center) ;
\draw[blue] (E.center) -- (H.center) ;
\draw[blue] (C.center) -- (I.center) ;

\draw[red] (G.center) -- (A.center) ;
\draw[red] (H.center) -- (F.center) ;
\draw[red] (F.center) -- (A.center) ;
\draw[red] (A.center) -- (I.center) ;
\draw[red] (A.center) -- (B.center) ;

\draw (J.center) -- (K.center) ;

\draw[densely dotted,>=stealth,<->] (L.center) -- (M.center) ;

\node[right] at ($(L)!0.5!(M)$) { $\lambda$} ;
\node[below right] at ($(A)!0.5!(F)$) { $l$} ;
\end{tikzpicture}}

\newcommand{\transformationgaugevertexA}{
\begin{tikzpicture}[scale=0.4,baseline=(pt_base.base)]

\begin{scope}[shift = {(8,0)}]

\coordinate (pt_base) at (0,0) ;

\draw[blue] (-2,1.25) -- (0,0) ;
\draw[blue] (-1,1.25) -- (0,0) ;
\draw[blue] (1,1.25) -- (0,0) ;
\draw[blue] (2,1.25) -- (0,0) ;
\draw[red] (0,-0.5) -- (0,0) ;

\draw[red] (0,-1.25) -- (0,-0.5) ;

\node[scale = 0.1] at (0,0) {\pointbullet} ;

\draw[blue] [densely dotted] (-0.5,0.9) -- (0.5,0.9) ;

\draw  (-1.5,0) -- (1.5,0) ;

\end{scope}

\draw[>=stealth,->] (3,0) -- (5,0) ; 

\begin{scope}

\coordinate (pt_base) at (0,0) ;

\draw[blue] (-2,1.25) -- (-0.625,0.5) ;
\draw[blue] (-1,1.25) -- (-0.3125,0.5) ;
\draw[blue] (1,1.25) -- (0.3125,0.5) ;
\draw[blue] (2,1.25) -- (0.625,0.5) ;

\draw[red] (-0.625,0.5) -- (0,0) ;
\draw[red] (-0.3125,0.5) -- (0,0) ;
\draw[red] (0.625,0.5) -- (0,0) ;
\draw[red] (0.3125,0.5) -- (0,0) ;
\draw[red] (0,-1.25) -- (0,0) ;

\node[scale = 0.1] at (0,0) {\pointbullet} ;

\draw[blue] [densely dotted] (-0.5,0.9) -- (0.5,0.9) ;

\draw  (-2,0.5) -- (2,0.5) ;

\end{scope}

\end{tikzpicture}}

\newcommand{\transformationgaugevertexB}{
\begin{tikzpicture}[scale=0.4,baseline=(pt_base.base)]

\begin{scope}[shift = {(8,0)}]

\coordinate (pt_base) at (0,0) ;

\draw[blue] (-2,1.25) -- (0,0) ;
\draw[blue] (-1,1.25) -- (0,0) ;
\draw[blue] (1,1.25) -- (0,0) ;
\draw[blue] (2,1.25) -- (0,0) ;
\draw[red] (0,-0.5) -- (0,0) ;

\draw[red] (0,-1.25) -- (0,-0.5) ;

\node[scale = 0.1] at (0,0) {\pointbullet} ;

\draw[blue] [densely dotted] (-0.5,0.9) -- (0.5,0.9) ;

\draw  (-1.5,0) -- (1.5,0) ;

\end{scope}

\draw[>=stealth,->] (3,0) -- (5,0) ; 

\begin{scope}

\coordinate (pt_base) at (0,0) ;

\draw[blue] (-2,1.25) -- (0,0) ;
\draw[blue] (-1,1.25) -- (0,0) ;
\draw[blue] (1,1.25) -- (0,0) ;
\draw[blue] (2,1.25) -- (0,0) ;
\draw[blue] (0,-0.5) -- (0,0) ;

\draw[red] (0,-1.25) -- (0,-0.5) ;

\node[scale = 0.1] at (0,0) {\pointbullet} ;

\draw[blue] [densely dotted] (-0.5,0.9) -- (0.5,0.9) ;

\draw  (-1,-0.5) -- (1,-0.5) ;

\end{scope}

\end{tikzpicture}}

\newcommand{\intersectiongaugeedgeedgecollaps}{
\begin{tikzpicture}[scale=0.4,baseline=(pt_base.base)]

\coordinate (pt_base) at (0,0) ;

\draw[blue] (-2,1.25) -- (-0.625,0.5) ;
\draw[blue] (-1,1.25) -- (-0.3125,0.5) ;
\draw[blue] (1,1.25) -- (0.3125,0.5) ;
\draw[blue] (2,1.25) -- (0.625,0.5) ;

\draw[red] (-0.625,0.5) -- (0,0) ;
\draw[red] (-0.3125,0.5) -- (0,0) ;
\draw[red] (0.625,0.5) -- (0,0) ;
\draw[red] (0.3125,0.5) -- (0,0) ;
\draw[red] (0,-1.25) -- (0,0) ;

\node[scale = 0.1] at (0,0) {\pointbullet} ;
\node[scale = 0.1] at (0,-1.25) {\pointbullet} ;
\node[xshift = - 7 pt , yshift = -8 pt] at (0,0) {\small $e_q$} ; 

\draw[blue] [densely dotted] (-0.5,0.9) -- (0.5,0.9) ;

\draw  (-2,0.5) -- (2,0.5) ;

\end{tikzpicture}}

\newcommand{\directlyadjacentgaugevertex}{
\begin{tikzpicture}[scale=0.4,baseline=(pt_base.base)]

\coordinate (pt_base) at (0,0) ;

\draw[blue] (-2,1.25) -- (0,0) ;
\draw[blue] (-1,1.25) -- (0,0) ;
\draw[blue] (1,1.25) -- (0,0) ;
\draw[blue] (2,1.25) -- (0,0) ;
\draw[red] (0,-0.5) -- (0,0) ;

\draw[red] (0,-1.25) -- (0,-0.5) ;

\node[scale = 0.1] at (0,0) {\pointbullet} ;
\node[scale = 0.1] at (1,1.25) {\pointbullet} ;
\node[below right] at (0,0) {\small $v_k$} ;

\node[xshift = -4 pt] at (0.5,0.9) {\small $e_p$} ;

\draw  (-1.5,0) -- (1.5,0) ;

\end{tikzpicture}}

\newcommand{\exampleunderlyingbroken}{
\begin{tikzpicture}[scale=0.2,baseline=(pt_base.base)]

\begin{scope}

\begin{scope}[shift = {(1.25,3)}]
\draw[blue] (-1.25,1.25) -- (-0.5,0.5) ;
\draw[blue] (1.25,1.25) -- (0.5,0.5) ;

\draw[red] (-0.5,0.5) -- (0,0) ;
\draw[red] (0.5,0.5) -- (0,0) ;
\draw[red] (0,-1.25) -- (0,0) ;

\draw (-1.25,0.5) -- (1.25,0.5) ;
\end{scope}

\begin{scope}[red]
\draw (-1.25,1.25) -- (0,0) ;
\draw (1.25,1.25) -- (0,0) ;
\draw (0,-1.25) -- (0,0) ;
\end{scope}

\begin{scope}[shift = {(-1.25,3)}]
\draw[blue] (0,1.25) -- (0,0.5) ;
\draw[red] (0,-1.25) -- (0,0.5) ;

\draw (-1.25,0.5) -- (0.75,0.5) ;
\end{scope}

\end{scope}

\draw[>=stealth,->] (3,1) -- (5,1) ; 

\begin{scope}[shift = {(8,0)}]

\coordinate (pt_base) at (0,0) ;

\begin{scope}
\draw (-1.25,1.25) -- (0,0) ;
\draw (1.25,1.25) -- (0,0) ;
\draw (0,-1.25) -- (0,0) ;
\end{scope}

\begin{scope}[shift = {(1.25,3)}]
\draw (-1.25,1.25) -- (0,0) ;
\draw (1.25,1.25) -- (0,0) ;
\draw (0,-1.25) -- (0,0) ;
\end{scope}

\end{scope}

\end{tikzpicture}}

\newcommand{\orderingarbrebicoloresignesA}{
\begin{tikzpicture}[scale = 0.2 ,baseline=(pt_base.base)]

\coordinate (pt_base) at (0,-0.5) ;
\node (A) at (-1.5,1) {} ;
\node (B) at (0,0) {} ;
\node (C) at (1.5,1) {} ;
\node (D) at (0,-1) {} ;
\node (E) at (0.25,2) {} ;
\node (F) at (2.75,2) {} ;
\node (G) at (-0.25,2) {} ;
\node (I) at (-2.75,2) {} ;
\node (O) at (0.9,0.6) {} ;
\node (P) at (-0.9,0.6) {} ;

\node (K) at (-2.75,0.6) {} ;
\node (L) at (2.75,0.6) {} ;

\node at (-1.5,0) {\tiny $e_1$} ;
\node at (1.5,0) {\tiny $e_2$} ;

\node (M) at (3,0.6) {} ;
\node (N) at (3,0) {} ;

\draw[blue] (F.center) -- (C.center) ;
\draw[blue] (E.center) -- (C.center) ;
\draw[blue] (O.center) -- (C.center) ;
\draw[blue] (A.center) -- (I.center) ;
\draw[blue] (A.center) -- (G.center) ;
\draw[blue] (A.center) -- (P.center) ;
\draw[red] (B.center) -- (P.center) ;
\draw[red] (B.center) -- (O.center) ;
\draw[red] (B.center) -- (D.center) ;

\draw (K.center) -- (L.center) ;

\end{tikzpicture}
}

\newcommand{\arbrebicoloresignesA}{
\begin{tikzpicture}[scale = 0.2 ,baseline=(pt_base.base)]

\coordinate (pt_base) at (0,-0.5) ;
\node (A) at (-1.5,1) {} ;
\node (B) at (0,0) {} ;
\node (C) at (1.5,1) {} ;
\node (D) at (0,-1) {} ;
\node (E) at (0.25,2) {} ;
\node (F) at (2.75,2) {} ;
\node (G) at (-0.25,2) {} ;
\node (I) at (-2.75,2) {} ;
\node (O) at (0.9,0.6) {} ;
\node (P) at (-0.9,0.6) {} ;

\node (K) at (-2.75,0.6) {} ;
\node (L) at (2.75,0.6) {} ;

\node (M) at (3,0.6) {} ;
\node (N) at (3,0) {} ;

\draw[blue] (F.center) -- (C.center) ;
\draw[blue] (E.center) -- (C.center) ;
\draw[blue] (O.center) -- (C.center) ;
\draw[blue] (A.center) -- (I.center) ;
\draw[blue] (A.center) -- (G.center) ;
\draw[blue] (A.center) -- (P.center) ;
\draw[red] (B.center) -- (P.center) ;
\draw[red] (B.center) -- (O.center) ;
\draw[red] (B.center) -- (D.center) ;

\draw (K.center) -- (L.center) ;

\end{tikzpicture}
}

\newcommand{\arbrebicoloresignesF}{
\begin{tikzpicture}[scale = 0.2 ,baseline=(pt_base.base)]

\coordinate (pt_base) at (0,-0.5) ;
\node (A) at (-1.5,1) {} ;
\node (B) at (0,0) {} ;
\node (C) at (1.5,1) {} ;
\node (D) at (0,-1) {} ;
\node (E) at (0.25,2) {} ;
\node (F) at (2.75,2) {} ;
\node (G) at (-0.25,2) {} ;
\node (I) at (-2.75,2) {} ;

\node (K) at (-2.75,0) {} ;
\node (L) at (2.75,0) {} ;

\draw[blue] (F.center) -- (C.center) ;
\draw[blue] (E.center) -- (C.center) ;
\draw[blue] (B.center) -- (C.center) ;
\draw[blue] (A.center) -- (I.center) ;
\draw[blue] (A.center) -- (G.center) ;
\draw[blue] (A.center) -- (B.center) ;
\draw[red] (B.center) -- (D.center) ;

\draw (K.center) -- (L.center) ;

\end{tikzpicture}
}

\newcommand{\arbrebicoloresignesD}{
\begin{tikzpicture}[scale = 0.2 ,baseline=(pt_base.base)]

\coordinate (pt_base) at (0,-0.5) ;
\node (A) at (-1.5,1) {} ;
\node (B) at (0,0) {} ;
\node (C) at (0.9,0.6) {} ;
\node (D) at (0,-1) {} ;
\node (E) at (0.25,2) {} ;
\node (F) at (2.75,2) {} ;
\node (G) at (-0.25,2) {} ;
\node (I) at (-2.75,2) {} ;
\node (O) at (0.9,0.6) {} ;
\node (P) at (-0.9,0.6) {} ;

\node (K) at (-2.75,0.6) {} ;
\node (L) at (2.75,0.6) {} ;

\node (M) at (3,0.6) {} ;
\node (N) at (3,0) {} ;

\draw[blue] (F.center) -- (C.center) ;
\draw[blue] (E.center) -- (C.center) ;
\draw[blue] (O.center) -- (C.center) ;
\draw[blue] (A.center) -- (I.center) ;
\draw[blue] (A.center) -- (G.center) ;
\draw[blue] (A.center) -- (P.center) ;
\draw[red] (B.center) -- (P.center) ;
\draw[red] (B.center) -- (O.center) ;
\draw[red] (B.center) -- (D.center) ;

\draw (K.center) -- (L.center) ;

\end{tikzpicture}
}

\newcommand{\arbrebicoloresignesE}{
\begin{tikzpicture}[scale = 0.2 ,baseline=(pt_base.base)]

\coordinate (pt_base) at (0,-0.5) ;
\node (A) at (-0.9,0.6) {} ;
\node (B) at (0,0) {} ;
\node (C) at (1.5,1) {} ;
\node (D) at (0,-1) {} ;
\node (E) at (0.25,2) {} ;
\node (F) at (2.75,2) {} ;
\node (G) at (-0.25,2) {} ;
\node (I) at (-2.75,2) {} ;
\node (O) at (0.9,0.6) {} ;
\node (P) at (-0.9,0.6) {} ;

\node (K) at (-2.75,0.6) {} ;
\node (L) at (2.75,0.6) {} ;

\node (M) at (3,0.6) {} ;
\node (N) at (3,0) {} ;

\draw[blue] (F.center) -- (C.center) ;
\draw[blue] (E.center) -- (C.center) ;
\draw[blue] (O.center) -- (C.center) ;
\draw[blue] (A.center) -- (I.center) ;
\draw[blue] (A.center) -- (G.center) ;
\draw[blue] (A.center) -- (P.center) ;
\draw[red] (B.center) -- (P.center) ;
\draw[red] (B.center) -- (O.center) ;
\draw[red] (B.center) -- (D.center) ;

\draw (K.center) -- (L.center) ;

\end{tikzpicture}
}

\newcommand{\arbrebicoloresignesI}{
\begin{tikzpicture}[scale = 0.24,baseline=(pt_base.base)]

\coordinate (pt_base) at (0,-0.5) ;

\begin{scope}
\node (K) at (-0.625,0.625) {} ;
\node (L) at (0,0) {} ;
\draw[blue] (-0.625,0.625) -- (0,1.25) ;
\draw[blue] (-1.25,1.25) -- (-0.3,0.3) ;
\draw[blue] (1.25,1.25) -- (0.3,0.3) ;
\draw[red] (0,0) -- (-0.3,0.3) ;
\draw[red] (0,0) -- (0.3,0.3) ;
\draw[red] (0,-1.25) -- (0,0) ;

\draw (-1.25,0.3) -- (1.25,0.3) ;
\end{scope}

\begin{scope}[shift = {(1.25,2.1)},blue]
\draw (-0.625,0.625) -- (0,0) ;
\draw (0.625,0.625) -- (0,0) ;
\draw (0,-0.625) -- (0,0) ;
\end{scope}
\end{tikzpicture}
}

\newcommand{\arbrebicoloresignesJ}{
\begin{tikzpicture}[scale = 0.24,baseline=(pt_base.base)]

\coordinate (pt_base) at (0,-0.5) ;

\begin{scope}
\node (K) at (-0.625,0.625) {} ;
\node (L) at (0,0) {} ;
\draw[blue] (0.625,0.625) -- (0,1.25) ;
\draw[blue] (-1.25,1.25) -- (-0.3,0.3) ;
\draw[blue] (1.25,1.25) -- (0.3,0.3) ;
\draw[red] (0,0) -- (-0.3,0.3) ;
\draw[red] (0,0) -- (0.3,0.3) ;
\draw[red] (0,-1.25) -- (0,0) ;

\draw (-1.25,0.3) -- (1.25,0.3) ;
\end{scope}

\begin{scope}[shift = {(-1.25,2.1)},blue]
\draw (-0.625,0.625) -- (0,0) ;
\draw (0.625,0.625) -- (0,0) ;
\draw (0,-0.625) -- (0,0) ;
\end{scope}
\end{tikzpicture}
}

\newcommand{\arbrebicoloresignesK}{
\begin{tikzpicture}[scale = 0.14,baseline=(pt_base.base)]

\coordinate (pt_base) at (0,-0.5) ;

\begin{scope}[red]
\draw (-1.6,1.25) -- (0,0) ;
\draw (1.6,1.25) -- (0,0) ;
\draw (0,-1.25) -- (0,0) ;
\end{scope}

\begin{scope}[shift = {(-1.6,2.8)}]
\draw[blue] (-1.25,1.25) -- (0,0) ;
\draw[blue] (1.25,1.25) -- (0,0) ;
\draw[blue] (0,-0.625) -- (0,0) ;
\draw[red] (0,-0.625) -- (0,-1.25) ;
\draw (-1.25,-0.625) -- (1.25,-0.625) ;
\end{scope}

\begin{scope}[shift = {(1.6,2.8)}]
\draw[blue] (-1.25,1.25) -- (0,0) ;
\draw[blue] (1.25,1.25) -- (0,0) ;
\draw[blue] (0,-0.625) -- (0,0) ;
\draw[red] (0,-0.625) -- (0,-1.25) ;
\draw (-1.25,-0.625) -- (1.25,-0.625) ;
\end{scope}
\end{tikzpicture}
}

\newcommand{\examplemappsi}{

\begin{tikzpicture}
\draw (-2,2) node{} -- (-1,1) node{} ;
\draw (-1,1) node{} -- (0,0) node{} node[midway,sloped,allow upside down,scale=0.1]{\pointbullet} node[midway,below left]{\tiny $M$} ;
\draw (2,2) node{} -- (0,0) node{} ;
\draw (0,0) node{} -- (0,-1) node{} ;
\draw (0,2) node{} -- (-1,1) node{}  ;

\node (O) at (0,0) {};
\node (A) at (-1,1) {} ;
\node[scale = 0.1] (a) at ($(A)!0.7!(O)$)  {};
\node (B) at (1,1) {} ;
\node[scale = 0.1] (b) at ($(B)!0.7!(O)$)  {\pointbullet};
\node (C) at (0,-1) {} ;
\node[scale = 0.1] (c) at ($(C)!0.7!(O)$)  {\pointbullet};

\draw [line width = 1.5 pt,red] (a.center) -- (O.center) ;
\draw [line width = 1.5 pt,red] (b.center) -- (O.center) ;
\draw [line width = 1.5 pt,red] (c.center) -- (O.center) ;

\node (D) at (-2,2) {};
\node (E) at (0,2) {} ;
\node[scale = 0.1] (d) at ($(D)!0.7!(A)$)  {\pointbullet};
\node[scale = 0.1] (e) at ($(E)!0.7!(A)$)  {\pointbullet};
\node (f) at ($(O)!0.7!(A)$)  {};

\draw [line width = 1.5 pt,red] (d.center) -- (A.center) ;
\draw [line width = 1.5 pt,red] (e.center) -- (A.center) ;
\draw [line width = 1.5 pt,red] (f.center) -- (A.center) ;

\node (a) at (-0.5,0.5) {};
\node[scale = 0.1] (b) at ($(B)!0.7!(O)$)  {\pointbullet};
\node[right] at (b) {\tiny $W^U(x_3)$} ;
\node[scale = 0.1] (c) at ($(C)!0.7!(O)$)  {\pointbullet};
\node[below right] at (c) {\tiny $W^S(y)$} ;
\node[scale = 0.1] (d) at ($(D)!0.7!(A)$)  {\pointbullet};
\node[left = 5 pt] at (d) {\tiny $W^U(x_1)$} ;
\node[scale = 0.1] (e) at ($(E)!0.7!(A)$)  {\pointbullet};
\node[right = 5pt] at (e) {\tiny $W^U(x_2)$} ;

\node[scale = 0.1 , green!40!gray] at (O) {\pointbullet} ; 

\node (o) at (0.1 , - 0.01) {} ;
\node (oo) at (-0.1 , - 0.01) {} ;

\draw[->,>=stealth,orange] (b) to[bend left] (O);
\draw[->,>=stealth,orange] (c) to[bend right] (o.center);
\draw[->,>=stealth,orange] (d) node{} to[bend right = 120] (oo.center) node{} ;
\draw[->,>=stealth,orange] (e) to[bend left] (O);
\draw[->,>=stealth,orange] (a) to[bend right] (O);

\node[orange] (p) at (0.2,0.7) {$\psi$} ;

\end{tikzpicture}
}

\newcommand{\examplemapf}{

\begin{tikzpicture}[scale=1.4]


\node (O) at (0,0) {};
\node (A) at (-1,1) {} ;
\node[scale = 0.1] (a) at ($(A)!0.7!(O)$)  {\pointbullet};
\node (B) at (1,1) {} ;
\node[scale = 0.1] (b) at ($(B)!0.7!(O)$)  {{\pointbullet}};
\node (C) at (0,-1) {} ;
\node[scale = 0.1] (c) at ($(C)!0.7!(O)$)  {{\pointbullet}};

\draw (A) -- (O) ; 
\draw (B) -- (O) ;

\draw [line width = 1.5 pt,red] (a.center) -- (O.center) ;
\draw [line width = 1.5 pt,red] (b.center) -- (O.center) ;
\draw [line width = 1.5 pt,red] (c.center) -- (O.center) ;

\node[scale = 0.1] at (a)  {\pointbullet};
\node[left] at (a) {\tiny $W^U(x_1)$} ;
\node[scale = 0.1] at (b)  {\pointbullet};
\node[right] at (b) {\tiny $W^U(x_2)$} ;
\node[scale = 0.1] at (c)  {\pointbullet};


\begin{scope}[shift = {(3,-2)}]

\node (P) at (0,0) {} ;
\node (D) at (-1,1) {} ;
\node[scale = 0.1] (d) at ($(D)!0.7!(P)$)  {{\pointbullet}};
\node (E) at (1,1) {} ;
\node[scale = 0.1] (e) at ($(E)!0.7!(P)$)  {\pointbullet};
\node (F) at (0,-1) {} ;
\node[scale = 0.1] (f) at ($(F)!0.7!(P)$)  {\pointbullet};

\draw (E) -- (P) ;
\draw (F) -- (P) ;

\draw [line width = 1.5 pt,red] (d.center) -- (P.center) ;
\draw [line width = 1.5 pt,red] (e.center) -- (P.center) ;
\draw [line width = 1.5 pt,red] (f.center) -- (P.center) ;

\node[scale = 0.1] at (d)  {\pointbullet};
\node[scale = 0.1] at (e)  {\pointbullet};
\node[right] at (e) {\tiny $W^U(x_3)$} ;
\node[scale = 0.1] at (f)  {\pointbullet};
\node[below right] at (f) {\tiny $W^S(y)$} ;

\end{scope}


\node[scale=0.15] (Z) at (0,-2) {\cross} ;
\node[below left] at (Z) {\tiny $z$} ;

\fill[opacity=0.1] (Z) circle (1.3) ; 
\draw[opacity=0.5] (Z) circle (1.3) ; 
\node[right] at (-1.3,-2) {\tiny $U_z$} ;


\node[scale = 0.1] (Q) at (0.3,-1.7) {\pointbullet} ;
\node[above right] at (Q) {\tiny $M$} ;
\draw (c.center) to[out=-90,in=135] node[pos=0.6,scale = 0.1] (G) {\pointbullet} (Q.center)  ;
\draw (Q.center) to[out=-45,in=135] node[pos=0.3,scale = 0.1] (H) {\pointbullet} (d.center)  ;


\draw[->,>=stealth,orange] (a) to[bend right] node[pos=0.5,left,orange] {\tiny $\psi_{e,\mathbb{X}_{t_2}}$} (c);
\draw[->,>=stealth,orange] (b) to[bend left] (c);
\draw[->,>=stealth,orange] (c) to[bend left] node[pos=0.5,right,orange] {\tiny $\phi^{L-1}$} (G);

\draw[->,>=stealth,blue] (e) to[bend right] node[pos=0.5,above,blue] {\tiny $\psi_{e,\mathbb{X}_{t_1}}$} (d) ;
\draw[->,>=stealth,blue] (f) to[bend left] (d);
\draw[->,>=stealth,blue] (d) to[bend left] node[pos=0.5,below,blue] {\tiny $\phi^{-(L-1)}$} (H);

\draw[->,>=stealth,purple] (Q) to[bend left=60] node[pos=0.5,left,purple] {\tiny $ev_\delta$} (G);
\draw[->,>=stealth,purple] (Q) to[bend left] (H);

\end{tikzpicture}
}

\newcommand{\examplemappsiun}{

\begin{tikzpicture}
\draw (-2,2) node{} -- (-1,1) node{} ;
\draw (-1,1) node{} -- (0,0) node{} ;
\draw (2,2) node{} -- (0,0) node{} ;
\draw (0,0) node{} -- (0,-1) node{} ;
\draw (0,2) node{} -- (-1,1) node{}  ;

\node (O) at (0,0) {};
\node (A) at (-1,1) {} ;
\node[scale = 0.1] (a) at ($(A)!0.7!(O)$)  {};
\node (B) at (1,1) {} ;
\node[scale = 0.1] (b) at ($(B)!0.7!(O)$)  {\pointbullet};
\node (C) at (0,-1) {} ;
\node[scale = 0.1] (c) at ($(C)!0.7!(O)$)  {\pointbullet};

\draw [line width = 1.5 pt,red] (a.center) -- (O.center) ;
\draw [line width = 1.5 pt,red] (b.center) -- (O.center) ;
\draw [line width = 1.5 pt,red] (c.center) -- (O.center) ;

\node (D) at (-2,2) {};
\node (E) at (0,2) {} ;
\node[scale = 0.1] (d) at ($(D)!0.7!(A)$)  {\pointbullet};
\node[scale = 0.1] (e) at ($(E)!0.7!(A)$)  {\pointbullet};
\node (f) at ($(O)!0.7!(A)$)  {};

\draw [line width = 1.5 pt,red] (d.center) -- (A.center) ;
\draw [line width = 1.5 pt,red] (e.center) -- (A.center) ;
\draw [line width = 1.5 pt,red] (f.center) -- (A.center) ;

\node (a) at (-0.5,0.5) {};
\node[scale = 0.1] (b) at ($(B)!0.7!(O)$)  {\pointbullet};
\node[right] at (b) {\tiny $W^U(x_3)$} ;
\node[scale = 0.1] (c) at ($(C)!0.7!(O)$)  {\pointbullet};
\node[below right] at (c) {\tiny $W^S(y)$} ;
\node[scale = 0.1] (d) at ($(D)!0.7!(A)$)  {\pointbullet};
\node[left = 5 pt] at (d) {\tiny $W^U(x_1)$} ;
\node[scale = 0.1] (e) at ($(E)!0.7!(A)$)  {\pointbullet};
\node[right = 5pt] at (e) {\tiny $W^U(x_2)$} ;

\node[scale = 0.1 , green!40!gray] at (O) {\pointbullet} ; 

\node (o) at (0.1 , - 0.01) {} ;
\node (oo) at (-0.1 , - 0.01) {} ;

\draw[->,>=stealth,orange] (b) to[bend left] (O);
\draw[->,>=stealth,orange] (c) to[bend right] (o.center);
\draw[->,>=stealth,orange] (d) node{} to[bend right = 60] (oo.center) node{} ;
\draw[->,>=stealth,orange] (e) to[bend left] (O);

\node[orange] (p) at (-1,-0.5) {$\phi_{\mathbb{X}_t}$} ;

\end{tikzpicture}
}

\newcommand{\examplemappsideux}{

\begin{tikzpicture}
\draw (-2,2) node{} -- (-1,1) node{} ;
\draw (-1,1) node{} -- (0,0) node{} ;
\draw (2,2) node{} -- (0,0) node{} ;
\draw (0,0) node{} -- (0,-1) node{} ;
\draw (0,2) node{} -- (-1,1) node{}  ;

\node (O) at (0,0) {};
\node (A) at (-1,1) {} ;
\node[scale = 0.1] (a) at ($(A)!0.7!(O)$)  {};
\node (B) at (1,1) {} ;
\node[scale = 0.1] (b) at ($(B)!0.7!(O)$)  {\pointbullet};
\node (C) at (0,-1) {} ;
\node[scale = 0.1] (c) at ($(C)!0.7!(O)$)  {\pointbullet};

\draw [line width = 1.5 pt,red] (a.center) -- (O.center) ;
\draw [line width = 1.5 pt,red] (b.center) -- (O.center) ;
\draw [line width = 1.5 pt,red] (c.center) -- (O.center) ;

\node (D) at (-2,2) {};
\node (E) at (0,2) {} ;
\node[scale = 0.1] (d) at ($(D)!0.7!(A)$)  {\pointbullet};
\node[scale = 0.1] (e) at ($(E)!0.7!(A)$)  {\pointbullet};
\node (f) at ($(O)!0.7!(A)$)  {};

\draw [line width = 1.5 pt,red] (d.center) -- (A.center) ;
\draw [line width = 1.5 pt,red] (e.center) -- (A.center) ;
\draw [line width = 1.5 pt,red] (f.center) -- (A.center) ;

\node (a) at (-0.5,0.5) {};
\node[scale = 0.1] (b) at ($(B)!0.7!(O)$)  {\pointbullet};
\node[right] at (b) {\tiny $W^U(x_3)$} ;
\node[scale = 0.1] (c) at ($(C)!0.7!(O)$)  {\pointbullet};
\node[below right] at (c) {\tiny $W^S(y)$} ;
\node[scale = 0.1] (d) at ($(D)!0.7!(A)$)  {\pointbullet};
\node[left = 5 pt] at (d) {\tiny $W^U(x_1)$} ;
\node[scale = 0.1] (e) at ($(E)!0.7!(A)$)  {\pointbullet};

\node[scale = 0.1 , green!40!gray] at (O) {\pointbullet} ; 

\node (o) at (0.1 , - 0.01) {} ;
\node (oo) at (-0.1 , - 0.01) {} ;

\draw[->,>=stealth,orange] (b) to[bend right] (e);
\draw[->,>=stealth,orange] (c) to[bend left] (e);
\draw[->,>=stealth,orange] (d) node{} to[bend left] (e) ;

\node[orange] (p) at (-1,-0.5) {$\psi_{e_2,\mathbb{X}_t}$} ;

\end{tikzpicture}
}

\newcommand{\examplemappsitrois}{

\begin{tikzpicture}
\draw (-2,2) node{} -- (-1,1) node{} ;
\draw (-1,1) node{} -- (0,0) node{} ;
\draw (2,2) node{} -- (0,0) node{} ;
\draw (0,0) node{} -- (0,-1) node{} ;
\draw (0,2) node{} -- (-1,1) node{}  ;

\node (O) at (0,0) {};
\node (A) at (-1,1) {} ;
\node[scale = 0.1] (a) at ($(A)!0.7!(O)$)  {};
\node (B) at (1,1) {} ;
\node[scale = 0.1] (b) at ($(B)!0.7!(O)$)  {\pointbullet};
\node (C) at (0,-1) {} ;
\node[scale = 0.1] (c) at ($(C)!0.7!(O)$)  {\pointbullet};

\draw [line width = 1.5 pt,red] (a.center) -- (O.center) ;
\draw [line width = 1.5 pt,red] (b.center) -- (O.center) ;
\draw [line width = 1.5 pt,red] (c.center) -- (O.center) ;

\node (D) at (-2,2) {};
\node (E) at (0,2) {} ;
\node[scale = 0.1] (d) at ($(D)!0.7!(A)$)  {\pointbullet};
\node[scale = 0.1] (e) at ($(E)!0.7!(A)$)  {\pointbullet};
\node (f) at ($(O)!0.7!(A)$)  {};

\draw [line width = 1.5 pt,red] (d.center) -- (A.center) ;
\draw [line width = 1.5 pt,red] (e.center) -- (A.center) ;
\draw [line width = 1.5 pt,red] (f.center) -- (A.center) ;

\node (a) at (-0.5,0.5) {};
\node[scale = 0.1] (b) at ($(B)!0.7!(O)$)  {\pointbullet};
\node[right] at (b) {\tiny $W^U(x_3)$} ;
\node[scale = 0.1] (c) at ($(C)!0.7!(O)$)  {\pointbullet};
\node[scale = 0.1] (d) at ($(D)!0.7!(A)$)  {\pointbullet};
\node[left = 5 pt] at (d) {\tiny $W^U(x_1)$} ;
\node[scale = 0.1] (e) at ($(E)!0.7!(A)$)  {\pointbullet};
\node[right = 5pt] at (e) {\tiny $W^U(x_2)$} ;

\node[scale = 0.1 , green!40!gray] at (O) {\pointbullet} ; 

\node (o) at (0.1 , - 0.01) {} ;
\node (oo) at (-0.1 , - 0.01) {} ;

\draw[->,>=stealth,orange] (b) to[bend left = 70] (c);
\draw[->,>=stealth,orange] (d) to[bend right] (c);
\draw[->,>=stealth,orange] (e) node{} to[bend left = 50] (c) ;

\node[orange] (p) at (-1,-0.5) {$\psi_{e_0,\mathbb{X}_t}$} ;

\end{tikzpicture}
}

\newcommand{\flecheun}{
\begin{tikzpicture}
\draw[->,>=stealth] (0,0) -- (-1,0);
\end{tikzpicture}
}

\newcommand{\flechedeux}{
\begin{tikzpicture}
\draw[->,>=stealth] (0,0) -- (0,-1);
\end{tikzpicture}
}

\newcommand{\compactificationWSy}{
\begin{tikzpicture}[scale = 3.5,baseline=(pt_base.base)]

\coordinate (pt_base) at (0,0) ;

\draw[fill, opacity = 0.1, orange] (0,0) -- (1,0) -- (1,1) -- (0,1) -- (0,0) ;

\node[scale=0.15] (z) at (0,0) {\cross} ;
\node[scale=0.15] (y) at (1,0) {\cross} ;

\draw (z) node{} to node[pos=0.5,scale=0.1] {\midarrow} node[pos=0.5,below= 10 pt] {$\gamma$} (y) node{}  ;
\draw[blue]  (0,1) node{} -- (z) node{} node[midway,left] { $W^S(z)$} ;

\draw(0.05,1) node{} to[bend right] node[pos=0.5,scale=0.1]{\pointbullet} node[pos=0.5,scale=0.75,xshift=-13pt] {\flecheun} node[pos=0.5,scale=0.75,xshift=-13pt,yshift=-13pt] {$-\dot{\gamma}$} (y) node{}  ;

\node[orange] (inst) at (0.7,0.7) {$W^S(y)$} ;

\node[below] (z) at (0,0) {\small $z$} ;
\node[below] (y) at (1,0) {\small $y$} ;

\end{tikzpicture}
}

\newcommand{\compactificationWUx}{
\begin{tikzpicture}[scale = 3.5,baseline=(pt_base.base)]

\coordinate (pt_base) at (0,0) ;

\draw[fill, opacity = 0.1, orange] (0,0) -- (1,0) -- (1,1) -- (0,1) -- (0,0) ;

\node[orange] (inst) at (0.7,0.7) {$W^U(x)$} ;

\node[scale=0.15] (x) at (0,1) {\cross} ;
\node[scale=0.15] (z) at (0,0) {\cross} ;

\draw (x) node{} to node[pos=0.5,sloped,allow upside down,scale=0.1] {\midarrow} node[pos=0.5,left= 10 pt] {$\gamma$} (z) node{}  ;
\draw[blue]  (z) node{} -- (1,0) node{} node[midway,below] { $W^U(z)$} ;

\draw (0.05,1) node{} to[bend right] node[pos=0.5,scale=0.1]{\pointbullet} node[pos=0.5,scale=0.75,yshift=-13pt] {\flechedeux} node[pos=0.5,scale=0.75,xshift=-13pt,yshift=-13pt] {$\dot{\gamma}$} (1,0.01) node{}  ;

\node[left] (x) at (0,1) {\small $x$} ;
\node[left] (z) at (0,0) {\small $z$} ;

\end{tikzpicture}
}

\newcommand{\diagramphilosophy}{
\begin{tikzpicture}[baseline=(pt_base.base)]

\coordinate (pt_base) at (0,-1) ;

\node[text width=6 cm, align = center] (A) at (5,-1) {Set of operations in symplectic topology (e.g. on the Fukaya category)} ;
\node[text width=6 cm, align = center] (B) at (-5,-1) {Moduli spaces of pseudo-holomorphic curves (e.g. disks with Lagrangian boundary conditions)} ;
\node[text width=6 cm, align = center] (C) at (-5,1.5) {Compactified moduli spaces of curves (e.g. $\overline{\mathcal{M}}_{n,1}$)} ;
\node[text width=6 cm, align = center] (D) at (5,1.5) {Operadic object (e.g. operad \Ainf )} ;

\draw[->,>=stealth,densely dotted] (D) -- (A) node[midway,right] {\Small \textit{Encodes}} ;
\draw[->,>=stealth,densely dotted] (B) -- (A) node[midway,below,text width=4 cm, align = center] {\Small \textit{Counting 0-dimensional moduli spaces}} ;
\draw[->,>=stealth] (C) -- (B) node[midway,left] {\Small \textit{Floer theory}} ;
\draw[->,>=stealth] (C) -- (D) node[midway,above] {\Small \textit{Functor $C^*_{cell}$}} ;
\end{tikzpicture}
}

\newcommand{\brokenperturbedmorsegradientreeun}{
\begin{tikzpicture}[scale = 0.3,baseline=(pt_base.base)]

\coordinate (pt_base) at (0,-0.5) ;

\begin{scope}
\draw (-1.25,1.25) -- (0,0) ;
\draw (1.25,1.25) -- (0,0) ;
\draw (0,-1.25) -- (0,0) ;
\end{scope}

\begin{scope}[xshift = -1.25 cm , yshift = 1.4 cm]
\draw (0,0 ) -- (0,1) ; 
\end{scope}

\begin{scope}[xshift = -1.25 cm , yshift = 3.8 cm]
\draw (-1.25,1.25) -- (0,0) ;
\draw (1.25,1.25) -- (0,0) ;
\draw (0,-1.25) -- (0,0) ;
\end{scope}

\begin{scope}[yshift = 5.2 cm]
\draw (0,0 ) -- (0,1) ;
\end{scope}

\end{tikzpicture} }

\newcommand{\brokenperturbedmorsegradientreedeux}{
\begin{tikzpicture}[scale = 0.3,baseline=(pt_base.base)]

\coordinate (pt_base) at (0,-0.5) ;

\begin{scope}
\draw (-1.25,1.25) -- (0,0) ;
\draw (1.25,1.25) -- (0,0) ;
\draw (0,-1.25) -- (0,0) ;
\end{scope}

\begin{scope}[xshift = -1.25 cm , yshift = 1.4 cm]
\draw (0,0 ) -- (0,1) ; 
\end{scope}

\begin{scope}[xshift = -1.25 cm , yshift = 2.55 cm]
\draw (0,0 ) -- (0,1) ;
\end{scope}

\begin{scope}[xshift = -1.25 cm , yshift = 4.95 cm]
\draw (-1.25,1.25) -- (0,0) ;
\draw (1.25,1.25) -- (0,0) ;
\draw (0,-1.25) -- (0,0) ;
\end{scope}

\begin{scope}[xshift = 1.25 cm , yshift = 1.4 cm]
\draw (0,0 ) -- (0,1) ;
\end{scope}

\end{tikzpicture} }

\newcommand{\brokenperturbedmorsegradientreetrois}{
\begin{tikzpicture}[scale = 0.3,baseline=(pt_base.base)]

\node (O) at (0,0) {} ;
\node (o) at (-1.25,5) {} ; 
\node[scale = 0.08] (A) at (0,-1.25) {\cross} ;
\node[below] at (A) {\Small $y$} ; 
\node[scale = 0.08] (B) at (1.25,1.25) {\cross} ; 
\node[above] at (B) {\Small $x_3$} ; 
\node[scale = 0.08] (C) at (-1.25,1.25) {\cross} ; 
\node[left] at (C) {\Small $z_2$} ; 
\node[scale = 0.08] (D) at (-1.25,3.75) {\cross} ; 
\node[left] at (D) {\Small $z_1$} ; 
\node[scale = 0.08] (E) at (-2.5,6.25) {\cross} ; 
\node[above] at (E) {\Small $x_1$} ; 
\node[scale = 0.08] (F) at (0,6.25) {\cross} ; 
\node[above] at (F) {\Small $x_2$} ; 

\draw (A) -- (O.center) ;
\draw (B) -- (O.center) ;
\draw (C) -- (O.center) ;
\draw (C) -- (D) ;
\draw (D) -- (o.center) ;
\draw (E) -- (o.center) ;
\draw (F) -- (o.center) ;
\end{tikzpicture} }

\newcommand{\examplemapsigma}{

\begin{tikzpicture}
\draw (-2,2) node{} -- (-1,1) node{} ;
\draw (-1,1) node{} -- (0,0) node{} ;
\draw (2,2) node{} -- (0,0) node{} ;
\draw (0,0) node{} -- (0,-1) node{} ;
\draw (0,2) node{} -- (-1,1) node{}  ;

\node (O) at (0,0) {};
\node (A) at (-1,1) {} ;
\node[scale = 0.1] (a) at ($(A)!0.7!(O)$)  {};
\node (B) at (1,1) {} ;
\node[scale = 0.1] (b) at ($(B)!0.7!(O)$)  {\pointbullet};
\node (C) at (0,-1) {} ;
\node[scale = 0.1] (c) at ($(C)!0.7!(O)$)  {\pointbullet};

\draw [line width = 1.5 pt,red] (a.center) -- (O.center) ;
\draw [line width = 1.5 pt,red] (b.center) -- (O.center) ;
\draw [line width = 1.5 pt,red] (c.center) -- (O.center) ;

\node (D) at (-2,2) {};
\node (E) at (0,2) {} ;
\node[scale = 0.1] (d) at ($(D)!0.7!(A)$)  {\pointbullet};
\node[scale = 0.1] (e) at ($(E)!0.7!(A)$)  {\pointbullet};
\node (f) at ($(O)!0.7!(A)$)  {};

\draw [line width = 1.5 pt,red] (d.center) -- (A.center) ;
\draw [line width = 1.5 pt,red] (e.center) -- (A.center) ;
\draw [line width = 1.5 pt,red] (f.center) -- (A.center) ;

\node (a) at (-0.5,0.5) {};
\node[scale = 0.1] (b) at ($(B)!0.7!(O)$)  {\pointbullet};
\node[scale = 0.1] (c) at ($(C)!0.7!(O)$)  {\pointbullet};
\node[below right] at (c) {\tiny $W^S(y)$} ;
\node[scale = 0.1] (d) at ($(D)!0.7!(A)$)  {\pointbullet};;
\node[scale = 0.1] (e) at ($(E)!0.7!(A)$)  {\pointbullet};

\node[scale = 0.1 , green!40!gray] at (O) {\pointbullet} ; 

\node (o) at (0.1 , - 0.01) {} ;
\node (oo) at (-0.1 , - 0.01) {} ;

\draw[->,>=stealth,orange] (c) to[bend right] (b);
\draw[->,>=stealth,orange] (c) to[bend left] (d);
\draw[->,>=stealth,orange] (c) to[bend right=60] (e) ;

\node[orange] (p) at (-1,-0.5) {$\sigma_{e_0,\mathbb{X}_{t^1}}$} ;

\end{tikzpicture}
}

\maketitle

\begin{abstract}
Elaborating on works by Abouzaid and Mescher, we prove that for a Morse function on a smooth compact manifold, its Morse cochain complex can be endowed with an $\Omega B As$-algebra structure by counting moduli spaces of perturbed Morse gradient trees. This rich structure descends to its already known $\Ainf$-algebra structure. We then introduce the notion of \ombas -morphism between two \ombas -algebras and prove that given two Morse functions, one can construct an $\Omega B As$-morphism between their associated $\Omega B As$-algebras by counting moduli spaces of two-colored perturbed Morse gradient trees. This morphism induces a standard \Ainf -morphism between the induced \Ainf -algebras. We work with integer coefficients, and provide to this extent a detailed account on the sign conventions for \Ainf\ (resp. $\Omega B As$)-algebras and \Ainf\ (resp. $\Omega B As$)-morphisms, using polytopes (resp. moduli spaces) which explicitly realize the dg-operadic objects encoding them. Our proofs also involve showing at the level of polytopes that an \ombas -morphism between \ombas -algebras naturally induces an \Ainf -morphism between \Ainf -algebras. This paper is adressed to people acquainted with either differential topology or algebraic operads, and written in a way to be hopefully understood by both communities. It comes in particular with a short survey on operads, \Ainf -algebras and \Ainf -morphisms, the associahedra and the multiplihedra. All the details on transversality, gluing maps, signs and orientations for the moduli spaces defining the algebraic structures on the Morse cochains are thorougly carried out. It moreover lays the basis for a second article in which we solve the problem of finding a satisfactory homotopic notion of higher morphisms between \Ainf -algebras and between \ombas -algebras, and show how this higher algebra of \Ainf\ and \ombas -algebras naturally arises in the context of Morse theory.
\end{abstract}

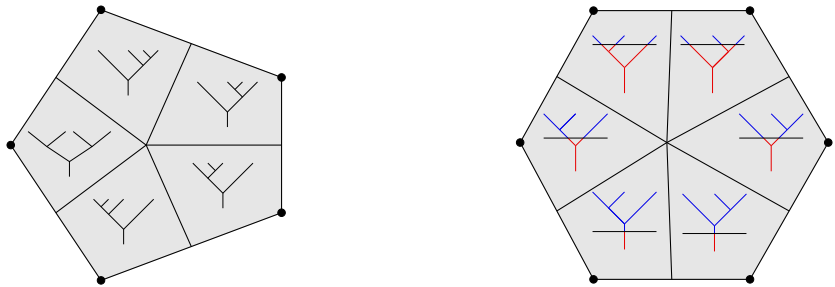
\vspace{30pt}

\begin{figure}[h]
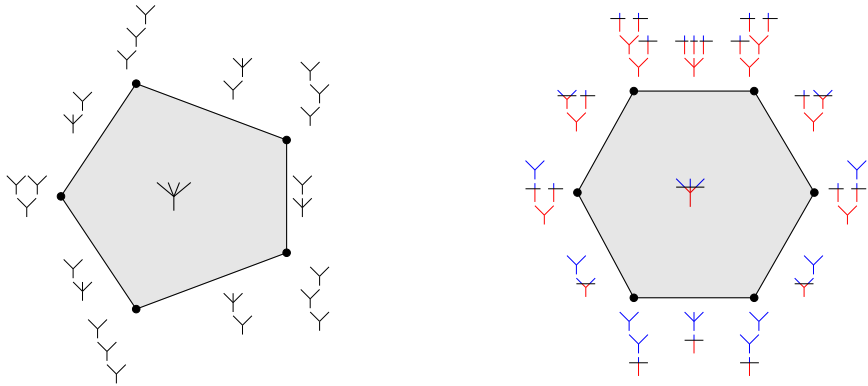

    \centering
    \begin{subfigure}{0.4\textwidth}
    \centering
        \associaedretrois
    \end{subfigure} ~
    \begin{subfigure}{0.4\textwidth}
    \centering
        \multiplaedretrois
    \end{subfigure}
    \caption*{\textit{The associahedron $K_4$ and the multiplihedron $J_3$ ...}}
\end{figure}

\newpage

\setcounter{tocdepth}{1}

\tableofcontents

\vspace{30pt}

\begin{figure}[h]
    \centering
    \begin{subfigure}{0.4\textwidth}
    \centering
        \Tquatrestratbis
    \end{subfigure} ~
    \begin{subfigure}{0.4\textwidth}
    \centering
        \CTtroisstratbis
    \end{subfigure}
    \vspace{10pt}
    \caption*{\textit{... and their $\Omega B As$-cell decompositions}}
\end{figure}

\newpage

\begin{leftbar}
\part*{Introduction}\label{p:intro}
\end{leftbar}

\paragraph{\textit{\textbf{Outline of the paper and main results}}} --- Our first part begins with concise and self-contained recollections on the theory of algebraic (non-symmetric) operads, that we subsequently specialize to the case of \Ainf -algebras, \Ainf -morphisms between them and their homotopy theory. We introduce in particular the convenient setting of operadic bimodules to define the operadic bimodule \infmor\ encoding \Ainf -morphisms between \Ainf -algebras. We then recall how the operad \Ainf\ (resp. the operadic bimodule \infmor ) can be realized using families of polytopes, known as the associahedra (resp. multiplihedra). The associahedra can themselves be realized as geometric moduli spaces : the compactified moduli spaces of metric stable ribbon trees $\overline{\mathcal{T}}_n$. These moduli spaces come with a refined cell decomposition encoding the operad $\Omega B As$. Likewise, the multiplihedra can be realized as the compactified moduli spaces of two-colored metric stable ribbon trees $\overline{\mathcal{CT}}_n$. Endowing these moduli spaces with a refined cell decomposition, we introduce a new operadic bimodule : the operadic bimodule \ombasmor , encoding \ombas -morphisms between \ombas -algebras.

\theoremstyle{definition}
\newtheorem*{alg:def:ombas-morph}{Definition~\ref{alg:def:ombas-morph}}
\begin{alg:def:ombas-morph}
The operadic bimodule \ombasmor\ is the quasi-free $(\Omega B As , \Omega B As)$-operadic bimodule generated by the set of two-colored stable ribbon trees
\[ \Omega B As - \mathrm{Morph} := \mathcal{F}^{\Omega B As, \Omega B As}( \arbreopunmorph , \arbrebicoloreL , \arbrebicoloreM , \arbrebicoloreN , \cdots ) \ ,  \]
where a two-colored stable ribbon tree $t_g$ with $e(t)$ internal edges and whose gauge crosses $j$ vertices has degree $|t_g| := j - e(t) -1$. The differential of a two-colored stable ribbon tree $t_g$ is given by the signed sum of all two-colored stable ribbon trees obtained from $t_g$ under the rule prescribed by the top dimensional strata in the boundary of $\overline{\mathcal{CT}}_n(t_g)$.
\end{alg:def:ombas-morph} 
  
\noindent The \ombas\ framework provides another template to study algebras which are homotopy-associative, together with morphisms between them which preserve the product up to homotopy. This is followed by a comprehensive study on the \Ainf\ and \ombas\ sign conventions. In the \Ainf\ case, we show how the two usual sign conventions for \Ainf -algebras and \Ainf -morphisms are naturally induced by the shifted bar construction viewpoint. Using the Loday realizations of the associahedra~\cite{masuda-diagonal-assoc} and the Forcey-Loday realizations of the multiplihedra~\cite{masuda-diagonal-multipl}, we give a complete proof of the following two folklore propositions :

\theoremstyle{theorem}
\newtheorem*{alg:prop:loday-assoc-multi-signs}{Propositions~\ref{alg:prop:loday-assoc-signs}~and~\ref{alg:prop:loday-multi-signs}}
\begin{alg:prop:loday-assoc-multi-signs}
The Loday realizations of the associahedra and the Forcey-Loday realizations of the multiplihedra determine the usual sign conventions for \Ainf -algebras and \Ainf -morphisms between them.
\end{alg:prop:loday-assoc-multi-signs}

\noindent On the \ombas\ side, we start by recalling the formulation of the operad \ombas\ by Markl and Shnider~\cite{markl-assoc}. We then proceed to study the moduli spaces of stable two-colored metric ribbon trees $\mathcal{CT}_n(t_g)$ and compute the signs arising in the top dimensional strata of their boundary in Propositions~\ref{alg:prop:int-collapse-signs}~to~\ref{alg:prop:below-break-signs}. This allows us to complete our definition of the operadic bimodule \ombasmor\ by making explicit the signs for the action-composition maps and the differential.
We finally give an alternative and more geometric construction of the morphism of operads $\Ainf \rightarrow \ombas$ defined in~\cite{markl-assoc}, using the realizations of the associahedra as geometric moduli spaces. We then build an explicit morphism of operadic bimodules $\infmor \rightarrow \ombasmor$ applying the same ideas to the moduli spaces realizing the multiplihedra.

\theoremstyle{theorem}
\newtheorem*{alg:prop:markl-shnider}{Propositions~\ref{alg:prop:markl-shnider-un}~and~\ref{alg:prop:markl-shnider-deux}}
\begin{alg:prop:markl-shnider}
There exist a geometric morphism of operads $\Ainf \rightarrow \ombas$ and a geometric morphism of operadic bimodules $\infmor \rightarrow \ombasmor$.
\end{alg:prop:markl-shnider}

\emph{Morse theory} corresponds to the study of manifolds endowed with a \emph{Morse function}, i.e. a function whose critical points are non-degenerate. Given a smooth compact manifold $M$, Fukaya constructed in~\cite{fukaya-morse-homotopy} an \Ainf -category whose objects are functions $f_i$ on $M$, whose spaces of morphisms between two functions $f_i$ and $f_j$ (such that $f_i-f_j$ is Morse) are the Morse cochain complexes $C^*(f_i-f_j)$, and whose higher multiplications are defined by counting moduli spaces of Morse ribbon trees. Adapting this construction to the case of a single Morse function $f$ on $M$, Abouzaid defines in~\cite{abouzaid-plumbings} an \Ainf -algebra structure on the Morse cochains $C^*(f)$ by counting moduli spaces of \emph{perturbed Morse gradient ribbon trees}. His work was subsequently continued by Mescher in~\cite{mescher-morse}. 

\noindent In the second part of this paper, we adapt the constructions of Abouzaid~\cite{abouzaid-plumbings}, using the terminology of Mescher~\cite{mescher-morse}, to perform two constructions on the Morse cochains $C^*(f)$. Firstly, we introduce the notion of smooth choices of perturbation data $\mathbb{X}_n$ on the moduli spaces $\mathcal{T}_n$ that we use to define the moduli spaces of perturbed Morse gradient trees $\mathcal{T}_t^{\mathbb{X}_t}(y ; x_1,\dots,x_n)$ modeled on a stable ribbon tree type $t$. 

\theoremstyle{theorem}
\newtheorem*{geo:th:existence-compact}{Theorems~\ref{geo:th:exist-admissible-perturbation-data-Tn}~and~\ref{geo:th:compactification-Tn}}
\begin{geo:th:existence-compact}
Under some generic assumptions on the choice of perturbation data $\{ \mathbb{X}_n \}_{n \geqslant 2}$, the moduli spaces $\mathcal{T}_t^{\mathbb{X}_t}(y ; x_1,\dots,x_n)$ are orientable manifolds. If they have dimension 0, they are compact. If they have dimension 1, they can be compactified to compact manifolds with boundary, whose boundary is modeled on the boundary of the moduli spaces $\mathcal{T}_n(t)$.
\end{geo:th:existence-compact}

\noindent We then show that under a generic choice of perturbation data $\{ \mathbb{X}_n \}_{n \geqslant 2}$ the Morse cochains $C^*(f)$ can be endowed with an \ombas -algebra structure, by counting 0-dimensional moduli spaces of Morse gradient ribbon trees.

\theoremstyle{theorem}
\newtheorem*{geo:th:struct-ombas-alg}{Theorem~\ref{geo:th:ombas-alg}}
\begin{geo:th:struct-ombas-alg} 
Defining for every $n$ and every stable ribbon tree type $t$ of arity $n$ the operation $m_t$ as 
\begin{align*}
m_t : C^*(f) \otimes \cdots \otimes C^*(f) &\longrightarrow C^*(f) \\
x_1 \otimes \cdots \otimes x_n &\longmapsto \sum_{|y|= \sum_{i=1}^n|x_i| - e(t)} \# \mathcal{T}_t^\mathbb{X}(y ; x_1,\cdots,x_n) \cdot y \ ,
\end{align*} 
these operations endow the Morse cochains  $C^*(f)$ with an $\Omega B As$-algebra structure.
\end{geo:th:struct-ombas-alg}

\noindent This \ombas -algebra structure is more canonical than the \Ainf -algebra structure of Abouzaid, as the \ombas -cell decomposition of the associahedra is the natural cell decomposition arising when realizing the moduli spaces of stable metric ribbon trees in Morse theory. This cell decomposition is also more appropriate for a rigorous proof of Theorem~\ref{geo:th:exist-admissible-perturbation-data-Tn}, than the \Ainf -cell decomposition used in~\cite{abouzaid-plumbings}. We recover the \Ainf -algebra structure of Abouzaid using the morphism $\Ainf \rightarrow \ombas$ of Proposition~\ref{alg:prop:markl-shnider-un}.

\noindent Given now two Morse functions $f$ and $g$, we can perform the same constructions in Morse theory using this time the moduli spaces $\mathcal{CT}_n$ as blueprints. The counterparts of Theorems~\ref{geo:th:exist-admissible-perturbation-data-Tn}~and~\ref{geo:th:compactification-Tn} still hold. Moreover, given two generic choices of perturbation data $\mathbb{X}^f$ and $\mathbb{X}^g$, we construct an \ombas -morphism between the \ombas -algebras $C^*(f)$ and $C^*(g)$ by counting 0-dimensional moduli spaces of two-colored Morse gradient trees. This construction provides a first geometric and explicit instance of the newly defined notion of \ombas -morphism.

\theoremstyle{theorem}
\newtheorem*{geo:th:struct-ombas-morph}{Theorem~\ref{geo:th:ombas-morph}}
\begin{geo:th:struct-ombas-morph}
Let $(\mathbb{Y}_n)_{n \geqslant 1}$ be a generic choice of perturbation data on the moduli spaces $\mathcal{CT}_n$.
Defining for every $n$ and every two-colored stable ribbon tree type $t_g$ of arity $n$ the operations $\mu_{t_g}$ as
\begin{align*}
\mu_{t_g}^{\mathbb{Y}} : C^*(f) \otimes \cdots \otimes C^*(f) &\longrightarrow C^*(g) \\
x_1 \otimes \cdots \otimes x_n &\longmapsto \sum_{|y|= \sum_{i=1}^n|x_i| + |t_g|} \# \mathcal{CT}_{t_g}^\mathbb{Y}(y ; x_1,\cdots,x_n) \cdot y \ .
\end{align*} 
these operations fit into an \ombas -morphism $\mu^{\mathbb{Y}} : (C^*(f),m_t^{\mathbb{X}^f}) \rightarrow (C^*(g),m_t^{\mathbb{X}^g})$.
\end{geo:th:struct-ombas-morph}

\noindent This \ombas -morphism yields in particular an \Ainf -morphism between two \Ainf -algebras, using the morphism of Proposition~\ref{alg:prop:markl-shnider-deux}. 
These constructions are followed by a section dedicated to a comprehensive proof of Theorems~\ref{geo:th:exist-admissible-perturbation-data-Tn} and~\ref{geo:th:existence-perturb-CTn}, which clarifies and completes the constructions of~\cite{abouzaid-plumbings}. Our last section on signs and orientations is dedicated to a thorough sign check for Theorems~\ref{geo:th:ombas-alg}~and~\ref{geo:th:ombas-morph}. We show that we have in fact defined a twisted \ombas -algebra structure on the Morse cochains, and a twisted \ombas -morphism between two Morse cochains complexes : when the manifold $M$ is odd-dimensional, the word "twisted" can be dropped.

\theoremstyle{definition}
\newtheorem*{geo:def:twisted-ombas}{Definition~\ref{geo:def:twisted-ombas}}
\begin{geo:def:twisted-ombas}
A \emph{twisted \Ainf -algebra} is a dg-\Z -module $A$ endowed with two different differentials $\partial_1$ and $\partial_2$, and a sequence of degree $2-n$ operations $m_n : A^{\otimes n} \rightarrow A$ such that
\[ [ \partial , m_n ] = - \sum_{\substack{i_1+i_2+i_3=n \\ 2 \leqslant i_2 \leqslant n-1}} (-1)^{i_1 + i_2i_3} m_{i_1+1+i_3} (\ide^{\otimes i_1} \otimes m_{i_2} \otimes \ide^{\otimes i_3} ) \ , \]
where $[ \partial , \cdot ]$ denotes the bracket for the maps $ (A^{\otimes n} , \partial_1) \rightarrow (A , \partial_2)$.
A \emph{twisted $\Omega B As$-algebra} and a \emph{twisted $\Omega B As$-morphism} are defined similarly.
\end{geo:def:twisted-ombas}

\noindent Our computations are performed using the convenient viewpoint of signed short exact sequences of vector bundles. This last section also gives us the opportunity to recall in detail the basic method to compute the relations satisfied by algebraic operations defined in the context of Morse theory or symplectic topology : counting the points on the boundary of an oriented 1-dimensional manifold.  We moreover pay a particular attention to the construction of explicit gluing maps for the 1-dimensional moduli spaces of perturbed Morse gradient trees.

Finally, the third and last part is composed of a series of developments on the algebraic and geometric constructions performed in the first two parts. We show in particular that : 

\theoremstyle{theorem}
\newtheorem*{fd:prop}{Proposition~\ref{fd:prop:quasi-iso-Morse}}
\begin{fd:prop}
The twisted $\Omega B As$-morphism $\mu^{\mathbb{Y}} : (C^*(f),m_t^{\mathbb{X}^f}) \longrightarrow (C^*(g),m_t^{\mathbb{X}^g})$ constructed in Theorem~\ref{geo:th:ombas-morph} is a quasi-isomorphism.
\end{fd:prop}

\noindent We also give a brief overview on the \Ainf -structures appearing in symplectic topology through Floer theory, which is sometimes presented as an infinite-dimensional analogue of Morse theory. In the last section we formulate two problems naturally arising from our constructions. Problem 1 is solved in our second article~\cite{mazuir-II} while Problem 2 is still a work in progress.

\paragraph{\textit{\textbf{Towards article II}}} --- This article completes the existing works on strongly homotopy associative structures arising from Morse theory and clarifies the analytical and algebraic technicalities that they involve. It moreover lays the ground for a second article~\cite{mazuir-II} dealing with two questions. First, understand and define a suitable homotopic notion of higher morphisms between \Ainf -algebras, which would give a satisfactory description of the higher algebra of \Ainf -algebras. Secondly, elaborating on the work of Abouzaid and Mescher on perturbed Morse gradient trees, realize these higher morphisms through moduli spaces in Morse theory.

\paragraph{\textit{\textbf{Acknowledgements}}} My first thanks go to my advisor Alexandru Oancea, for his continuous help and support through the settling of this series of papers. I also express my gratitude to Bruno Vallette for his constant reachability and his suggestions and ideas on the algebra underlying this work. I specially thank Jean-Michel Fischer and Guillaume Laplante-Anfossi who repeatedly took the time to offer explanations on higher algebra and $\infty$-categories. I finally adress my thanks to Florian Bertuol, Thomas Massoni, Amiel Peiffer-Smadja and Victor Roca Lucio for useful discussions. 

\newpage

\begin{leftbar}
\part{Algebra}\label{p:algebra}
\end{leftbar}

\section{Operadic algebra}\label{alg:s:op-alg}

Our first section is devoted to some basic recollections on operadic algebra, and the particular case of the operad \Ainf . The specialist already acquainted with these notions will only have to read sections~\ref{alg:ss:op-bimod} and~\ref{alg:ss:ainf-morph}, which introduce the \emph{operadic bimodule} viewpoint on \Ainf -morphisms through the $(\Ainf , \Ainf )$-operadic bimodule \infmor . All the signs of this section are worked out in section~\ref{alg:ss:signs-ainf-bar}, and will temporarily be written $\pm$ here.

We let in the rest of this section $\mathcal{C}$ be one the following two monoidal categories : the category of differential graded \Z -modules with cohomological convention $(\mathtt{dg-\Z -mod},\otimes)$ and the category of polytopes $(\mathtt{Poly},\times)$, introduced in detail in subsection~\ref{alg:sss:cat-poly}. We will write $\otimes$ for the tensor product on $\mathcal{C}$, and $I$ for its identity element.
Sections~\ref{alg:ss:op} and~\ref{alg:ss:p-alg} are derived from~\cite{loday-vallette-algebraic-operads}. Apart from the operadic bimodule viewpoint, most of the material presented in sections~\ref{alg:ss:op-ainf} and~\ref{alg:ss:ainf-morph} is inspired from~\cite{loday-vallette-algebraic-operads}~and~\cite{vallette-algebra}.

\subsection{Operads} \label{alg:ss:op}

\subsubsection{Definition} \label{alg:sss:def}

\begin{definition}
A \emph{(non-symmetric) $\mathcal{C}$-operad} $P$ consists in the data of a collection of objects $\{ P_n \}_{n \geqslant 1}$ of $\mathcal{C}$ together with a unit element $e \in P_1$ and with compositions
\[ P_k \otimes P_{i_1} \otimes \cdots \otimes P_{i_k} \underset{c_{i_1,\dots,i_k}}{\longrightarrow} P_{i_1 + \dots + i_k} \]
which are unital and associative. 
The objects $P_n$ are to be thought as spaces encoding arity $n$ operations while the compositions $c_{i_1,\dots,i_k}$ define how to compose these operations together.
\end{definition}

Operads can be defined in an equivalent fashion using partial compositions instead of total compositions. An operad is then the data of a collection of objects $\{ P_n\}_{n \geqslant 1}$ together with a unit element $e \in P_1$ and with partial composition maps
\[ \circ_i : P_k \otimes P_h \longrightarrow P_{h+k-1} \ , \ 1 \leqslant i \leqslant k  \]
which are unital and associative.
Finally a morphism of operads $P \rightarrow Q$ is a sequence of maps $P_n \rightarrow Q_n$ compatible with the compositions and preserving the identity.

\subsubsection{Schur functors} \label{alg:sss:schur}

There is a third equivalent definition of operads using the notion of Schur functors. Call any collection $P = \{ P_n \}$ of objects of $\mathcal{C}$ a \emph{$\N$-module}. To each $\N$-module one can associate its \emph{Schur functor}, which is the endofunctor $S_P : \mathcal{C} \rightarrow \mathcal{C}$ defined as
\[ C \longmapsto \bigoplus_{n=1}^\infty P_n \otimes C^{\otimes n} \ . \]

Given two \N -modules $P$ and $Q$, composing their Schur functors gives the following formula
\[ S_P \circ S_Q : C \longrightarrow \bigoplus_{n=1}^\infty \left( \bigoplus_{i_1+\dots+i_k=n} P_k \otimes Q_{i_1} \otimes \cdots \otimes Q_{i_k} \right) \otimes C^{\otimes n} \ . \]
In other words, there is a \N -module associated to the composition of the Schur functors of two \N -modules, and it is given by 
\[ P \circ Q = \{ \bigoplus_{i_1+\dots+i_k=n} P_k \otimes Q_{i_1} \otimes \cdots \otimes Q_{i_k} \}_{n \geqslant 1} \ . \]

The category $(\mathrm{End}(\mathcal{C}), \circ,Id_{\mathcal{C}})$, endowed with composition of endofunctors, is a monoidal category. In particular, there is a well-defined notion of monoid in $\mathrm{End}(\mathcal{C})$. A monoid structure on an endofunctor $F : \mathcal{C} \rightarrow \mathcal{C}$ is the data of natural transformations $\mu_F : F \circ F \rightarrow F$ and $e : Id_{\mathcal{C}} \rightarrow F$, which satisfy the usual commutative diagrams for monoids. 
This viewpoint yields the following equivalent definition of an operad. Albeit tedious, it will prove useful in the following section when considering operadic modules.

\begin{definition}
A \emph{$\mathcal{C}$-operad} is the data of a \N -module $P = \{ P_n \}$ of $\mathcal{C}$ together with a monoid structure on its Schur functor $S_P$.
\end{definition}

\subsection{$P$-algebras} \label{alg:ss:p-alg}

Let $A$ be a dg-\Z -module and $n \geqslant 1$. Define the graded \Z -module $\Hom (A^{\otimes n},A)^i$ of $i$-graded maps $A^{\otimes n} \rightarrow A$, and endow it with the differential $[ \partial , f] = \partial f - (-1)^{|f|} f \partial$. The \N -module $\End_A(n) := \Hom (A^{\otimes n},A)$ in dg-\Z -modules can then naturally be endowed with an operad structure, where composition maps are defined as one expects. Let $P$ be a $\mathtt{(dg-\Z -mod)}$-operad. A \emph{structure of $P$-algebra} on $A$ is defined to be the datum of a morphism of operads 
\[ P \longrightarrow \End_A \ , \]
in other words the datum of a way to interpret each operation of $P_n$ in $\Hom (A^{\otimes n},A)$, such that abstract composition in $P$ coincides with actual composition in $\End_A$.

A morphism of $P$-algebras between $A$ and $B$ is then simply a dg-map $f : A \rightarrow B$, which commutes with every operation of $P_n$ interpreted in $A$ and $B$. In other words, for every 
$m_n \in P_n$, 
\[ m_n^B \circ f^{\otimes n} = f \circ m_n^A \ . \]

\subsection{Operadic bimodules} \label{alg:ss:op-bimod}

\subsubsection{Definition with Schur functors} \label{alg:sss:def-with-schur}

Let now $(\mathcal{D},\otimes_{\mathcal{D}},I)$ be any monoidal category, and $(A,\mu_A)$ and $(B,\mu_B)$ be two monoids in $\mathcal{D}$. Reproducing the diagrams of usual algebra, one can define the notion of an $(A,B)$-bimodule in $\mathcal{D}$. It is simply the data of an object $R$ of $\mathcal{D}$, together with action maps $\lambda : A \otimes R \rightarrow R$ and $\mu : R \otimes B \rightarrow R$ which are compatible with the product on $A$ and $B$, act trivially under their identity elements and satisfy the obvious associativity conditions.

Take for instance $\mathcal{D}$ to be the category $\mathtt{dg-\Z -mod}$. A monoid in $\mathcal{D}$ is then a unital associative differential graded algebra, and the notion of bimodules in the previous paragraph then coincides with the usual notion of bimodules over dg-algebras.

\begin{definition}
Given $P$ and $Q$ two operads seen as their Schur functors $S_P$ and $S_Q$, let $R = \{ R_n \}$ be a \N -module of $\mathcal{C}$ seen as its Schur functor $S_R$. 
A \emph{$(P,Q)$-operadic bimodule structure} on $R$ is a $(S_P,S_Q)$-bimodule structure $\lambda : S_P \circ S_R \rightarrow S_R$ and $\mu : S_R \circ S_Q \rightarrow S_R$ on $S_R$ in $(\mathrm{End}(\mathcal{C}), \circ,Id_{\mathcal{C}})$. 
\end{definition}

\subsubsection{Operadic bimodules with operations} \label{alg:sss:op-bimod-op}

This definition is of course of no use for actual computations. Unraveling the definitions, we get an equivalent definition for $(P,Q)$-operadic bimodules. 

\begin{definition}
A \emph{$(P,Q)$-operadic bimodule structure} on $R$ is the data of action-composition maps
\begin{align*}
R_k \otimes Q_{i_1} \otimes \cdots \otimes Q_{i_k} &\underset{\mu_{i_1,\dots,i_k}}{\longrightarrow} R_{i_1 + \dots + i_k} \ ,  \\
P_h \otimes R_{j_1} \otimes \cdots \otimes R_{j_h} &\underset{\lambda_{j_1,\dots,j_h}}{\longrightarrow} R_{j_1 + \dots + j_h} \ ,
\end{align*}
which are compatible with one another, with identities, and with compositions in $P$ and $Q$.
\end{definition}

\noindent Note that the action of $Q$ on $R$ can be reduced to partial action-composition maps
\[ \circ_i : R_k \otimes Q_h \longrightarrow R_{h+k-1} \ \ 1 \leqslant i \leqslant k \ , \]
as $Q$ has an identity. This cannot be done for the action of $P$ on $R$, as $R$ does not necessarily have an identity.

\subsubsection{The $(\End_B , \End_A )$-operadic bimodule $\Hom (A,B)$} \label{alg:sss:homAB}

Let $A$ and $B$ be two dg-\Z -modules. We have seen that they each determine an operad, $ \End_A$ and $ \End_B$ respectively. Then the \N -module $\Hom (A ,B) := \{ \Hom (A^{\otimes n},B) \}_{n \geqslant 1}$ in dg-\Z -modules is a $(\End_B , \End_A)$-operadic bimodule where the action-composition maps are defined as one could expect.

\subsection{The operad \Ainf} \label{alg:ss:op-ainf}

\subsubsection{Suspension of a dg-\Z -module} \label{alg:sss:susp}

Let $A$ be a graded \Z -module. We define $sA$ to be the graded \Z -module $(sA)^i := A^{i - 1}$. In other words, $|sa|=|a|-1$. It is merely a notation that gives a convenient way to handle certain degrees. Note for instance that a degree $2-n$ map $A^{\otimes n} \rightarrow A$ is simply a degree $+1$ map $(sA)^{\otimes n} \rightarrow sA$.
This will be used thoroughly in the rest of this part.

\subsubsection{\Ainf -algebras} \label{alg:sss:ainf-alg}

Let $A$ be a dg-\Z -module with differential $m_1$. Recall that we are working in the cohomological framework hence $m_1$ has degree $+1$. A structure of \Ainf -algebra on $A$ is the data of a collection of degree $2-n$ maps 
\[m_n : A^{\otimes n} \longrightarrow A \ , \ n \geqslant 1 , \]
extending $m_1$ and which satisfy the following equations, called the \Ainf -equations
\[ \left[ m_1 , m_n \right] = \sum_{\substack{i_1+i_2+i_3=n \\ 2 \leqslant i_2 \leqslant n-1}} \pm m_{i_1+1+i_3} (\ide^{\otimes i_1} \otimes m_{i_2} \otimes \ide^{\otimes i_3} ) . \]
We refer to section~\ref{alg:ss:signs-ainf-bar} for the signs. Representing $m_n$ as \arbreop{0.15} , this equation reads as
\[ [ m_1 , \arbreop{0.2} ] = \sum_{\substack{h + k = n+1 \\ 2 \leqslant h \leqslant n-1 \\ 1 \leqslant i \leqslant k}} \pm \eqainf   \ . \]

We have in particular that
\begin{align*}
\left[ m_1 , m_2 \right] &= 0  \ , \\  
\left[ m_1 , m_3 \right] &= m_2 ( \ide \otimes m_2 - m_2\otimes \ide ) \ .
\end{align*}
Defining $H^*(A)$ to be the cohomology of $A$ relative to $m_1$, the last two equations show that $m_2$ descends to an associative product on $H^*(A)$. An \Ainf -algebra is simply a correct notion of a dg-algebra whose product is associative up to homotopy. Indeed to define such a notion, we have to keep track of all the higher homotopies coming with the fact that the product is associative up to homotopy : these higher homotopies are exactly the $m_n$. 

\subsubsection{The operad \Ainf} \label{alg:sss:ainf-op}

The \Ainf -algebra structure defined previously is actually governed by the following operad :
\begin{definition}
The \emph{operad \Ainf} is the quasi-free $\mathtt{dg-\Z -mod}$-operad generated in arity $n \geqslant 2$ by one operation $m_n$ of degree $2-n$ and whose differential is defined by
\[ \partial (m_n) = \sum_{\substack{i_1+i_2+i_3=n \\ 2 \leqslant i_2 \leqslant n-1}} \pm m_{i_1+1+i_3} (\ide^{\otimes i_1} \otimes m_{i_2} \otimes \ide^{\otimes i_3} ) \ . \]
\end{definition}

This is often written as $\Ainf = \mathcal{F}( \arbreopdeux , \arbreoptrois, \arbreopquatre , \cdots )$ where
\[ \partial ( \arbreop{0.2} ) = \sum_{\substack{h+k = n+1 \\ 2 \leqslant h \leqslant n-1 \\ 1 \leqslant i \leqslant k}} \pm \eqainf \ . \] 
Recall that quasi-free means that the operad is freely generated by the operations \arbreop{0.15} as a graded object, with the additional datum of a differential on its generating operations that is non-canonical. 
We then check that an \Ainf -algebra structure on a dg-\Z -module $A$ amounts simply to a morphism of operads $\Ainf \rightarrow \End_A$.

\subsubsection{The bar construction} \label{alg:sss:bar-constr}

\Ainf -algebras can also be defined using the \emph{bar construction}. Define the reduced tensor coalgebra of a graded \Z -module $V$ to be 
\[ \overline{T} V := V \oplus V^{\otimes 2} \oplus \cdots \]
endowed with the coassociative comultiplication
\[ \Delta_{\overline{T} V} (v_1 \dots v_n) := \sum_{i=1}^{n-1} v_1 \dots v_i \otimes v_{i+1} \dots v_n \ . \]

Then, we have a correspondence
\[
\begin{array}{c@{}c@{}c}
 \left\{\begin{array}{c}
         \text{collections of morphisms of degree $2-n$} \\
         \text{$m_n : A^{\otimes n} \rightarrow A \ , \ n \geqslant 1$} \\
  \end{array}\right\}
  & \longleftrightarrow 
  & \left\{\begin{array}{c}
         \text{collections of morphisms of degree $+1$} \\
         \text{$b_n : (sA)^{\otimes n} \rightarrow sA \ , \ n \geqslant 1$} \\
  \end{array}\right\} \\
  & & \updownarrow \\
  & & \left\{\begin{array}{c}
         \text{coderivations $D$ of degree $+1$ of $\overline{T} (sA)$}
  \end{array}\right\}
\end{array} \ .
\] 
Indeed, to each family of maps $b_n : (sA)^{\otimes n} \rightarrow sA$ of degree $+1$ one can associate a map $D : \overline{T}(sA) \rightarrow \overline{T}(sA)$ of degree $+1$ whose restriction to the $(sA)^{\otimes n}$ summand is given by
\[ \sum_{i_1+i_2+i_3 = n} \pm \ide^{\otimes i_1} \otimes b_{i_2} \otimes \ide^{\otimes i_3} \ . \]
Then the map $D$ is a coderivation of $\overline{T}(sA)$.

There is a second correspondence
\[
\begin{array}{c@{}c@{}c}
 \left\{\begin{array}{c}
         \text{collections of morphisms of degree $2-n$} \\
         \text{$m_n : A^{\otimes n} \rightarrow A \ , \ n \geqslant 1,$} \\
         \text{satisfying the \Ainf -equations} \\
  \end{array}\right\}
  & \longleftrightarrow 
  & \left\{\begin{array}{c}
         \text{coderivations $D$ of degree $+1$ of $\overline{T} (sA)$} \\
         \text{such that $D^2 = 0$} \\
  \end{array}\right\} \ .
\end{array}
\]    
Hence, the following proposition
\begin{proposition}
There is a one-to-one correspondence between \Ainf -algebra structures on $A$ and coderivations $D : \overline{T}(sA) \rightarrow \overline{T}(sA)$ of degree $+1$ which square to 0.
\end{proposition}

\subsection{\Ainf -morphisms} \label{alg:ss:ainf-morph} 

\subsubsection{dg-morphisms between \Ainf -algebras} \label{alg:sss:dg-morph}

Using the definition of section~\ref{alg:ss:p-alg}, a morphism between two \Ainf -algebras $A$ and $B$ is simply a dg-morphism $f : A \rightarrow B$ which is compatible with all the $m_n$.
This notion of morphism is however not satisfactory from an homotopy-theoretic point of view. Indeed, an \Ainf -algebra being an algebra whose product is associative up to homotopy, the correct homotopy notion of a morphism between two \Ainf -algebras would be that of a map which preserves the product $m_2$ up to homotopy, i.e. of a dg-morphism $f_1 : A \rightarrow B$ together with higher coherent homotopies, the first one satisfying
\[ [\partial , f_2 ] = f_1 m_2^A - m_2^B (f_1 \otimes f_1) \ . \]

\subsubsection{\Ainf -morphisms} \label{alg:sss:ainf-morph}

\begin{definition}
An \emph{\Ainf -morphism} between two \Ainf -algebras $A$ and $B$ is a dg-coalgebra morphism $F : (\overline{T}(sA),D_A) \rightarrow (\overline{T}(sB),D_B)$ between their bar constructions. 
\end{definition}

As previously, we have a one-to-one correspondence 
\[
\begin{array}{c@{}c@{}c}
 \left\{\begin{array}{c}
         \text{collections of morphisms of degree $1-n$} \\
         \text{$f_n : A^{\otimes n} \rightarrow B \ , \ n \geqslant 1,$} \\
  \end{array}\right\}
  & \longleftrightarrow 
  & \left\{\begin{array}{c}
         \text{morphisms of graded coalgebras} \\
         F : \overline{T} (sA) \rightarrow \overline{T} (sB)
  \end{array}\right\}
\end{array} \ .
\] 
The component of $F$ mapping $(sA)^{\otimes n}$ to $(sB)^{\otimes s}$ is given by
\[ \sum_{i_1+\cdots+i_s=n} \pm f_{i_1} \otimes \cdots \otimes f_{i_s} \ . \]
A coalgebra morphism preserves the differentials if and only if for all $n \geqslant 1$,
\begin{align*}
\sum_{i_1+i_2+i_3=n} \pm f_{i_1+1+i_3} (\ide^{\otimes i_1} \otimes m^A_{i_2} \otimes \ide^{\otimes i_3}) = \sum_{i_1 + \cdots + i_s = n} \pm m^B_s ( f_{i_1} \otimes \cdots \otimes f_{i_s}) \ . \tag{$\star$} \label{alg:eq:ainf-eq-morph}
\end{align*} 
These equations can be rewritten as 
\begin{align*}
\left[ m_1 , f_n \right] =  \sum_{\substack{i_1+i_2+i_3=n \\ i_2 \geqslant 2}} \pm f_{i_1+1+i_3} (\ide^{\otimes i_1} \otimes m_{i_2}^A \otimes \ide^{\otimes i_3}) + \sum_{\substack{i_1 + \cdots + i_s = n \\ s \geqslant 2 }} \pm m_s^B ( f_{i_1} \otimes \cdots \otimes f_{i_s}) \ .  \tag{$\star$}
\end{align*} 
This yields the following equivalent definition : 
\begin{definition}
An \emph{\Ainf -morphism} between two \Ainf -algebras $A$ and $B$ is a family of maps $f_n : A^{\otimes n} \rightarrow B$ of degree $1-n$ satisfying equations~\ref{alg:eq:ainf-eq-morph}.
\end{definition}

See section~\ref{alg:ss:signs-ainf-bar} for signs. We check that we recover in particular 
$ [\partial , f_2 ] = f_1 m_2^A - m_2^B (f_1 \otimes f_1) \ . $
As a result, an \Ainf -morphism of \Ainf -algebras induces a morphism of associative algebras on the level of cohomology.
An \emph{\Ainf -quasi-isomorphism} is then defined to be an \Ainf -morphism inducing an isomorphism in cohomology.

\subsubsection{Composing \Ainf -morphisms} \label{alg:sss:composition}

Given two coalgebra morphisms $F : \overline{T}V \rightarrow \overline{T}W$ and $G : \overline{T}W \rightarrow \overline{T}Z$, the family of morphisms associated to $G \circ F$ is given by 
\[ (G \circ F)_n := \sum_{i_1 + \cdots + i_s = n} \pm g_s (f_{i_1} \otimes \cdots \otimes f_{i_s}) \ . \]
Hence, the composition of two \Ainf -morphisms $f : A \rightarrow B$ and $g: B \rightarrow C$ is defined to be 
\[ (g \circ f)_n := \sum_{i_1 + \cdots + i_s = n} \pm g_s (f_{i_1} \otimes \cdots \otimes f_{i_s}) \ . \]
In particular one can define $\mathtt{\Ainf -alg}$, the category of \Ainf -algebras with \Ainf -morphisms between them, whose composition is defined by the previous formula.

\subsubsection{The (\Ainf ,\Ainf )-operadic bimodule encoding $\Ainf$-morphisms} \label{alg:sss:ainf-ainf-bimod}

In fact there is an (\Ainf ,\Ainf )-operadic bimodule encoding the notion of \Ainf -morphisms of \Ainf -algebras. 

\begin{definition}
The operadic bimodule \infmor\ is the quasi-free $(\Ainf ,\Ainf )$-operadic bimodule generated in arity $n \geqslant 1$ by one operation $f_n$ of degree $1-n$ and whose differential is defined by
\[ \partial (f_n) = \sum_{\substack{i_1+i_2+i_3=n \\ i_2 \geqslant 2}} \pm f_{i_1+1+i_3} (\ide^{\otimes i_1} \otimes m_{i_2} \otimes \ide^{\otimes i_3}) + \sum_{\substack{i_1 + \cdots + i_s = n \\ s \geqslant 2 }} \pm m_s ( f_{i_1} \otimes \cdots \otimes f_{i_s}) \ . \]
\end{definition}

Representing the generating operations of the operad \Ainf\ acting on the right in blue \arbreopbleu{0.15} and the ones of the operad \Ainf\ acting on the left in red \arbreoprouge{0.15}, we represent $f_n$ by \arbreopmorph{0.15}. This operadic bimodule can then be written as
\[ \infmor = \mathcal{F}^{\Ainf , \Ainf}(\arbreopunmorph , \arbreopdeuxmorph , \arbreoptroismorph , \arbreopquatremorph , \cdots ) \ ,  \]
with differential defined as
\[ \partial ( \arbreopmorph{0.2} ) = \sum_{\substack{h+k = n+1 \\ 1 \leqslant i \leqslant k \\ h \geqslant 2 }} \pm \eqainfmorphun + \sum_{\substack{i_1 + \cdots + i_s = n \\ s \geqslant 2 }} \pm \eqainfmorphdeux . \] 

Consider $A$ and $B$ two \Ainf -algebras, which we can see as two morphisms of operads $\Ainf \rightarrow \End_A$ and $\Ainf \rightarrow \End_B$. Recall from subsection~\ref{alg:sss:homAB} that $\Hom (A,B)$ is a $(\End_B , \End_A)$-operadic bimodule. The previous two morphisms of operads make $\Hom (A , B)$ into an (\Ainf ,\Ainf )-operadic bimodule. An \Ainf -morphism between $A$ and $B$ is then simply a morphism of (\Ainf ,\Ainf )-operadic bimodules 
\[ \infmor \longrightarrow \Hom (A , B) \ . \]
It is in that sense that $\infmor$ is the (\Ainf ,\Ainf )-operadic bimodule encoding the notion of \Ainf -morphisms of \Ainf -algebras.

\subsubsection{The framework of two-colored operads} \label{alg:sss:two-col-op}

In fact, our choice of notation \arbreopmorph{0.15} reveals that the natural framework to work with the operad \Ainf\ and the operadic bimodule \infmor\ is provided by the quasi-free two-colored operad 
\[ A_\infty^2 := \mathcal{F} (\arbreopdeuxcol{red} , \arbreoptroiscol{red} , \arbreopquatrecol{red}, \cdots, \arbreopdeuxcol{blue} , \arbreoptroiscol{blue} , \arbreopquatrecol{blue} , \cdots, \arbreopunmorph , \arbreopdeuxmorph , \arbreoptroismorph , \arbreopquatremorph , \cdots ) \ , \]
where the differential on the generating operations is given by the previous formulae.
A two-colored operad can be roughly defined as an operad whose operations have entries and output labeled either in red or in blue, and whose operations can only be composed along the same color. See~\cite{yau-colored} for a complete definition. 

\subsection{Homotopy theory of \Ainf -algebras} \label{alg:ss:homotopy}

\Ainf -algebras with \Ainf -morphisms between them provide a suitable framework to study homotopy theory of dg-associative algebras. This is because the two-colored operad $A_\infty^2$ is a resolution 
\[ A_\infty^2 \tilde{\longrightarrow} As^2 \ , \]
of the two-colored operad encoding associative algebras with morphisms of algebras, and a fibrant-cofibrant object in the model category of two-colored operads in dg-\Z -modules. See~\cite{markl-homotopy-diagram}. We illustrate these statements with two fundamental theorems. We refer moreover to~\cite{markl-transferring} for a more general version of Theorem~\ref{alg:th:htt}.

\begin{theorem}[Homotopy transfer theorem \cite{kadeishvili-theory}] \label{alg:th:htt}
Let $(A,\partial_A)$ and $(H,\partial_H)$ be two cochain complexes. Suppose that $H$ is a \emph{deformation retract} of $A$, that is that they fit into a diagram
\begin{center} 
\begin{tikzcd}
\arrow[loop left,"h"](A,\partial_A) \arrow[r, shift left, "p"] & \arrow[l,shift left,"i"] (H,\partial_H) \ ,
\end{tikzcd}
\end{center}
where $\mathrm{id}_A - ip = [ \partial,h]$. Then if $(A,\partial_A)$ is endowed with an associative algebra structure, $H$ can be made into an \Ainf -algebra such that $i$ and $p$ extend to \Ainf -morphisms. 
\end{theorem}

\begin{theorem}[Fundamental theorem of \Ainf -quasi-isomorphisms \cite{lefevre-hasegawa}] For every \Ainf -quasi-isomor\-phism $f : A \rightarrow B$ there exists an \Ainf -quasi-isomorphism $B \rightarrow A$ which inverts $f$ on the level of cohomology. 
\end{theorem}

\section{Operads in polytopes}\label{alg:s:poly-op}

We recall in the first section the monoidal category $\mathtt{Poly}$ defined in~\cite{masuda-diagonal-assoc}, which yields a good framework to handle operadic calculus in a category whose objects are polytopes.
We then introduce in sections~\ref{alg:ss:assoc}~and~\ref{alg:ss:multipl} the two main combinatorial objects of this article : the \emph{associahedra} and the \emph{multiplihedra}, which are polytopes that respectively encode \Ainf -algebras and \Ainf -morphisms between them. 
Explicit realizations of the associahedra and the multiplihedra will be given in sections~\ref{alg:ss:loday-assoc}~and~\ref{alg:ss:forcey-loday-multipl}.

\subsection{Three monoidal categories and their operadic algebra} \label{alg:ss:three-monoid}

\subsubsection{Differential graded \Z -modules and CW-complexes} \label{alg:sss:dg-cw}

Consider $\mathtt{dg-\Z -mod}$ to be the category with objects differential graded \Z -modules with cohomological convention, and morphisms the morphisms of dg-\Z -modules. It is a monoidal category with the classical tensor product of dg-\Z -modules and unit the underlying field seen as a dg-\Z -module concentrated in degree 0.

Likewise, define $\mathtt{CW}$ to be the category whose objects are finite CW-complexes and whose morphisms are CW-maps between CW-complexes. This category is again a monoidal category with product the usual cartesian product and unit the point $*$.
The cellular chain functor $C^{cell}_{*} : \mathtt{CW} \rightarrow \mathtt{dg-\Z -mod}$ is then strong monoidal, i.e. it satisfies
\[ C^{cell}_{*} (P \times Q) = C^{cell}_{*}(P) \otimes C^{cell}_{*}(Q) \ . \]
To be consistent with the cohomological degree convention on \Ainf -algebras, we will actually work with the strong monoidal functor
\[ C^{cell}_{-*} : \mathtt{CW} \longrightarrow \mathtt{dg-\Z -mod} \ , \]
where $C^{cell}_{-*}(P)$ is simply the \Z -module $C^{cell}_{*}(P)$ taken with its opposite grading.

\subsubsection{The category of polytopes (\cite{masuda-diagonal-assoc})} \label{alg:sss:cat-poly}

Define a \emph{polytope} to be the convex hull of a finite number of points in a Euclidean space $\R^n$. 
A \emph{polytopal complex} is then a finite collection $\mathcal{P}$ of polytopes satisfying three conditions :
\begin{enumerate}[label=(\roman*)]
\item $\emptyset \in \mathcal{P}$ ,
\item if $P \in \mathcal{P}$ then all the faces of $P$ are also in $\mathcal{P}$ ,
\item if $P$ and $Q$ are two polytopes of $\mathcal{P}$ then the intersection $P \cap Q$ belongs to $\mathcal{P}$.
\end{enumerate}
The realisation of a polytopal complex is simply
\[ | \mathcal{P} | := \bigcup_{P \in \mathcal{P}} P \ . \]
Given $P$ a polytope, we say in particular that a polytopal complex $\mathcal{Q}$ is a polytopal subdivision of $P$ if $| \mathcal{Q} | = P$. Every polytope $P$ comes with a polytopal complex $\mathcal{L}(P)$ consisting of all its faces, which realizes a polytopal subdivision of $P$.

Following~\cite{masuda-diagonal-assoc}, we then define the category $\mathtt{Poly}$ as :
\begin{enumerate}[label={}]
\item \textbf{Objects.} Polytopes.
\item \textbf{Morphisms.} A continuous map $f : P \rightarrow Q$ which is a homeomorphism $P \rightarrow | \mathcal{D} |$ where $\mathcal{D}$ is a polytopal subcomplex of $\mathcal{L}(Q)$ and $f^{-1}(\mathcal{D})$ is a polytopal subdivision of $P$. Such a map will be called a \emph{polytopal map}.
\end{enumerate}
This is a monoidal category with product the usual cartesian product and unit the polytope reduced to a point $*$. It is in fact a monoidal subcategory of $\mathtt{CW}$. 

\subsubsection{From operadic algebra in $\mathtt{Poly}$ to operadic algebra in $\mathtt{dg-\Z -mod}$} \label{alg:sss:op-alg-poly-to-dg}

Let $\{ X_n \}$ be a $\mathtt{Poly}$-operad, that is a collection of polytopes $X_n$ together with polytopal maps 
\[ \circ_i : X_k \times X_h \longrightarrow X_{h+k-1} \ , \]
satisfying the compatibility conditions of partial compositions.
Then, the functor $C^{cell}_{-*}$ yields a new $\mathtt{dg-\Z -mod}$-operad $\{ P_n \}$ defined by $P_n := C^{cell}_{-*}(X_n)$ and whose partial compositions are
\[ \circ_i : C^{cell}_{-*}(X_k) \otimes  C^{cell}_{-*}(X_h) \tilde{\longrightarrow} C^{cell}_{-*}(X_k \times X_h) \underset{C^{cell}_{-*}(\circ_i)}{\longrightarrow} C^{cell}_{-*}(X_{h+k-1}) \ . \]

In the same way, let $\{ X_n \}$ and $\{ Y_n \}$ be two $\mathtt{Poly}$-operads, and $\{ Z_n \}$ be a $(\{ X_n \} , \{ Y_n \})$-operadic bimodule, that is a collection of polytopes $\{ Z_n \}$ together with polytopal action-composition maps
\begin{align*}
X_s \times Z_{i_1} \times \cdots \times Z_{i_s} &\overset{\mu}{\longrightarrow} Z_{i_1 + \dots + i_s} \ , \\
Z_k \times Y_h &\underset{\circ_i}{\longrightarrow} Z_{h+k-1} \ ,  
\end{align*}
which are compatible with the composition maps of  $\{ X_n \}$ and $\{ Y_n \}$.
Then, the functor $C^{cell}_{-*}$ yields a new operadic-bimodule in $\mathtt{dg-\Z -mod}$ as follows. Denote $P_n = C^{cell}_{-*}(X_n)$ and $Q_n = C^{cell}_{-*}(Y_n)$. These are both operads in $\mathtt{dg-\Z -mod}$. Defining $R_n := C^{cell}_{-*}(Z_n)$, this is a $(P,Q)$-operadic bimodule with action-composition maps defined by
\begin{align*}
C^{cell}_{-*}(X_{s}) \otimes C^{cell}_{-*}(Z_{i_1}) \otimes \cdots \otimes C^{cell}_{-*}(Z_{i_s}) &\tilde{\longrightarrow} C^{cell}_{-*}(X_s \times Z_{i_1} \times \cdots \times Z_{i_s}) \overset{C^{cell}_{-*}(\mu)}{\longrightarrow} C^{cell}_{-*}(Z_{i_1 + \dots + i_s}) \ , \\
C^{cell}_{-*}(Z_k) \otimes C^{cell}_{-*}(Y_h) &\tilde{\longrightarrow} C^{cell}_{-*}(Z_k \times Y_h) \underset{C^{cell}_{-*}(\circ_i)}{\longrightarrow} C^{cell}_{-*}(Z_{h+k-1}) \ .
\end{align*}

\subsection{The associahedra} \label{alg:ss:assoc}

The $\mathtt{dg-\Z -mod}$-operad $\Ainf$ actually stems from a $\mathtt{Poly}$-operad : 
\begin{theorem}[\cite{masuda-diagonal-assoc}]
There exists a collection of polytopes, called the \emph{associahedra} and denoted $\{ K_n \}$, endowed with a structure of operad in the category~$\mathtt{Poly}$ and whose image under the functor $C^{cell}_{-*}$ yields the operad \Ainf .
\end{theorem}
\noindent We refer to section~\ref{alg:ss:loday-assoc} in the appendix for a detailed construction and a proof that $\Ainf (n) = C^{cell}_{-*}(K_n)$, and only list noteworthy properties of these polytopes in the following paragraphs.

As $\Ainf (n) = C^{cell}_{-*}(K_n)$, we know that $K_n$ has to have a unique cell $[K_n]$ of dimension $n-2$ whose image under $\partial_{cell}$ is the \Ainf -equation, that is such that
\[ \partial_{cell} [K_n] = \sum \pm \circ_i ([K_k] \otimes [K_h]) \ . \]
In fact, these polytopes are constructed such that the boundary of $K_n$ is exactly
\[ \partial K_n = \bigcup_{\substack{h+k = n+1 \\ 2 \leqslant h \leqslant n-1}} \bigcup_{1 \leqslant i \leqslant k} K_k \times_i K_h \ , \]
where $\times_i$ is in fact the standard $\times$ cartesian product, and such that partial compositions are then simply polytopal inclusions of $K_k \times K_h$ in the boundary of $K_{h+k-1}$.

The first three associahedra $K_2$, $K_3$ and $K_4$ are represented in figure~\ref{alg:fig:associahedra}, labeling their cells by the operations they define in $\Ainf$ when seen in $C^{cell}_{-*}(K_n)$. 

\begin{figure}[h]
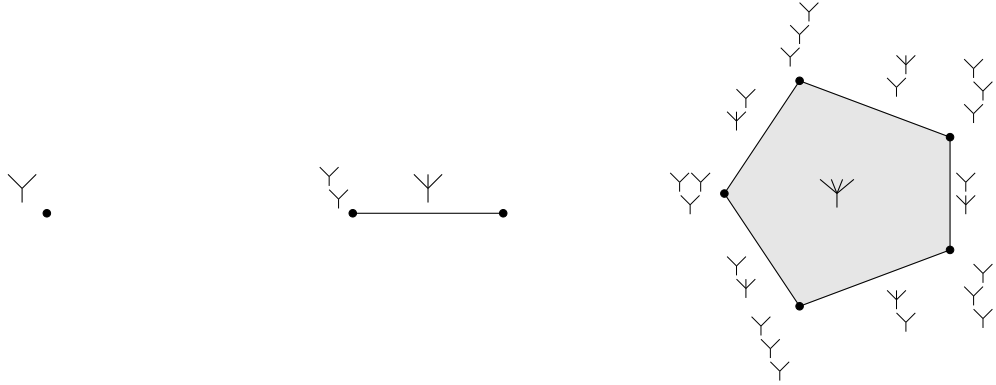
 
    \centering
    \begin{subfigure}{0.3\textwidth}
    \centering
       \associaedreun 
    \end{subfigure} ~
    \begin{subfigure}{0.3\textwidth}
    \centering
        \associaedredeux
    \end{subfigure} ~
    \begin{subfigure}{0.3\textwidth}
    \centering
        \associaedretrois
    \end{subfigure}
    \caption{The associahedra $K_2$, $K_3$ and $K_4$} \label{alg:fig:associahedra}
\end{figure}

\subsection{The multiplihedra} \label{alg:ss:multipl}

Just like the operad \Ainf, the $\mathtt{dg-\Z -mod}$-operadic bimodule \infmor\ is the image under the functor $C^{cell}_{-*}$ of a $\mathtt{Poly}$-operadic bimodule :
\begin{theorem}[\cite{masuda-diagonal-multipl}]
There exists a collection of polytopes, called the \emph{multiplihedra} and denoted $\{ J_n \}$, endowed with a structure of $(\{ K_n \} , \{ K_n \})$-operadic bimodule, i.e. with polytopal action-composition maps 
\begin{align*}
K_s \times J_{i_1} \times \cdots \times J_{i_s} &\overset{\mu}{\longrightarrow} J_{i_1 + \dots + i_s} \ , \\
J_k \times K_h &\underset{\circ_i}{\longrightarrow} J_{h+k-1} \ ,
\end{align*}
whose image under the functor $C^{cell}_{-*}$ yields the $(\Ainf , \Ainf )$-operadic bimodule \infmor . 
\end{theorem}

We refer this time to section~\ref{alg:ss:forcey-loday-multipl} for details and conclude again by listing the main noteworthy properties of the $J_n$. Knowing that $\infmor (n) = C^{cell}_{-*}(J_n)$, we know that $J_n$ has to have a unique $n-1$-dimensional cell $[J_n]$ whose image under $\partial_{cell}$ is the \Ainf -equation for $\Ainf$-morphisms, that is such that
\[ \partial_{cell} [J_n] = \sum \pm \circ_i ([J_k] \otimes [K_h]) \ + \sum \pm \mu ( [K_s] \otimes [J_{i_1}] \otimes \dots \otimes [J_{i_s}]  )  \ . \]
In fact, the polytopes $J_n$ have the following properties 
\begin{enumerate}[label=(\roman*)]
\item the boundary of $J_n$ is exactly
\[ \partial J_n = \bigcup_{\substack{h+k=n+1 \\ h \geqslant 2}} \bigcup_{1 \leqslant i \leqslant k} J_k \times_i K_h \cup \bigcup_{\substack{i_1 + \dots + i_s = n \\ s \geqslant 2}} K_s \times J_{i_1} \times \cdots \times J_{i_s} \ , \]
where $\times_k$ is the standard cartesian product $\times$,
\item action-compositions are polytopal inclusions of faces in the boundary of $J_n$.
\end{enumerate}
The first three polytopes $J_1$, $J_2$ and $J_3$ are represented in figure~\ref{alg:fig:multiplihedra}, labeling their cells by the operations they define in $\infmor$.

\begin{figure}[h]
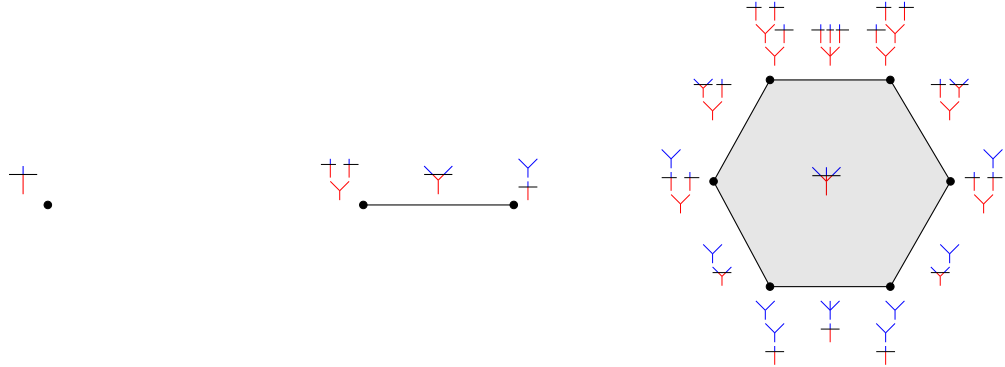
 
    \centering
    \begin{subfigure}{0.3\textwidth}
    \centering
       \multiplaedreun 
    \end{subfigure} ~
    \begin{subfigure}{0.3\textwidth}
    \centering
        \multiplaedredeux
    \end{subfigure} ~
    \begin{subfigure}{0.3\textwidth}
    \centering
        \multiplaedretrois
    \end{subfigure}
    \caption{The multiplihedra $J_1$, $J_2$ and $J_3$} \label{alg:fig:multiplihedra}
\end{figure}

\section{Moduli spaces of metric trees} \label{alg:s:mod-spac}

The associahedra and the multiplihedra are the polytopes governing the structures of \Ainf -algebras and \Ainf -morphisms between them. We show in this section that these polytopes can in fact be realized as geometric moduli spaces : the associahedra are the compactified moduli spaces of stable metric ribbon trees $\overline{\mathcal{T}}_n$, while the multiplihedra are the compactified moduli spaces of stable two-colored metric ribbon trees $\overline{\mathcal{CT}}_n$.

These moduli spaces will come with two cell decompositions : their \Ainf -cell decomposition, corresponding to the cell decomposition of the associahedra (resp. multiplihedra), and a refined decomposition, called the $\Omega B As$-cell decomposition. This second cell decomposition recovers the operad $\Omega B As$ in the case of $\overline{\mathcal{T}}_n$, and an $(\Omega B As , \Omega B As)$-operadic bimodule denoted $\Omega B As-\mathrm{Morph}$ in the case of $\overline{\mathcal{CT}}_n$. They are respectively related to the operad $\Ainf$ and the operadic bimodule \infmor\ by a morphism of operads $\Ainf \rightarrow \ombas$ and a morphism of operadic bimodules $\infmor \rightarrow \Omega B As-\mathrm{Morph}$ (Propositions~\ref{alg:prop:markl-shnider-un} and \ref{alg:prop:markl-shnider-deux}).

\subsection{The associahedra and metric ribbon trees} \label{alg:ss:assoc-metr-tree}

We refer to section~2 of~\cite{mau-woodward} and section~7 of~\cite{abouzaid-plumbings} for the moduli space viewpoint on the associahedra.

\subsubsection{Definitions} \label{alg:sss:def-metr-ribbon-tree}

We begin by giving the definitions of the trees we will need in the rest of the section. The best way to understand them is with the examples depicted in figure~\ref{alg:fig:metric-ribbon-tree}.
\begin{definition}
\begin{enumerate}[label=(\roman*)]
\item A \emph{(rooted) ribbon tree}, is the data of a tree together with a cyclic ordering on the edges at each vertex of the tree and a distinguished vertex adjacent to an external edge called the \emph{root}. This external edge is then called the \emph{outgoing edge}, while all the other external edges are called the \emph{incoming edges}. For a ribbon tree $t$, we will write $E(t)$ for the set of its internal edges, $\overline{E}(t)$ for the set of all its edges, and $e(t)$ for its number of internal edges.
\item A \emph{metric ribbon tree} is the data of a ribbon tree, together with a length $l_e \in ]0,+\infty[$ for each of its internal edges $e$. The external edges are thought as having length equal to $+ \infty$.
\item A ribbon tree is called \emph{stable} if all its inner vertices are at least trivalent. It is called \emph{binary} if all its inner vertices are trivalent. We denote $SRT_n$ the set of all stable ribbon trees, and $BRT_n$ the set of all binary ribbon trees. Note in particular that for a binary tree $t \in BRT_n$ we have that $e(t)=n-2$.
\end{enumerate}  
\end{definition}

\begin{figure}[h] 
    \centering
    \begin{subfigure}[b]{0.2\textwidth}
    	\centering
        \exampleribbontree
        \caption*{~\\A ribbon tree}
    \end{subfigure}
    \hspace{5pt}
    \begin{subfigure}[b]{0.2\textwidth}
    	\centering
        \examplemetricribbontree
        \caption*{~\\A metric ribbon tree}
    \end{subfigure}    
    \hspace{5pt}    
    \begin{subfigure}[b]{0.2\textwidth}
    	\captionsetup{justification=centering}
    	\centering
        \examplestablemetricribbontree
        \caption*{A stable metric \\ ribbon tree}
    \end{subfigure}
    \hspace{5pt}
    \begin{subfigure}[b]{0.2\textwidth}
    	\captionsetup{justification=centering}
    	\centering
        \examplebinarymetricribbontree
        \caption*{A binary metric \\ribbon tree}
    \end{subfigure} 
    \caption{} \label{alg:fig:metric-ribbon-tree}
\end{figure}

\subsubsection{Moduli spaces of stable metric ribbon trees} \label{alg:sss:mod-spac-ribbon-tree}

\begin{definition}
Define $\mathcal{T}_n$ to be \emph{moduli space of stable metric ribbon trees with $n$ incoming edges}. For each stable ribbon tree type $t$, we define moreover $\mathcal{T}_n(t) \subset \mathcal{T}_n$ to be the moduli space $$\mathcal{T}_n(t) := \{ \text{stable metric ribbon trees of type $t$} \} \ .$$
\end{definition}
We then have that 
\[ \mathcal{T}_n = \bigcup_{t \in SRT_n} \mathcal{T}_n(t) \ . \]
Writing $e(t)$ the number of internal edges for a ribbon tree of type $t$, each $\mathcal{T}_n(t)$ is naturally topologized as $]0,+\infty[^{e(t)}$, and they form a stratification of $\mathcal{T}_n$. This is illustrated in figures~\ref{alg:fig:compactification-stratum-T}~and~\ref{alg:fig:compact-mod-space-ombas}.

Interpreting a length in $]0,+\infty[^{e(t)}$ which goes towards 0 as the contraction of the corresponding edge of $t$, the strata $\mathcal{T}_n(t)$ can in fact be consistently glued together. With this observation, one can prove that the space $\mathcal{T}_n$ is in fact itself homeomorphic to $\R^{n-2}$. Allowing lengths of internal edges to go to $+ \infty$, this moduli space can be compactified into a $(n-2)$-dimensional CW-complex $\overline{\mathcal{T}}_n$, where $\mathcal{T}_n$ is seen as its unique $(n-2)$-dimensional stratum. The codimension 1 stratum of this CW-complex is given by
\[ \bigcup_{\substack{h+k = n+1 \\ 2 \leqslant h \leqslant n-1}} \bigcup_{1 \leqslant i \leqslant k} \mathcal{T}_{k} \times_i \mathcal{T}_{h} \ , \]
where $\times_i$ is the standard cartesian product $\times$, and the $i$ means that the outgoing edge of a tree in $\mathcal{T}_{h}$ connects to the $i$-th incoming edge of a tree in $\mathcal{T}_{k}$. It corresponds to metric trees with one internal edge of infinite length. More generally, the codimension $m$ stratum is given by metric trees with $m$ internal edges of infinite lengths. This cell decomposition of $\overline{\mathcal{T}}_n$ will be called its \emph{\Ainf -cell decomposition}.

\begin{theorem}
The moduli space $\overline{\mathcal{T}}_n$ endowed with its \Ainf -cell decomposition is isomorphic as a CW-complex to the associahedron $K_n$.
\end{theorem} 
\noindent This was first noticed in section 1.4. of Boardman-Vogt~\cite{boardman-vogt}. See two examples on figure~\ref{alg:fig:compact-mod-space-ombas}.

\subsubsection{The second cell decomposition of $\overline{\mathcal{T}}_n$} \label{alg:sss:second-cell-ribbon}

In fact the previous compactification can be obtained by first compactifying each cell $\mathcal{T}_n(t)$ individually and then gluing consistently all compactifications together. For $t \in RT_n$, the stratum $\mathcal{T}_n(t)$ is homeomorphic to $]0,+\infty[^{e(t)}$ and its compactification in $\overline{\mathcal{T}}_n$ is homeomorphic to $[0,+\infty]^{e(t)}$. A length equal to 0 simply corresponds to collapsing one edge of $t$ and a length equal to $+\infty$ is interpreted as breaking this edge. This is illustrated in the instance of a cell of $\mathcal{T}_4(t)$ in figure~\ref{alg:fig:compactification-stratum-T}.

\begin{figure}[h]
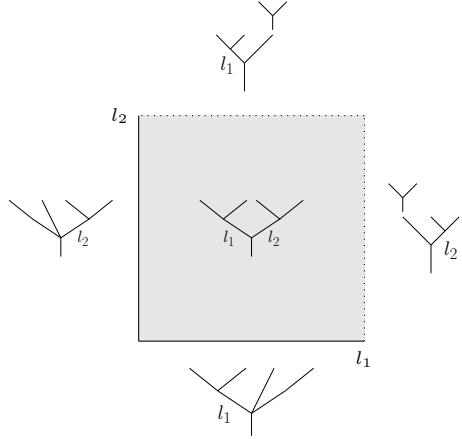
 
\centering
\examplestratum
\caption{Compactification of a stratum of $\mathcal{T}_4$} \label{alg:fig:compactification-stratum-T}
\end{figure}

\begin{definition} A \emph{broken ribbon tree} is a ribbon tree some of whose internal edges may be broken. Equivalently, it is the datum of a finite collection of (unbroken) ribbon trees together with a way of arranging this collection into a new tree (with broken edges). A broken ribbon tree is said to be \emph{stable} if every unbroken ribbon tree forming it is stable.
\end{definition}

The viewpoint introduced in the previous paragraph yields a new cell decomposition of $\overline{\mathcal{T}}_n$, an example of which is given in figure~\ref{alg:fig:compact-mod-space-ombas}.  Its cells are indexed by broken stable ribbon trees, a broken stable ribbon tree with $i$ finite internal edges labeling an $i$-dimensional cell.

\begin{figure}[h]
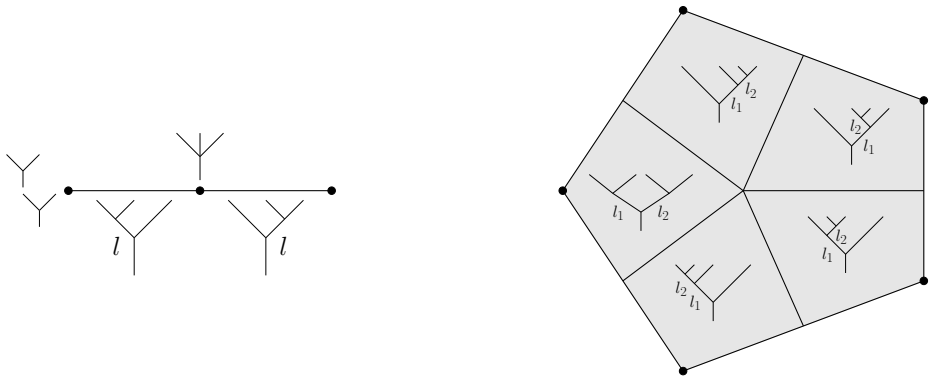
 
    \begin{subfigure}[b]{0.4\textwidth}
    		\centering
        \Ttroisstrat
    \end{subfigure}
    \hspace{10pt}
    \begin{subfigure}[b]{0.4\textwidth}
    \centering
        \Tquatrestrat
    \end{subfigure}
    \caption{The compactified moduli spaces $\overline{\mathcal{T}}_3$ and $\overline{\mathcal{T}}_4$ with their cell decomposition by broken stable ribbon tree type} \label{alg:fig:compact-mod-space-ombas}
\end{figure}

\subsubsection{The operad $\Omega B As$} \label{alg:sss:op-ombas}

Endowing the $\overline{\mathcal{T}}_n$ with this new cell decomposition, the maps
\[ \overline{\mathcal{T}}_k \times \overline{\mathcal{T}}_h \underset{\circ_i}{\longrightarrow} \overline{\mathcal{T}}_{h+k-1} \]
are then cellular maps, and hence form a new operad in $\mathtt{CW}$. Taking its image under the functor $C_{-*}^{cell}$ yields an operad in $\mathtt{dg-\Z -mod}$ : \emph{the operad \ombas}. We refer to section~\ref{alg:ss:signs-ombas} for a complete description of this operad and its sign conventions. 

\begin{definition}
The \emph{operad \ombas} is the quasi-free operad generated by the set of stable ribbon trees, where a stable ribbon tree $t$ has degree $|t| := -e(t)$. Its differential on a stable ribbon tree $t$ is given by the signed sum of all stable ribbon trees obtained from $t$ by breaking or collapsing exactly one of its internal edges.
\end{definition}
In other words, it is the quasi-free operad
\[ \ombas := \mathcal{F} (\premiertermecobarbarA , \premiertermecobarbarD , \premiertermecobarbarB , \premiertermecobarbarC , \cdots , SRT_n, \cdots ) \]
where for instance
\begin{align*}
| \examplediffcobarbarA | &= -2 \ , \\
\partial ( \examplediffcobarbarA ) &= \pm \examplediffcobarbarB \pm \examplediffcobarbarC \pm \examplediffcobarbarD \pm \examplediffcobarbarE \ . 
\end{align*}  

As the choice of notation \ombas\ suggests, this $\mathtt{dg-\Z -mod}$-operad is in fact the bar-cobar construction of the operad $As$, usually denoted $\Omega B As$. To put it shortly, the classical cobar-bar adjunction for standard algebras and coalgebras
\[ \Omega : \mathtt{conilpotent \ dg-coalgebras} \leftrightharpoons \mathtt{augmented \ dg-algebras} : B \ , \]
admits a counterpart in the realm of operads and cooperads 
\[ \Omega : \mathtt{coaugmented \ dg-cooperads} \leftrightharpoons \mathtt{augmented \ dg-operads} : B \ , \]
and the previously obtained operad is exactly equal to $\Omega B As$.
We refer the curious reader to section 6.5 in Loday-Vallette~\cite{loday-vallette-algebraic-operads}, for more details on that matter. 

\subsubsection{From the operad \Ainf\ to the operad $\Omega B As$} \label{alg:sss:op-ainf-to-ombas}

The $\mathtt{dg-\Z -mod}$-operads \Ainf\ and $\Omega B As$ are in fact related by the following proposition : 
\begin{proposition} \label{alg:prop:markl-shnider-un}
There exists a morphism of operads $\Ainf \rightarrow \Omega B As$ given on the generating operations of \Ainf\ by 
\[ m_n \longmapsto \sum_{t \in BRT_n} \pm m_t \ .  \]
\end{proposition}

This morphism stems from the image under the functor $C_{-*}^{cell}$ of the identity map 
$\ide : (\overline{\mathcal{T}}_n)_{\Ainf} \rightarrow (\overline{\mathcal{T}}_n)_{\Omega B As}$ refining the cell decomposition on $\overline{\mathcal{T}}_n$. The formula on $m_n$ then simply corresponds to associating to the $n-2$-dimensional cell of $\overline{\mathcal{T}}_n$ with the \Ainf -cell decomposition, the signed sum of all $n-2$-dimensional cells of $\overline{\mathcal{T}}_n$ with the $\Omega B As$-cell decomposition.

This geometric construction of the morphism $\Ainf \rightarrow \ombas$ is an adaptation of the algebraic construction by Markl and Shnider in~\cite{markl-assoc} and is detailed in subsection~\ref{alg:sss:morph-of-op}. Proposition~\ref{alg:prop:markl-shnider-un} dates in fact back to~\cite{getzler-jones-operads}, and is built in the theory of Koszul duality, as explained in sections 7 and 9 of~\cite{loday-vallette-algebraic-operads}. 
We moreover point out that the morphism $\Ainf \rightarrow \ombas$ will be crucial in the rest of this paper. It implies indeed that in order to construct a structure of \Ainf -algebra on a cochain complex, it is enough to endow it with a structure of $\Omega B As$-algebra.

\subsection{The multiplihedra and two-colored metric ribbon trees} \label{alg:ss:multipl-two-col-metr-tree}

We have seen in the previous section that the polytopes $K_n$ can be realized as the compactified moduli spaces of stable metric ribbon trees. So can the polytopes $J_n$ : they are the compactified moduli spaces of stable two-colored metric ribbon trees. 

\subsubsection{Two-colored metric ribbon trees} \label{alg:sss:def-gauged-tree}

\begin{definition}
A \emph{stable two-colored metric ribbon tree} or \emph{stable gauged metric ribbon tree} is defined to be a stable metric ribbon tree together with a length $\lambda \in \R$. 
This length is to be thought of as a gauge drawn over the metric tree, at distance $\lambda$ from its root, where the positive direction is pointing down. 
\end{definition}

The gauge divides the tree into two parts, each of which we think of as being colored in a different color. See an instance on figure~\ref{alg:fig:example-stable-two-col}.
This definition, despite being visual, will prove difficult to manipulate when trying to compactify moduli spaces of stable two-colored metric ribbon trees. We thus proceed to give an equivalent definition, which will provide a natural way of compactifying these moduli spaces. The equivalence between the two definitions is depicted on an example in figure~\ref{alg:fig:example-stable-two-col}.

\begin{definition}
\begin{enumerate}[label=(\roman*)]
\item \emph{A two-colored ribbon tree} is defined to be a ribbon tree together with a distinguished subset of vertices $E_{col}(T)$ called the \emph{colored vertices}. This set is such that, either there is exactly one colored vertex in every non-self crossing path from an incoming edge to the root and none in the path from the outgoing edge to the root, or there is no colored vertex in any non-self crossing path from an incoming edge to the root and exactly one in the path from the outgoing edge to the root. These colored vertices are to be thought as the intersection points of the gauge with the ribbon tree.
\item A two-colored ribbon tree is called \emph{stable} if all its non-colored vertices are at least trivalent. We denote $SCRT_n$ the set of all stable two-colored ribbon trees, and $CBRT_n$ the set of all two-colored binary ribbon trees whose gauge does not cross any vertex of the underlying binary ribbon tree.
\item A \emph{two-colored metric ribbon tree} is the data of a length for all internal edges $l_e \in ]0 , +\infty [$, such that the lengths of all non self-crossing paths from a colored vertex to the root are all equal.
\end{enumerate}
\end{definition}

\begin{figure}[h!]
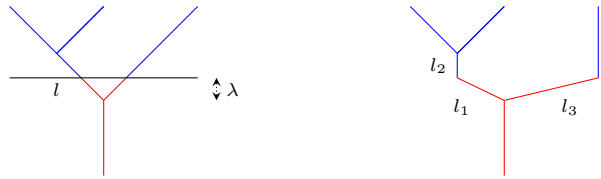
 
    \centering
    \begin{subfigure}[b]{0.2\textwidth}
        \centering 
        \exampletcribbontreeun
    \end{subfigure}
    \hspace{50pt}
    \begin{subfigure}[b]{0.2\textwidth}
        \exampletcribbontreedeux
    \end{subfigure}
    \caption{An example of a stable two-colored metric ribbon tree with the two definitions : here $l_1=l_3=-\lambda$ and $l=l_1 + l_2$} \label{alg:fig:example-stable-two-col}
\end{figure}

These two definitions of two-colored metric ribbon trees are easily seen to be equivalent, by viewing the colored vertices as the intersection points between the gauge and the edges. In the rest of the paper, the notations $t_c$ and $t_g$ will both stand for a two-colored stable ribbon tree, seen respectively from the colored vertices and from the gauged viewpoint. The symbol $t$ will then denote the underlying stable ribbon tree.

\subsubsection{Moduli spaces of stable two-colored metric ribbon trees} \label{alg:sss:mod-spac-two-col-tree}

The results presented in this subsection can be found in section 7 of Mau-Woodward~\cite{mau-woodward}, where they are formulated in the two-colored viewpoint.

\begin{definition}
For $n \geqslant 2$, we define $\mathcal{CT}_n$ to be the \emph{moduli space of stable two-colored metric ribbon trees}. It has a cell decomposition by stable two-colored ribbon tree type,
\[ \mathcal{CT}_n = \bigcup_{t_c \in SCRT_n} \mathcal{CT}_n(t_c) \ . \]
We also denote $\mathcal{CT}_1 := \{ \arbreopunmorph \}$ the space whose only element is the unique two-colored ribbon tree of arity~1.
\end{definition}

The space $\mathcal{CT}_n$ is homeomorphic to $\R^{n-1}$ : $\mathcal{T}_n$ is homeomorphic to $\R^{n-2}$ and, using the gauge description, the datum of a gauge adds a factor $\R$. Allowing again internal edges of metric trees to go to $+ \infty$ by using the second definition for two-colored metric ribbon trees, this moduli space $\mathcal{CT}_n$ can be compactified into a $(n-1)$-dimensional CW-complex $\overline{\mathcal{CT}}_n$. It has one $n-1$ dimensional stratum given by $\mathcal{CT}_n$. Its codimension 1 stratum is given by
\[ \bigcup_{i_1+\dots+i_s=n} \mathcal{T}_s \times \mathcal{CT}_{i_1} \times \dots \times \mathcal{CT}_{i_s} \cup \bigcup_{i_1+i_2+i_3=n} \mathcal{CT}_{i_1+1+i_3} \times \mathcal{T}_{i_2} \ . \]
This cell decomposition of $\overline{\mathcal{CT}}_n$ will be called its \emph{\Ainf -cell decomposition}.
Two sequences of stable two-colored metric ribbon trees converging in the compactification $\overline{\mathcal{CT}}_3$ are represented in figure~\ref{alg:fig:sequences-two-col-compact}.

\begin{figure}[h!]
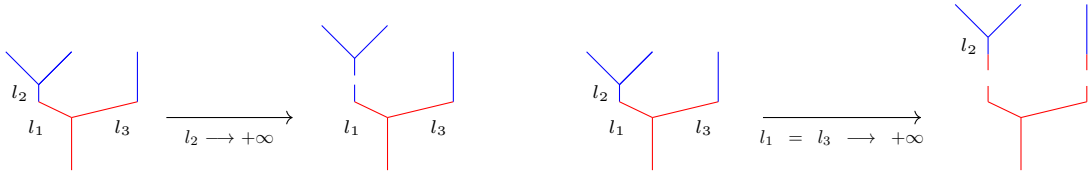
 
    \begin{subfigure}[b]{0.45\textwidth}
    	\centering
        \twocolbordun
    \end{subfigure}
    \hspace{10pt}
    \begin{subfigure}[b]{0.45\textwidth}
    	\centering
        \twocolborddeux
    \end{subfigure}
    \caption{Two sequences of stable two-colored metric ribbon trees converging in the compactification $\overline{\mathcal{CT}}_3$} \label{alg:fig:sequences-two-col-compact}
\end{figure}

\begin{theorem}[\cite{mau-woodward}]
The moduli space $\overline{\mathcal{CT}}_n$ endowed with its \Ainf -cell decomposition is isomorphic as a CW-complex to the multiplihedron $J_n$. 
\end{theorem}
\noindent This theorem is illustrated in figure~\ref{alg:fig:compact-mod-space-CT-ombas}.

\subsubsection{The second cell decomposition of $\overline{\mathcal{CT}}_n$} \label{alg:sss:second-cell-two-col}

As for $\overline{\mathcal{T}}_n$, the compactified moduli space $\overline{\mathcal{CT}}_n$ can be endowed with a refined cell decomposition. This subsection sums up some of the main results of section~\ref{alg:ss:mod-space-CTm}, where we provide an extensive study of the strata of this refined cell decomposition.

Let $t_g$ be a gauged stable ribbon tree. Writing again $e(t)$ for the number of internal edges of the underlying stable ribbon tree, the stratum $\mathcal{CT}_n(t_g)$ is a polyhedral cone in $\R^{e(t)+1}$. For instance, 
\[ \mathcal{CT}_4(\arbrebicoloreA) = \{ (\lambda , l_1 , l_2 ) \text{ such that } l_1 > 0 \ ; \ l_2 > 0 \ ; \ 0 < - \lambda < l_1 , l_2 \} \ . \]
Denote $j$ the number of vertices $v$ of $t$ crossed by the gauge as depicted below 
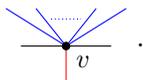
\[ \localgaugeedgeCbis \ .  \]
There is for instance one vertex intersected by the gauge in~\examplestratumCTABbis .
The stratum $\mathcal{CT}_n(t_g)$ then has dimension $e(t)+1-j$, but is not naturally isomorphic to $]0,+\infty[^{e(t)+1-j}$, in the sense that its compactification will not coincide with a $(e(t)+1-j)$-dimensional cube. 

Switching now to the colored vertices viewpoint, the polyhedral cones $\mathcal{CT}_n(t_c)$ can be compactified, by allowing lengths of internal edges to go towards 0 or $+\infty$. The compactification $\overline{\mathcal{CT}}_n$ is simply obtained by gluing the previous compactifications. See an instance of the compactification of $\mathcal{CT}_3(\arbrebicoloreB) = \{ (\lambda , l) \text{ such that } l > 0 \ ; \  - \lambda > l \}$ in figure~\ref{alg:fig:compact-stratum-CT}.

\begin{figure}
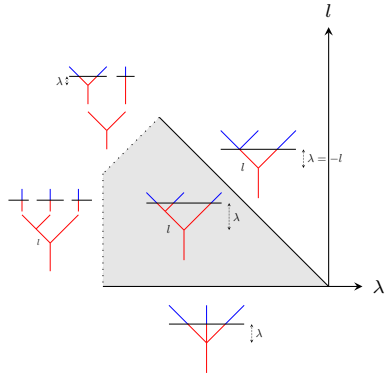
 
\centering
\examplestratumCT
\caption{Compactification of a stratum of $\mathcal{CT}_3$} \label{alg:fig:compact-stratum-CT}
\end{figure}

This yields a new cell decomposition of $\overline{\mathcal{CT}}_n$, where each cell is labeled by a broken two-colored stable ribbon tree. A two-colored stable ribbon tree $t_g$ with $e(t)$ internal edges and whose gauge crosses $j$ vertices labels a $e(t)+1-j$-dimensional cell. The dimension of a cell labeled by a broken two-colored tree can then simply be obtained by adding the dimensions associated to each of the pieces of the broken tree. The cell decompositions for $\overline{\mathcal{CT}}_2$ and $\overline{\mathcal{CT}}_3$ are represented in figure~\ref{alg:fig:compact-mod-space-CT-ombas}.

\begin{figure}[h!]
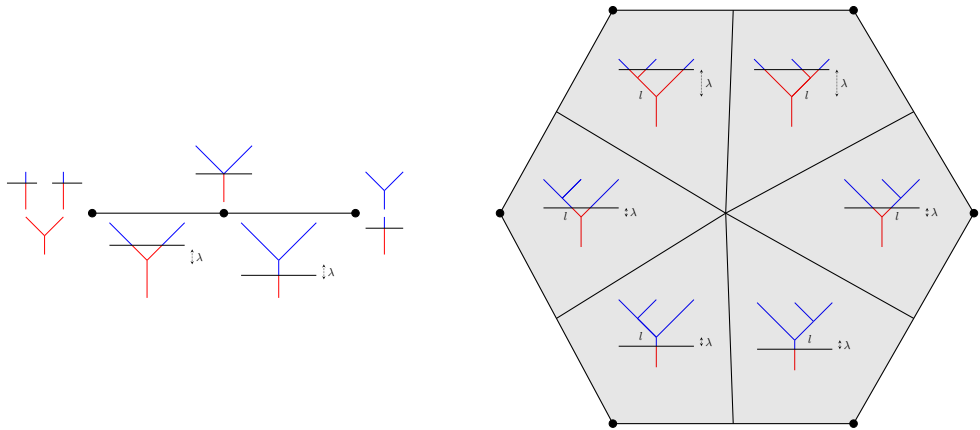
 
    \centering
    \begin{subfigure}[b]{0.3\textwidth}
    	\centering
        \CTdeuxstrat
    \end{subfigure}
    \hspace{10pt}
    \begin{subfigure}[b]{0.5\textwidth}
    	\centering
        \CTtroisstrat
    \end{subfigure}
    \caption{The compactified moduli spaces $\overline{\mathcal{CT}}_2$ and $\overline{\mathcal{CT}}_3$ with their cell decomposition by stable two-colored ribbon tree type} \label{alg:fig:compact-mod-space-CT-ombas}
\end{figure}

Endowing the moduli spaces $\overline{\mathcal{T}}_n$ with their $\Omega B As$-cell decomposition and the moduli spaces $\overline{\mathcal{CT}}_n$ with this new cell decomposition, the maps 
\begin{align*}
\overline{\mathcal{T}}_s \times \overline{\mathcal{CT}}_{i_1} \times \cdots \times \overline{\mathcal{CT}}_{i_s} &\longrightarrow \overline{\mathcal{CT}}_{i_1 + \dots + i_s} \ , \\
\overline{\mathcal{CT}}_k \times \overline{\mathcal{T}}_h &\underset{\circ_i}{\longrightarrow} \overline{\mathcal{CT}}_{h+k-1} \ ,
\end{align*}
are cellular : the \N -module $\{ \overline{\mathcal{CT}}_n \}$ is a $(\{ \overline{\mathcal{T}}_n \}, \{ \overline{\mathcal{T}}_n \} )$-operadic bimodule for this new cell decomposition. 

\subsubsection{The operadic bimodule $\Omega B As - \mathrm{Morph}$} \label{alg:sss:op-bimod-ombas}

The functor $C^{cell}_{-*}$ sends the previous operadic bimodule in $\mathtt{CW}$ to an $(\Omega B As , \Omega B As)$-operadic bimodule in $\mathtt{dg-\Z -mod}$, that we will denote $\Omega B As - \mathrm{Morph}$. We refer to section~\ref{alg:ss:op-bimod-ombasmorph} for a complete description of $\Omega B As - \mathrm{Morph}$ and explicit sign computations. 

\begin{definition} \label{alg:def:ombas-morph}
The operadic bimodule \ombasmor\ is the quasi-free $(\Omega B As , \Omega B As)$-operadic bimodule generated by the set of two-colored stable ribbon trees. A two-colored stable ribbon tree $t_g$ with $e(t)$ internal edges and whose gauge crosses $j$ vertices has degree $|t_g| := j - e(t) -1$. The differential of a two-colored stable ribbon tree $t_c$ is given by the signed sum of all two-colored stable ribbon trees obtained from $t_c$ under the rule prescribed by the top dimensional strata in the boundary of $\overline{\mathcal{CT}}_n(t_c)$.
\end{definition} 

Before giving tedious written details for the differential rule, we refer the reader to figure~\ref{alg:fig:compact-stratum-CT} and to the upcoming example.
Consider the following two-colored stable ribbon tree \arbrebicoloreA . Which codimension 1 phenomena can happen ?
\begin{enumerate}[label=(\roman*)]
\item The gauge can be moved to cross exactly one vertex of \arbrebicoloreC\ : these situations are given by \arbrebicoloreD , \arbrebicoloreE\ and \arbrebicoloreF .
\item An internal edge can break above the gauge : \arbrebicoloreI\ and \arbrebicoloreJ .
\item Both internal edges can break below the gauge : \arbrebicoloreK .
\end{enumerate}
Note that unlike for $\mathcal{CT}_3(\arbrebicoloreB)$, no internal edge can collapse in this example : that would be a codimension 2 phenomenon. These two examples list all four possible codimension 1 phenomena that can happen : the gauge moves to cross exactly one additional vertex of the underlying stable ribbon tree (gauge-vertex) ; an internal edge located above the gauge or intersecting it breaks or, when the gauge is below the root, the outgoing edge breaks between the gauge and the root (above-break) ; edges (internal or incoming) that are possibly intersecting the gauge, break below it, such that there is exactly one edge breaking in each non-self crossing path from an incoming edge to the root (below-break) ; an internal edge that does not intersect the gauge collapses (int-collapse).

In other words, we constructed the quasi-free $(\Omega B As, \Omega B As)$-operadic bimodule
\[ \Omega B As - \mathrm{Morph} := \mathcal{F}^{\Omega B As, \Omega B As}( \arbreopunmorph , \arbrebicoloreL , \arbrebicoloreM , \arbrebicoloreN , \cdots , SCRT_n,\cdots) \ ,  \]
where for instance
\begin{align*}
 | \arbrebicoloreA | &= -3 \ , \\
 \partial (\arbrebicoloreA) &= \pm \arbrebicoloreD \pm \arbrebicoloreE \pm \arbrebicoloreF \pm \arbrebicoloreI \pm \arbrebicoloreJ \pm \arbrebicoloreK \ .
\end{align*}
Note that the symbol \arbreopdeuxmorph\ used here is the same as the one used for the only arity 2 generating operation of \infmor . It will however be clear from the context what \arbreopdeuxmorph\ stands for in the rest of this paper.

Consider $A$ and $B$ two \ombas -algebras, which we can see as two morphisms of operads $\ombas \rightarrow \End_A$ and $\ombas \rightarrow \End_B$. We then define an \emph{\ombas -morphism} $A \rightarrow B$ to be a morphism of $(\ombas , \ombas )$-operadic bimodules $\ombasmor \rightarrow \Hom (A , B)$. It is equivalent to a collection of operations $\mu_{t_g} : A^{\otimes n} \rightarrow B$, $t_g \in SCRT_n$, satisfying the equations prescribed by the differential on \ombasmor . Note that in order to define the category $\mathtt{\ombas -alg}$ of \ombas -algebras with \ombas -morphisms between them, it remains to define the composition of two \ombas -morphisms. This question will be explored in an upcoming article.

\subsubsection{From \infmor\ to $\Omega B As-\mathrm{Morph}$} \label{alg:sss:infmor-to-ombasmor}

The morphism of operads $\Ainf \rightarrow \Omega B As$ makes the $(\Omega B As, \Omega B As)$-operadic bimodule $\Omega B As - \mathrm{Morph}$ into an $(\Ainf, \Ainf)$-operadic bimodule. 

\begin{proposition} \label{alg:prop:markl-shnider-deux}
There exists a morphism of $(\Ainf , \Ainf )$-operadic bimodules $\infmor \rightarrow \Omega B As - \mathrm{Morph}$ given on the generating operations of \infmor\ by
\[ f_n \longmapsto \sum_{t_g \in CBRT_n} \pm f_{t_g} \ .\]
\end{proposition}

As a result, to construct an \Ainf -morphism between two \Ainf -algebras whose \Ainf -algebra structure comes from an $\Omega B As$-algebra structure, it is enough to construct an $\Omega B As$-morphism between them. As in subsection~\ref{alg:sss:op-ainf-to-ombas}, this morphism stems again from the image under the functor $C_{-*}^{cell}$ of the identity morphism on $\overline{\mathcal{CT}}_n$ refining its cell decomposition. The formula for $f_n$ is obtained by sending the $n-1$-dimensional cell of $\overline{\mathcal{CT}}_n$ appearing in the \Ainf -cell decomposition, to the signed sum of all $n-1$-dimensional cells $\overline{\mathcal{CT}}_n$ appearing in the \ombas -cell decomposition.  We refer to subsection~\ref{alg:sss:morph-op-bimod-ainf-ombasmorph} for a complete proof and the details on signs.

\subsection{Résumé} \label{alg:ss:resume}

The moduli space of stable metric ribbon trees $\mathcal{T}_n$ can be compactified by allowing lengths of internal edges to go towards $+ \infty$. This compactification comes with two cell decompositions. The first one, by considering the moduli spaces $\mathcal{T}_n$ as $(n-2)$-dimensional strata, yields a CW-complex isomorphic to the associahedron $K_n$. Its realization under the functor $C^{cell}_{-*}$ then yields the operad \Ainf . The second one is obtained by considering the stratification of $\mathcal{T}_n$ by strata labeled by stable ribbon tree types. It is sent under the functor $C^{cell}_{-*}$ to the operad $\Omega B As$. These two operads in $\mathtt{dg-\Z -mod}$ are then related by a morphism of operads $\Ainf \rightarrow \Omega B As$.

The moduli space of stable two-colored metric ribbon trees $\mathcal{CT}_n$ can be compactified by allowing lengths to go towards $+ \infty$. There are again two cell decompositions for this compactification. Considering the moduli spaces $\mathcal{CT}_n$ as $(n-1)$-dimensional strata yields a first CW-complex isomorphic to the multiplihedron $J_n$. Its image under $C^{cell}_{-*}$ is the $( \Ainf , \Ainf)$-operadic bimodule \infmor . Likewise, considering the stratification of $\mathcal{CT}_n$ by strata labeled by two-colored stable ribbon tree types, we obtain a second cell decomposition. The functor $C^{cell}_{-*}$ sends it to the $(\Omega B As , \Omega B As)$-operadic bimodule $\Omega B As - \mathrm{Morph}$. The morphism of operads $\Ainf \rightarrow \Omega B As$ makes $\Omega B As - \mathrm{Morph}$ into a $( \Ainf , \Ainf)$-operadic bimodule. It is related to \infmor\ by a morphism of operadic bimodules $\infmor \rightarrow \Omega B As - \mathrm{Morph}$.

\section{Signs and polytopes for \Ainf -algebras and \Ainf -morphisms}\label{alg:s:sign-pol-ainf}

The goal of this section is twofold : work out all the signs written as $\pm$ in the \Ainf -equations in section~\ref{alg:s:op-alg} and provide explicit realizations for the associahedra and multiplihedra as polytopes. We begin by introducing the basic Koszul sign rules to work in a graded algebraic framework, and explain how to compute signs by comparing orientations on the boundary of a manifold with boundary. We then recall two equivalent sign conventions for \Ainf -algebras and \Ainf -morphisms and show how they naturally ensue from the bar construction viewpoint. We subsequently detail explicit polytopal realizations of the associahedra and the multiplihedra, introduced in~\cite{masuda-diagonal-assoc}~and~\cite{masuda-diagonal-multipl}, and conclude by showing that these polytopes determine indeed the \Ainf -sign conventions previously defined. 

\subsection{Basic conventions for signs and orientations} \label{alg:ss:basic-conv}

\subsubsection{Koszul sign rule} \label{alg:sss:Koszul-sign}

All formulae in this section will be written using the Koszul sign rule that we briefly recall. We will work exclusively with cohomological conventions.

Given $A$ and $B$ two dg \Z -modules, the differential on $A \otimes B$ is defined as 
\[ \partial_{A \otimes B} (a \otimes b) = \partial_A a \otimes b + (-1)^{|a|} a \otimes \partial_B b \ . \]
Given $A$ and $B$ two dg \Z -modules, we consider the graded \Z -module $\mathrm{Hom}(A,B)$ whose degree $r$ component is given by all maps $A \rightarrow B$ of degree $r$. We endow it with the differential 
\[ \partial_{\mathrm{Hom}(A,B)}(f) := \partial_B \circ f - (-1)^{|f|} f \circ \partial_A =: [\partial , f] \ . \]
Given $f : A \rightarrow A'$ and $g : B \rightarrow B'$ two graded maps between dg-\Z -modules, we set
\[ (f \otimes g) (a \otimes b) = (-1)^{|g| |a|} f(a) \otimes g(b) \ . \]
Finally, given $f : A \rightarrow A'$, $f' : A' \rightarrow A''$, $g : B \rightarrow B'$ and $g' : B' \rightarrow B''$, we define
\[ (f' \otimes g') \circ (f \otimes g) = (-1)^{|g'||f|}(f' \circ f) \otimes (g' \circ g) \ . \]
We check in particular that with this sign rule, the differential on a tensor product $A_1 \otimes \cdots \otimes A_n$ is given by
\[ \partial_{A_1 \otimes \cdots \otimes A_n} = \sum_{i=1}^n \ide_{A_1} \otimes \cdots \otimes \partial_{A_i} \otimes \cdots \otimes \ide_{A_n} \ . \]

\subsubsection{Orientation of the boundary of a manifold with boundary} \label{alg:sss:or-boundary}

Let $(M,\partial M)$ be an oriented $n$-manifold with boundary. We choose to orient its boundary $\partial M$ as follows : given $x \in \partial M$, a basis $e_1,\dots,e_{n-1}$ of $T_x(\partial M)$, and an outward pointing vector $\nu \in T_xM$, the basis $e_1,\dots,e_{n-1}$ is positively oriented if and only if the basis $\nu,e_1,\dots,e_{n-1}$ is a positively oriented basis of $T_xM$. Note that in the particular case when the manifold with boundary is a half-space inside the Euclidean space $\R^{n}$, defined by an inequality
\[ \sum_{i=1}^na_ix_i \leqslant C \ , \]
the vector $(a_1,\dots,a_n)$ is outward-pointing.

We recover under this convention the classical singular and cubical differentials. Take $X$ a topological space. Given a singular simplex $\sigma : \Delta^n \rightarrow X$, its differential is classically defined as 
\[ \partial_{sing} (\sigma) := \sum_{i=0}^n (-1)^i \sigma_i \ , \]
where $\sigma_i$ stands for the restriction $[0<\cdots < \hat{i} < \cdots < n] \hookrightarrow \Delta^n \rightarrow X$. Realizing $\Delta^n$ as a polytope in $\R^n$ and orienting it with the canonical orientation of $\R^n$, we check that its boundary reads exactly as
\[ \partial \Delta^n = \bigcup_{i=0}^n (-1)^i \Delta^{n-1}_i \ , \]
where $\Delta^{n-1}_i$ is the $(n-1)$-simplex corresponding to the face $[0<\cdots < \hat{i} < \cdots < n]$. The sign $(-1)^i$ means that the orientation of $\Delta^{n-1}_i$ induced by its canonical identification with $\Delta^{n-1}$ and its orientation as the boundary of $\Delta^n$, differ by a $(-1)^i$ sign.

Similarly, given a singular cube $\sigma : I^n \rightarrow X$, its differential is
\[ \partial_{cub} \sigma := \sum_{i=1}^n (-1)^i ( \sigma_{i,0} - \sigma_{i,1}) \ , \]
where $\sigma_{i,0}$ denotes the singular cube $I^{n-1} \rightarrow X$ obtained from $\sigma$ by setting its $i$-th entry to $0$, and $\sigma_{i,1}$ is defined similarly. We check again that considering $I^{n} \subset \R^n$ as a polytope of $\R^n$, its boundary reads as 
\[ \partial I^n = \bigcup_{i=1}^n (-1)^i ( I^{n-1}_{i,0} \cup - I^{n-1}_{i,1} ) \ , \]
where $I^{n-1}_{i,0}$ is the face of $I^n$ obtained by setting the $i$-th coordinate equal to 0, and $I^{n-1}_{i,1}$ is defined likewise.

\subsubsection{Coorientations} \label{alg:sss:coorientation}

Our convention for orienting the boundary of an oriented manifold with boundary $(M , \partial M)$ can in fact be rephrased as follows : the boundary $\partial M$ is cooriented by the outward pointing vector field $\nu$.

More generally consider an oriented manifold $N$ and a submanifold $S \subset N$. A coorientation of $S$ is defined to be an orientation of the normal bundle to $S$. Given any complement bundle $\nu_S$ to $TS$ in $TN |_S$, 
\[ TN |_S = \nu_S \oplus TS \ , \]
this orientation induces in turn an orientation on $\nu_S$, the normal bundle being canonically isomorphic to $\nu_S$.
The manifold $S$ is then orientable if and only if it is coorientable. This can be proven using the first Stiefel-Whitney class for instance. Given a coorientation for $S$, \emph{the induced orientation on S} is set to be the one whose concatenation with that of $\nu_S$, in the order $(\nu_S , TS)$, gives the orientation on $TN |_S$.

\subsection{Signs for \Ainf -algebras and \Ainf -morphisms using the bar construction} \label{alg:ss:signs-ainf-bar}

There exist various conventions on signs for \Ainf -algebras and \Ainf -morphisms between them, which can seem inexplicable when met out of context. The goal of this section is twofold : to give a comprehensive account of the two sign conventions coming from the bar construction, and to state our choice of signs for the rest of the paper. The eager reader can straightaway jump to subsection~\ref{alg:sss:signs-ainf-choice}, where our choice of signs is given.

\subsubsection{\Ainf -algebras} \label{alg:sss:signs-ainf-alg}

We will first be interested in the following two sign conventions for \Ainf -algebras :
\begin{align*} 
\left[ m_1 , m_n \right] &=  - \sum_{\substack{i_1+i_2+i_3=n \\ 2 \leqslant i_2 \leqslant n-1}} (-1)^{i_1i_2 + i_3} m_{i_1+1+i_3} (\ide^{\otimes i_1} \otimes m_{i_2} \otimes \ide^{\otimes i_3} ) \ , \tag{A}  \\
\left[ m_1 , m_n \right] &= - \sum_{\substack{i_1+i_2+i_3=n \\ 2 \leqslant i_2 \leqslant n-1}} (-1)^{i_1 + i_2i_3} m_{i_1+1+i_3} (\ide^{\otimes i_1} \otimes m_{i_2} \otimes \ide^{\otimes i_3} ) \ , \tag{B} 
\end{align*}
which can we rewritten as
\begin{align*} 
\sum_{i_1+i_2+i_3=n} (-1)^{i_1i_2 + i_3} m_{i_1+1+i_3} (\ide^{\otimes i_1} \otimes m_{i_2} \otimes \ide^{\otimes i_3} ) &= 0 \ , \tag{A} \\
\sum_{i_1+i_2+i_3=n} (-1)^{i_1 + i_2i_3} m_{i_1+1+i_3} (\ide^{\otimes i_1} \otimes m_{i_2} \otimes \ide^{\otimes i_3} ) &= 0 \ . \tag{B}
\end{align*}

First, note that these two sign conventions are equivalent in the following sense : given a sequence of operations $m_n : A^{\otimes n} \rightarrow A$ satisfying equations (A), we check that the operations $m_n':=(-1)^{\binom{n}{2}}m_n$ satisfy equations (B). This sign change does not come out of the blue, and appears in the following proof that these equations come indeed from the bar construction.

Introduce the suspension and desuspension maps 
\begin{align*}
s : A &\longrightarrow sA & w : sA &\rightarrow A \\
a &\longmapsto sa & sa &\longmapsto a  \ ,
\end{align*} 
which are respectively of degree $-1$ and $+1$. We check that with the Koszul sign rule, 
\[ w^{\otimes n} \circ s^{\otimes n}  = (-1)^{\binom{n}{2}} \ide_{A^{\otimes n}} \ . \]

Then, note that a degree $2-n$ map $m_n : A^{\otimes n} \rightarrow A$ yields a degree $+1$ map $b_n := s m_n w^{\otimes n} : (sA)^{\otimes n} \rightarrow sA$.
Consider now a collection of degree $2-n$ maps $m_n : A^{\otimes n } \rightarrow A$, and the associated degree $+1$ maps $b_n : (sA)^{\otimes n} \rightarrow sA$. Denoting $D$ the unique coderivation on $\overline{T}(sA)$ associated to the $b_n$, the equation $D^2 = 0$ is then equivalent to the equations
\[ \sum_{i_1+i_2+i_3=n} b_{i_1+1+i_3} (\ide^{\otimes i_1} \otimes b_{i_2} \otimes \ide^{\otimes i_3} ) = 0 \ . \]
There are now two ways to unravel the signs from these equations.

The first way consists in simply replacing the $b_i$ by their definition. It leads to the (A) sign conventions : 
\begin{align*}
&\sum_{i_1+i_2+i_3=n} b_{i_1+1+i_3} (\ide^{\otimes i_1} \otimes b_{i_2} \otimes \ide^{\otimes i_3} )  \\ 
= &\sum_{i_1+i_2+i_3=n} sm_{i_1+1+i_3}(w^{\otimes i_1} \otimes w \otimes w^{\otimes i_3}) (\ide^{\otimes i_1} \otimes sm_{i_2}w^{\otimes i_2} \otimes \ide^{\otimes i_3} ) \\
= &\sum_{i_1+i_2+i_3=n} (-1)^{i_3}sm_{i_1+1+i_3}(w^{\otimes i_1} \otimes m_{i_2}w^{\otimes i_2} \otimes w^{\otimes i_3} ) \\
= &\sum_{i_1+i_2+i_3=n} (-1)^{i_3+i_1i_2}sm_{i_1+1+i_3}(\ide^{\otimes i_1} \otimes m_{i_2} \otimes \ide^{\otimes i_3} )(w^{\otimes i_1} \otimes w^{\otimes i_2} \otimes w^{\otimes i_3}) \\
= &s \left( \sum_{i_1+i_2+i_3=n} (-1)^{i_1i_2+i_3}m_{i_1+1+i_3}(\ide^{\otimes i_1} \otimes m_{i_2} \otimes \ide^{\otimes i_3} ) \right) w^{\otimes n} \ .
\end{align*}
The second way consists in first composing and post-composing by $w$ and $s^{\otimes n}$ and then replacing the $b_i$ by their definition. It leads to the (B) sign conventions and makes the $(-1)^{\binom{n}{2}}$ sign change appear: 
\begin{align*}
&\sum_{i_1+i_2+i_3=n} w b_{i_1+1+i_3} (\ide^{\otimes i_1} \otimes b_{i_2} \otimes \ide^{\otimes i_3} ) s^{\otimes n}  \\ 
= &\sum_{i_1+i_2+i_3=n} w b_{i_1+1+i_3} (\ide^{\otimes i_1} \otimes b_{i_2} \otimes \ide^{\otimes i_3} ) (s^{\otimes i_1} \otimes s^{\otimes i_2} \otimes s^{\otimes i_3}) \\ 
= &\sum_{i_1+i_2+i_3=n} (-1)^{i_1}w b_{i_1+1+i_3} (s^{\otimes i_1} \otimes b_{i_2} s^{\otimes i_2} \otimes s^{\otimes i_3} ) \\ 
= &\sum_{i_1+i_2+i_3=n} (-1)^{i_1}w s m_{i_1+1+i_3} w^{\otimes i_1 + 1 + i_3} (s^{\otimes i_1} \otimes sm_{i_2} w^{\otimes i_2} s^{\otimes i_2} \otimes s^{\otimes i_3} ) \\ 
= &\sum_{i_1+i_2+i_3=n} (-1)^{i_1} m_{i_1+1+i_3} w^{\otimes i_1 + 1 + i_3} (s^{\otimes i_1} \otimes (-1)^{ \binom{i_2}{2}}sm_{i_2} \otimes s^{\otimes i_3} ) \\
= &\sum_{i_1+i_2+i_3=n} (-1)^{i_1+i_2i_3} m_{i_1+1+i_3} w^{\otimes i_1 + 1 + i_3} s^{\otimes i_1+1+i_3} (\ide^{\otimes i_1} \otimes (-1)^{ \binom{i_2}{2}}m_{i_2} \otimes \ide^{\otimes i_3} ) \\
= &\sum_{i_1+i_2+i_3=n} (-1)^{i_1+i_2i_3} (-1)^{ \binom{i_1+1+i_3}{2}}m_{i_1+1+i_3} (\ide^{\otimes i_1} \otimes (-1)^{ \binom{i_2}{2}}m_{i_2} \otimes \ide^{\otimes i_3} ) \ .
\end{align*}

\subsubsection{\Ainf -morphisms} \label{alg:sss:signs-ainf-morph}

We now dwell into the two sign conventions for \Ainf -morphisms that are coming with the bar construction viewpoint. They are as follows :
\begin{align*} 
\left[ m_1 , f_n \right] =  &\sum_{\substack{i_1+i_2+i_3=n \\ i_2 \geqslant 2}} (-1)^{i_1i_2 + i_3} f_{i_1+1+i_3} (\ide^{\otimes i_1} \otimes m_{i_2} \otimes \ide^{\otimes i_3}) \tag{A} \\ &- \sum_{\substack{i_1 + \cdots + i_s = n \\ s \geqslant 2 }} (-1)^{\epsilon_A} m_s ( f_{i_1} \otimes \cdots \otimes f_{i_s}) \ ,  \\
\left[ m_1 , f_n \right] = &\sum_{\substack{i_1+i_2+i_3=n \\ i_2 \geqslant 2}} (-1)^{i_1 + i_2i_3} f_{i_1+1+i_3} (\ide^{\otimes i_1} \otimes m_{i_2} \otimes \ide^{\otimes i_3})  \tag{B} \\ &- \sum_{\substack{i_1 + \cdots + i_s = n \\ s \geqslant 2 }} (-1)^{\epsilon_B} m_s ( f_{i_1} \otimes \cdots \otimes f_{i_s}) \ ,  
\end{align*}
which can we rewritten as
\begin{align*} 
\sum_{i_1+i_2+i_3=n} (-1)^{i_1i_2 + i_3} f_{i_1+1+i_3} (\ide^{\otimes i_1} \otimes m_{i_2} \otimes \ide^{\otimes i_3}) &= \sum_{i_1 + \cdots + i_s = n} (-1)^{\epsilon_A} m_s ( f_{i_1} \otimes \cdots \otimes f_{i_s}) \ , \tag{A} \\
\sum_{i_1+i_2+i_3=n} (-1)^{i_1 + i_2i_3} f_{i_1+1+i_3} (\ide^{\otimes i_1} \otimes m_{i_2} \otimes \ide^{\otimes i_3}) &= \sum_{i_1 + \cdots + i_s = n} (-1)^{\epsilon_B} m_s ( f_{i_1} \otimes \cdots \otimes f_{i_s}) \ , \tag{B}
\end{align*}
where 
\begin{align*}
\epsilon_A = \sum_{u = 1}^s i_u \left( \sum_{ u < t \leqslant s} (1-i_t) \right) \ , && \epsilon_B = \sum_{u=1}^s (s-u) (1-i_u) \ .
\end{align*}

These two sign conventions are again equivalent : given a sequence of operations $m_n$ and $f_n$ satisfying equations (A), we check that the operations $m_n':=(-1)^{\binom{n}{2}}m_n$ and $f_n':=(-1)^{\binom{n}{2}}f_n$ satisfy equations (B). The $(-1)^{\binom{n}{2}}$ twist will again appear in the following proof, from the formula $w^{\otimes n} \circ s^{\otimes n}  = (-1)^{\binom{n}{2}} \ide_{A^{\otimes n}}$.

Consider now two dg-modules $A$ and $B$, together with a collection of degree $2-n$ maps $m_n : A^{\otimes n } \rightarrow A$ and $m_n : B^{\otimes n } \rightarrow B$ (we use the same notation for sake of readability), and a collection of degree $1-n$ maps $f_n : A^{\otimes n } \rightarrow B$. We associate again to the $m_n$ the degree $+1$ maps $b_n$, and also associate to the $f_n$ the degree 0 maps $F_n := sf_nw^{\otimes n} : (sA)^{\otimes n} \rightarrow sB$. We denote $D_A$ and $D_B$ the unique coderivations acting respectively on $\overline{T}(sA)$ and $\overline{T}(sB)$, and $F : \overline{T}(sA) \rightarrow \overline{T}(sB)$ the unique coalgebra morphism associated to the $F_n$. The equation $FD_A=D_BF$ is then equivalent to the equations
\[ \sum_{i_1+i_2+i_3=n} F_{i_1+1+i_3} (\ide^{\otimes i_1} \otimes b_{i_2} \otimes \ide^{\otimes i_3}) = \sum_{i_1 + \cdots + i_s = n} b_s ( F_{i_1} \otimes \cdots \otimes F_{i_s}) \ . \]
There are again two ways to unravel the signs from these equations, which will lead to conventions (A) and (B). The proofs proceed exactly as in subsection~\ref{alg:sss:signs-ainf-alg}.

\subsubsection{Composition of \Ainf -morphisms} \label{alg:sss:signs-ainf-comp}

Let $f_n : A^{\otimes n} \rightarrow B$ and $g_n : B^{\otimes n} \rightarrow C$ be two \Ainf -morphisms under conventions (A). The arity $n$ component of their composition $g \circ f$ is defined as
\begin{align*} 
\sum_{i_1 + \cdots + i_s = n} (-1)^{\epsilon_A} g_s ( f_{i_1} \otimes \cdots \otimes f_{i_s}) \ , \tag{A} 
\end{align*}
where $\epsilon_A$ is as previously.

Let $f_n : A^{\otimes n} \rightarrow B$ and $g_n : B^{\otimes n} \rightarrow C$ be two \Ainf -morphisms under conventions (B). The arity $n$ component of their composition $g \circ f$ is this time defined as
\begin{align*} 
\sum_{i_1 + \cdots + i_s = n} (-1)^{\epsilon_B} g_s ( f_{i_1} \otimes \cdots \otimes f_{i_s}) \ , \tag{B} 
\end{align*}
where $\epsilon_B$ is as previously.

We check that in each case, this newly defined morphism satisfies the \Ainf -equations, respectively under the sign conventions (A) and (B). This can again be proven using the bar construction and applying the previous transformations.

\subsubsection{Choice of convention in this paper} \label{alg:sss:signs-ainf-choice}

We will work in the rest of this paper under the set of conventions (B). The operations $m_n$ of an \Ainf -algebra will satisfy equations 
\[ \left[ m_1 , m_n \right] = - \sum_{\substack{i_1+i_2+i_3=n \\ 2 \leqslant i_2 \leqslant n-1}} (-1)^{i_1 + i_2i_3} m_{i_1+1+i_3} (\ide^{\otimes i_1} \otimes m_{i_2} \otimes \ide^{\otimes i_3} ) \ , \]
an \Ainf -morphism between two \Ainf -algebras will satisfy equations
\[ \resizebox{\hsize}{!}{$\displaystyle{ \left[ m_1 , f_n \right] = \sum_{\substack{i_1+i_2+i_3=n \\ i_2 \geqslant 2}} (-1)^{i_1 + i_2i_3} f_{i_1+1+i_3} (\ide^{\otimes i_1} \otimes m_{i_2} \otimes \ide^{\otimes i_3})  - \sum_{\substack{i_1 + \cdots + i_s = n \\ s \geqslant 2 }} (-1)^{\epsilon_B} m_s ( f_{i_1} \otimes \cdots \otimes f_{i_s}) \ , }$} \]
and two \Ainf -morphisms will be composed as
\[ \sum_{i_1 + \cdots + i_s = n} (-1)^{\epsilon_B} g_s ( f_{i_1} \otimes \cdots \otimes f_{i_s}) \ , \]
where $\epsilon_B = \sum_{u=1}^s (s-u) (1-i_u)$.

This choice of conventions will be accounted for in the next two sections : the signs are the ones which arise naturally from the realizations of the associahedra and the multiplihedra à la Loday. We also point out that a choice of convention for the signs on \Ainf -algebras completely determines the conventions on \Ainf -morphisms and their composition.

\subsection{Loday associahedra and signs} \label{alg:ss:loday-assoc}

\Ainf -structures were introduced for the first time in two seminal papers by Stasheff on homotopy associative H-spaces~\cite{stasheff-homotopy}. In the first paper of the series, he defined cell complexes $K_n \subset I^{n-2}$ which govern $A_n$-structures on topological spaces, and hence realize the associahedra as cell complexes. The associahedra were later realized as polytopes by Haiman in~\cite{haiman-assoc}, Lee in~\cite{lee-assoc} or Loday in~\cite{loday-assoc}. They were recently endowed with an operad structure in the category $\mathtt{Poly}$ by Masuda, Thomas, Tonks and Vallette in~\cite{masuda-diagonal-assoc}, using the notion of weighted Loday realizations. 

Following~\cite{masuda-diagonal-assoc}, we explain the construction of these realizations. We then show that the sign convention (B) for \Ainf -algebras is determined by these realizations : this gives a more geometric explanation of these signs, which does not come from a $(-1)^{\binom{n}{2}}$ twist after reading the signs on the bar construction. This also provides an explicit proof with signs of the statement in~\cite{masuda-diagonal-assoc}, that these polytopes are sent to the operad \Ainf\ by the functor $C_{-*}^{cell}$ (Proposition~\ref{alg:prop:loday-assoc-signs}). These realizations moreover achieve the first step towards constructing the morphism of operads of Markl-Shnider $\Ainf \rightarrow \Omega B As$.

\subsubsection{Realizations of the associahedra à la Loday} \label{alg:sss:real-a-la-loday}

\begin{definition}[\cite{masuda-diagonal-assoc}]
Given $n \geqslant 1$, define a weight $\mathbf{\omega}$ to be a list of $n$ positive integers $(\omega_1,\dots,\omega_n)$. The \emph{Loday realization of weight $\mathbb{\omega}$} of $K_n$ is defined as the common intersection in $\R^{n-1}$ of the hyperplane of equation
\[ H_\omega : \sum_{i=1}^{n-1} x_i = \sum_{1 \leqslant k < l \leqslant n} \omega_k \omega_l \ \]
and of the half-spaces of equation 
\[ D_{i_1,i_2,i_3} : x_{i_1+1} + \cdots + x_{i_1+i_2-1} \geqslant \sum_{i_1+1 \leqslant k < l \leqslant i_1+i_2} \omega_k \omega_l \ , \]
for all $i_1+i_2+i_3 = n$ and $2 \leqslant i_2 \leqslant n-1$. This polytope is denoted $K_\mathbf{\omega}$.
\end{definition}

\begin{figure}[h]
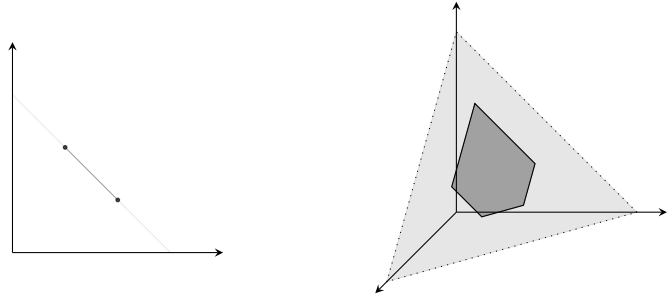
 
    \centering
    \begin{subfigure}{0.3\textwidth}
    \centering
       \associaedrelodaydeux
    \end{subfigure} ~
    \begin{subfigure}{0.3\textwidth}
    \centering
        \associaedrelodaytrois
    \end{subfigure} ~
    \caption{The Loday realizations $K_{(1,1)}$ and $K_{(1,1,1)}$ : the lighter grey depicts $H_\omega$, while the darker grey stands for $K_\omega$.} \label{alg:fig:loday-real}
\end{figure}

The Loday realizations $K_{(1,1)}$ and $K_{(1,1,1)}$ are represented in figure~\ref{alg:fig:loday-real}. The polytope $K_\mathbf{\omega}$ being defined as an intersection of half-spaces inside the $(n-2)$-dimensional space $H_\omega$, it has dimension $n-2$. In fact, denoting $\mathbf{1}_n$ the weight of length $n$ whose entries are all equal to 1, it is one of the main results of~\cite{masuda-diagonal-assoc} that the collection of polytopes $(K_{\mathbf{1}_n})_{n \geqslant 1}$ can be made into an operad in the category $\mathtt{Poly}$. 
The goal of this section is to show the following proposition : 

\begin{proposition} \label{alg:prop:loday-assoc-signs}
The Loday associahedra determine the sign conventions (B) for \Ainf -algebras.
\end{proposition}

That is, after orienting each polytope $K_n := K_{\mathbf{1}_n}$ the boundary of $K_n$ reads as 
\[ \partial K_n = - \bigcup_{\substack{i_1+i_2+i_3=n \\ 2 \leqslant i_2 \leqslant n-1}} (-1)^{i_1 + i_2 i_3} K_{i_1 + 1 + i_3} \times K_{i_2} \ , \]
where $K_{i_1 + 1 + i_3} \times K_{i_2}$ is sent to $m_{i_1+1+i_3} (\ide^{\otimes i_1} \otimes m_{i_2} \otimes \ide^{\otimes i_3} )$ under the functor $C^{cell}_{-*}$. The signs mean that after comparing the product orientation on $K_{i_1 + 1 + i_3} \times K_{i_2}$ induced by the orientations of $K_{i_1 + 1 + i_3}$ and $K_{i_2}$, to the orientation of the boundary of $K_n$, they differ by the sign $-(-1)^{i_1+i_2i_3}$. 
We explain now how to obtain the set-theoretic decomposition of the boundary 
$$\partial K_n = \bigcup_{\substack{i_1+i_2+i_3=n \\ 2 \leqslant i_2 \leqslant n-1}} K_{i_1 + 1 + i_3} \times K_{i_2} \ ,$$ and inspect the signs in the next section.

The top dimensional strata in the boundary of some $K_\omega$ are obtained by allowing exactly one of the inequalities 
\[ x_{i_1+1} + \cdots + x_{i_1+i_2-1} \geqslant \sum_{i_1+1 \leqslant k < l \leqslant i_1 + i_2} \omega_k \omega_l \ , \]
to become an equality. We write $H_{i_1,i_2,i_3}$ for these hyperplanes. Defining two new weights
\begin{align*}
\overline{\omega} &:= (\omega_1,\dots,\omega_{i_1},\omega_{i_1+1} + \cdots + \omega_{i_1 + i_2},\omega_{i_1+i_2+1},\dots,\omega_n) \ , \\
 \widetilde{\omega} &:= (\omega_{i_1+1},\dots,\omega_{i_1+i_2}) \ ,
\end{align*}
the map 
\begin{align*}
\theta : \R^{i_1 + i_3} \times \R^{i_2-1} &\longrightarrow \R^{n-1} \\
(x_1,\dots,x_{i_1 + i_3}) \times (y_1,\dots,y_{i_2-1}) &\longmapsto (x_1,\dots,x_{i_1},y_1,\dots,y_{i_2-1},x_{i_1+1},\dots,x_{i_1+i_3})
\end{align*}
induces a bijection between $K_{\overline{\omega}} \times K_{\widetilde{\omega}}$ and the codimension 1 face of $K_\omega$ corresponding to the intersection with $H_{i_1,i_2,i_3}$. 

\subsubsection{Recovering signs from these realizations} \label{alg:sss:recover-signs-loday}

The directing hyperplane $\overline{H}_\omega$ of the affine hyperplane $H_\omega$ has basis 
\[ e_j^\omega = (1,0,\cdots,0,-1_{j+1},0,\cdots,0) \ , \]
where $-1$ is in the $j+1$-th spot, and we add a superscript $\omega$ for later use. We choose this basis as a positively oriented basis for $\overline{H}_\omega$ : this defines our orientation of $K_\omega$.
Choosing any $(a_1,\dots,a_{n-1}) \in H_\omega$, the basis $e_j^\omega$ parametrizes $H_\omega$ under the map
\[ (y_1 , \dots , y_{n-2}) \longmapsto (\sum_{j=1}^{n-2}y_j + a_1 , - y_1 + a_2 , \dots, -y_{n-2}+a_{n-1}) \ . \]
Hence in the coordinates of the basis $e_j^\omega$, the half-space $H_\omega \cap D_{i_1,i_2,i_3}$ reads as 
\begin{align*}
\text{when $i_1 = 0$ : }& - y_{i_2-1} - \cdots - y_{n-2} \leqslant C \ , \\
\text{when $i_1 \geqslant 1$ : }&  y_{i_1} + \cdots + y_{i_1+i_2-2} \leqslant C \ ,
\end{align*}
where $C$ denotes some constant that we are not interested in. Hence, in the basis $e_j^\omega$, an outward pointing vector for the boundary $H_\omega \cap H_{i_1,i_2,i_3}$ is 
\begin{align*}
\text{when $i_1 = 0$ : }& \nu := (0 , \dots , 0 , -1_{i_2 -1} , \dots , -1_{n-2}) \ , \\
\text{when $i_1 \geqslant 1$ : }& \nu := (0 , \dots , 0 , 1_{i_1} , \dots , 1_{i_1+i_2-2} , 0 ,\dots , 0) \ .
\end{align*}

We have chosen orienting bases for the directing hyperplanes $\overline{H}_\omega$, and computed all outward pointing vectors for the boundaries in these bases. It only remains to study the image of these bases under the maps $\theta$. We write $e_j^{\overline{\omega}}$ for the orienting basis of $K_{\overline{\omega}}$ and $e_j^{\widetilde{\omega}}$ for the one of $K_{\widetilde{\omega}}$. We distinguish two cases.

When $i_1=0$, the map $\theta$ reads as 
\[ \theta (x_1,\dots,x_{i_3},y_1,\dots,y_{i_2-1}) = (y_1,\dots,y_{i_2-1},x_{1},\dots,x_{i_3}) \ , \]
and we compute that :
\begin{align*}
\theta (e_j^{\overline{\omega}}) = - e^\omega_{i_2 - 1} + e^\omega_{j + i_2 - 1} && \theta (e_j^{\widetilde{\omega}}) = e^\omega_{j} \ .
\end{align*}
The determinant then has value
\[ \mathrm{det}_{e^\omega_j} \left( \nu , \theta (e_j^{\overline{\omega}}),\theta (e_j^{\widetilde{\omega}}) \right) = -i_3 (-1)^{i_2 i_3} \ . \]
Thus, we recover the $-(-1)^{i_1+i_2 i_3} K_{i_1 + 1 + i_3} \times K_{i_2}$ oriented component of the boundary.

When $i_1 \geqslant 1$, the map $\theta$ now reads as 
\[ \theta (x_1,\dots,x_{i_3},y_1,\dots,y_{i_2-1}) = (x_1,\dots,x_{i_1},y_1,\dots,y_{i_2-1},x_{i_1+1},\dots,x_{i_1+i_3}) \ , \]
and we compute that :
\begin{align*}
j \leqslant i_1 -1 \ , \ \theta (e_j^{\overline{\omega}}) = e^\omega_{j} && j \geqslant i_1 \ , \ \theta (e_j^{\overline{\omega}}) = e^\omega_{j+i_2 - 1} && \theta(e_j^{\widetilde{\omega}}) = e_{j+i_1}^{\omega}- e_{i_1}^\omega \ . 
\end{align*}
This time, 
\[ \mathrm{det}_{e^\omega_j} \left( \nu , \theta (e_j^{\overline{\omega}}),\theta (e_j^{\widetilde{\omega}}) \right) = - (i_2 - 1)(-1)^{i_1 + i_2 i_3} \ . \]
We find again the $-(-1)^{i_1+i_2 i_3} K_{i_1 + 1 + i_3} \times K_{i_2}$ oriented component of the boundary, which concludes the proof of~Proposition~\ref{alg:prop:loday-assoc-signs}.

\subsection{Forcey-Loday multiplihedra and signs} \label{alg:ss:forcey-loday-multipl}

Iwase and Mimura realized the multiplihedra as cell complexes in~\cite{iwase-mimura} following the hints of Stasheff in~\cite{stasheff-homotopy}. The multiplihedra were later realized as polytopes in~\cite{forcey-multipl}. This will be adapted in an upcoming paper by Masuda, Vallette and the author~\cite{masuda-diagonal-multipl}, which uses again the notion of weighted Loday realizations.

The goal of this section is to show that the sign convention (B) for \Ainf -morphisms is naturally determined by the weighted Loday realizations of~\cite{masuda-diagonal-multipl}. In this regard, we lay out the explicit construction of~\cite{masuda-diagonal-multipl}, and follow the same lines of proof as in the previous section. This also provides a proof with signs that these polytopes are sent to the operadic bimodule \infmor\ by the functor $C_{-*}^{cell}$ (Proposition~\ref{alg:prop:loday-multi-signs}).

\subsubsection{Forcey-Loday realizations of the multiplihedra} \label{alg:sss:forcey-loday-real}

\begin{definition}[\cite{masuda-diagonal-multipl}]
Given $n \geqslant 1$, choose a weight $\mathbf{\omega}=(\omega_1,\dots,\omega_n)$. The \emph{Forcey-Loday realization} of weight $\mathbb{\omega}$ of $J_n$ is defined as the intersection in $\R^{n-1}$ of the half-spaces of equation 
\[ D_{i_1,i_2,i_3} : x_{i_1+1} + \cdots + x_{i_1+i_2-1} \geqslant \sum_{i_1+1 \leqslant k < l \leqslant i_1+i_2} \omega_k \omega_l \ , \]
for all $i_1+i_2+i_3 = n$ and $i_2 \geqslant 2$, with the half-spaces of equation
\[ D^{i_1,\dots,i_s} : x_{i_1} + x_{i_1 + i_2} + \cdots + x_{i_1+\cdots+i_{s-1}} \leqslant 2 \sum_{1 \leqslant t < u \leqslant s} \Omega_{t} \Omega_{u} \  \]
for all $i_1+\cdots+i_s = n$, with each $i_t \geqslant 1$ and $s \geqslant 2$, and where $\Omega_t := \sum_{a=1}^{i_t}\omega_{i_1 + \cdots + i_{t-1} + a}$. This polytope is denoted $J_\mathbf{\omega}$. 
\end{definition}

\begin{figure}[h]
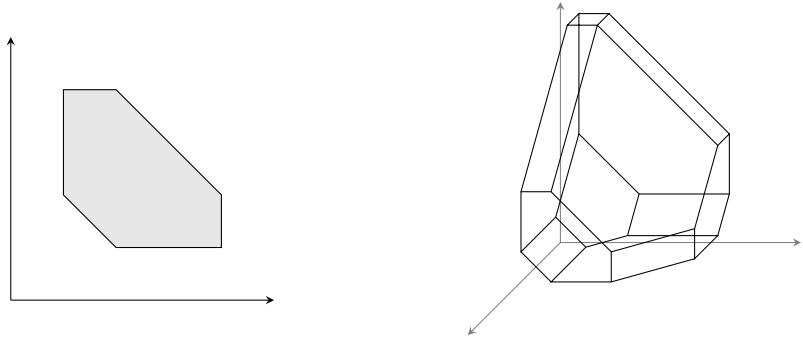
 
    \centering
    \begin{subfigure}{0.3\textwidth}
    \centering
       \multiplaedreforceylodaytrois
    \end{subfigure} ~
    \hspace{30 pt}
    \begin{subfigure}{0.3\textwidth}
    \centering
        \multiplaedreforceylodayquatre
    \end{subfigure} ~
    \caption{The Forcey-Loday realizations $J_{(1,1,1)}$ and $J_{(1,1,1,1)}$} \label{alg:fig:forcey-loday-real}
\end{figure}

The Forcey-Loday realizations $J_{(1,1,1)}$ and $J_{(1,1,1,1)}$ are depicted in figure~\ref{alg:fig:forcey-loday-real}. The polytope $J_\mathbf{\omega}$ being an intersection of half-spaces in $\R^{n-1}$, it has dimension $n-1$. Setting $J_n := J_{\mathbf{1}_n}$, it is proven in~\cite{masuda-diagonal-multipl} that the collection of polytopes $\{ J_n \}_{n \geqslant 1}$ can be made into a $(\{ K_n \} , \{ K_n \})$-operadic bimodule in the category $\mathtt{Poly}$.

\begin{proposition}  \label{alg:prop:loday-multi-signs}
The Forcey-Loday realizations determine the sign conventions (B) for \Ainf -morphisms. 
\end{proposition}

More precisely our goal is to prove that, after orienting the $K_n$ as before and choosing an orientation for the $J_n$, the boundary of $J_n$ reads as 
\[ \partial J_n = \bigcup_{\substack{i_1+i_2+i_3=n \\ i_2 \geqslant 2}} (-1)^{i_1 + i_2 i_3} J_{i_1 + 1 + i_3} \times K_{i_2} \cup - \bigcup_{\substack{i_1 + \cdots + i_s = n \\ s \geqslant 2}} (-1)^{\epsilon_B} K_s \times J_{i_1} \times \cdots \times J_{i_s}  \ , \]
where $\epsilon_B$ is as in subsection~\ref{alg:sss:signs-ainf-choice} ; $K_{i_1 + 1 + i_3} \times K_{i_2}$ is sent to $f_{i_1+1+i_3} (\ide^{\otimes i_1} \otimes m_{i_2} \otimes \ide^{\otimes i_3} )$ while $K_s \times J_{i_1} \times \cdots \times J_{i_s}$ is sent to $ m_s ( f_{i_1} \otimes \cdots \otimes f_{i_s})$ by the functor $C^{cell}_{-*}$. 

We conclude this section with a proof of the set-theoretic equality for the boundary
$$ \partial J_n = \bigcup_{\substack{i_1+i_2+i_3=n \\ i_2 \geqslant 2}} J_{i_1 + 1 + i_3} \times K_{i_2} \cup  \bigcup_{\substack{i_1 + \cdots + i_s = n \\ s \geqslant 2}} K_s \times J_{i_1} \times \cdots \times J_{i_s}  \ ,$$ and postpone the processing of signs to the next subsection.
The top dimensional strata in the boundary of a $J_\omega$ are obtained by allowing exactly one of the inequalities 
\begin{align*} x_{i_1+1} + \cdots + x_{i_1+i_2-1} \geqslant \sum_{i_1+1 \leqslant k < l \leqslant i_1 + i_2} \omega_k \omega_l \ , \\
x_{i_1} + x_{i_1 + i_2} + \cdots + x_{i_1+\cdots+i_{s-1}} \leqslant 2 \sum_{1 \leqslant t < u \leqslant s} \Omega_{t} \Omega_{u} \ ,
\end{align*}
to become an equality. We write $H_{i_1,i_2,i_3}$ and $H^{i_1,\dots,i_s}$ for these hyperplanes. 

Begin with the $H_{i_1,i_2,i_3}$ component. Defining two new weights
\begin{align*}
\overline{\omega} &:= (\omega_1,\dots,\omega_{i_1},\omega_{i_1+1} + \cdots + \omega_{i_1 + i_2},\omega_{i_1+i_2+1},\dots,\omega_n) \ , \\
 \widetilde{\omega} &:= (\omega_{i_1+1},\dots,\omega_{i_1+i_2}) \ ,
\end{align*}
the map 
\begin{align*}
\theta : \R^{i_1 + i_3} \times \R^{i_2-1} &\longrightarrow \R^{n-1} \\
(x_1,\dots,x_{i_1 + i_3}) \times (y_1,\dots,y_{i_2-1}) &\longmapsto (x_1,\dots,x_{i_1},y_1,\dots,y_{i_2-1},x_{i_1+1},\dots,x_{i_1+i_3})
\end{align*}
induces a bijection between $J_{\overline{\omega}} \times K_{\widetilde{\omega}}$ and the codimension 1 face of $J_\omega$ corresponding to the intersection with $H_{i_1,i_2,i_3}$. 

In the case of the $H^{i_1,\dots,i_s}$ component, we define the weights
\begin{align*}
\overline{\omega} &:= (\sqrt{2}\Omega_1,\dots,\sqrt{2}\Omega_s) \ , \\
\widetilde{\omega}_t &:= (\omega_{i_1+\cdots + i_{t-1} + 1},\dots,\omega_{i_1+\cdots + i_{t-1} + i_t}) \ , \ 1 \leqslant t \leqslant s \ .
\end{align*}
This time, the map
\[ \theta : \R^{s-1} \times \R^{i_1-1} \times \cdots \times \R^{i_s-1} \longrightarrow \R^{n-1} \]
sends an element $(x_1,\dots,x_{s-1}) \times (y_1^1,\dots,y_{i_1-1}^1) \times \cdots \times (y_1^s,\dots,y_{i_s-1}^s)$ to
\[ (y_1^1,\dots,y_{i_1-1}^1,x_1,y_1^2,\dots,y_{i_2-1}^2,x_2,y_1^3,\dots,x_{s-1},y_1^s,\dots,y_{i_s-1}^s) \ . \]
It induces a bijection between $K_{\overline{\omega}} \times J_{\widetilde{\omega}_1} \times \cdots \times J_{\widetilde{\omega}_s}$ and the codimension 1 face of $J_\omega$ corresponding to the intersection with $H^{i_1,\dots,i_s}$. 

\subsubsection{Processing the signs for these realizations} \label{alg:sss:process-signs}

We set the orientation on $\R^{n-1}$, and hence on $J_\omega$, to be such that the vectors
\[ f_j^\omega := (0,0,\cdots,0,-1_{j},0,\cdots,0) \ , \]
define a positively oriented basis of $\R^{n-1}$. In the coordinates of the basis $f_j^\omega$, the half-space $D_{i_1,i_2,i_3}$ reads as 
\[ z_{i_1+1} + \cdots + z_{i_1+i_2-1} \leqslant - \sum_{i_1+1 \leqslant k < l \leqslant i_1+i_2} \omega_k \omega_l \ , \]
and the half-space $D^{i_1,\dots,i_s}$ as
\[  - z_{i_1} - z_{i_1 + i_2} - \cdots - z_{i_1+\cdots+i_{s-1}} \leqslant 2 \sum_{1 \leqslant t < u \leqslant s} \Omega_{t} \Omega_{u} \]
In this basis, an outward pointing vector for the boundary $H_{i_1,i_2,i_3}$ is then
\[  \nu := ( 0 ,\dots , 0 , 1_{i_1 + 1} , \dots , 1_{i_1 + i_2-1} , 0 , \dots , 0 )   \ , \]
while an outward pointing vector for the boundary $H^{i_1,\cdots,i_s}$ is 
\[ \nu := ( 0 , \dots , 0 , -1_{i_1} , 0 , \dots , 0 , -1_{i_1 + i_2} , 0 , \dots \dots , 0 , -1_{i_1 + i_2 + \cdots + i_{s-1}} , 0 , \dots , 0 ) \ . \]
Now that we have chosen positively oriented bases for the $J_\omega$, and chosen outward pointing vectors for each component of their boundaries, we conclude again by computing the image of these bases under the maps $\theta$. 

In the case of a boundary component $H_{i_1,i_2,i_3}$, 
\begin{align*}
j \leqslant i_1 \ , \ \theta (f_j^{\overline{\omega}}) = f^\omega_{j} && j \geqslant i_1 + 1 \ , \ \theta (f_j^{\overline{\omega}}) = f^\omega_{j+i_2 - 1} && \theta (e_j^{\widetilde{\omega}}) = - f_{i_1 + 1}^\omega + f^\omega_{i_1 + j+1} \ .
\end{align*}
The determinant against the basis $f_j^\omega$ then has value
\[ \mathrm{det}_{f^\omega_j} \left( \nu , \theta (f_j^{\overline{\omega}}),\theta (e_j^{\widetilde{\omega}}) \right) = (i_2 - 1) (-1)^{i_1 + i_2 i_3} \ . \]
Thus, we recover the $(-1)^{i_1+i_2 i_3} J_{i_1 + 1 + i_3} \times K_{i_2}$ oriented component of the boundary.

Finally, in the case of a boundary component $H^{i_1,\dots,i_s}$, we compute that
\begin{align*}
\theta (e_j^{\overline{\omega}}) = - f^\omega_{i_1} + f^\omega_{i_1 + \cdots + i_{j+1}}  && \theta(f_j^{\widetilde{\omega}_t}) = f^\omega_{j+i_1 + \cdots + i_{t-1}} \ . 
\end{align*}
This time, 
\[ \mathrm{det}_{f^\omega_j} \left( \nu , \theta (e_j^{\overline{\omega}}),\theta (f_j^{\widetilde{\omega}_1}) , \dots , \theta (f_j^{\widetilde{\omega}_s}) \right) = - (s - 1)(-1)^{\epsilon_B} \ . \]
We find again the $-(-1)^{\epsilon_B} K_{s} \times J_{i_1} \times \cdots \times J_{i_s}$ oriented component of the boundary, which concludes the proof of Proposition~\ref{alg:prop:loday-multi-signs}.

\section{Signs and moduli spaces for $\Omega B As$-algebras and $\Omega B As$-morphisms} \label{alg:s:sign-pol-ombas}

This section completes section~\ref{alg:s:mod-spac} by explicitly describing the two families of moduli spaces of metric trees $\mathcal{T}_n(t)$ and $\mathcal{CT}_n(t_g)$, working out the induced signs for $\Omega B As$-algebras and $\Omega B As$-morphisms and eventually constructing the morphisms of Propositions~\ref{alg:prop:markl-shnider-un} and~\ref{alg:prop:markl-shnider-deux}.

More precisely, we begin by recalling the definition of the operad $\Omega B As$ from Markl-Shnider, using the formalism of orientations on broken stable ribbon trees. This establishes a direct link to the moduli spaces $\mathcal{T}_n(t)$. Using the fact that the dual decomposition on the associahedron coincides with its $\Omega B As$ decomposition, we give a new proof of the morphism of operads $\Ainf \rightarrow \Omega B As$, that relies uniquely on polytopes and not on sign computations. We then attend to the definition of the operadic bimodule $\Omega B As - \mathrm{Morph}$. This goes through a long and comprehensive study of the signs ensuing from orientations of the codimension 1 strata of the compactified moduli spaces $\overline{\mathcal{CT}}_n(t_{g})$. We finally define the morphism of operadic bimodules $\Ainf - \mathrm{Morph} \rightarrow \Omega B As - \mathrm{Morph}$, using again solely the realizations of the multiplihedra from~\cite{masuda-diagonal-multipl}. This is an opportunity to state a MacLane's coherence theorem encoded by the multiplihedra, while the classical MacLane's coherence theorem on monoidal categories is encoded by the associahedra (see subsection~\ref{alg:sss:maclane-coherence}). 

\subsection{The operad $\Omega B As$} \label{alg:ss:signs-ombas}

\subsubsection{Definition of the operad $\Omega B As$}\label{alg:sss:def-ombas}

The definition of the operad $\Omega B As$ that we now lay out is the one given by Markl and Shnider in~\cite{markl-assoc}. We only expose the material necessary to our construction, and refer to their paper for further details and proofs. In the rest of the section, the notation $t$ stands for a stable ribbon tree, and the notation $t_{br}$ denotes a broken stable ribbon tree. Observe that a stable ribbon tree is a broken stable ribbon tree with 0 broken edge. As a result, all constructions performed for broken stable ribbon trees in the upcoming subsections will hold in particular for stable ribbon trees.

\begin{definition}[\cite{markl-assoc}]
Given a broken stable ribbon tree $t_{br}$, an \emph{ordering} of $t_{br}$ is defined to be an ordering of its $i$ finite internal edges $e_{1} , \dots , e_{i}$. Two orderings are said to be equivalent if one passes from one ordering to the other by an even permutation. An \emph{orientation} of $t_{br}$ is then defined to be an equivalence class of orderings, and written $\omega := e_{1} \wedge \cdots \wedge e_{i}$. Each tree $t_{br}$ has exactly two orientations. Given an orientation $\omega$ of $t_{br}$ we will write $-\omega$ for the second orientation on $t_{br}$, called its \emph{opposite orientation}.
\end{definition}

\begin{definition}[\cite{markl-assoc}]
Consider the \Z -module freely generated by the pairs $(t_{br},\omega)$ where $t_{br}$ is a broken stable ribbon tree and $\omega$ an orientation of $t_{br}$. We define the arity $n$ space of operations $\Omega B As(n)_*$ to be the quotient of this \Z -module under the relation
\[ (t_{br} , - \omega) = - (t_{br} , \omega) \ . \]
A pair $(t_{br},\omega)$ where $t_{br}$ has $i$ finite internal edges, is defined to have degree $-i$. The partial compositions are then
\[ (t_{br},\omega) \circ_k (t_{br}',\omega ') = (t_{br} \circ_k t_{br}' , \omega \wedge \omega ') \ , \]
where the tree $t_{br} \circ_k t_{br}'$ is the broken ribbon tree obtained by grafting $t_{br}'$ to the $k$-th incoming edge of $t_{br}$, and the edge resulting from the grafting is broken. The differential $\partial_{\Omega B As}$ on $\Omega B As (n)_*$ is finally set to send an element $( t_{br}, e_{1} \wedge \cdots \wedge e_{i})$ to 
\[ \sum_{j=1}^i (-1)^{j} \left( (t_{br}/e_{j} , e_{1} \wedge \cdots \wedge \hat{e_{j}} \wedge \cdots \wedge e_{i}) - ((t_{br})_j , e_{1} \wedge \cdots \wedge \hat{e_{j}} \wedge \cdots \wedge e_{i}) \right) \ , \]
where $t_{br}/e_{j}$ is the tree obtained from $t$ by collapsing the edge $e_{j}$ and $(t_{br})_j$ is the tree obtained from $t_{br}$ by breaking the edge $e_{j}$. It can be checked that the collection of dg-\Z -modules $\Omega B As(n)_*$ defines indeed an operad in $\mathtt{dg-\Z -mod}$.
\end{definition}

Choosing a distinguished orientation for every stable ribbon tree $t \in SRT$, this definition of the operad $\Omega B As$ yields the definition as the quasi-free operad
\[ \mathcal{F} (\premiertermecobarbarA , \premiertermecobarbarD , \premiertermecobarbarB , \premiertermecobarbarC , \cdots , SRT_n, \cdots ) \ ,  \]
given in subsection~\ref{alg:sss:op-ombas}. Our definition with the pairs $(t,\omega)$, albeit more tedious at first sight, allows however for easier computations of signs.

\subsubsection{Canonical orientations for the binary ribbon trees (\cite{markl-assoc})} \label{alg:sss:canon-or-ribbon-tree}

For a fixed $n \geqslant 2$, the set of binary ribbon trees $BRT_n$ can be endowed with a partial order that Tamari introduced in his thesis~\cite{tamari-monoides}. 

\begin{definition}
The \emph{Tamari order} on $BRT_n$ is the partial order generated by the covering relations
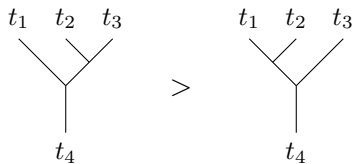
\[ \coveringrelationA \ \ \ > \ \ \ \coveringrelationB \]
where $t_1$, $t_2$, $t_3$ and $t_4$ are binary ribbon trees. 
\end{definition}

The left-hand side in the above covering relation will be called a \emph{right-leaning configuration}, and the right-hand side a \emph{left-leaning configuration}. 
Hence given two trees $t$ and $t'$ in $BRT_n$, the inequality $t \geqslant t'$ holds if and only one can pass from $t$ to $t'$ by successive transformations of a right-leaning configuration into a left-leaning configuration.
For example in the case of $BRT_4$, we obtain the Hasse diagram in figure~\ref{alg:fig:tamari-poset}. 

\begin{figure}[h] 
\centering
\begin{subfigure}{0.4\textwidth}
	\centering
	\Tamariquatre
\end{subfigure} 
\begin{subfigure}{0.4\textwidth}
	\centering
	\Tamariorientationsquatre
\end{subfigure}
\caption{On the left, the Hasse diagram of the Tamari poset, where the maximal element is written at the top. On the right, all the canonical orientations for $BRT_4$ computed going down the Tamari poset.} \label{alg:fig:tamari-poset}
\end{figure} 

The Tamari poset has a unique maximal element and a unique minimal element, respectively given by the right-leaning and left-leaning combs, denoted $t_{max}$ and $t_{min}$. Given moreover a binary ribbon tree $t$, its immediate neighbours are by definition the trees obtained from $t$ by either transforming exactly one right-leaning configuration of $t$ into a left-leaning configuration, or transforming exactly one left-leaning configuration of $t$ into a right-leaning configuration.

The canonical orientation on the maximal binary tree is defined as
\[ \rightleaningcomb \ \ \omega_{can} := e_1 \wedge \cdots \wedge e_{n-2}  \ . \]
Using the Tamari order, we can now build inductively canonical orientations on all binary trees. We start at the maximal binary ribbon tree, and use the following rule on the covering relations
\[ \coveringrelationorientationsA  \ \ \omega = \cdots \wedge e \wedge \cdots  \ \ \ \longrightarrow \ \ \  \coveringrelationorientationsB \ \ - \omega = \cdots \wedge (-e) \wedge \cdots  \ , \]
to define the orientations of its immediate neighbours. We then repeat this rule while going down the Tamari poset until the minimal binary tree is reached. This process is consistent (see subsection~\ref{alg:sss:maclane-coherence}), i.e. it does not depend on the path taken in the Tamari poset from the maximal binary tree to the binary tree whose orientation is being defined. A full example for $BRT_4$ is illustrated in figure~\ref{alg:fig:tamari-poset}.

\begin{definition}[\cite{markl-assoc}]
The orientations obtained under this process are called the \emph{canonical orientations} and written $\omega_{can}$. 
\end{definition}

\subsubsection{The moduli spaces $\overline{\mathcal{T}}_n$ realize the operad $\Omega B As$} \label{alg:sss:mod-space-Tn-real}

We explained in subsection~\ref{alg:sss:second-cell-ribbon} that the compactified moduli space $\overline{\mathcal{T}}_n$ comes with a fine cell decomposition, which is labeled by all broken stable ribbon trees with $n$ incoming edges. Consider then a cell $\overline{\mathcal{T}}_n(t_{br}) \subset \overline{\mathcal{T}}_n$, where $t_{br}$ is a broken stable ribbon tree. An ordering of its finite internal edges $e_{1} , \dots , e_{i}$ induces an isomorphism
\[ \overline{\mathcal{T}}_n(t_{br}) \ \tilde{\longrightarrow} \ [ 0 , + \infty ]^{i} \ , \]
where the length $l_{e_{j}}$ is seen as the $j$-th coordinate in $[ 0 , + \infty ]^{i}$. This ordering induces in particular an orientation on $\mathcal{T}_n(t_{br})$, by taking the image of the canonical orientation of $] 0 , + \infty [^{i}$ under the isomorphism. We check that two orderings of $t_{br}$ define the same orientation on $\mathcal{T}_n(t_{br})$ if and only if they are equivalent : in other words, an orientation of $t_{br}$ amounts to an orientation of $\mathcal{T}_n(t_{br})$.

Consider now the \Z -module freely generated by the pairs $$(\overline{\mathcal{T}}_n(t_{br}) , \text{choice of orientation $\omega$ on the cell $\overline{\mathcal{T}}_n(t_{br})$}) \ ,$$ where $t_{br}$ is a broken stable ribbon tree. The complex $C_{-*}^{cell}(\overline{\mathcal{T}}_n)$ can simply be defined to be the quotient of this \Z -module under the relation
\[ - (\overline{\mathcal{T}}_n(t_{br}) , \omega) = (\overline{\mathcal{T}}_n(t_{br}) , - \omega) \ . \]
The differential of an element $(\overline{\mathcal{T}}_n(t_{br}) , \omega)$ is moreover given by the classical cubical differential on $[ 0 , + \infty ]^{i}$. Defining the cell chain complex in this way, it becomes tautological that :
\begin{proposition} 
The functor $C_{-*}^{cell}$ sends the operad $\overline{\mathcal{T}}_n$ to the operad $\Omega B As$.
\end{proposition}

What's more, it can be easily seen that given a binary ribbon tree $t$, the cells labeled by the immediate neighbours to the tree $t$ in the Tamari order are exactly the cells having a codimension 1 stratum in common with the cell $\overline{\mathcal{T}}_n(t)$.

\subsubsection{The morphism of operads $\Ainf \rightarrow \Omega B As$} \label{alg:sss:morph-of-op}

The moduli space $\overline{\mathcal{T}}_n$ endowed with its \Ainf -cell decomposition is isomorphic to the Loday realization $K_n$ of the associahedron. In fact, tedious computations show that under this isomorphism, the $\Omega B As$-decomposition is sent to the dual subdivision of $K_n$. See appendix C of~\cite{loday-vallette-algebraic-operads} and an illustration in figure~\ref{alg:fig:compact-mod-space-ombas} for instance. The goal of this section is to prove the following proposition : 
\begin{proposition}
The map $\ide : (\overline{\mathcal{T}}_n)_{\Ainf} \rightarrow (\overline{\mathcal{T}}_n)_{\Omega B As}$ is sent under the functor $C_{-*}^{cell}$ to the morphism of operads $\Ainf \rightarrow \Omega B As$ acting as
\[ m_n \longmapsto \sum_{t \in BRT_n} (t, \omega_{can}) \ . \]
\end{proposition}

For this purpose, we will work with the Loday realizations of the associahedra. We will show that taking the restriction of the orientation of $K_n$ chosen in section~\ref{alg:ss:loday-assoc} to the top dimensional cells of its dual subdivision yields the canonical orientations on these cells in the $\overline{\mathcal{T}}_n$ viewpoint.

We begin by proving this statement for the cell labeled by the right-leaning comb $t_{max}$. Consider the orientation on the cell $\overline{\mathcal{T}}_n(t_{max})$ induced by the canonical ordering $e_1 , \dots , e_{n-2}$ under the isomorphism
\[ \overline{\mathcal{T}}_n(t_{max}) \ \tilde{\longrightarrow} \ [0,+\infty]^{n-2} \ . \]
The face of $\overline{\mathcal{T}}_n(t_{max})$ associated to the breaking of the $i$-th edge corresponds to the face $H_{i,n-i,0}$ when seen in the Loday polytope. An outward-pointing vector for the face $H_{i,n-i,0}$ is moreover
\[ \nu_i := (0 , \dots , 0 , 1_{i} , \dots , 1_{n-2}) \ , \]
where coordinates are taken in the basis $e_j^\omega$.
The orientation defined by the canonical basis of $[0,+\infty]^{n-2}$ being exactly the one defined by the ordered list of the outwarding-point vectors to the $+ \infty$ boundary, it is sent to the orientation of the basis $(\nu_1,\dots,\nu_{n-2})$ in the Loday polytope. We then check that 
\[ \mathrm{det}_{e_j^\omega} ( \nu_j) = 1 \ . \]
Hence the orientation of $K_n$ and the one induced by the canonical orientation are the same for the cell $\overline{\mathcal{T}}_n(t_{max})$.

\begin{figure}[h]
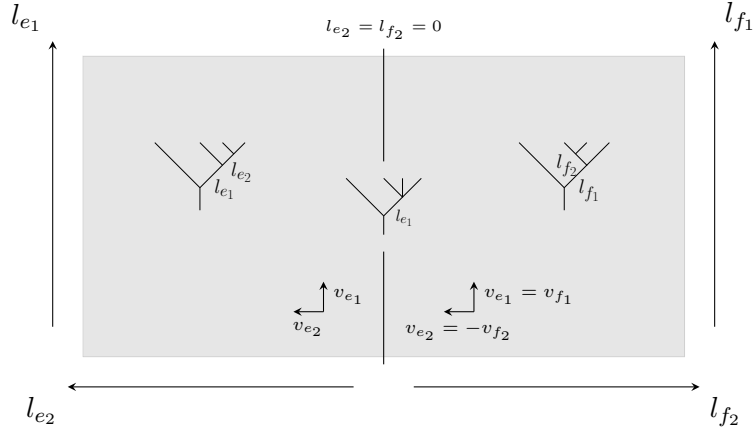
 
\centering
\coveringrelationsgeom
\caption{Gluing the cells $\overline{\mathcal{T}}_n(t_{max})$ and $\overline{\mathcal{T}}_n(t)$ along their common boundary : on this diagram, a vector of the form $v_e$ is the vector orienting the axis associated to the length $l_e$} \label{alg:fig:covering-rel}
\end{figure}

As explained in the previous subsection, the cells labeled by the immediate neighbours of the right-leaning comb $t_{max}$ in the Tamari order are exactly the cells having a codimension 1 stratum in common with this cell. Choose an immediate neighbour $t$, and write $e$ for the edge that has been collapsed to obtain the common codimension 1 stratum. We detail the process to obtain the induced orientation on $\overline{\mathcal{T}}_n(t)$ following figure~\ref{alg:fig:covering-rel}. Gluing the cells $\overline{\mathcal{T}}_n(t_{max})$ and $\overline{\mathcal{T}}_n(t)$ along their common boundary, we obtain a new copy of $[0,+\infty]^{n-2}$ which can be divided into two halves $t_{max}$ and $t$. We then orient the total space $[0,+\infty]^{n-2}$ as the $t_{max}$ half. Reading the induced orientation on the $t$ half, it is the one obtained from the $t_{max}$ half by reversing the axis associated to the edge $e$. By construction, this orientation is exactly the one obtained by restricting the global orientation on $K_n$ to an orientation on $\mathcal{T}_n(t)$.

Finally, going down the Tamari order, we can read the induced orientation on the top dimensional cells one immediate neighbour after another. And the rule to do this step-by-step process is exactly the one given in~\ref{alg:sss:canon-or-ribbon-tree} on the covering relations. Hence, by construction, the global orientation on $K_n$ restricts to the canonical orientations on binary trees, which concludes the proof of Proposition~\ref{alg:prop:markl-shnider-un}.

\subsection{The moduli spaces $\mathcal{CT}_n(t_{br,g})$} \label{alg:ss:mod-space-CTm}

We give a detailed definition of the moduli spaces of gauged stable metric ribbon trees $\mathcal{CT}_n(t_{g})$, introduced in part~\ref{alg:ss:multipl-two-col-metr-tree}. Building on these explicit realizations, we then thoroughly compute the signs appearing in the codimension 1 strata of the compactified moduli spaces $\overline{\mathcal{CT}}_n(t_{g})$. This yields in particular the signs which will appear in subsection~\ref{alg:sss:def-op-bimod-ombasmorph}, in the definition of the differential on the operadic bimodule $\Omega B As - \mathrm{Morph}$.

\subsubsection{Definition}\label{alg:sss:CTm-def}

In the rest of the section, we will write $t_{br,g}$ for a broken gauged stable ribbon tree, and $t_g$ for an unbroken gauged stable ribbon tree. 
\begin{definition}
We set \arbreopunmorph\ to be the unique stable gauged tree of arity 1, and will call it the \emph{trivial gauged tree}. We define the underlying broken stable ribbon tree $t_{br}$ of a $t_{br,g}$ to be the ribbon tree obtained by first deleting all the \arbreopunmorph\ in $t_{br,g}$, and then forgetting all the remaining gauges of $t_{br,g}$. We refer moreover to a gauge in $t_{br,g}$ which is associated to a non-trivial gauged tree, as a \emph{non-trivial gauge} of $t_{br,g}$. 
\end{definition}

\begin{figure}[h]
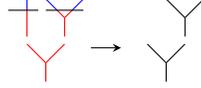
 
\exampleunderlyingbroken
\caption{An instance of association $t_{br,g} \mapsto t_{br}$}
\end{figure} \label{alg:fig:underlying-broken}

We now define the moduli spaces $\mathcal{CT}_n(t_{br,g})$ in three steps. Consider a gauged stable ribbon tree $t_g$ whose gauge does not intersect any of its vertices. Locally at any vertex directly adjacent to the gauge, the intersection between the gauge and the edges of $t$ corresponds to one of the following two cases
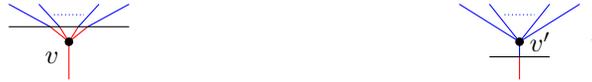
\begin{align*}
\localgaugeedgeA && \localgaugeedgeB \ .
\end{align*}
Write $r$ for the root, the unique vertex adjacent to the outgoing edge. For a vertex $v$, we denote $d(r,v)$ the distance separating it from the root : the sum of the lengths of the edges appearing in the unique non self-crossing path going from $r$ to $v$.  Associating lengths $l_e > 0$ to all edges of $t$, we then associate the following inequalities to the two above cases
\begin{align*}
- \lambda > d ( r,v) && - \lambda < d(r,v') \ .
\end{align*}
Note that this set of inequalities amounts to seeing the gauge as going towards $-\infty$ when going up, and towards $+\infty$ as going down.
The moduli space $\mathcal{CT}_n(t_{g})$ is then defined as 
\[ \mathcal{CT}_n(t_{g}) := \left\{ ( \lambda , \{ l_e \}_{e \in E(t)} ) \  , \ \lambda \in \R  , \ l_e > 0  , \ - \lambda > d ( r,v)  , \ - \lambda < d(r,v') \right\} \ , \]
where the set of inequalities on $\lambda$ is prescribed by the gauged tree $t_g$.

Consider now a gauged stable ribbon tree $t_g$ whose gauge may intersect some of its vertices. To the two previous local pictures, one has to add the case 
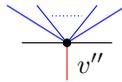
\[ \localgaugeedgeC  \]
to which we associate the equality
\[ - \lambda = d(r,v'') \ . \]
The moduli space $\mathcal{CT}_n(t_{g})$ is this time defined as 
\[ \mathcal{CT}_n(t_{g}) := \left\{ ( \lambda , \{ l_e \}_{e \in E(t)} ) \  , \ \lambda \in \R  , \ l_e > 0  , \ - \lambda > d ( r,v)  , \ - \lambda < d(r,v') , \ - \lambda = d(r,v'') \right\} \ , \]
where the set of equalities and inequalities on $\lambda$ is prescribed by the gauged tree $t_g$.

Finally, consider a gauged broken stable ribbon tree $t_{br,g}$, whose gauges may intersect some of its vertices. We order the non-trivial unbroken gauged ribbon trees appearing in $t_{br,g}$ from left to right, as 
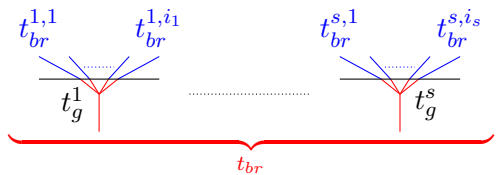
\[
\rouge{\underbrace{\ordreunbrokengaugedtrees}_{\rouge{t_{br}}}}
\] 
where $t_{br}^{1,1} , \dots , t_{br}^{1,i_1} , \dots , t_{br}^{s,1} , \dots , t_{br}^{s,i_s}$ and $t_{br}$ are broken stable ribbon trees, and the non-trivial unbroken gauged ribbon trees are represented in the picture as gauged corollae $t_g^1 , \dots , t_g^s$ for the sake of readability.
We write moreover $r_1, \dots , r_s$ and $\lambda_1 , \dots , \lambda_s$ for their respective roots and gauges.
The moduli space $\mathcal{CT}_n(t_{br,g})$ is this time defined as 
\[ \mathcal{CT}_n(t_{br,g}) := 
		\left\{ \begin{array}{c}
       ( \lambda_1 , \dots , \lambda_s , \{ l_e \}_{e \in E(t_{br})} ) \  , \ \lambda_i \in \R  , \ l_e > 0  , \\
  - \lambda_i > d ( r_i,v)  , \ - \lambda_i < d(r_i,v') , \ - \lambda_i = d(r_i,v'')
  \end{array} \right\} \ ,
\]
where the set of equalities and inequalities on $\lambda_i$ is prescribed by the unbroken gauged tree $t_g^i$.

\subsubsection{Orienting the moduli spaces $\mathcal{CT}_n(t_{br,g})$} \label{alg:sss:or-CTm}

\begin{definition}
Define an \emph{orientation} on a broken gauged stable ribbon tree $t_{br,g}$, to be an orientation $e_{1} \wedge \dots \wedge e_{i}$ on $t_{br}$. 
\end{definition}

We now explain how to orient the moduli spaces $\mathcal{CT}_n(t_{br,g})$, following the previous three steps approach.
Begin with a gauged stable ribbon tree $t_g$ whose gauge does not intersect any of its vertices. An orientation $\omega$ on $t_g$ identifies $\mathcal{CT}_n(t_{g})$ with a polyhedral cone 
\[ \mathcal{CT}_n(t_{g}) \subset ]- \infty , + \infty [ \times ] 0 , + \infty [^{e(t)} \ , \]
defined by the inequalities $- \lambda > d ( r,v)$ and $- \lambda < d(r,v')$. This polyhedral cone has dimension $e(t)+1$, and we choose to orient it as an open subset of $]- \infty , + \infty [ \times ] 0 , + \infty [^{e(t)}$ endowed with its canonical orientation.

Consider now a gauged stable ribbon tree $t_g$ whose gauge may intersect some of its vertices. This time, an orientation $\omega$ on $t_g$ identifies $\mathcal{CT}_n(t_{g})$ with a polyhedral cone 
\[ \mathcal{CT}_n(t_{g}) \subset ]- \infty , + \infty [ \times ] 0 , + \infty [^{e(t)} \ , \]
defined by the inequalities $- \lambda > d ( r,v)$ and $- \lambda < d(r,v')$, to which we add the equalities $- \lambda = d(r,v'')$. If there are exactly $j$ gauge-vertex intersections in the gauged tree $t_g$, this polyhedral cone has dimension $e(t)+1 - j$. Order now the $j$ intersections from left to right 
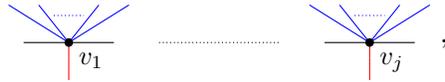
\[ \ordreunbrokengaugedtreeintersectionsA \ , \]
and consider the tree $t_g'$ obtained by replacing these intersections by 
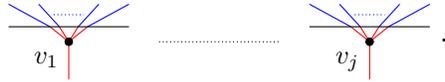
\[ \ordreunbrokengaugedtreeintersectionsB \ . \]
One can see $t_g$ as lying in the boundary of $t_g'$, by allowing the inequalities $- \lambda > d ( r,v_k)$ to become equalities $- \lambda = d ( r,v_k)$ for $k=1,\dots,j$. This determines in particular $j$ vectors $\nu_k$ corresponding to the outwarding-pointing vectors to the boundary of the half-space $- \lambda \geqslant d ( r,v_k)$. We finally choose to coorient (and hence orient) $\mathcal{CT}_n(t_{g} )$ inside $]- \infty , + \infty [ \times ] 0 , + \infty [^{e(t)}$ with the vectors $( \nu_1 , \dots , \nu_j)$. 

Lastly, consider a gauged broken stable ribbon tree $t_{br,g}$, whose gauges may intersect some of its vertices. Suppose there are exactly $s$ non-trivial unbroken gauged trees $t_g^1 , \dots , t_g^s$ appearing in $t_{br,g}$, which are ordered from left to right as previously. Suppose also that in each tree $t_g^i$, there are $j_i$ gauge-vertex intersections. An orientation $\omega$ on $t_{br,g}$ identifies $\mathcal{CT}_n(t_{br,g})$ with a polyhedral cone 
\[ \mathcal{CT}_n(t_{br,g}) \subset ]- \infty , + \infty [^s \times ] 0 , + \infty [^{e(t_{br})} \ , \]
defined by the set of equalities and inequalities on the $\lambda_i$, and where the factor $]- \infty , + \infty [^s$ corresponds to $(\lambda_1, \dots , \lambda_s)$. This polyhedral cone has dimension $e(t_{br}) + s - \sum_{i=1}^s j_i$. Now, as in the previous paragraph, order all gauge-vertex intersections from left to right in every tree $t_g^i$, and construct a new tree $t_{br,g}'$. Seeing $\mathcal{CT}_n(t_{br,g})$ as lying in the boundary of $\mathcal{CT}_n(t_{br,g}')$, this determines again a collection of outward-pointing vectors $\nu_{i,1} , \dots , \nu_{i,j_i}$ for $i= 1 , \dots , s$. We then coorient $\mathcal{CT}_n(t_{br,g} )$ inside $]- \infty , + \infty [^s \times ] 0 , + \infty [^{e(t_{br})}$ with the vectors $( \nu_{1,1} , \dots , \nu_{1,j_1} , \dots , \nu_{s,1} , \dots , \nu_{s,j_s})$.

\begin{definition}
We define $\mathcal{CT}_n(t_{br,g}, \omega )$ to be the moduli space $\mathcal{CT}_n(t_{br,g} )$ endowed with the previous orientation. 
\end{definition}

We moreover insist on the fact that for a given broken stable ribbon tree type $t_{br}$ all gauged trees $t_{br,g}$ whose underlying ribbon tree is $t_{br}$ form polyhedral cones $\subset ]- \infty , + \infty [^s \times ] 0 , + \infty [^{e(t_{br})}$, and the collection of these polyhedral cones is a partition of 
$ ]- \infty , + \infty [^s \times ] 0 , + \infty [^{e(t_{br})}$. This is illustrated in figure~\ref{alg:fig:decompo-polyedrale}.

\begin{figure}[h] 
\centering
\exampledecompositionpolyedrale
\caption{} \label{alg:fig:decompo-polyedrale}
\end{figure}

\subsubsection{Compactification} \label{alg:sss:compact-CTm}

Recall from section~\ref{alg:ss:multipl-two-col-metr-tree} that each broken gauged ribbon tree $t_{br,g}$ can be seen as a broken two-colored ribbon tree $t_{br,c}$. Using the two-colored metric trees viewpoint, the compactification of $\mathcal{CT}_n(t_{br,c})$ is defined by allowing lengths of internal edges to go towards $0$ or $+ \infty$, where combinatorics are induced by the equalities defined by the colored vertices. The compactification rule for gauged metric trees is then simply defined by transporting the compactification rule from the two-colored viewpoint to the gauged viewpoint. We do not give further details here, as we won't need them in our upcoming computations.

For a gauged stable ribbon tree $t_g$, the compactified moduli space $\overline{\mathcal{CT}}_n(t_{g})$ has codimension 1 strata given by the four components (int-collapse), (gauge-vertex), (above-break) and (below-break). Choose an orientation $\omega$ for $t_{g}$. As for the moduli spaces $\mathcal{T}_n(t,\omega)$, the question is now to determine which signs appear in the boundary of the compactification of the oriented moduli space $\mathcal{CT}_n(t_{g}, \omega )$. We will inspect this matter in the four upcoming sections, computing the signs for each boundary component. Note that this time the compactification is much more elaborate than the cubical compactification of the $\mathcal{T}_n(t, \omega )$, and as a result we will not be able to write nice and elegant formulae. We will rather give recipes to compute the signs in each case. 

\subsubsection{The (int-collapse) boundary component} \label{alg:sss:CTm-int-collapse-bound}

Consider a gauged stable ribbon tree $t_g$. The (int-collapse) boundary corresponds to the collapsing of an internal edge that does not intersect the gauge of the tree $t$. Choosing an ordering $\omega = e_1 \wedge \cdots \wedge e_i$, suppose that it is the $p$-th edge of $t$ which collapses. Write moreover $(t/e_p)_g$ for the resulting gauged tree, and $\omega_p := e_1 \wedge \cdots \wedge \widehat{e_p} \wedge \cdots \wedge e_i$ for the induced ordering on the edges of $t/e_p$.

We begin by considering the case of a gauged tree $t_g$ whose gauge does not intersect any of its vertices. Suppose first that the collapsing edge is located above the gauge. A neighbordhood of the boundary can then be parametrized as 
\begin{align*}
]-1,0] \times \mathcal{CT}_n((t/e_p)_{g}, \omega_p ) &\longrightarrow \overline{\mathcal{CT}}_n(t_{g}, \omega ) \\
(\delta , \lambda , l_1 , \dots , \widehat{l_p} , \dots , l_i ) &\longmapsto (\lambda , l_1 , \dots , l_p := - \delta , \dots , l_i ) \ .
\end{align*}
This map has sign $(-1)^{p+1}$, and the component $\mathcal{CT}_n((t/e_p)_{g}, \omega_p )$ consequently bears a $(-1)^{p+1}$ sign in the boundary of $\overline{\mathcal{CT}}_n(t_{g}, \omega )$.

Suppose next that the collapsing edge is located below the gauge. We define a parametrization of a neighborhood of the boundary 
\[ ]-1,0] \times \mathcal{CT}_n((t/e_p)_{g}, \omega_p ) \longrightarrow \overline{\mathcal{CT}}_n(t_{g}, \omega ) \]
as follows : $\lambda$ is sent to $\lambda + \delta$ ; if the edge $e_q$ is located directly below a gauge-edge intersection
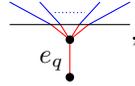
\[ \intersectiongaugeedgeedgecollaps , \]
then we send $l_q$ to $l_q - \delta$ ; for all the other edges $e_q$ of $(t/e_p)$, we send $l_q$ to $l_q$ ; finally, we set $l_p := -\delta$. We check again that this map has sign $(-1)^{p+1}$. Hence, in general, for a gauged tree $t_g$ whose gauge does not intersect any of its vertices, the component $\mathcal{CT}_n((t/e_p)_{g}, \omega_p )$ bears a $(-1)^{p+1}$ sign in the boundary of $\overline{\mathcal{CT}}_n(t_{g}, \omega )$.

Move on to the case of a gauged stable ribbon tree $t_g$ whose gauge may intersect some of its vertices. Order the $j$ gauge-vertex intersections from left to right as depicted in subsection~\ref{alg:sss:or-CTm}. We are going to distinguish three cases, but will eventually end up with the same sign in each case. Suppose to begin with that the collapsing edge $e_p$ is located above the gauge, and is not adjacent to a gauge-vertex intersection. Then, denoting $(t/e_p)_{g}'$ the tree obtained via the same process as $t_{g}'$, we check that the first parametrization introduced in this section 
\[ \Phi : \ ]-1,0] \times \mathcal{CT}_n((t/e_p)_{g}', \omega_p ) \longrightarrow \overline{\mathcal{CT}}_n(t_{g}', \omega ) \ , \]
restricts to a parametrization of a neighborhood of the boundary
\[ \phi : \ ]-1,0] \times \mathcal{CT}_n((t/e_p)_{g}, \omega_p ) \longrightarrow \overline{\mathcal{CT}}_n(t_{g}, \omega ) \ . \]
We also check that $\Phi$ sends the outward-pointing vectors $\nu_k^{(t/e_p)}$ associated to the gauge-vertex intersections in $(t/e_p)_g$, to the outward-pointing vectors $\nu_k^{t}$ associated to the gauge-vertex intersections in $t_g$. Computing the sign of $\phi$ amounts to computing the sign of $\Phi$ and then exchanging the direction $\delta$ with the outward-pointing vectors $\nu_1^{t} , \dots , \nu_j^{t}$. The total sign is hence $(-1)^{p+1+j}$.

Suppose, as second case, that the collapsing edge $e_p$ is located above the gauge, and directly adjacent to a gauge-vertex intersection.
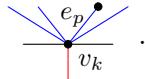
\[ \directlyadjacentgaugevertex \ . \]
We cannot use the trees $(t/e_p)_{g}'$ and $t_{g}'$ as in the last paragraph, as the gauge would then cut the edge $e_p$ in the gauged tree $t_{g}'$. A small change is required. We form the tree $t_{g}''$ as the tree $t_{g}'$, but instead of moving the gauge up at the vertex $v_k$, we move it down. The tree $(t/e_p)_{g}''$ is defined similarly. Applying the same argument as previously, we compute again a $(-1)^{p+1+j}$ sign for the boundary.

Finally, suppose that the collapsing edge $e_p$ is located below the gauge. It may this time be directly adjacent to a gauge-vertex intersection. Introducing again the trees $(t/e_p)_{g}'$ and $t_{g}'$, and using this time the second parametrization introduced in this section, we find a $(-1)^{p+1+j}$ sign for the boundary. Note that there is a small adjustment to make in the proof for the outward-pointing vectors. Indeed, the outward-pointing vector $\nu_k^{(t/e_p)}$ gets again sent to the outward-pointing vector $\nu_k^{t}$, except if the edge $e_p$ is located in the non-self crossing path going from the vertex $v_k$ intersected by the gauge to the root. For such an intersection, the vector $\nu_k^{(t/e_p)}$ is sent to $\nu_k^{t} - e_p$ by the map $\Phi$, where $e_p$ is the positive direction for the length $l_p$. Though the vector $\nu_k^{t} - e_p$ is  not equal to $\nu_k^{t}$, it is still outward-pointing to the half-space $- \lambda \geqslant d(r , v_k)$. As a result, $\Phi ( \nu_1^{(t/e_p)} ) , \dots , \Phi ( \nu_j^{(t/e_p)} )  $ defines indeed the same coorientation of $\mathcal{CT}_n(t_{g}, \omega )$ as $\nu_1^{t} , \dots , \nu_j^{t}$.

\begin{proposition} \label{alg:prop:int-collapse-signs}
For a gauged stable ribbon tree $t_g$ whose gauge intersects $j$ vertices, the boundary component $\mathcal{CT}_n((t/e_p)_{g}, \omega_p )$ corresponding to the collapsing of the $p$-th edge of $t$ bears a $(-1)^{p+1+j}$ sign in the boundary of $\overline{\mathcal{CT}}_n(t_{g}, \omega )$.
\end{proposition}

\subsubsection{The (gauge-vertex) boundary component} \label{alg:sss:CTm-gauge-vertex-bound}

Consider a gauged stable ribbon tree $t_g$ whose gauge may intersect some of its vertices. We order the gauge-vertex intersections from left to right as depicted in subsection~\ref{alg:sss:or-CTm}. The (gauge-vertex) boundary corresponds to the gauge crossing exactly one additional vertex of $t$. We suppose that this intersection takes place between the $k$-th and $k+1$-th intersections of $t_g$. We write moreover $t_g^{0}$ for the resulting gauged tree, and introduce again the tree $t_g'$ of subsection~\ref{alg:sss:or-CTm}.

\begin{proposition}
Suppose the crossing results from a move
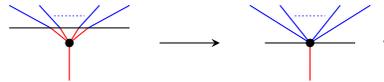
\[ \transformationgaugevertexA \ . \]
Then the boundary component $\mathcal{CT}_n(t^{0}_{g}, \omega )$ has sign $(-1)^{j+k}$ in the boundary of $\overline{\mathcal{CT}}_n(t_{g}, \omega )$.
\end{proposition}
Indeed the orientation induced on $\mathcal{CT}_n(t^{0}_{g}, \omega )$ in the boundary of $\overline{\mathcal{CT}}_n(t_{g}, \omega )$, is defined by the coorientation $(\nu_1 , \dots , \nu_k , \widehat{\nu} , \nu_{k+1} , \dots , \nu_j , \nu)$ inside $\mathcal{CT}_n(t_{g}', \omega )$. The orientation defined by $\omega$ on $\mathcal{CT}_n(t^{0}_{g}, \omega )$, is the one defined by the coorientation $(\nu_1 , \dots , \nu_k , \nu , \nu_{k+1} , \dots , \nu_j )$ inside $\mathcal{CT}_n(t_{g}', \omega )$. Hence, these two orientations differ by a $(-1)^{j+k}$ sign.

\begin{proposition}
Suppose the crossing results from a move
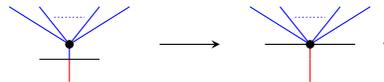
\[ \transformationgaugevertexB \ . \]
Then the boundary component $\mathcal{CT}_n(t^{0}_{g}, \omega )$ has sign $(-1)^{j+k+1}$ in the boundary of $\overline{\mathcal{CT}}_n(t_{g}, \omega )$.
\end{proposition}
Again the orientation induced on $\mathcal{CT}_n(t^{0}_{g}, \omega )$ in the boundary of $\overline{\mathcal{CT}}_n(t_{g}, \omega )$, is defined by the coorientation $(\nu_1 , \dots , \nu_k , \widehat{\nu} , \nu_{k+1} , \dots , \nu_j , - \nu)$ inside $\mathcal{CT}_n(t_{g}', \omega )$. The orientation defined by $\omega$ on $\mathcal{CT}_n(t^{0}_{g}, \omega )$, is the one defined by the coorientation $(\nu_1 , \dots , \nu_k , \nu , \nu_{k+1} , \dots , \nu_j )$ inside $\mathcal{CT}_n(t_{g}', \omega )$. Hence, these two orientations differ by a $(-1)^{j+k+1}$ sign. 

\subsubsection{The (above-break) boundary component} \label{alg:sss:CTm-above-break-bound}

The (above-break) boundary corresponds either to the breaking of an internal edge of $t$, that is located above the gauge or intersects the gauge, or, when the gauge is below the root, to the outgoing edge breaking between the gauge and the root. Choosing an ordering $\omega = e_1 \wedge \cdots \wedge e_i$, suppose that it is the $p$-th edge of $t$ which breaks and write moreover $(t_p)_g$ for the resulting broken gauged tree.

We begin by considering the case of a gauged tree $t_g$ whose gauge does not intersect any of its vertices. Suppose first that the breaking edge does not intersect the gauge. A neighborhood of the boundary can then be parametrized as 
\begin{align*}
]0 , + \infty] \times \mathcal{CT}_n((t_p)_{g}, \omega_p ) &\longrightarrow \overline{\mathcal{CT}}_n(t_{g}, \omega ) \\
(\delta , \lambda , l_1 , \dots , \widehat{l_p} , \dots , l_i ) &\longmapsto (\lambda , l_1 , \dots , l_p := \delta , \dots , l_i ) \ .
\end{align*}
This map has sign $(-1)^{p}$. In the case when the breaking edge does intersect the gauge, a neighbordhood of the boundary can be parametrized as 
\begin{align*}
]0 , + \infty] \times \mathcal{CT}_n((t_p)_{g}, \omega_p ) &\longrightarrow \overline{\mathcal{CT}}_n(t_{g}, \omega ) \\
(\delta , \lambda , l_1 , \dots , \widehat{l_p} , \dots , l_i ) &\longmapsto (\lambda , l_1 , \dots , l_p := \delta - \lambda , \dots , l_i ) \ ,
\end{align*} 
where we set this time $l_p := \delta - \lambda$ in order for the inequality $- \lambda < d ( r,v')$ to hold in this case. This parametrization again has sign $(-1)^{p}$.

The case of a gauged tree $t_g$ whose gauge may intersect some of its vertices is treated as in subsection~\ref{alg:sss:CTm-int-collapse-bound}. We check again that the parametrization maps $\Phi$ introduced in the previous paragraph, restrict to parametrizations of a neighborhood of the boundary
\[ ]0 , + \infty] \times \mathcal{CT}_n((t_p)_{g}, \omega_p ) \longrightarrow \overline{\mathcal{CT}}_n(t_{g}, \omega ) \ , \]
and that $\Phi$ sends moreover the coorientation of $\mathcal{CT}_n((t_p)_{g}, \omega_p )$ to the coorientation of $\mathcal{CT}_n(t_{g}, \omega )$. These coorientations introduce as previously an additional $(-1)^j$ sign.

Finally, suppose that the gauge of $t_g$ intersects its outgoing edge and compute the sign of the (above-break) boundary component corresponding to the gauge going towards $+ \infty$. A parametrization of a neighborhood of the boundary is simply given by 
\begin{align*}
]0 , + \infty] \times \mathcal{CT}_n((t_0)_{g}, \omega_p ) &\longrightarrow \overline{\mathcal{CT}}_n(t_{g}, \omega ) \\
(\delta , l_1 , \dots , l_i ) &\longmapsto (\lambda := \delta , l_1 , \dots , l_i ) \ .
\end{align*} 
This map has sign $1$.

\begin{proposition}
For a gauged stable ribbon tree $t_g$ whose gauge intersects $j$ vertices, the boundary component $\mathcal{CT}_n((t_p)_{g}, \omega_p )$ corresponding to the breaking of the $p$-th edge of $t$ bears a $(-1)^{p+j}$ sign in the boundary of $\overline{\mathcal{CT}}_n(t_{g}, \omega )$, where we set $e_0$ for the outgoing edge of $t$.
\end{proposition}

\subsubsection{The (below-break) boundary component} \label{alg:sss:CTm-below-break-bound}

The (below-break) boundary corresponds to the breaking of edges of $t$ that are located below the gauge or intersect it, such that there is exactly one edge breaking in each non-self crossing path from an incoming edge to the root.  Write $(t_{br})_g$ for the resulting broken gauged tree. Consider now an ordering $\omega = e_1 \wedge \cdots \wedge e_i$ of $t_g$. We order again from left to right the $s$ non-trivial unbroken gauged trees $t_g^1 , \dots , t_g^s$ of $(t_{br})_g$, and denote moreover $e_{j_1} , \dots , e_{j_s}$ the internal edges of $t$ whose breaking produce the trees $t_g^1 , \dots , t_g^s$. Beware that we do not necessarily have that $j_1 < \cdots < j_s$.  We assume in the next paragraphs that $j_1 = 1 , \dots , j_s = s$, and will explain how to deal with the general case at the end of this section. We set to this extent $\omega_{br} := e_{s+1} \wedge \cdots \wedge e_{i}$.

We introduce two more pieces of notation. We will denote $\mathcal{E}_\infty$ the set of incoming edges of $t$ which are crossed by the gauge and correspond to the trivial gauged trees in $(t_{br})_g$. In other words, the set of edges which are breaking in the (below-break) boundary component associated to $(t_{br})_g$ is $\mathcal{E}_\infty \cup \{ e_{j_1} , \dots , e_{j_s} \}$. For an edge $e$, internal or external, we will moreover write $w_e$ for the vertex adjacent to $e$ which is closest to the root $r$ of $t$, and set $w_u := w_{e_u}$ for $u = 1 , \dots , s$.

Start by considering the case of a gauged tree $t_g$ whose gauge does not intersect any of its vertices. Suppose first that among the breaking internal edges, none of them intersects the gauge. We define a parametrization of a neighbourhood of the boundary 
\begin{align*}
]0 , + \infty] \times \mathcal{CT}_n((t_{br})_{g}, \omega_{br} ) &\longrightarrow \overline{\mathcal{CT}}_n(t_{g}, \omega )
\end{align*}
by sending $( \delta , \lambda_1 , \dots , \lambda_s , l_{s+1} , \dots , l_i)$ to the element of $\mathcal{CT}_n(t_{g}, \omega)$ whose entries are defined as 
\begin{align*}
\lambda &:= - \delta + \sum_{u=1}^s \left( \lambda_u - d(r,w_u) \right) - \sum_{e \in \mathcal{E}_\infty} d(r,w_e)  \ , \\
l_v &:= \delta + \sum_{\substack{u = 1 , \dots , s \\ u \neq v}} \left( - \lambda_u + d(r , w_u) \right) + \sum_{e \in \mathcal{E}_\infty} d(r,w_e) \  &&\text{for $v = e_1 , \dots , e_s$} \ , \\
l_k &:= l_k \  &&\text{for $k = s+1 , \dots , i$} \ .
\end{align*}
We compute that this map has sign $-1$. 

Suppose now that among the breaking internal edges of $t_g$, some of them may intersect the gauge. We denote $\mathcal{N}_{\cap} \subset \{ 1 , \dots , s \}$ for the set of indices corresponding to the breaking internal edges which intersect the gauge, and $\mathcal{N}_{\emptyset} \subset \{ 1 , \dots , s \}$ for the set of indices corresponding to the breaking of internal edges which do not intersect the gauge. We define this time a parametrization of a neighbourhood of the boundary
\begin{align*}
]0 , + \infty] \times \mathcal{CT}_n((t_{br})_{g}, \omega_{br} ) &\longrightarrow \overline{\mathcal{CT}}_n(t_{g}, \omega )
\end{align*}
by sending $( \delta , \lambda_1 , \dots , \lambda_s , l_{s+1} , \dots , l_i)$ to the element of $\mathcal{CT}_n(t_{g}, \omega)$ whose entries are set to be 
\begin{align*}
\lambda &:= - \delta + \sum_{u \in \mathcal{N}_{\emptyset}} \left( \lambda_u - d(r,w_u) \right) - \sum_{u \in \mathcal{N}_{\cap}} d(r,w_u) - \sum_{e \in \mathcal{E}_\infty} d(r,w_e)  \ , \\
l_v &:= \delta + \sum_{\substack{u \in \mathcal{N}_{\emptyset} \\ u \neq v}} \left( - \lambda_u + d(r , w_u) \right) + \sum_{u \in \mathcal{N}_{\cap}} d(r,w_u) + \sum_{e \in \mathcal{E}_\infty} d(r,w_e) \  &&\text{for $v \in \mathcal{N}_{\emptyset}$} \ , \\
l_v &:= \delta + \lambda_v + \sum_{u \in \mathcal{N}_{\emptyset}} \left( - \lambda_u + d(r , w_u) \right) + \sum_{\substack{u \in \mathcal{N}_{\cap} \\ u \neq v}} d(r,w_u) + \sum_{e \in \mathcal{E}_\infty} d(r,w_e) \  &&\text{for $v \in \mathcal{N}_{\cap}$} \ , \\
l_k &:= l_k \  &&\text{for $k = s+1 , \dots , i$} \ .
\end{align*}
We compute that this map has again sign $-1$. 

Consider now the case of a gauged tree $t_g$ whose gauge intersects $j$ of its vertices. We check as in the previous sections that the parametrization maps introduced in the previous paragraphs, restrict to parametrizations of a neighborhood of the boundary
\[ ]0 , + \infty] \times \mathcal{CT}_n((t_{br})_{g}, \omega_{br} ) \longrightarrow \overline{\mathcal{CT}}_n(t_{g}, \omega ) \ , \]
and that these maps send moreover the coorientation of $\mathcal{CT}_n((t_{br})_{g}, \omega_{br} )$ to the coorientation of $\mathcal{CT}_n(t_{g}, \omega )$. These coorientations introduce an additional $(-1)^j$ sign.

We have thus computed the sign of the (below-break) boundary when $j_1 = 1 , \dots , j_s = s$. Now, consider the general case where we dot no necessarily have that $j_1 = 1 , \dots , j_s = s$. We denote $\varepsilon ( j_1 , \dots , j_s  ; \omega) $ the sign obtained after modifying $\omega$ by moving $e_{j_k}$ to the $k$-th spot in $\omega$, and write $\omega_0$ for the newly obtained orientation on $t_g$. Twisting the orientation on $\mathcal{CT}_n(t_{g}, \omega )$ by $(-1)^{\varepsilon ( j_1 , \dots , j_s ; \omega)}$ amounts to identifying it with $\mathcal{CT}_n(t_{g}, \omega_0 )$. We can apply the previous constructions and find the desired sign for the associated (below-break) component. 

\begin{proposition} \label{alg:prop:below-break-signs}
For a gauged stable ribbon tree $t_g$ whose gauge intersects $j$ vertices, the boundary component $\mathcal{CT}_n((t_{br})_{g}, \omega_{br} )$ corresponding to the breaking of the internal edges $e_{j_1} , \dots , e_{j_s}$ of $t$ bears a $(-1)^{\varepsilon ( j_1 , \dots , j_s ; \omega) + 1 + j}$ sign in the boundary of $\overline{\mathcal{CT}}_n(t_{g}, \omega )$.
\end{proposition}

\subsection{The operadic bimodule $\Omega B As - \mathrm{Morph}$} \label{alg:ss:op-bimod-ombasmorph}

\subsubsection{Definition of the operadic bimodule $\Omega B As - \mathrm{Morph}$} \label{alg:sss:def-op-bimod-ombasmorph}

We choose to define the operadic bimodule $\Omega B As - \mathrm{Morph}$ with the formalism of orientations on gauged trees, so that it be compatible with the definition of Markl-Shnider for the operad $\Omega B As$. As before, $t_{br,g}$ will stand for a broken gauged stable ribbon tree, while $t_g$ will denote an unbroken gauged stable ribbon tree. We also respectively write $t_{br}$ and $t$ for the underlying stable ribbon trees.

\begin{definition}[Spaces of operations and action-composition maps]
Consider the \Z -module freely generated by the pairs $(t_{br,g},\omega)$. We define the arity $n$ space of operations $\Omega B As - \mathrm{Morph}(n)_*$ to be the quotient of this \Z -module under the relation
\[ (t_{br,g} , - \omega) = - (t_{br,g} , \omega) \ . \]
An element $(t_{br,g},\omega)$ where $t_{br,g}$ has $e(t_{br})$ finite internal edges and $g$ non-trivial gauges which intersect $j$ vertices of $t_{br}$ is defined to have degree $j-(e(t_{br})+g)$.
The operad $\Omega B As$ then acts on $\Omega B As -\mathrm{Morph}$ as follows
\begin{align*}
(t_{br,g},\omega) \circ_i (t_{br}',\omega ') &= (t_{br,g} \circ_i t_{br}' , \omega \wedge \omega ') \ , \\
\mu ((t_{br},\omega),(t_{br,g}^1,\omega_1),\dots,(t_{br,g}^s,\omega_s)) &= (-1)^{\dagger} ( \mu (t_{br} , t_{br,g}^1 \dots , t_{br,g}^s) , \omega \wedge \omega_1 \wedge \cdots \wedge \omega_s ) \ ,
\end{align*}
where the tree $t_{br,g} \circ_i t_{br}'$ is the gauged broken ribbon tree obtained by grafting $t_{br}'$ to the $i$-th incoming edge of $t_{br,g}$ and $\mu (t_{br} , t_{br,g}^1 \dots , t_{br,g}^s)$ is the gauged broken ribbon tree defined by grafting each $t_{br,g}^j$ to the $j$-th incoming edge of $t_{br}$. Writing $g_i$ for the number of non-trivial gauges and $j_i$ for the number of gauge-vertex intersections of $t_{br,g}^i$, $i= 1 , \dots , s$, and setting $t_{br}^0 := t_{br}$ and $g_0 = j_0 = 0$, 
\[ \dagger := \sum_{i=1}^s g_i  \sum_{l=0}^{i-1} e(t_{br}^l)  + \sum_{i=1}^s j_i \sum_{l=0}^{i-1} (e (t_{br}^l) + g_l - j_l ) \ , \]
or equivalently
\[ \dagger = \sum_{i=1}^s g_i  \left( |t_{br}| + \sum_{l=1}^{i-1} |t_{br}^l | \right)  + \sum_{i=1}^s j_i \left( |t_{br}| + \sum_{l=1}^{i-1} |t_{br,g}^l| \right) \ . \]
\end{definition}

Choosing a distinguished orientation for every gauged stable ribbon tree $t_g \in SCRT$, this definition of the operadic bimodule $\Omega B As -\mathrm{Morph}$ amounts to defining it as the free operadic bimodule in graded \Z -modules
\[ \mathcal{F}^{\Omega B As, \Omega B As}( \arbreopunmorph , \arbrebicoloreL , \arbrebicoloreM , \arbrebicoloreN , \cdots , SCRT_n,\cdots) \ .  \]
It remains to define a differential on the generating operations $(t_g , \omega )$ to recover the definition given in subsection~\ref{alg:sss:op-bimod-ombas}. 

\begin{definition}[Differential]
The differential of a gauged stable ribbon tree $(t_{g},\omega)$ is defined as the signed sum of all codimension~1 contributions
\[ \resizebox{\hsize}{!}{$\displaystyle{ \partial (t_{g},\omega) = \sum \pm (int-collapse) + \sum \pm (gauge-vertex) + \sum \pm (above-break) + \sum \pm (below-break) \ ,}$} \] 
where the signs are as computed in Propositions~\ref{alg:prop:int-collapse-signs}~to~\ref{alg:prop:below-break-signs}. 
\end{definition}

For instance, choosing the ordering $e_1 \wedge e_2$ on 
\[ \orderingarbrebicoloresignesA \ , \] the signs in the computation of subsection~\ref{alg:sss:op-bimod-ombas} are
\begin{align*}
 \partial \left( \arbrebicoloresignesA , e_1 \wedge e_2 \right) = &\left( \arbrebicoloresignesF , e_1 \wedge e_2 \right) - \left( \arbrebicoloresignesD , e_1 \wedge e_2 \right) - \left( \arbrebicoloresignesE , e_1 \wedge e_2 \right) \\
 &+ \left( \arbrebicoloresignesI , e_1 \right) - \left( \arbrebicoloresignesJ , e_2 \right) - \left( \arbrebicoloresignesK , \emptyset \right) \ .
\end{align*}

\subsubsection{The moduli spaces $\overline{\mathcal{CT}}_n$ realize the operadic bimodule $\Omega B As - \mathrm{Morph}$} \label{alg:sss:mod-space-CTn-realize}

We only have to check that the signs for the action-composition maps of $\Omega B As - \mathrm{Morph}$ are indeed the ones determined by the moduli spaces $\overline{\mathcal{CT}}_n$, to conclude that the moduli spaces $\overline{\mathcal{CT}}_n$ endowed with their fine cell decomposition realize the operadic bimodule $\Omega B As - \mathrm{Morph}$ under the functor $C_{-*}^{cell}$.

The computation for $\circ_i$ is straighforward. Consider now the map
\begin{align*}
\mu : \mathcal{T}(t_{br},\omega) \times \mathcal{CT}(t_{br,g}^1,\omega_1) \times \cdots \times \mathcal{CT} (t_{br,g}^s,\omega_s) &\longrightarrow \mathcal{CT} ( \mu (t_{br} , t_{br,g}^1 \dots , t_{br,g}^s) , \omega \wedge \omega_1 \wedge \cdots \wedge \omega_s ) \\
\left( L_\omega , (\Lambda_1 , L_{\omega_1}) , \dots , (\Lambda_s , L_{\omega_s}) \right) &\longmapsto \left( \Lambda_1 , \dots , \Lambda_s , L_{\omega} , L_{\omega_1} , \dots , L_{\omega_s} \right) \ ,
\end{align*}
where $L_{\omega_i}$ stands for the list of lengths of $t_{br}^i$ according to the ordering $\omega_i$, and $\Lambda _i := (\lambda_{i,1} , \dots , \lambda_{i,g_i})$ stands for the list of non-trivial gauges of $t_{br,g}^i$. We compute that, in the absence of gauge vertex intersections, this map has sign 
\[ (-1)^{\sum_{i=1}^s g_i  \sum_{l=0}^{i-1} e(t_{br}^l)} \ . \]
Assuming that there are some gauge-vertex intersections, the combinatorics of coorientations introduce an additional sign
\[ (-1)^{\sum_{i=1}^s j_i \sum_{l=0}^{i-1} (e (t_{br}^l) + g_l - j_l )} \ . \]
In total, we recover the sign $(-1)^{\dagger}$, which concludes the proof.

\subsubsection{Canonical orientations for the gauged binary ribbon trees} \label{alg:sss:can-or-gauged}

For a fixed $n \geqslant 2$, the set of gauged binary ribbon trees $CBRT_n$ can be endowed with a partial order, inspired by the Tamari order on $BRT_n$. It is introduced in~\cite{masuda-diagonal-multipl}. 

\begin{definition}[\cite{masuda-diagonal-multipl}]
The \emph{Tamari order on $CBRT_n$} is the partial order generated by the covering relations
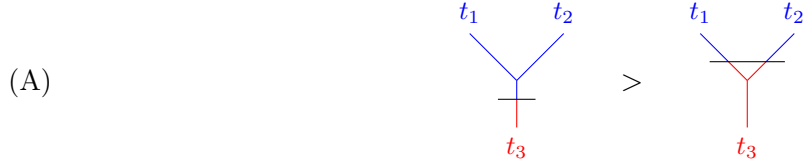
\begin{align*}
\coveringrelationAgaugedun \ \ \ > \ \ \ \coveringrelationBgaugedun \tag{A}
\end{align*} 
where $t_1$, $t_2$ and $t_3$ are binary ribbon trees,
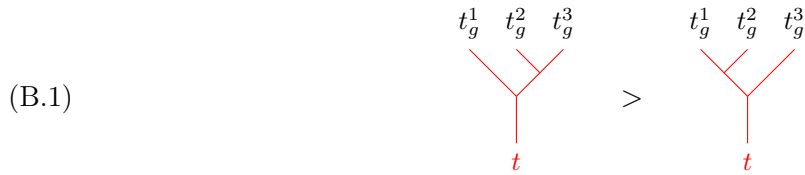
\begin{align*}
\coveringrelationAgaugeddeux \ \ \ > \ \ \ \coveringrelationBgaugeddeux \tag{B.1}
\end{align*} 
where $t_g^1$, $t_g^2$, $t_g^3$ are gauged binary ribbon trees and $t$ is a binary ribbon tree, and
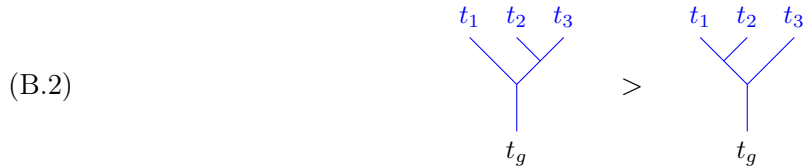
\begin{align*}
\coveringrelationAgaugedtrois \ \ \ > \ \ \ \coveringrelationBgaugedtrois \tag{B.2}
\end{align*} 
where $t_1$, $t_2$, $t_3$ are binary ribbon trees and $t_g$ is a gauged binary ribbon tree.
\end{definition}

For example in the case of $CBRT_4$, we obtain the Hasse diagram in figure~\ref{alg:fig:tamari-poset-gauged}. This Tamari-like poset has a unique maximal element and a unique minimal element, respectively given by the right-leaning comb whose gauge intersects the outgoing edge, and the left-leaning comb whose gauge intersects all incoming edges.

\begin{figure}[h]
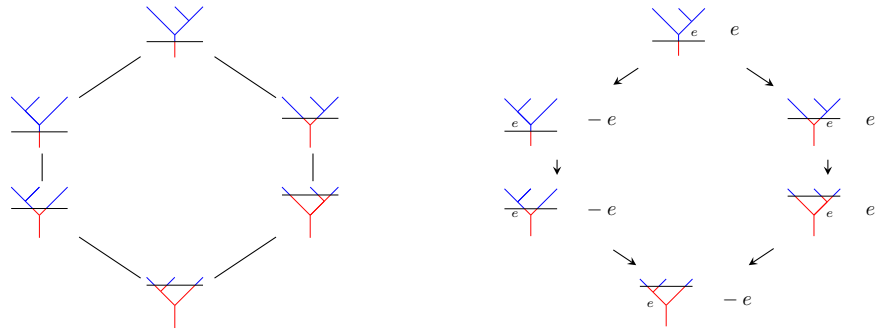
 
\centering
\begin{subfigure}{0.4\textwidth}
	\centering
	\TamariquatreCT
\end{subfigure} 
\begin{subfigure}{0.4\textwidth}
	\centering
	\TamariorientationsquatreCT
\end{subfigure}
\caption{On the left, the Hasse diagram of the poset $CBRT_3$, where the maximal element is written at the top. On the right, all the canonical orientations for $CBRT_3$ computed going down the poset.} \label{alg:fig:tamari-poset-gauged}
\end{figure} 

The canonical orientation on the maximal gauged binary tree is defined as
\[ \rightleaningcombgauged \ \ \omega_{can} := e_1 \wedge \cdots \wedge e_{n-2}  \ . \]
Using this Tamari-like order, we can now build inductively canonical orientations on all gauged binary trees. We start at the maximal gauged binary tree, and transport the orientation $\omega_{can}$ to its immediate neighbours as follows : the immediate neighbours of $t_g^{max}$ obtained under the covering relation (A) are endowed with the orientation $\omega_{can}$, while the ones obtained under the covering relations (B) are endowed with the orientation $-\omega_{can}$. We then repeat this operation while going down the poset until the minimal gauged binary tree is reached. This process is consistent (see next section), i.e. it does not depend on the path taken in the poset from $t_g^{max}$ to the gauged binary tree whose orientation is being defined. A full example for $CBRT_3$ is illustrated in figure~\ref{alg:fig:tamari-poset-gauged}.

\begin{definition}
The such obtained orientations will again be called the \emph{canonical orientations} and written $\omega_{can}$. They coincide in fact with the canonical orientations on the underlying binary trees. 
\end{definition}

\subsubsection{MacLane's coherence} \label{alg:sss:maclane-coherence}

We stated in subsections~\ref{alg:sss:canon-or-ribbon-tree}~and~\ref{alg:sss:can-or-gauged} that our process of transforming orientations is consistent, i.e. it does not depend on the path taken in the Tamari poset from the maximal tree to the tree whose orientation is being defined. In fact, our rules to transform orientations under the covering relations enable us to transport the orientation $\omega$ of any (gauged) tree $t_{(g)}$ to any (gauged) tree $t'_{(g)}$, along a path in the Tamari poset. The following result then holds~: for a given oriented (gauged) tree $(t_{(g)},\omega)$, any two paths in the Tamari poset from $t_{(g)}$ to $t_{(g)}'$ yield the same orientation on $t_{(g)}'$.

As pointed out by Markl and Shnider in~\cite{markl-assoc}, an adaptation of the proof of MacLane's coherence theorem shows that it is enough to prove that the diagram described by $K_4$ commutes to conclude that this statement holds for $BRT_n$. And this is the case as shown in figure~\ref{alg:fig:tamari-poset}. In the case of $CBRT_n$, an adaptation of these arguments shows this time that it is enough to prove that the diagrams described by $K_4$ and $J_3$ commute in order to conclude. This is again the case. 

A conceptual explanation for these two "coherence theorems" can be given as follows. In the case of $BRT_n$, a path between two trees $t$ and $t'$ in the Tamari poset corresponds to a path in the 1-skeleton of $K_n$. The faces of the 2-skeleton of $K_n$ consist moreover of the products
\begin{align*}
&K_2 \times \cdots \times K_2 \times K_3 \times K_2 \times \cdots \times K_2 \times K_3 \times K_2 \times \cdots \times K_2 \ , \\
&K_2 \times \cdots \times K_2 \times K_4 \times K_2 \times \cdots \times K_2 \ .
\end{align*}
The first type of face corresponds to a square diagram that tautologically commutes, while the second type of face corresponds to the $K_4$ diagram. Given now two paths from $t$ to $t'$, they delineate a family of faces in the 2-skeleton of $K_n$. Translating this into algebra, as all faces translate into commuting diagrams, the two paths produce the same orientation.

\subsubsection{The morphism of operadic bimodules $\infmor \rightarrow \Omega B As - \mathrm{Morph}$} \label{alg:sss:morph-op-bimod-ainf-ombasmorph}

The moduli space $\overline{\mathcal{CT}}_n$ endowed with its \Ainf -cell decomposition is isomorphic to the Forcey-Loday realization $J_n$ of the multiplihedron. Forcey shows in~\cite{forcey-multipl} that under this isomorphism, the $\Omega B As$-decomposition is sent to the dual subdivision of $J_n$. This is illustrated on figure~\ref{alg:fig:compact-mod-space-CT-ombas} for instance. The goal of this section is again to show that :
\begin{proposition}
The map $\ide : (\overline{\mathcal{CT}}_n)_{\Ainf} \rightarrow (\overline{\mathcal{CT}}_n)_{\Omega B As}$ is sent under the functor $C_{-*}^{cell}$ to the morphism of operadic bimodules $\infmor \rightarrow \Omega B As - \mathrm{Morph}$ acting as
\[ f_n \longmapsto \sum_{t_g \in CBRT_n} (t_g, \omega_{can}) \ . \]
\end{proposition}

We prove that taking the restriction of the orientation of $J_n$ chosen in section~\ref{alg:ss:forcey-loday-multipl} to the top dimensional cells of its dual subdivision, yields the canonical orientations on these cells in the $\overline{\mathcal{CT}}_n$ viewpoint. We follow in this regard the exact same line of proof as in subsection~\ref{alg:sss:morph-of-op}.

This statement is at first shown for the maximal gauged binary tree $t^{max}_g$, the right-leaning comb whose gauge crosses the outgoing edge. The orientation on the cell $\overline{\mathcal{CT}}_n(t_g^{max})$ induced by the canonical orientation $e_1 \wedge \cdots \wedge e_{n-2}$ defines an isomorphism
\[ \overline{\mathcal{CT}}_n(t_g^{max}) \ \tilde{\longrightarrow} \ [0,+\infty] \times [0,+\infty]^{n-2} \ , \]
where the factor $[0,+\infty]$ corresponds to the gauge $\lambda$, and the factor $[0,+\infty]^{n-2}$ to the lengths of the inner edges. 
The face of $\overline{\mathcal{CT}}_n(t_g^{max})$ associated to the gauge going to $+\infty$ corresponds to the face $H_{0,n,0}$ when seen in the Forcey-Loday polytope, while the face associated to the  breaking of the $i$-th edge corresponds to the face $H_{i,n-i,0}$. An outward-pointing vector for the face $H_{i,n-i,0}$ is moreover
\[ \nu_i := (0 , \dots , 0 , 1_{i+1} , \dots , 1_{n-1}) \ , \]
where coordinates are taken in the basis $f_j^\omega$.
The orientation defined by the canonical basis of $[0,+\infty] \times [0,+\infty]^{n-2}$ is exactly the one defined by the ordered list of the outward-pointing vectors to the $+\infty$ boundary. This orientation is thus sent to the orientation defined by the basis $(\nu_0,\dots,\nu_{n-2})$ in the Forcey-Loday polytope. It remains to check that 
\[ \mathrm{det}_{f_j^\omega} ( \nu_j) = 1 \ . \]
As a result, the orientation induced by $J_n$ and the one defined by the canonical orientation coincide for the cell $\overline{\mathcal{CT}}_n(t_g^{max})$.

The rest of the proof is a mere adaptation of the proof of subsection~\ref{alg:sss:morph-of-op}. The cells labeled by the gauged binary trees which are immediate neighbours of the maximal gauged binary tree, are exactly the ones having a codimension 1 stratum in common with $\overline{\mathcal{CT}}_n(t_g^{max})$. Choosing one such tree $t_g$, and gluing the cells $\overline{\mathcal{CT}}_n(t_g)$ and $\overline{\mathcal{CT}}_n(t_g^{max})$ along their common boundary, one can read the induced orientation on $\overline{\mathcal{CT}}_n(t_g)$. In the case when the immediate neighbour $t_g$ is obtained under the covering relation (A), the cells $\mathcal{CT}_n(t_g)$ and $\mathcal{CT}_n(t_g^{max})$ are in fact both oriented as subspaces of $]-\infty,+\infty[ \times ]0,+\infty[^{n-2}$.  In the case when the immediate neighbour $t_g$ is obtained under the covering relations (B), we send the reader back to subsection~\ref{alg:sss:morph-of-op} for explanations on why a $-1$ twist of the orientation has to be introduced. In each case, the induced orientation is exactly the canonical orientation on $\overline{\mathcal{CT}}_n(t_g)$. This argument can now be repeated going down the poset, and the induced orientation will always coincide with the canonical orientation on the cell, which concludes the proof of Proposition~\ref{alg:prop:markl-shnider-deux}.

\newpage

\begin{leftbar}
\part{Geometry} \label{p:geometry}
\end{leftbar}

\setcounter{section}{0}

\section{\Ainf\ and $\Omega B As$-algebra structures on the Morse cochains} \label{geo:s:ainf-ombas-alg-Morse}

Let $M$ be an oriented closed Riemannian manifold endowed with a Morse function $f$ together with a Morse-Smale metric. Following~\cite{hutchings-floer}, the Morse cochains $C^*(f)$ form a deformation retract of the singular cochains on $M$. The cup product naturally endows the singular cochains $C^*_{sing}(M)$ with a dg-algebra structure. The homotopy transfer theorem then ensures that it can be transferred to an \Ainf -algebra structure on the Morse cochains $C^*(f)$. The following question then naturally arises. The differential on the Morse cochains is defined by a count of moduli spaces of gradient trajectories connecting critical points of $f$. Is it possible to define higher multiplications $m_n$ on $C^*(f)$ by a count of moduli spaces such that they fit in a structure of \Ainf -algebra ?

We have seen in the previous part that the polytopes encoding the operad \Ainf\ are the associahedra and that they can be realized as the compactified moduli spaces of stable metric ribbon trees. A natural candidate would thus be an interpretation of metric ribbon trees in Morse theory. A naive approach would be to define trees each edge of which corresponds to a Morse gradient trajectory as in figure~\ref{geo:fig:arbre-gradient-exemple}. These moduli spaces are however not well defined, as two trajectories coming from two distinct critical points cannot intersect. A second problem is that moduli spaces of trajectories issued from the same critical point do not intersect transversely. In his article~\cite{abouzaid-plumbings}, Abouzaid bypasses this problem by perturbing the equation around each vertex, so that a transverse intersection can be achieved. See also~\cite{mescher-morse}. This is illustrated in figure~\ref{geo:fig:arbre-gradient-exemple}. 

\begin{figure}[h] 
\centering
\arbredegradientexemple 
\caption{} \label{geo:fig:arbre-gradient-exemple}
\end{figure}

Trees obtained in this way will be called \emph{perturbed Morse gradient trees}. Let $t$ be a stable ribbon tree type and $y,x_1,\dots,x_n$ a collection of critical points of the Morse function $f$. We prove in this section that for a generic choice of perturbation data $\mathbb{X}_t$ on the moduli space $\mathcal{T}_n(t)$, the moduli space of perturbed Morse gradient trees modeled on $t$ and connecting $x_1 , \dots ,x_n$ to $y$, denoted $\mathcal{T}_t(y ; x_1,\dots,x_n)$, is an orientable manifold (Proposition~\ref{geo:prop:trois-items-ombas-alg}). Under some additional generic assumptions on the choices of perturbation data $\mathbb{X}_t$, these moduli spaces are compact in the 0-dimensional case, and can be compactified to compact manifolds with boundary in the 1-dimensional case (Theorems~\ref{geo:th:exist-admissible-perturbation-data-Tn}~and~\ref{geo:th:compactification-Tn}). We are finally able to define operations on the Morse cochains $C^*(f)$ by counting the 0-dimensional moduli spaces of Morse gradient trees : these operations define an \ombas -algebra structure on $C^*(f)$ (Theorem~\ref{geo:th:ombas-alg}). Our constructions are carried out using the formalism introduced in~\cite{abouzaid-plumbings} and some terminology of~\cite{mescher-morse}. Technical details are moreover postponed to sections~\ref{geo:s:transversality}~and~\ref{geo:s:signs-or}.

Note that in Floer theory, \Ainf -structures arise from the fact that moduli spaces of closed pointed disks naturally yield the \Ainf -cell decompositions of the associahedra. This is not the case in our situation, where it is the $\Omega B As$-cell decompositions that naturally arise.

\subsection{Conventions} \label{geo:ss:conventions}

We refer to section~\ref{geo:ss:basic-mod-space-Morse} for additional details on the moduli spaces introduced in this section. We will study Morse theory of the Morse function $f : M \rightarrow \R$ using its negative gradient vector field $-\nabla f$. Denote $d$ the dimension of the manifold $M$ and $\phi^s$ the flow of $-\nabla f$. For a critical point $x$ define its unstable and stable manifolds
\begin{align*}
W^U(x) &:= \{ z \in M , \ \lim_{s \rightarrow - \infty} \phi^s(z) = x  \} \\
W^S(x) &:= \{ z \in M , \ \lim_{s \rightarrow + \infty} \phi^s(z) = x  \} \ .
\end{align*}
Their dimensions are such that $\mathrm{dim} (W^U(x)) + \mathrm{dim} (W^S(x)) = d$. 
We then define the \emph{degree of a critical point $x$ }to be $|x| := \mathrm{dim} (W^S(x))$. This degree is often referred to as the \emph{coindex of $x$} in the litterature. 

We will moreover work with Morse cochains. For two critical point $x \neq y$, define 
\[ \mathcal{T} (y;x) := W^S(y) \cap W^U(x) / \R  \]
to be the moduli space of negative gradient trajectories connecting $x$ to $y$. Denote moreover $\mathcal{T}(x;x) = \emptyset$. Under the Morse-Smale assumption on $f$ and the Riemannian metric on $M$, for $x \neq y$ the moduli space $\mathcal{T} (y;x)$ has dimension $\mathrm{dim} \left( \mathcal{T} (y;x) \right) = |y| - |x| - 1$.
The Morse differential $\partial_{Morse} : C^*(f) \rightarrow C^*(f)$ is then defined to count descending negative gradient trajectories
\[  \partial_{Morse} (x) :=\sum_{|y| = |x| + 1} \# \mathcal{T} (y;x) \cdot y  \ . \]

\subsection{Perturbed Morse gradient trees} \label{geo:ss:pert-Morse-tree}

\begin{definition}[\cite{abouzaid-plumbings}]
Let $T:=(t,\{ l_e \}_{e \in E(t)})$ be a metric tree, where $\{ l_e \}_{e \in E(t)}$ are the lengths of its internal edges. A \emph{choice of perturbation data} on $T$ consists of the following data :
\begin{enumerate}[label=(\roman*)]
\item a vector field
\[ [ 0 , l_e ] \times M \underset{\mathbb{X}_e}{\longrightarrow} T M \ , \]
that vanishes on $[ 1 , l_e -1 ]$, for every internal edge $e$ of $t$ ; 
\item a vector field 
\[ [ 0 , +\infty [ \times M \underset{\mathbb{X}_{e_0}}{\longrightarrow} T M \ , \]
that vanishes away from $[0,1]$, for the outgoing edge $e_0$ of $t$ ; 
\item a vector field 
\[ ] - \infty , 0 ] \times M \underset{\mathbb{X}_{e_i}}{\longrightarrow} T M \ , \]
that vanishes away from $[-1,0]$, for every incoming edge $e_i \ (1 \leqslant i \leqslant n)$ of $t$.
\end{enumerate}
\end{definition}

Note that when $l_e \leqslant 2$, the vanishing condition on $[ 1 , l_e -1 ]$ is empty, that is we do not require any specific vanishing property for $\mathbb{X}_e$. For brevity's sake we will write $D_e$ for all segments $[ 0 , l_e ]$ as well as  for all semi-infinite segments $] - \infty , 0 ]$ and $[ 0 , +\infty [$ in the rest of the paper.

\begin{definition}[\cite{abouzaid-plumbings}]
A \emph{perturbed Morse gradient tree} $T^{Morse}$ associated to $(T,\mathbb{X})$ is the data for each edge $e$ of $t$ of a smooth map $\gamma_e : D_e \rightarrow M$ 
such that $\gamma_e$ is a trajectory of the perturbed negative gradient $-\nabla f + \mathbb{X}_e$, i.e.
\[ \dot{\gamma}_e (s) = -\nabla f (\gamma_e (s)) + \mathbb{X}_e(s,\gamma_e (s)) \ , \]
and such that the endpoints of these trajectories coincide as prescribed by the edges of the tree $T$.
\end{definition}

\begin{figure}[h!]
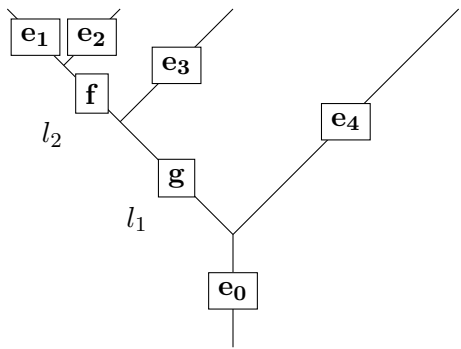
 
\centering
\perturbedgradienttreecorrespondence
\caption{Choosing perturbation data $\mathbb{X}$ for this metric tree, we have that $\phi_{1,\mathbb{X}}=\phi_{g,\mathbb{X}}^{l_1} \circ \phi_{f,\mathbb{X}}^{l_2} \circ \phi_{e_1,\mathbb{X}}^1$, $\phi_{2,\mathbb{X}}=\phi_{g,\mathbb{X}}^{l_1} \circ \phi_{f,\mathbb{X}}^{l_2} \circ \phi_{e_2,\mathbb{X}}^1$, $\phi_{3,\mathbb{X}}=\phi_{g,\mathbb{X}}^{l_1} \circ \phi_{e_3,\mathbb{X}}^1$ and $\phi_{4,\mathbb{X}}= \phi_{e_4,\mathbb{X}}^1$ } \label{geo:fig:flow-map}
\end{figure}

A perturbed Morse gradient tree $T^{Morse}$ associated to $(T,\mathbb{X})$ is determined by the data of the time -1 points on its incoming edges plus the time 1 point on its outgoing edge. Indeed, for each edge $e$ of $t$, we write $\phi_{e,\mathbb{X}}$ for the flow of $-\nabla f + \mathbb{X}_e$. We moreover define for every incoming edge $e_i$ $(1 \leqslant i \leqslant n)$ of $T$, the diffeomorphism $\phi_{i,\mathbb{X}}$ to be the composition of all flows obtained by following the time -1 point of the metric tree on $e_i$ along the only non-self crossing path connecting it to the root. We also set $\phi_{0,\mathbb{X}}$ for the flow of $\phi_{e_0,\mathbb{X}}$ at time -1, where $e_0$ is the outgoing edge of $t$. This is depicted on figure~\ref{geo:fig:flow-map}. Setting 
\[ \Phi_{T,\mathbb{X}} : M \times \cdots \times M \underset{\phi_{0,\mathbb{X}} \times \cdots \times  \phi_{n,\mathbb{X}}}{\longrightarrow} M \times \cdots \times M \ , \]
and $\Delta$ for the thin diagonal of $M \times \cdots \times M$, it is then clear that :

\begin{proposition}[\cite{abouzaid-plumbings}]
There is a one-to-one correspondence
\[
\begin{array}{c@{}c@{}c}
 \left\{\begin{array}{c}
         \text{perturbed Morse gradient trees} \\
         \text{associated to $(T,\mathbb{X})$} \\ 
  \end{array}\right\}
  & \longleftrightarrow 
  & \begin{array}{c}
         (\Phi_{T,\mathbb{X}})^{-1}(\Delta) \\
  \end{array} \ .
\end{array}
\]  
\end{proposition}

The vector fields on the external edges are equal to $-\nabla f$ away from a length 1 segment, hence the trajectories associated to these edges all converge to critical points of the function $f$.
For critical points $y$ and $x_1,\dots ,x_n$, the map $\Phi_{T,\mathbb{X}}$ can be restricted to
\[ W^S(y) \times W^U(x_1) \times \cdots \times W^U(x_n) \ , \]
such that the inverse image of the diagonal yields all perturbed Morse gradient trees associated to $(T,\mathbb{X})$ connecting $x_1,\dots ,x_n$ to $y$.

\subsection{Moduli spaces of perturbed Morse gradient trees} \label{geo:ss:mod-space-pert-Morse-tree}

Recall that $E(t)$ stands for the set of internal edges of $t$, and $\overline{E}(t)$ for the set of all its edges. We previously saw that a choice of perturbation data on a metric ribbon tree $T:=(t,\{ l_e \}_{e \in E(t)})$ is the data of maps $\mathbb{X}_{T,f} : D_f \times M  \longrightarrow TM$, for every edge $f \in \overline{E}(t)$ of $t$. Define the cone $C_f \subset \mathcal{T}_n(t) \times \R \simeq \R^{e(t)+1}  $ to be
\begin{enumerate}[label=(\roman*)]
\item $\{ (( l_e )_{e \in E(t)},s) \text{ such that } 0 \leqslant s \leqslant l_f \}$ if $f$ is an internal edge ;
\item $\{ (( l_e )_{e \in E(t)},s) \text{ such that } s \leqslant 0 \}$ if $f$ is an incoming edge ;
\item $\{ (( l_e )_{e \in E(t)},s) \text{ such that } s \geqslant 0 \}$ if $f$ is the outgoing edge.
\end{enumerate}
Then a choice of perturbation data for every metric ribbon tree in $\mathcal{T}_n(t)$ yields a map
\[ \mathbb{X}_{t,f} : C_f \times M \longrightarrow TM \ , \]
for every edge $f$ of $t$. This choice of perturbation data is said to be \emph{smooth} if all these maps are smooth.

\begin{definition}
Let $\mathbb{X}_t$ be a smooth choice of perturbation data on $\mathcal{T}_n(t)$. For critical points $y$ and $x_1,\dots ,x_n$, we define the moduli space
\[ \mathcal{T}_t^{\mathbb{X}_t}(y ; x_1,\dots,x_n) := \left\{\begin{array}{c}
         \text{perturbed Morse gradient trees associated to $(T,\mathbb{X}_T)$} \\
         \text{and connecting $x_1,\dots,x_n$ to $y$, for  $T \in \mathcal{T}_n(t)$ } \\
  \end{array}\right\} . \]
\end{definition}

Introduce now the map
\[ \phi_{\mathbb{X}_t} : \mathcal{T}_n(t) \times W^S(y) \times W^U(x_1) \times \cdots \times W^U(x_n) \longrightarrow M^{\times n+1} \ , \]
whose restriction to every $T \in \mathcal{T}_n(t)$ is as defined previously :

\begin{proposition} \label{geo:prop:trois-items-ombas-alg}
\begin{enumerate}[label=(\roman*)]
\item The moduli space $\mathcal{T}_t^{\mathbb{X}_t}(y ; x_1,\dots,x_n)$ can be rewritten as 
\[ \mathcal{T}_t^{\mathbb{X}_t}(y ; x_1,\dots,x_n) = \phi_{\mathbb{X}_t}^{-1}(\Delta) \ , \]
where $\Delta$ is the thin diagonal of $M^{\times n+1}$. \label{item:moduli-spaces:A}
\item Given a choice of perturbation data $\mathbb{X}_t$ making $\phi_{\mathbb{X}_t}$ transverse to the diagonal $\Delta$, the moduli space $\mathcal{T}_t^{\mathbb{X}_t}(y ; x_1,\dots,x_n)$ is an orientable manifold of dimension 
\[ \dim \left( \mathcal{T}_t(y ; x_1,\dots,x_n) \right) = e(t) + |y| - \sum_{i=1}^n|x_i| \ . \] \label{item:moduli-spaces:B}
\item Choices of perturbation data $\mathbb{X}_t$ such that $\phi_{\mathbb{X}_t}$ is transverse to $\Delta$ exist. \label{item:moduli-spaces:C}
\end{enumerate}
\end{proposition} 

Item~\ref{item:moduli-spaces:A} is straightforward and item~\ref{item:moduli-spaces:B} stems from the fact that if $\phi_{\mathbb{X}_t}$ transverse to $\Delta$, the moduli spaces $\mathcal{T}_t^{\mathbb{X}_t}(y ; x_1,\dots,x_n)$ are manifolds of codimension
\[ \mathrm{codim} \left( \mathcal{T}_t(y ; x_1,\dots,x_n) \right) = \mathrm{codim}_{M^{\times n +1}} (\Delta) = nd  \ , \]
where $d := \dim(M)$. 
Note that we have chosen to grade the Morse cochains using the coindex in order for this convenient dimension formula to hold. We refer to sections~\ref{geo:s:transversality} for details on item~\ref{item:moduli-spaces:C}.

\subsection{Compactifications} \label{geo:ss:compact-ribbon-Morse}

We now would like to compactify the moduli spaces $\mathcal{T}_t^{\mathbb{X}_t}(y ; x_1,\dots,x_n)$ that have dimension 1 to 1-dimensional manifolds with boundary. They are defined as the inverse image in $\mathcal{T}_n(t) \times  W^S(y) \times W^U(x_1) \times \cdots \times W^U(x_n)$ of the diagonal $\Delta$ under $\phi_{\mathbb{X}_t}$ . The boundary components in the compactification should hence come from those of $\mathcal{T}_n(t)$, of the $W^U(x_i)$, and of $W^S(y)$ : that is they will respectively come from internal edges of the perturbed Morse gradient tree collapsing, or breaking at a critical point (boundary of $\mathcal{T}_n(t)$), its semi-infinite incoming edges breaking at a critical point (boundary of $W^U(x_i)$) and its semi-infinite outgoing edge breaking at a critical point (boundary of $W^S(y)$). Some of these phenomena are represented on figure~\ref{geo:fig:examples-perturbed-ribbon-breaking}.

\begin{figure}[h]
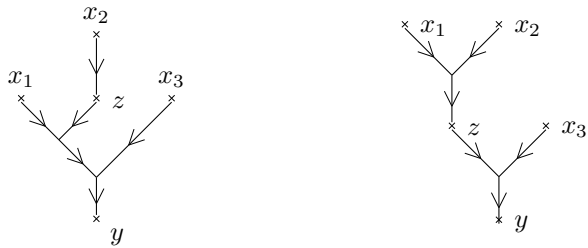
 
    \centering
    \begin{subfigure}{0.3\textwidth}
    \centering
       \examplebreakingMorseribbontreeun
       \end{subfigure} ~
    \begin{subfigure}{0.3\textwidth}
    \centering
        \examplebreakingMorseribbontreedeux
    \end{subfigure}
    \caption{Two examples of perturbed Morse gradient trees breaking at a critical point} \label{geo:fig:examples-perturbed-ribbon-breaking}
\end{figure}

Choose smooth perturbation data $\mathbb{X}_t$ for all $t \in SRT_i, \ 2 \leqslant i \leqslant n$. We denote $\mathbb{X}_n := (\mathbb{X}_t)_{t \in SRT_n}$ and call it a \emph{choice of perturbation data on the moduli space $\mathcal{T}_n$}. We construct the boundary of the compactification of the moduli space $\mathcal{T}^{\mathbb{X}_t}_t(y ; x_1,\dots,x_n)$ by using the perturbation data $(\mathbb{X}_t)^{t \in SRT_i}_{1 \leqslant i \leqslant n }$. It is given by the spaces
\begin{enumerate}[label=(\roman*)]
\item corresponding to an internal edge collapsing (int-collapse) : $$\mathcal{T}^{\mathbb{X}_{t'}}_{t'}(y ; x_1,\dots,x_n)$$ where $t' \in SRT_n$ are all the trees obtained by collapsing exactly one internal edge of $t$ ;
\item corresponding to an internal edge breaking (int-break) : $$\mathcal{T}^{\mathbb{X}_{t_1}}_{t_1}(y ; x_1,\dots,x_{i_1},z,x_{i_1+i_2+1},\dots,x_n) \times \mathcal{T}^{\mathbb{X}_{t_2}}_{t_2}(z ; x_{i_1+1},\dots,x_{i_1+i_2}) ,$$
where $t_2$ is seen to lie above the $i_1 + 1$-incoming edge of $t_1$ ;
\item corresponding to an external edge breaking (Morse) : $$\mathcal{T}(y;z) \times \mathcal{T}^{\mathbb{X}_t}_t(z ; x_1,\dots,x_n) \ \text{ and } \  \mathcal{T}^{\mathbb{X}_t}_t(y ; x_1,\dots,z,\dots , x_n) \times \mathcal{T}(z;x_i)  \ .$$
\end{enumerate} 
While the (Morse) boundary simply comes from the fact that external edges are Morse trajectories away from a length 1 segment, the analysis for the (int-collapse) and (int-break) boundaries requires to refine our definitions of perturbation data. It namely appears here why we had to choose more perturbation data than $\mathbb{X}_t$, as they will appear in the boundary of the compactified moduli space.

We begin by tackling the conditions coming with the (int-collapse) boundary. Let $t$ be a stable ribbon tree type and consider a choice of perturbation data on $\mathcal{T}_n(t)$ : it is a choice of perturbation data $\mathbb{X}_T$ for every $T \in \mathcal{T}_n(t) \simeq ]0,+\infty[^{e(t)}$. Denote $coll(t) \subset SRT_n$ the set of all trees obtained by collapsing internal edges of $t$. A choice of perturbation data $(\mathbb{X}_{t'})_{t' \in coll(t)}$ then corresponds to a choice of perturbation data $\mathbb{X}_T$ for every $T \in [0,+\infty[^{e(t)}$. Following section~\ref{geo:ss:mod-space-pert-Morse-tree}, such a choice of perturbation data is equivalent to a map 
\[ \tilde{\mathbb{X}}_{t,f} : \tilde{C}_f \times M \longrightarrow TM \ , \]
for every edge $f$ of $t$, where $\tilde{C}_f \subset [0,+\infty[^{e(t)} \times \R$ is defined in a similar fashion to $C_f$.

\begin{definition}
A choice of perturbation data $(\mathbb{X}_{t'})_{t' \in coll(t)}$ is said to be \emph{smooth} if all maps $\tilde{\mathbb{X}}_{t,f}$ are smooth. A choice of perturbation data $\mathbb{X}_n $ is said to be \emph{smooth} if for every $t \in SRT_n$, the choice of perturbation data $(\mathbb{X}_{t'})_{t' \in coll(t)}$ is smooth.
\end{definition}

We now tackle the conditions coming with the (int-break) boundary. We work again with a fixed stable ribbon tree type $t$. Consider a choice of perturbation data $\mathbb{X}_t=(\mathbb{X}_{t,e})_{e \in \overline{E}(t)}$ on $\mathcal{T}_n(t)$. We have to specify what happens on the $\mathbb{X}_{t,e}$ when the length of an internal edge $f$ of $t$, denoted $l_f$, goes towards $+\infty$. Write $t_1$ and $t_2$ for the trees obtained by breaking $t$ at the edge $f$.
\begin{enumerate}[label=(\roman*)]
\item For $e \in \overline{E}(t)$ and $\neq f$, assuming for instance that $e \in t_1$, we require that $$\lim_{l_f \rightarrow + \infty} \mathbb{X}_{t,e} = \mathbb{X}_{t_1,e} \ .$$ \label{geo:item:cond-un-ombas-alg}
\item For $f=e$, $\mathbb{X}_{t,f}$ yields two parts when $l_f \rightarrow +\infty$ : the part corresponding to the infinite edge in $t_1$ and the part corresponding to the infinite edge in $t_2$. We then require that they coincide respectively with $\mathbb{X}_{t_1,f}$ and $\mathbb{X}_{t_2,f}$. \label{geo:item:cond-deux-ombas-alg}
\end{enumerate}
Two examples illustrating these two cases are detailed in the following paragraphs.

Begin with an example of the first case, where $e \neq f$. This is represented on figure~\ref{geo:fig:example-perturbation-break-un}. We only represent the perturbation $\mathbb{X}_{t,f_3}$ on this figure for clarity's sake. 
The perturbation datum $\mathbb{X}_{t,f_3}^{\infty}$ could a priori depend on $l_{f_1}$ : the requirement $\mathbb{X}_{t,f_3}^{\infty}=\mathbb{X}_{t_1,f_3}$ says in particular that it is independent of $l_{f_1}$.
\begin{figure}[h] 
\centering
\exampleperturbationbreakun
\caption{} \label{geo:fig:example-perturbation-break-un}
\end{figure}

Similarly, we illustrate the second case, where $e=f$, on figure~\ref{geo:fig:example-perturbation-break-deux}. 
A priori, $\mathbb{X}_{t,f_2}^{+}$ and $\mathbb{X}_{t,f_2}^{-}$ can depend on both $l_{f_1}$ and $l_{f_3}$ : the requirement $\mathbb{X}_{t,f_2}^{+}=\mathbb{X}_{t_2,f_2}$ says exactly that $\mathbb{X}_{t,f_2}^{+}$ is independent of $l_{f_3}$, and similarly for $\mathbb{X}_{t,f_2}^{-}=\mathbb{X}_{t_1,f_2}$ with respect to $l_{f_1}$.
\begin{figure}[h] 
\centering
\exampleperturbationbreakdeux
\caption{} \label{geo:fig:example-perturbation-break-deux}
\end{figure}

\begin{definition}
A choice of perturbation data $(\mathbb{X}_i)_{2 \leqslant i \leqslant n}$ is said to be \emph{gluing-compatible} if it satisfies conditions~\ref{geo:item:cond-un-ombas-alg}~and~\ref{geo:item:cond-deux-ombas-alg} for lengths of edges going toward $+\infty$. A choice of perturbation data $(\mathbb{X}_n)_{n \geqslant 2}$ being both smooth and gluing-compatible, and such that all maps $\phi_{\mathbb{X}_t}$ are transverse to $\Delta$, is said to be \emph{admissible}.
\end{definition}

\begin{theorem} \label{geo:th:exist-admissible-perturbation-data-Tn}
Admissible choices of perturbation data on the moduli spaces $\mathcal{T}_n$ exist.
\end{theorem}

\begin{theorem} \label{geo:th:compactification-Tn}
Let $(\mathbb{X}_n)_{n \geqslant 2}$ be an admissible choice of perturbation data. The 0-dimensional moduli spaces $\mathcal{T}^{\mathbb{X}_t}_t(y ; x_1,\dots,x_n)$ are compact. The 1-dimensional moduli spaces $\mathcal{T}^{\mathbb{X}_t}_t(y ; x_1,\dots,x_n)$ can be compactified to 1-dimensional manifolds with boundary $\overline{\mathcal{T}}^{\mathbb{X}_t}_t(y ; x_1,\dots,x_n)$, whose boundary is described at the beginning of this section.
\end{theorem}

We refer to section~\ref{geo:s:transversality} for a proof of Theorem~\ref{geo:th:exist-admissible-perturbation-data-Tn}. Theorem~\ref{geo:th:compactification-Tn} is proven in chapter 6 of~\cite{mescher-morse}. Using the results of~\cite{wehrheim-morse}, we could in fact try to prove that all moduli spaces $\mathcal{T}^{\mathbb{X}_t}_t(y ; x_1,\dots,x_n)$ can be compactified to compact manifolds with corners. The analysis involved therein goes however beyond the scope of this paper.

Consider now a stable ribbon tree $t$ together with an internal edge $f \in E(t)$ and write $t_1$ and $t_2$ for the trees obtained by breaking $t$ at the edge $f$, where $t_2$ is seen to lie abpve $t_1$. Given critical points $y,z,x_1,\dots,x_n$ suppose moreover that the moduli spaces $\mathcal{T}_{t_1}(y ; x_1,\dots,x_{i_1},z,x_{i_1+i_2+1},\dots,x_n)$ and $\mathcal{T}_{t_2}(z ; x_{i_1+1},\dots,x_{i_1+i_2})$ are 0-dimensional. Let $T_1^{Morse}$ and $T_2^{Morse}$ be two perturbed Morse gradient trees which belong respectively to the former and the latter moduli spaces. Theorem~\ref{geo:th:compactification-Tn} implies in particular that there exists $R > 0$ and an embedding
\[ \# _{ T_1^{Morse} , T_2^{Morse} } : [ R , + \infty ] \longrightarrow \overline{\mathcal{T}}_{t}(y ; x_1,\dots ,x_n) \]
parametrizing a neighborhood of the boundary $\{ T_1^{Morse} \} \times \{ T_2^{Morse} \} \subset \partial \overline{\mathcal{T}}_{t}^{Morse}$, i.e. sending $+ \infty$ to $(T_1^{Morse} , T_2^{Morse}) \in \partial \overline{\mathcal{T}}_{t}^{Morse}$. Such a map is called a \emph{gluing map} for $T_1^{Morse}$ and $T_2^{Morse}$. Explicit gluing maps are constructed in subsection~\ref{geo:sss:gluing-and-orientations}.

\subsection{$\Omega B As$ -algebra structure on the Morse cochains} \label{geo:ss:ombas-str-Morse}

We now have all the necessary material to define an $\Omega B As$-algebra structure on the Morse cochains $C^*(f)$. 

\begin{theorem} \label{geo:th:ombas-alg}
Let $\mathbb{X} := (\mathbb{X}_n)_{n \geqslant 2}$ be an admissible choice of perturbation data. Defining for every $n$ and $t \in SRT_n$ the operations $m_t$ as 
\begin{align*}
m_t : C^*(f) \otimes \cdots \otimes C^*(f) &\longrightarrow C^*(f) \\
x_1 \otimes \cdots \otimes x_n &\longmapsto \sum_{|y|= \sum_{i=1}^n|x_i| - e(t)} \# \mathcal{T}_t^\mathbb{X}(y ; x_1,\cdots,x_n) \cdot y \ ,
\end{align*} 
they endow the Morse cochains  $C^*(f)$ with an $\Omega B As$-algebra structure.
\end{theorem}

The proof of this theorem is detailed in section~\ref{geo:ss:twisted-ainf-alg-Morse}. Putting it shortly, counting the boundary points of the 1-dimensional orientable compactified moduli spaces $\overline{\mathcal{T}}_t^\mathbb{X}(y ; x_1,\cdots,x_n)$ whose boundary is described in the previous section yields the $\Omega B As$-equations 
\[ [ \partial_{Morse} , m_t ] = \sum_{t' \in coll(t)} \pm m_{t'} + \sum_{t_1 \#_i t_2 = t} \pm m_{t_1} \circ_i m_{t_2} \ . \]
In fact, the collection of operations $\{ m_t \}$ does not exactly define an $\Omega B As$-algebra structure : one of the two differentials $\partial_{Morse}$ appearing in the bracket $[ \partial_{Morse} , \cdot ]$ has to be twisted by a specific sign for the $\Omega B As$-equations to hold. We will speak about a \emph{twisted $\Omega B As$-algebra structure}. In the case when $M$ is odd-dimensional, this twisted $\Omega B As$-algebra is exactly an $\Omega B As$-algebra. 

If we want to recover an \Ainf -algebra structure on the Morse cochains, it suffices to apply the morphism of operads $\Ainf \rightarrow \ombas$ described in section~\ref{alg:sss:op-ainf-to-ombas}. In his paper~\cite{abouzaid-plumbings}, Abouzaid constructs a geometric \Ainf -morphism $C^*_{sing}(M) \rightarrow C^*(f)$, where the Morse cochains are endowed with the \Ainf -algebra structure constructed in this subsection. This \Ainf -morphism is in fact a quasi-isomorphism. This implies in particular that the Morse cochains $C^*(f)$ endowed with the \Ainf -algebra structure constructed in this subsection are quasi-isomorphic as an \Ainf -algebra to the Morse cochains endowed with the \Ainf -algebra structure induced by the homotopy transfer theorem. His construction of the \Ainf -morphism $C^*_{sing}(M) \rightarrow C^*(f)$ could be adapted to our present framework, and lifted to an \ombas -morphism. We will however not give more details on that matter.   

\section{\Ainf\ and $\Omega B As$-morphisms between the Morse cochains} \label{geo:s:ainf-ombas-morph-Morse}

Let $M$ be an oriented closed Riemannian manifold endowed with a Morse function $f$ together with a Morse-Smale metric. We have proven in the previous section that, upon choosing admissible perturbation data on the moduli spaces of stable metric ribbon trees $\mathcal{T}_n(t)$, we can define moduli spaces of perturbed Morse gradient trees, whose count will define the operations $m_t, \ t \in SRT$, of an $\Omega B As$-algebra structure on the Morse cochains $C^*(f)$. 

Consider now another Morse function $g$ on $M$. Apply again the homotopy transfer theorem to $C^*(f)$ and $C^*(g)$, which are deformation retracts of the singular cochains on $M$. Endowing them with their induced \Ainf -algebra structures, the theorem yields a diagram
\[ (C^*(f),m_n^{ind}) \tilde{\longrightarrow} (C^*_{sing}(M),\cup) \tilde{\longrightarrow} (C^*(g),m_n^{ind}) \ , \]
where each arrow is an \Ainf -quasi-isomorphism, hence an \Ainf -quasi-isomorphism $(C^*(f),m_n^{ind}) \rightarrow (C^*(g),m_n^{ind})$.
Let $\mathbb{X}^g$ be an admissible perturbation data for $g$. This motivates the following question : endowing $C^*(f)$ and $C^*(g)$ with their $\Omega B As$-algebra structures, can we construct an $\Omega B As$-morphism
\[ (C^*(f),m_t^{\mathbb{X}^f}) \longrightarrow (C^*(g),m_t^{\mathbb{X}^g}) \ ? \]

\begin{figure}[h]
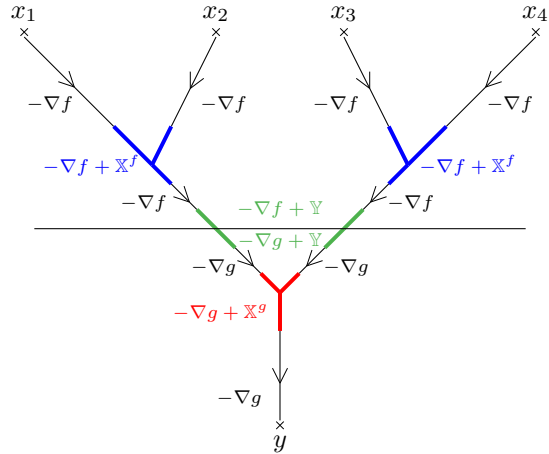
 
\centering
\arbredegradientbicoloreexemple
\caption{An example of a perturbed two-colored Morse gradient tree, where the $x_i$ are critical points of $f$ and $y$ is a critical point of $g$} \label{geo:fig:arbre-gradient-bicolore}
\end{figure}

While stable metric ribbon trees control $\Omega B As$-algebra structures, we have seen that two-colored stable metric ribbon trees control $\Omega B As$-morphisms. The answer to the previous question is then of course positive, and the morphism will be constructed using moduli spaces of \emph{two-colored perturbed Morse gradient trees}. As in section~\ref{geo:s:ainf-ombas-alg-Morse}, two-colored Morse gradient trees will be defined by perturbing Morse gradient equations around each vertex of the tree, where the Morse gradient is $-\nabla f$ above the gauge, and $-\nabla g$ below the gauge. This is illustrated in figure~\ref{geo:fig:arbre-gradient-bicolore}. The figure is incorrect, because we won't choose the perturbation to be equal to $\mathbb{X}^f$ above the gauge and to $\mathbb{X}^g$ below, but gives the correct intuition on the construction we unfold in this section.

The structure of this section follows the same lines as the previous section, and the only difficulty will consist in adapting properly our arguments to the combinatorics of two-colored ribbon trees. Under a generic choice of perturbation data on the moduli spaces $\mathcal{CT}_n$, the moduli spaces of two-colored perturbed Morse gradient trees connecting $x_1 , \dots , x_n \in \mathrm{Crit}(g)$ to $y \in \mathrm{Crit}(g)$, that we denote $\mathcal{CT}_{t_g}(y ; x_1,\dots,x_n)$, are orientable manifolds. They are moreover compact when 0-dimensional and can be compactified to compact manifolds with boundary when 1-dimensional (Theorems~\ref{geo:th:existence-perturb-CTn}~and~\ref{geo:th:compact-CTn}). Counting 0-dimensional moduli spaces of two-colored Morse gradient trees then defines an \ombas -morphism from $C^*(f)$ to $C^*(g)$ (Theorem~\ref{geo:th:ombas-morph}). 

\subsection{Notation} \label{geo:ss:notation}

A two-colored ribbon tree will be written $t_g$ using the gauge viewpoint, and $t_c$ using the colored vertices viewpoint. The tree $t_g$ then comes with an underlying stable ribbon tree $t$, while the tree $t_c$ is already a ribbon tree (though not necessarily stable because of its colored vertices).

A two-colored stable metric ribbon tree $T$ will be written $(t_g,(l_e)_{e\in E(t)},\lambda)$ using the gauge viewpoint. The lengths associated to the underlying metric ribbon tree with colored vertices will then be written $L_{f_c}((l_e)_{e\in E(t)},\lambda)$ where $f_c \in E(t_c)$. For instance, on figure~\ref{alg:fig:example-stable-two-col},
\begin{align*}
L_1 = - \lambda && L_2 = l + \lambda && L_3 = - \lambda \ .
\end{align*}
For the sake of readability, we do not write the dependence on $((l_e)_{e\in E(t)},\lambda)$ in the sequel.

\subsection{Perturbed two-colored Morse gradient trees} \label{geo:ss:pert-two-col-tree}

\begin{definition}
Let $T_g=(t_g,(l_e)_{e\in E(t)},\lambda)$ be a two-colored metric ribbon tree. A \emph{choice of perturbation data} $\mathbb{Y}$ on $T_g$ is defined to be a choice of perturbation data on the associated metric ribbon tree $(t_c,L_{f_c})$ in the sense of section~\ref{geo:ss:pert-Morse-tree}. 
\end{definition}

\begin{definition}
A \emph{two-colored perturbed Morse gradient tree} $T_g^{Morse}$ associated to a pair two-colored metric ribbon tree and perturbation data $(T_g,\mathbb{Y})$ is the data 
\begin{enumerate}[label=(\roman*)]
\item for each edge $f_c$ of $t_c$ which is above the gauge, of a smooth map
\begin{align*}
D_{f_c}  \underset{\gamma_{f_c}}{\longrightarrow} M \ , 
\end{align*}
such that $\gamma_{f_c}$ is a trajectory of the perturbed negative gradient $-\nabla f + \mathbb{Y}_{f_c}$,
\item for each edge $f_c$ of $t_c$ which is below the gauge, of a smooth map
\begin{align*}
D_{f_c} \underset{\gamma_{f_c}}{\longrightarrow} M \ , 
\end{align*}
such that $\gamma_{f_c}$ is a trajectory of the perturbed negative gradient $-\nabla g + \mathbb{Y}_{f_c}$,
\end{enumerate}
and such that the endpoints of these trajectories coincide as prescribed by the edges of the tree $t_c$.
\end{definition} 

Note that the above definitions still work for \arbreopunmorph . A choice of perturbation data for \arbreopunmorph\ is the data of vector fields
\begin{align*}
[0 , +\infty [ \times M &\underset{\mathbb{Y}_{+}}{\longrightarrow} T M \ , \\
] - \infty , 0] \times M &\underset{\mathbb{Y}_{-}}{\longrightarrow} T M \ ,
\end{align*}
which vanish away from a length 1 segment, and a two-colored perturbed Morse gradient tree associated to $(\arbreopunmorph , \mathbb{Y})$ is then simply the data of two smooth maps
\begin{align*}
]-\infty,0] &\underset{\gamma_-}{\longrightarrow} M \ , \\
[ 0 , +\infty [ &\underset{\gamma_+}{\longrightarrow} M \ ,
\end{align*}
such that $\gamma_{-}$ is a trajectory of $-\nabla f + \mathbb{Y}_-$ and $\gamma_{+}$ is a trajectory of $-\nabla g + \mathbb{Y}_+$.

There is also an equivalent formulation for two-colored perturbed Morse gradient trees, by following the flows of $-\nabla f + \mathbb{Y}$ and $-\nabla g + \mathbb{Y}$ along the the metric ribbon tree $(t_c,L_{f_c})$. That is, a two-colored perturbed Morse gradient tree is determined by the data of the time -1 points on its incoming edges plus the time 1 point on its outgoing edge. Introduce again the map
\[ \Phi_{T_g,\mathbb{Y}} : M \times \cdots \times M \underset{\phi_{0,\mathbb{Y}} \times \cdots \times  \phi_{n,\mathbb{Y}}}{\longrightarrow} M \times \cdots \times M \ , \]
defined as before, and set $\Delta$ for the diagonal of $M^{ \times n+1}$

\begin{proposition}
There is a one-to-one correspondence
\[
\begin{array}{c@{}c@{}c}
 \left\{\begin{array}{c}
         \text{two-colored perturbed Morse gradient trees} \\
         \text{associated to $(T_g,\mathbb{Y})$} \\ 
  \end{array}\right\}
  & \longleftrightarrow 
  & \begin{array}{c}
         (\Phi_{T_g,\mathbb{Y}})^{-1}(\Delta) \\
  \end{array} \ .
\end{array}
\]  
\end{proposition}

The vector fields on the incoming edges are equal to $-\nabla f$ away from a length 1 segment, hence the trajectories associated to these edges all converge to critical points of the function $f$, while the vector field on the outgoing edge is equal to $-\nabla g$ away from a length 1 segment, hence the trajectory associated to these edge converges to a critical point of the function $g$.
For critical points $y$ of the function $g$ and $x_1,\dots ,x_n$ of the function $f$, the map $\Phi_{T,\mathbb{Y}}$ can be restricted to
\[ W^S_g(y) \times W^U_f(x_1) \times \cdots \times W^U_f(x_n) \ , \]
such that the inverse image of the diagonal yields all two-colored perturbed Morse gradient trees associated to $(T,\mathbb{Y})$ connecting $x_1,\dots ,x_n$ to $y$.

\subsection{Moduli spaces of two-colored perturbed Morse gradient trees} \label{geo:ss:mod-space-pert-two-col-tree}

Choose a two-colored stable ribbon tree $t_g \in SCRT_n$ whose underlying stable ribbon tree is $t$ and whose associated ribbon tree with colored vertices is $t_c$. We write $(*)_{t_g}$ for the set of inequalities and equalities on $\{ l_e \}_{e \in E(t)}$ and $\lambda$, which define the polyedral cone $\mathcal{CT}_n(t_g) \subset \R^{e(t)+1}$. See part~\ref{p:algebra} section~\ref{alg:ss:mod-space-CTm} for more details. Define for all $f_c \in \overline{E}(t_c)$, the cone $C_{f_c} \subset \mathcal{CT}_n(t_g) \times \R \subset \R^{e(t)+1} \times \R  $ to be
\begin{enumerate}[label=(\roman*)]
\item $\{ (( l_e )_{e \in E(T)},\lambda,s) \text{ such that } (*)_{t_g} \ , \ 0 \leqslant s \leqslant L_{f_c}(( l_e )_{e \in E(T)},\lambda) \}$ if $f_c$ is an internal edge ;
\item $\{ (( l_e )_{e \in E(T)},\lambda,s) \text{ such that } (*)_{t_g} \ , \ s \leqslant 0 \}$ if $f_c$ is an incoming edge ;
\item $\{ (( l_e )_{e \in E(T)},\lambda,s) \text{ such that } (*)_{t_g} \ , \ s \geqslant 0 \}$ if $f_c$ is the outgoing edge.
\end{enumerate}
Then a choice of perturbation data for every two-colored metric ribbon tree in $\mathcal{CT}_n(t_g)$, yields maps $\mathbb{Y}_{t_g,f_c} : C_{f_c} \times M \longrightarrow TM$ for every edge $f_c$ of $t_c$. These perturbation data are said to be \emph{smooth} if all these maps are smooth.

\begin{definition}
Let $\mathbb{Y}_{t_g}$ be a smooth choice of perturbation data on the stratum $\mathcal{CT}_n(t_g)$. Given $y \in \mathrm{Crit}(g)$ and $x_1,\dots ,x_n \in \mathrm{Crit}(f)$, we define the moduli spaces
\[ \mathcal{CT}_{t_g}^{\mathbb{Y}_{t_g}}(y ; x_1,\dots,x_n) := \left\{\begin{array}{c}
         \text{two-colored perturbed Morse gradient trees associated to $(T_g,\mathbb{Y}_{T_g})$} \\
         \text{and connecting $x_1,\dots,x_n$ to $y$ for  $T_g \in \mathcal{CT}_n(t_g)$ } \\
  \end{array}\right\} . \]
\end{definition}  

Using the smooth map 
\[ \phi_{\mathbb{Y}_{t_g}} : \mathcal{CT}_n(t_g) \times W^S(y) \times W^U(x_1) \times \cdots \times W^U(x_n) \longrightarrow M^{\times n+1} \ , \]
this moduli space can be rewritten as 
\[ \mathcal{CT}_{t_g}^{\mathbb{Y}_{t_g}}(y ; x_1,\dots,x_n) = \phi_{\mathbb{Y}_{t_g}}^{-1}(\Delta) \ . \]

\begin{proposition}
\begin{enumerate}[label=(\roman*)]
\item Given a choice of perturbation data $\mathbb{Y}_{t_g}$ making $\phi_{\mathbb{Y}_{t_g}}$ transverse to the diagonal $\Delta \subset M^{\times n+1}$, the moduli spaces $\mathcal{CT}_{t_g}^{\mathbb{Y}_{t_g}}(y ; x_1,\dots,x_n)$ are orientable manifolds of dimension 
\[ \dim \left( \mathcal{CT}_{t_g}(y ; x_1,\dots,x_n) \right) = + |y| - \sum_{i=1}^n|x_i| -|t_g| \ . \]
\item Choices of perturbation data $\mathbb{Y}_{t_g}$ such that $\phi_{\mathbb{Y}_{t_g}}$ is transverse to the diagonal $\Delta$ exist. 
\end{enumerate}
\end{proposition}

\noindent The proof of this proposition is again postponed to section~\ref{geo:s:transversality}.

\subsection{Compactifications} \label{geo:ss:compact-two-col-Morse}

We finally proceed to compactify the moduli spaces $\mathcal{CT}_{t_g}^{\mathbb{Y}_{t_g}}(y ; x_1,\dots,x_n)$ that have dimension 1 to 1-dimensional manifolds with boundary. Their boundary components are going to be given by those coming from the compactification of $\mathcal{CT}_n(t_g)$, and the compactifications of the $W^U(x_i)$ and of $W^S(y)$.

Choose admissible perturbation data $\mathbb{X}^f$ and $\mathbb{X}^g$ for the functions $f$ and $g$. Choose moreover smooth perturbation data $\mathbb{Y}_{t_g}$ for all $t_g \in SCRT_i, \ 1 \leqslant i \leqslant n$. We will again denote $\mathbb{Y}_n := (\mathbb{Y}_{t_g})_{t_g \in SCRT_n}$, and call it a choice of perturbation data on $\mathcal{CT}_n$. Fixing a two-colored stable ribbon tree $t_g \in SCRT_n$ we would like to compactify the 1-dimensional moduli space $\mathcal{CT}_{t_g}^{\mathbb{Y}_{t_g}}(y ; x_1,\dots,x_n)$ using the perturbation data $\mathbb{X}^f$, $\mathbb{X}^g$ and $(\mathbb{Y}_i)_{1 \leqslant i \leqslant n}$. Its boundary will be given by the following phenomena
\begin{enumerate}[label=(\roman*)]
\item an external edge breaks at a critical point (Morse) :
$$\mathcal{T}(y;z) \times \mathcal{CT}_{t_g}^{\mathbb{Y}_{t_g}}(z ; x_1,\dots,x_n) \ \text{ and } \  \mathcal{CT}_{t_g}^{\mathbb{Y}_{t_g}}(y ; x_1 , \dots , z , \dots , x_n) \times \mathcal{T}(z;x_i)  \ ;$$
\item an internal edge of the tree $t$ collapses (int-collapse) :
$$\mathcal{CT}^{\mathbb{Y}_{t_g'}}_{t_g'}(y ; x_1,\dots,x_n)$$ where $t_g' \in SCRT_n$ are all the two-colored trees obtained by collapsing exactly one internal edge, which does not cross the gauge ;
\item the gauge moves to cross exactly one additional vertex of the underlying stable ribbon tree (gauge-vertex) :
$$\mathcal{CT}^{\mathbb{Y}_{t_g'}}_{t_g'}(y ; x_1,\dots,x_n)$$ where $t_g' \in SCRT_n$ are all the two-colored trees obtained by moving the gauge to cross exactly one additional vertex of $t$ ;
\item an internal edge located above the gauge or intersecting it breaks or, when the gauge is below the root, the outgoing edge breaks between the gauge and the root (above-break) : $$ \mathcal{CT}^{\mathbb{Y}_{t^1_g}}_{t^1_g}(y ; x_1,\dots,x_{i_1},z,x_{i_1+i_2+1},\dots,x_n) \times \mathcal{T}^{\mathbb{X}^f_{t^2}}_{t^2}(z ; x_{i_1+1},\dots,x_{i_1+i_2}) \ ;$$
\item edges (internal or incoming) that are possibly intersecting the gauge, break below it, such that there is exactly one edge breaking in each non-self crossing path from an incoming edge to the root (below-break) : $$ \mathcal{T}^{\mathbb{X}^g_{t^0}}_{t^1}(y ; y_{1},\dots,y_{s}) \times \mathcal{CT}^{\mathbb{Y}_{t^1_g}}_{t^1_g}(y_1 ; x_1,\dots) \times \cdots \times \mathcal{CT}^{\mathbb{Y}_{t^s_g}}_{t^s_g}(y_s ; \dots,x_n) \ .$$
\end{enumerate}

The (Morse) boundaries are again a simple consequence of the fact that external edges are Morse trajectories away from a length 1 segment. Perturbation data that behave well with respect to the (int-collapse) and (gauge-vertex) boundaries are defined using simple adjustments of the discussion in section~\ref{geo:ss:compact-ribbon-Morse}. Hence, it only remains to specify the required behaviours under the breaking of edges.

\begin{figure}[h] 
    \centering
    \begin{subfigure}{\textwidth}
    \centering
       \exampleperturbationCTbreakun
       \caption*{(above-break) case (i)}
    \end{subfigure} ~
    \begin{subfigure}{\textwidth}
    \centering
        \exampleperturbationCTbreakdeux
       \caption*{(above-break) case (ii)}
    \end{subfigure} ~
    \begin{subfigure}{\textwidth}
    \centering
        \exampleperturbationCTbreaktrois
       \caption*{(above-break) case (iii)}
    \end{subfigure}
    \caption{} \label{geo:fig:above-break-Morse}
\end{figure}

We begin with the (above-break) boundary. Writing $t_c$ for the two-colored ribbon tree associated to $t_g$, it corresponds to the breaking of an internal edge $f_c$ of $t_c$ situated above the set of colored vertices. Denote $t^1_c$ and $t^2$ the trees obtained by breaking $t_c$ at the edge $f_c$, where $t^2$ is seen to lie above $t^1_c$. We have to specify, for each edge $e_c \in \overline{E}(t_c)$, what happens to the perturbation $\mathbb{Y}_{t_c,e}$ at the limit.
\begin{enumerate}[label=(\roman*)]
\item For $e_c \in \overline{E}(t^2)$ and $\neq f_c$, we require that $$\lim \mathbb{Y}_{t_c,e_c} = \mathbb{X}^f_{t^2,e_c} \ .$$
\item For $e_c \in \overline{E}(t^1_c)$ and $\neq f_c$, we require that $$\lim \mathbb{Y}_{t_c,e_c} = \mathbb{Y}_{t^1_c,e_c} \ .$$
\item For $f_c=e_c$, $\mathbb{Y}_{t_c,f_c}$ yields two parts at the limit : the part corresponding to the outgoing edge of $t^1$ and the part corresponding to the incoming edge of $t^1_c$. We then require that they coincide respectively with the perturbation $\mathbb{X}^f_{t^2,e_c}$ and $\mathbb{Y}_{t^1_c,e_c}$.
\end{enumerate}
Leaving the notations aside, an example of each case is illustrated in figure~\ref{geo:fig:above-break-Morse}.

\begin{figure}[h] 
    \centering
    \begin{subfigure}{\textwidth}
    \centering
       \exampleperturbationCTbreakquatre
       \caption*{(below-break) case (i)}
    \end{subfigure} ~
    \begin{subfigure}{\textwidth}
    \centering
        \exampleperturbationCTbreakcinq
       \caption*{(below-break) case (ii)}
    \end{subfigure} ~
    \begin{subfigure}{\textwidth}
    \centering
        \exampleperturbationCTbreaksix
       \caption*{(below-break) case (iii)}
    \end{subfigure}
    \caption{} \label{geo:fig:below-break-Morse}
\end{figure}

We conclude with the (below-break) boundary. Denote $t^1_g,\dots,t^s_g$ and $t^0$ the trees obtained by the chosen breaking of $t_g$ below the gauge, where $t^1_g,\dots,t^s_g$ are seen to lie above $t^0$.
\begin{enumerate}[label=(\roman*)]
\item For $e_c \in \overline{E}(t^i_c)$ and not among the breaking edges, we require that $$\lim \mathbb{Y}_{t_c,e_c} = \mathbb{Y}_{t^i_c,e_c} \ .$$
\item For $e_c \in \overline{E}(t^1)$ and not among the breaking edges, we require that $$\lim \mathbb{Y}_{t_c,e_c} = \mathbb{X}^g_{t^0,e_c} \ .$$
\item For $f_c$ among the breaking edges, $\mathbb{Y}_{t_c,f_c}$ yields two parts at the limit : the part corresponding to the outgoing edge of a $t^j_c$ and the part corresponding to the incoming edge of $t^0$. We then require that they coincide respectively with the perturbation $\mathbb{Y}_{t^j_c}$ and $\mathbb{X}^g_{t^0}$.
\end{enumerate}
This is again illustrated on figure~\ref{geo:fig:below-break-Morse}.

\begin{definition}
A choice of perturbation data $\mathbb{Y}$ on the moduli spaces $\mathcal{CT}_n$ is said to be \emph{smooth} if it is compatible with the (int-collapse) and (gauge-vertex) boundaries. A smooth choice of perturbation data is said to be \emph{gluing-compatible w.r.t. $\mathbb{X}^f$ and $\mathbb{X}^g$} if it satisfies the (above-break) and (below-break) conditions described in this section. Smooth and consistent choices of perturbation data $(\mathbb{Y}_n)_{n \geqslant 1}$ such that all maps $\phi_{\mathbb{Y}_{t_g}}$ are transverse to the diagonal $\Delta$ are called \emph{admissible w.r.t. $\mathbb{X}^f$ and $\mathbb{X}^g$} or simply \emph{admissible}.
\end{definition}

\begin{theorem} \label{geo:th:existence-perturb-CTn}
Given admissible choices of perturbation data $\mathbb{X}^f$ and $\mathbb{X}^g$ on the moduli spaces $\mathcal{T}_n$, choices of perturbation data on the moduli spaces $\mathcal{CT}_n$ that are admissible w.r.t. $\mathbb{X}^f$ and $\mathbb{X}^g$ exist.
\end{theorem}

\begin{theorem} \label{geo:th:compact-CTn}
Let $(\mathbb{Y}_n)_{n \geqslant 1}$ be an admissible choice of perturbation data on the moduli spaces $\mathcal{CT}_n$. The 0-dimensional moduli spaces $\mathcal{CT}^{\mathbb{Y}_{t_g}}_{t_g}(y ; x_1,\dots,x_n)$ are compact. The 1-dimensional moduli spaces $\mathcal{CT}^{\mathbb{Y}_{t_g}}_{t_g}(y ; x_1,\dots,x_n)$ can be compactified to 1-dimensional manifolds with boundary, whose boundary is described at the beginning of this section..
\end{theorem}

Theorem~\ref{geo:th:existence-perturb-CTn} is proven in section~\ref{geo:s:transversality}. Theorem~\ref{geo:th:compact-CTn} is a consequence of the results in chapter 6 of~\cite{mescher-morse}. We moreover point out that Theorem~\ref{geo:th:compact-CTn} implies in particular the existence of gluing maps 
\begin{align*}
\#^{above-break}_{ T_g^{1,Morse} , T^{2,Morse} } : [ R , + \infty ] \longrightarrow \overline{\mathcal{CT}}_{t_g}(y ; x_1,\dots ,x_n) \\
\#^{below-break}_{ T^{0,Morse} ,T_g^{1,Morse}, \dots , T_g^{s,Morse} } : [ R , + \infty ] \longrightarrow \overline{\mathcal{CT}}_{t_g}(y ; x_1,\dots ,x_n)
\end{align*}
where notations are as in section~\ref{geo:ss:compact-ribbon-Morse}. Such gluing maps are constructed in subsection~\ref{geo:sss:gluing-two-colored}.

\subsection{The $\Omega B As$-morphism between Morse cochains} \label{geo:ss:ombas-morph-Morse}

Let $\mathbb{X}^f$ and $\mathbb{X}^g$ be admissible choices of perturbation data for the Morse functions $f$ and $g$. Denote $(C^*(f),m_t^{\mathbb{X}^f})$ and $(C^*(g),m_t^{\mathbb{X}^g})$ the $\Omega B As$-algebras constructed in section~\ref{geo:ss:ombas-str-Morse}. 

\begin{theorem} \label{geo:th:ombas-morph}
Let $(\mathbb{Y}_n)_{n \geqslant 1}$ be a choice of perturbation on the moduli spaces $\mathcal{CT}_n$ that is admissible w.r.t. $\mathbb{X}^f$ and $\mathbb{X}^g$.
Defining for every $n$ and $t_g \in SCRT_n$ the operations $\mu_{t_g}$ as
\begin{align*}
\mu_{t_g}^{\mathbb{Y}} : C^*(f) \otimes \cdots \otimes C^*(f) &\longrightarrow C^*(g) \\
x_1 \otimes \cdots \otimes x_n &\longmapsto \sum_{|y|= \sum_{i=1}^n|x_i| + |t_g|} \# \mathcal{CT}_{t_g}^\mathbb{Y}(y ; x_1,\cdots,x_n) \cdot y \ .
\end{align*} 
they fit into an \ombas -morphism $\mu^{\mathbb{Y}} : (C^*(f),m_t^{\mathbb{X}^f}) \rightarrow (C^*(g),m_t^{\mathbb{X}^g})$.
\end{theorem}

Again, the collection of operations $\{ \mu_{t_g} \}$ does not exactly define an $\Omega B As$-morphism but rather a \emph{twisted $\Omega B As$-morphism}. In the case when $M$ is odd-dimensional, this twisted $\Omega B As$-morphism is exactly an $\Omega B As$-morphism between two $\Omega B As$-algebras.
All sign computations are detailed in section~\ref{geo:s:signs-or}. If we want to go back to the more classical algebraic framework of \Ainf -algebras, an \Ainf -morphism between the induced \Ainf -algebra structures on the Morse cochains is simply obtained under the morphism of operadic bimodules $\infmor \rightarrow \Omega B As - \mathrm{Morph}$.

\section{Transversality} \label{geo:s:transversality}

The goal of this section is to prove Theorems~\ref{geo:th:exist-admissible-perturbation-data-Tn} and~\ref{geo:th:existence-perturb-CTn}. In this regard, we recall at first the parametric transversality lemma and then build an admissible choice of perturbation data $(\mathbb{X}_n)_{n \geqslant 2}$ on the moduli spaces $\mathcal{T}_n$, proceeding by induction on the number of internal edges $e(t)$ of a stable ribbon tree $t$. It moreover appears in our construction that all arguments adapt nicely to the framework of two-colored trees and admissible choices of perturbation data $(\mathbb{Y}_n)_{n \geqslant 1}$ on the moduli spaces $\mathcal{CT}_n$. 

\subsection{Parametric transversality lemma} \label{geo:sss:functional-analysis}

We begin by recalling Smale's generalization of the classical Sard theorem. See~\cite{smale-sard} or~\cite{mcduff-salamon} for a detailed proof :

\begin{theorem}[Sard-Smale theorem]
Let $X$ and $Y$ be separable Banach manifolds. Suppose that $f : X \rightarrow Y$ is a Fredholm map of class $C^l$ with $l \geqslant \mathrm{max}(1 , \mathrm{ind}(f) + 1 )$. Then the set $Y_{reg}(f)$ of regular values of $f$ is residual in $Y$ in the sense of Baire.
\end{theorem}

\noindent This theorem implies in particular the following corollary in transversality theory, that will constitute the cornerstone of our proof of Theorem~\ref{geo:th:exist-admissible-perturbation-data-Tn} : 

\begin{corollary}[Parametric transversality lemma]
Let $\mathfrak{X}$ be a Banach space, $M$ and $N$ two finite-dimensional manifolds and $S \subset N$ a submanifold of $N$. Suppose that $f : \mathfrak{X} \times M \rightarrow N$ is a map of class $C^l$ with $l \geqslant \mathrm{max}(1 , \mathrm{dim}(M) + \mathrm{dim}(S) - \mathrm{dim}(N) + 1 )$ and that it is transverse to $S$. Then the set
\[ \mathfrak{X}_{\pitchfork S} := \{ \mathbb{X} \in \mathfrak{X} \text{ such that } f_{\mathbb{X}} \pitchfork S \} \]
is residual in $\mathfrak{X}$ in the sense of Baire.
\end{corollary}

\begin{proof}
The map $f$ being transverse to $S$, the inverse image $f^{-1}(S)$ is a Banach submanifold of $\mathfrak{X} \times M$. Consider the standard projection $p_\mathfrak{X} : \mathfrak{X} \times M \rightarrow \mathfrak{X}$ and denote $\pi := p_\mathfrak{X}|_{f^{-1}(S)}$. Following Lemma 19.2 in~\cite{abraham-transversal}, this map is Fredholm and has index $\mathrm{dim}(M) + \mathrm{dim}(S) - \mathrm{dim}(N) $. Moreover, drawing from an argument in section 3.2. of~\cite{mcduff-salamon}, there is an equality $\mathfrak{X}_{reg}(\pi) = \mathfrak{X}_{\pitchfork S}$. One can then conclude by applying the Sard-Smale theorem to the map $\pi$.
\end{proof}

\subsection{Proof of theorem~\ref{geo:th:exist-admissible-perturbation-data-Tn}} \label{geo:ss:existence-of-admissible}

\subsubsection{The case $e(t)=0$} \label{geo:sss:etocase}

If $e(t) = 0$, the tree $t$ is a corolla. Fix an integer $l$ such that 
\[ l \geqslant \mathrm{max} \left( 1 , e(t) + |y| - \sum_{i=1}^n|x_i| + 1 \right) \ . \]
We define $C^l$-choices of perturbation data in a similar fashion to smooth choices of perturbation data.
A $C^l$-choice of perturbation data $\mathbb{X}_t$ on $\mathcal{T}_n(t)$ then simply corresponds to a $C^l$-choice of perturbation datum on each external edge of $t$. Define the parametrization space
\[ \mathfrak{X}_t^l  := \{ \text{$C^l$-perturbation data $\mathbb{X}_t$ on the moduli space $\mathcal{T}_n(t)$} \} \ . \]
This parametrization space is a Banach space. The linear combination of choices of perturbation data is simply defined as the linear combination of each perturbation datum $\mathbb{X}_{t,e}$ with $e$ an external edge of $t$. The vector space $\mathfrak{X}_t^l$ is moreover Banach as each perturbation datum $\mathbb{X}_{t,e}$ vanishes away from a length 1 segment in $D_e$.

Given critical points $y$ and $x_1, \dots , x_n$, introduce the $C^l$-map
\[ \phi_t : \mathfrak{X}_t^l \times \mathcal{T}_n(t) \times W^S(y) \times W^U(x_1) \times \cdots \times W^U(x_n) \longrightarrow M^{\times n+1} \ , \]
such that for every $\mathbb{X}_t \in \mathfrak{X}_t^l$, $\phi_t ( \mathbb{X}_t , \cdot ) = \phi_{\mathbb{X}_t}$. Note that we should in fact write $\phi_t^{y,x_1,\dots,x_n}$ as the domain of $\phi_t$ depends on $y,x_1,\dots,x_n$. The map $\phi_t$ is then a submersion. This is proven in Lemma~7.3. of~\cite{abouzaid-plumbings} and Abouzaid explains it informally in the following terms : "[this lemma] is the infinitesimal version of the fact that perturbing the gradient flow equation on a bounded subset of an edge integrates to an essentially arbitrary diffeomorphism".

In particular the map $\phi_t$ is transverse to the diagonal $\Delta \subset M^{\times n+1}$. Applying the parametric transversality theorem of subsection~\ref{geo:sss:functional-analysis},  there exists a residual set $\mathfrak{Y}_t^{l ; y , x_1 , \dots x_n} \subset \mathfrak{X}_t^l$  such that for every choice of perturbation data $\mathbb{X}_t \in \mathfrak{Y}_t^{l ; y , x_1 , \dots x_n}$ the map $\phi_{\mathbb{X}_t}$ is transverse to the diagonal $\Delta \subset M^{\times n+1}$. Considering the intersection
\[ \mathfrak{Y}_t^l := \bigcap_{y,x_1,\dots,x_n} \mathfrak{Y}_t^{y , x_1 , \dots x_n} \subset \mathfrak{X}_t \]
which is again residual, any $\mathbb{X}_t \in \mathfrak{Y}_t^l$ yields a $C^l$-choice of perturbation data on $\mathcal{T}_n(t)$ such that all the maps $\phi_{\mathbb{X}_t}$ are transverse to the diagonal $\Delta \subset M^{\times n+1}$. It remains to prove this statement in the smooth case. 

\subsubsection{Achieving smoothness à la Taubes}

Using an argument drawn from section 3.2. of~\cite{mcduff-salamon} and attributed to Taubes, we now prove that the set 
\[ \mathfrak{Y}_t := 
\left\{ \begin{array}{l}
 \text{smooth choices of perturbation data $\mathbb{X}_t$ on $\mathcal{T}_n(t)$ such that} \\\text{all the maps $\phi_{\mathbb{X}_t}$ are transverse to the diagonal $\Delta \subset M^{\times n+1}$ } \end{array} \right\} \]
is residual in the Fréchet space
\[ \mathfrak{X}_t  := \{ \text{smooth choices of perturbation data $\mathbb{X}_t$ on $\mathcal{T}_n(t)$} \} \ . \]

Choose an exhaustion by compact sets $L_0 \subset L_1 \subset L_2 \subset \cdots$ of the space $\mathcal{T}_n(t) \times W^S(y) \times W^U(x_1) \times \cdots \times W^U(x_n)$. Define 
\[ \mathfrak{Y}_{t,L_m} := 
\left\{ \begin{array}{l}
 \text{smooth choices of perturbation data $\mathbb{X}_t$ on $\mathcal{T}_n(t)$ such that} \\\text{all maps $\phi_{\mathbb{X}_t}$ are transverse on $L_m$ to the diagonal of $M^{\times n+1}$ } \end{array} \right\} \]
and note that
\[ \mathfrak{Y}_t = \bigcap_{m=0}^{+\infty} \mathfrak{Y}_{t,L_m} \ . \]
We will prove that each $\mathfrak{Y}_{t,L_m} \subset \mathfrak{Y}_{t}$ is open and dense in $\mathfrak{X}_{t}$ to conclude that $\mathfrak{Y}_{t}$ is indeed residual.

Fix $m \geqslant 0$. To prove that the set $\mathfrak{Y}_{t,L_m}$ is open in $\mathfrak{X}_{t}$ it suffices to prove that for every $l$, the set $\mathfrak{Y}_{t,L_m}^l$ is open in $\mathfrak{X}_{t}^l$, where $\mathfrak{Y}_{t,L_m}^l$ is defined by replacing "smooth" by "$C^l$" in the definition of $\mathfrak{Y}_{t,L_m}$. This last result is a simple consequence of the fact that "being transverse on a compact subset" is an open property : if the map $\phi_{\mathbb{X}_t^0}$ is transverse on $L_m$ to the diagonal $\Delta \subset M^{\times n+1}$ then for $\mathbb{X}_t \in \mathfrak{X}_{t}^l$ sufficiently close to $\mathbb{X}_t^0$ the map $\phi_{\mathbb{X}_t}$ is again transverse on $L_m$ to the diagonal on $L_m$. 

Let now $\mathbb{X}_t \in \mathfrak{X}_{t}$. As $\mathbb{X}_t \in \mathfrak{X}_{t}^l$ and the set $\mathfrak{Y}_t^l$ is dense in $\mathfrak{X}_t^l$, there exists a sequence $\mathbb{X}_t^l \in \mathfrak{Y}_{t}^l$ such that for all $l$
\[ ||  \mathbb{X}_t - \mathbb{X}_t^l ||_{\mathcal{C}^l} \leqslant 2^{-l} \ . \]
Note that $\mathbb{X}_t^l \in \mathfrak{Y}_{t,L_m}^l$.
Now since the set $\mathfrak{Y}_{t,L_m}^l$ is open in $\mathfrak{X}_{t}^l$ for the $\mathcal{C}^l$-topology, there exists $\varepsilon_l > 0$ such that for all $\mathbb{X}_t^{\prime l} \in \mathfrak{X}_{t}^l$ if
\[ ||  \mathbb{X}_t^l - \mathbb{X}_t^{\prime  l} ||_{\mathcal{C}^l} \leqslant \mathrm{min}( 2^{-l} , \varepsilon_l ) \ , \]
then $\mathbb{X}_t^{\prime  l} \in \mathfrak{Y}_{t,L_m}^l$. Choosing $\mathbb{X}_t^{\prime  l}$ to be smooth, this yields a sequence of smooth choices of perturbation data lying in $\mathfrak{Y}_{t,L_m}$ and converging to $\mathbb{X}_t$, which concludes the proof. 

\subsubsection{Induction step and conclusion} 

Let $k \geqslant 0$ and suppose that we have constructed an admissible choice of perturbation data $(\mathbb{X}_t^0)_{e(t) \leqslant k}$. This notation should not be confused with the notation $(\mathbb{X}_i)_{i \leqslant k}$ : the former corresponds to a choice of perturbation data on the strata $\mathcal{T}(t)$ of dimension $\leqslant k$ while the latter corresponds to a choice of perturbation data on the moduli spaces $\mathcal{T}_i$ with $i \leqslant k$. Let $t$ be a stable ribbon tree with $e(t) = k + 1$. We want to construct a choice of perturbation data $\mathbb{X}_t$ on $\mathcal{T}_n(t)$ which is smooth, gluing-compatible and such that each map $\phi_{\mathbb{X}_t}$ is transverse to the diagonal $\Delta \subset M^{\times n +1}$.

Under a choice of identification $\overline{\mathcal{T}}_n(t) \simeq [ 0 , +\infty ]^{e(t)}$, define $\underline{\mathcal{T}_n}(t) \subset \overline{\mathcal{T}}_n(t)$ as the inverse image of $[ 0 , +\infty [^{e(t)}$. Introduce the parametrization space 
\[  \mathfrak{X}_t^l := 
  \left\{\begin{array}{l}
         \text{$C^l$-perturbation data $\mathbb{X}_t$ on $\underline{\mathcal{T}_n}(t)$ such that} \\
         \text{$\mathbb{X}_t|_{\mathcal{T}(t')} = \mathbb{X}_{t'}^0$ for all $t' \in coll(t)$ and such that} \\
         \text{$\underset{{l_e \rightarrow + \infty}}{\mathrm{lim}} \mathbb{X}_t = \mathbb{X}^0_{t_1} \#_e \mathbb{X}^0_{t_2}$ for all $e\in E(t)$} \\
  \end{array}\right\} \ , \]
where $t_1 \#_e t_2 = t$, and $\mathrm{lim}_{l_e \rightarrow + \infty} \mathbb{X}_t = \mathbb{X}^0_{t_1} \#_e \mathbb{X}^0_{t_2}$ denotes the gluing-compatibility condition described in section~\ref{geo:ss:compact-ribbon-Morse}.
Following~\cite{mescher-morse} this parametrization space is an affine space which is Banach. One can indeed show that the $l_e \rightarrow + \infty$ conditions imply that each $\mathbb{X}_t \in \mathfrak{X}_t^l$ is bounded in the $C^l$-norm, and that the $C^l$-norm is thus well defined on $\mathfrak{X}_t^l$ although $\underline{\mathcal{T}_n}(t)$ is not compact.

Consider the $C^l$-map
\[ \phi_t : \mathfrak{X}_t^l \times \mathcal{T}_n(t) \times W^S(y) \times W^U(x_1) \times \cdots \times W^U(x_n) \longrightarrow M^{\times n+1} \ . \]
Using the same argument as in subsection~\ref{geo:sss:etocase}, the map $\phi_t$ is again transverse to the diagonal $\Delta \subset M^{\times n+1}$. 
Applying the parametric transversality theorem and proceeding as in the case $e(t)=0$, there exists a residual set $\mathfrak{Y}_t^l \subset \mathfrak{X}_t^l$ such that for every choice of perturbation data $\mathbb{X}_t \in \mathfrak{Y}_t^l$ the map $\phi_{\mathbb{X}_t}$ is transverse to the diagonal $\Delta \subset M^{\times n+1}$. Using the previous argument à la Taubes, we can moreover prove the same statement in the smooth context. By definition of the parametrization spaces $\mathfrak{X}_t$ this construction yields indeed an admissible choice of perturbation data $(\mathbb{X}_t)_{e(t) \leqslant k+1}$, which concludes the proof of Theorem~\ref{geo:th:exist-admissible-perturbation-data-Tn} by induction.

\section{Signs, orientations and gluing} \label{geo:s:signs-or}

We now complete and conclude the proofs of Theorems~\ref{geo:th:ombas-alg}~and~\ref{geo:th:ombas-morph}, by expliciting all orientations conventions on the moduli spaces of Morse gradient trees and computing the signs involved therein. We use to this extent the ad hoc formalism of signed short exact sequences of vector bundles. Particular attention will be paid to the behaviour of orientations under gluing in our proof.

\subsection{More on signs and orientations} \label{geo:ss:more-signs}

\subsubsection{Additional tools for orientations} \label{geo:sss:additional-tools} 

Consider a short exact sequence of vector spaces
\[ 0 \longrightarrow V_2 \longrightarrow W \longrightarrow V_1 \longrightarrow 0 \ . \]
It induces a direct sum decomposition $W = V_1 \oplus V_2 $. Suppose that the vector spaces $W$, $V_1$ and $V_2$ are oriented. We denote $(-1)^{\varepsilon}$ the sign obtained by comparing the orientation on $W$ to the one induced by the direct sum $V_1 \oplus V_2$. We will then say that the short exact sequence has sign $(-1)^{\varepsilon}$. In particular, when $(-1)^{\varepsilon}= 1$, we will say that the short exact sequence is \emph{positive}.

Now, consider two short exact sequences 
\[ 0 \longrightarrow V_2 \longrightarrow W \longrightarrow V_1 \longrightarrow 0 \ \text{ and } \ 0 \longrightarrow V_2' \longrightarrow W' \longrightarrow V_1' \longrightarrow 0 \ , \]
of respective signs $(-1)^{\varepsilon}$ and $(-1)^{\varepsilon'}$. Then the short exact sequence obtained by summing them
\[ 0 \longrightarrow V_2 \oplus V_2' \longrightarrow W \oplus W' \longrightarrow V_1 \oplus V_1' \longrightarrow 0 \ , \]
has sign $(-1)^{\varepsilon + \varepsilon' + \mathrm{dim}(V_1') \mathrm{dim}(V_2) }$. Indeed, the direct sum decomposition writes as
\[ W \oplus W' = (-1)^{\varepsilon} (V_1 \oplus V_2) \oplus(-1)^{\varepsilon'} (V_1' \oplus V_2') \simeq (-1)^{\varepsilon + \varepsilon' + \mathrm{dim}(V_1') \mathrm{dim}(V_2) } V_1 \oplus V_1' \oplus V_2 \oplus V_2' \ . \]

\subsubsection{Orientation and transversality} \label{geo:sss:orientation-transversality}

Given two manifolds $M,N$, a codimension $k$ submanifold $S \subset N$ and a smooth map
\[ \phi : M \longrightarrow N \]
which is tranverse to $S$, the inverse image $\phi^{-1}(S)$ is a codimension $k$ submanifold of $M$. Moreover, choosing a complementary $\nu_S$ to $TS$, the transversality assumption yields the following short exact sequence of vector bundles
\[ 0 \longrightarrow T \phi^{-1} (S) \longrightarrow T M |_{ \phi^{-1}(S)} \underset{d \phi}{\longrightarrow} \nu_S \longrightarrow 0 \ . \]
Suppose now that $M$, $N$ and $S$ are oriented. The orientations on $N$ and $S$ induce an orientation on $\nu_S$. The submanifold $\phi^{-1}(S)$ is then oriented by requiring that the previous short exact sequence be positive. We will refer to this choice of orientation as the \emph{natural orientation on $\phi^{-1}(S)$}. 

In the particular case of two submanifolds $S$ and $R$ of $M$ which intersect transversely, we will use the inclusion map $S \hookrightarrow M$, which is transverse to $R \subset M$, to define the intersection $S \cap R$. The orientation will then be defined using the positive short exact sequence 
\[ 0 \longrightarrow T ( S \cap R) \longrightarrow TS |_{S \cap R} \longrightarrow \nu_R \longrightarrow 0 \ , \]
or equivalently with the direct sum decomposition
\[ TS = \nu_R \oplus T ( S \cap R ) \ . \]
The intersection $R \cap S$ (in contrast to $S \cap R$) is oriented by interchanging $S$ and $R$ in the above discussion. The two orientations on the intersection differ then by a $(-1)^{\mathrm{codim}(S)\mathrm{codim}(R)}$ sign.

\subsection{Basic moduli spaces in Morse theory and their orientations} \label{geo:ss:basic-mod-space-Morse}

\subsubsection{Orienting the unstable and stable manifolds} \label{geo:sss:orient-stable-unstable}

Recall that for a critical point $x$ of a Morse function $f$, its unstable and stable manifolds are respectively defined as
\begin{align*}
W^U(x) &:= \{ z \in M , \ \lim_{s \rightarrow - \infty} \phi^s(z) = x  \} \\
W^S(x) &:= \{ z \in M , \ \lim_{s \rightarrow + \infty} \phi^s(z) = x  \} \ ,
\end{align*}
where we denote $\phi^s$ the flow of $-\nabla f$, and its degree is defined as $|x| := \mathrm{dim} (W^S(x))$.

The unstable and stable manifolds are respectively diffeomorphic to a $(d-|x|)$-dimensional ball and a $|x|$-dimensional ball. They are hence orientable. They intersect moreover transversely in a unique point, which is $x$. Assume now that the manifold $M$ is orientable and oriented. We choose for the rest of this section an arbitrary orientation on $W^U(x)$, and endow $W^S(x)$ with the unique orientation such that the concatenation of orientations $or_{W^U(x)} \wedge or_{W^S(x)}$ at $x$ coincides with the orientation $or_M$.

\subsubsection{Orienting the moduli spaces $\mathcal{T}(y;x)$} \label{geo:sss:orient-mod-space-Tyx}

For two critical points $x \neq y$, the moduli spaces of negative gradient trajectories $\mathcal{T}(y;x)$ can be defined in two ways. The first point of view hinges on the fact that \R\ acts on $W^S(y) \cap W^U(x)$, by defining $s \cdot p = \phi^s(p)$ for $s \in \R$ and $p \in W^S(y) \cap W^U(x)$. The moduli space $\mathcal{T}(y;x)$ is then defined by considering the quotient associated to this action, i.e. by defining $\mathcal{T}(y;x) := W^S(y) \cap W^U(x) / \R$. The second point of view is to consider the transverse intersection with the level set of a regular value $a$,
\[ \mathcal{T}(y;x) := W^S(y) \cap W^U(x) \cap f^{-1}(a) \ . \]

Using this description, and coorienting the level set $f^{-1}(a)$ with $- \nabla f$, the spaces $\mathcal{T}(y;x)$ can easily be oriented with the formalism of section~\ref{geo:sss:orientation-transversality} on transverse intersections : 
\[ TW^S(y) \simeq TW^S(x) \oplus T \left( W^S(y) \cap W^U(x) \right) \simeq TW^S(x) \oplus - \nabla f \oplus T \mathcal{T}(y;x) \ . \]
Note that the space $W^S(y) \cap W^U(x)$ consists in a union of negative gradient trajectories $\gamma : \R \rightarrow M$. We will therefore use the notation $\dot{\gamma}$ for $- \nabla f$, which will become handy in the next section.

We point out that the moduli spaces $\mathcal{T}(y;x)$ are constructed in a different way than the moduli spaces $\mathcal{T}_t(y ; x_1 , \dots , x_n)$ : they cannot naturally be viewed as an arity 1 case of the moduli spaces of gradient trees. This observation will be of importance in our upcoming discussion on signs for the $\Omega B As$-algebra structure on the Morse cochains.

Finally, the moduli spaces $\mathcal{T}(y;x)$ are manifolds of dimension
\[ \mathrm{dim}(\mathcal{T}(y;x)) = |y| -|x| - 1 \ , \]
which can be compactified to manifolds with corners $\overline{\mathcal{T}}(y ;x )$, by allowing convergence towards broken negative gradient trajectories. See for instance~\cite{wehrheim-morse}.
In the case where they are 1-dimensional, their boundary is given by the signed union
\[ \partial \overline{\mathcal{T}}(y;x) = \bigcup_{z \in \mathrm{Crit}(f)} - \mathcal{T}(y;z) \times \mathcal{T}(z;x) \ . \]
We moreover recall from section~\ref{geo:ss:conventions} that we work under the convention $\mathcal{T}(x;x) = \emptyset$.

\subsubsection{Compactifications of the unstable and stable manifolds} \label{geo:sss:compact-stable-unstable}

Using the moduli spaces $\mathcal{T}(y;x)$, we can now compactify the manifolds $W^S(y)$ and $W^U(x)$ to compact manifolds with corners $\overline{W^S}(y)$ and $\overline{W^U}(x)$. See~\cite{hutchings-floer} for instance. With our choices of orientations detailed in the previous section, the top dimensional strata in their boundary are given by 
\begin{align*}
\partial \overline{W^S}(y) &= \bigcup_{z \in \mathrm{Crit}(f)} (-1)^{|z|+1} W^S(z) \times \mathcal{T}(y;z)   \ , \\
\partial \overline{W^U}(x) &= \bigcup_{z \in \mathrm{Crit}(f)} (-1)^{(d-|z|)(|x|+1)} W^U(z) \times \mathcal{T}(z;x) \ ,
\end{align*}
where $d$ is the dimension of the ambient manifold $M$.

The pictures in the neighborhood of the critical point $z$ are represented in figure~\ref{geo:fig:compact-unstable-stable}.
For instance, in the case of $\partial \overline{W^S}(y)$, an element of $W^S(y)$ is seen as lying on a negative semi-infinite trajectory converging to $y$, and an outward-pointing vector to the boundary is given by $- \dot{\gamma}$. We hence have that
\[ - \dot{\gamma} \oplus T W^S(z) \oplus T \mathcal{T}(y;z)  = (-1)^{|z|}    T W^S(z) \oplus - \dot{\gamma} \oplus T \mathcal{T}(y;z) = (-1)^{|z|+1}  T W^S(y) \ . \]

\begin{figure}[h] 
    \centering
    \begin{subfigure}{0.4\textwidth}
    \centering
       \compactificationWSy 
    \end{subfigure} ~
    \begin{subfigure}{0.4\textwidth}
    \centering
        \compactificationWUx \ .
    \end{subfigure} 
    \caption{} \label{geo:fig:compact-unstable-stable}
\end{figure}

\subsubsection{Euclidean neighborhood of a critical point} \label{geo:sss:morse-theory}

Following~\cite{wehrheim-morse}, we will assume in the rest of this part that the pair (Morse function,metric) on the manifold $M$ is Euclidean. Denote $B^k_\delta := \{ x \in \R^k, \ |x| < \delta \}$. Such a pair is said to \emph{Euclidean} if it is Morse-Smale and is such that for each critical point $z \in \mathrm{Crit}(f)$ there exists a local chart $\phi : B_\delta^{d-|z|} \times B_\delta^{|z|} \tilde{\longrightarrow} U_z \subset M$, such that $\phi(0) = z$ and such that the function $f$ and the metric $g$ read as 
\begin{align*}
f ( x_1 , \dots , x_{n-|z|} , y_1 , \dots , y_{|z|} ) &= f(p) - \frac{1}{2} ( x_1^2 + \cdots + x_{n-|z|}^2) + \frac{1}{2} ( y_1^2 + \cdots + y_{|z|}^2) \\
g &= \sum_{i=1}^{n-|z|} dx_i\otimes dx_i + \sum_{i=1}^{|z|} dy_i \otimes dy_i 
\end{align*}
in the chart $\phi$. In this chart, we then have that 
\begin{align*}
W^U(z) &:= \{ y_1 = \cdots = y_{|z|} = 0 \} \\ 
W^S(z) &:= \{ x_1 = \cdots = x_{n-|z|} = 0 \} \ ,
\end{align*}
and $M = W^U(z) \times W^S(z)$. Hence any point of $U_z$ can be uniquely written as a sum $x+y$ where $x \in W^U(z)$ and $y \in W^S(z)$. Choosing now $s \in \R$ such that the the image of $x+y$ under the Morse flow map $\phi^{s}$ still lies in $U_z$, we have that
\[ \phi^s (x+y) = e^s x + e^{-s} y \ . \]
These observations will reveal crucial in the proof of subsection~\ref{geo:sss:gluing-and-orientations}.

\subsection{Preliminaries for section~\ref{geo:ss:twisted-ainf-alg-Morse}} \label{geo:ss:preliminaries}

\subsubsection{Counting the points on the boundary of an oriented 1-dimensional manifold} \label{geo:sss:counting-boundary-points}

Consider an oriented 1-dimensional manifold with boundary. Then its boundary $\partial M$ is oriented. Assume it can be written set-theoretically as a disjoint union
\[ \partial M = \bigsqcup_i N_i \ . \]
Suppose now that each $N_i$ comes with its own orientation, and write $(-1)^{\dagger_i}$ for the sign obtained by comparing this orientation to the boundary orientation. As oriented manifolds, the union writes as
\[ \partial M = \bigsqcup_i (-1)^{\dagger_i} N_i \ . \]
The $N_i$ being 0-dimensional, they can be seen as collections of points each coming with a $+$ or $-$ sign.
Noticing that an orientable 1-dimensional manifold with boundary is either a segment or a circle, and writing $\#N_i$ for the signed count of points of $N_i$, the previous equality finally implies that
\[ \sum (-1)^{\dagger_i} \# N_i = 0 \ . \]
This basic observation is key to constructing most algebraic structures arising in symplectic topology (and in particular Morse theory). 

For instance, for a critical point $x$, counting the boundary points of the 1-dimensional manifolds $\overline{\mathcal{T}}(y;x)$ proves that
\[ \partial^{Morse} \circ \partial^{Morse} (x) = \sum_{\substack{y \in \mathrm{Crit}(f) \\ |y| = |x| + 2}} \sum_{\substack{z \in \mathrm{Crit}(f) \\ |z| = |x| + 1}} \# \mathcal{T}(y;z) \# \mathcal{T}(z;x) \cdot y = 0 \ . \]
The equations for $\Omega B As$-algebras and $\Omega B As$-morphisms will be proven using this method in the following two subsections.

\subsubsection{Reformulating the $\Omega B As$-equations} \label{geo:sss:reformulating-equations}

We fix for each $t \in SRT_n$ an orientation $\omega_t$. Given a $t \in SRT_n$ the orientation $\omega_t$ defines an orientation of the moduli space $\mathcal{T}_n(t)$, and we write moreover $m_t$ for the operations $(t, \omega)$. The $\Omega B As$-equations for an $\Omega B As$-algebra then read as
\[ [ \partial , m_t ] = \sum_{t' \in coll(t)} (-1)^{\dagger_{\Omega B As}} m_{t'} + \sum_{t_1 \#_i t_2 = t} (-1)^{\dagger_{\Omega B As}} m_{t_1} \circ_i m_{t_2} \ , \]
where the notations for trees are as defined previously. The signs $(-1)^{\dagger_{\Omega B As}}$ are obtained as in section~\ref{alg:ss:signs-ombas}, by computing the signs of $\mathcal{T}_n(t')$ and $\mathcal{T}_{i_1 + 1 + i_3}(t_1) \times_i \mathcal{T}_{i_2}(t_2)$ in the boundary of $\mathcal{T}_n(t)$. We will not need to compute their explicit value, and will hence keep this useful notation $(-1)^{\Omega B As}$ to refer to them.

\subsubsection{Twisted \Ainf -algebras and twisted $\Omega B As$-algebras} \label{geo:sss:twisted-structure}

It is clear using this counting method, that the operations $m_t$ of section~\ref{geo:ss:ombas-str-Morse} will endow the Morse cochains $C^*(f)$ with a structure of $\Omega B As$-algebra over $\Z / 2$. Working over integers will prove more difficult, and we will prove a weaker result in this case. We introduce to this extent the notion of twisted \Ainf -algebras and twisted $\Omega B As$-algebras.

\begin{definition} \label{geo:def:twisted-ombas}
A \emph{twisted \Ainf -algebra} is a dg-\Z -module $A$ endowed with two different differentials $\partial_1$ and $\partial_2$, and a sequence of degree $2-n$ operations $m_n : A^{\otimes n} \rightarrow A$ such that
\[ [ \partial , m_n ] = - \sum_{\substack{i_1+i_2+i_3=n \\ 2 \leqslant i_2 \leqslant n-1}} (-1)^{i_1 + i_2i_3} m_{i_1+1+i_3} (\ide^{\otimes i_1} \otimes m_{i_2} \otimes \ide^{\otimes i_3} ) \ , \]
where $[ \partial , \cdot ]$ denotes the bracket for the maps $ (A^{\otimes n} , \partial_1) \rightarrow (A , \partial_2)$.
A \emph{twisted $\Omega B As$-algebra} is defined similarly.
\end{definition}

We make explicit the formulae obtained by evaluating the $\Omega B As$-equations on $A^{\otimes n}$, as we will need them in our next proof :
\begin{align*}
&- \partial_2 m_t (a_1 , \dots , a_n) + (-1)^{|t| + \sum_{j=1}^{i-1}|a_j|} m_t ( a_1 , \dots , a_{i-1} , \partial_1 a_i , a_{i+1} , \dots , a_n) \\
&+ \sum_{t_1 \# t_2 = t} (-1)^{\dagger_{\Omega B As} + |t_2| \sum_{j=1}^{i_1} |a_j|} m_{t_1} (a_1 , \dots , a_{i_1} , m_{t_2} (a_{i_1 + 1} , \dots , a_{i_1 + i_2} ) , a_{i_1 + i_2 + 1} , \dots , a_n) \\
&+ \sum_{t' \in coll(t)} (-1)^{\dagger_{\Omega B As}} m_{t'} (a_1 , \dots , a_n) \\
&= 0 \ .
\end{align*}
We refer to them as "twisted", as these algebras will occur in the upcoming lines by setting $\partial_2 := (-1)^{\sigma} \partial_1$, that is by simply twisting the differential $\partial_1$ by a specific sign.

Note that these two definitions cannot be phrased in terms of operads, as $\mathrm{Hom} ( (A , \partial_1) , (A , \partial_2) )$ is an $( \End_{(A , \partial_1)} , \End_{(A , \partial_2)} )$-operadic bimodule but is NOT an operad : the composition maps on $\mathrm{Hom} ( (A , \partial_1) , (A , \partial_2) )$ are associative, but they fail to be compatible with the differential $[ \partial , \cdot ]$. As a result, a  twisted \Ainf -algebra cannot be described as a morphism of operads from \Ainf\ to $\mathrm{Hom} ( (A , \partial_1) , (A , \partial_2) )$. However, a twisted $\Omega B As$-algebra structure always transfers to a twisted \Ainf -algebra structure. Indeed, while the functorial proof of~\ref{alg:sss:op-ainf-to-ombas} does not work anymore, we point out that the morphism of operads $\Ainf \rightarrow \Omega B As$ still contains the proof that a sequence of operations $m_t$ defining a twisted $\Omega B As$-algebra structure on $A$ can always be arranged in a sequence of operations $m_n$ defining a twisted \Ainf -algebra structure on $A$. 

\subsubsection{The maps $\psi_{e_i,\mathbb{X}_t}$} \label{geo:sss:gluing-choices-perturb}

Consider again a stable ribbon tree $t$ and order its external edges clockwise, starting with $e_0$ at the outgoing edge. Given a choice of perturbation data $\mathbb{X}_t$, we illustrate in figure~\ref{geo:fig:different-flow-maps} a mean to visualize the map 
\[ \phi_{\mathbb{X}_t} : \mathcal{T}_n(t) \times W^S(y) \times W^U(x_1) \times \cdots \times W^U(x_n) \longrightarrow M^{\times n+1}  \]
defined in section~\ref{geo:ss:mod-space-pert-Morse-tree}. We introduce a family of maps defined in a similar fashion. Consider $e_i$ an incoming edge of $t$. Define the map
\[ \psi_{e_i,\mathbb{X}_t} : \mathcal{T}_n(t) \times W^S(y) \times W^U(x_1) \times \cdots \times \widehat{W^U(x_i)} \times \cdots \times \cdots \times W^U(x_n) \longrightarrow M^{\times n}  \]
to be the map which for a fixed metric tree $T$ takes a point of a $W^U(x_j)$ for $j \neq i$ to the point in $M$ obtained by following the only non-self crossing path from the time $-1$ point on $e_j$ to the time $-1$ point on $e_i$ in $T$ through the perturbed gradient flow maps associated to $\mathbb{X}_T$, and which takes a point of $W^s(y)$ to the point in $M$ obtained by following the only non-self crossing path from the time $1$ point on $e_0$ to the time $-1$ point on $e_i$ in $T$ through the perturbed gradient flow maps associated to $\mathbb{X}_T$. The map $\psi_{e_0,\mathbb{X}_t}$ is defined similarly for the outgoing edge $e_0$. These two definitions are two be understood as depicted on two examples in figure~\ref{geo:fig:different-flow-maps}. 

\begin{figure}[h]
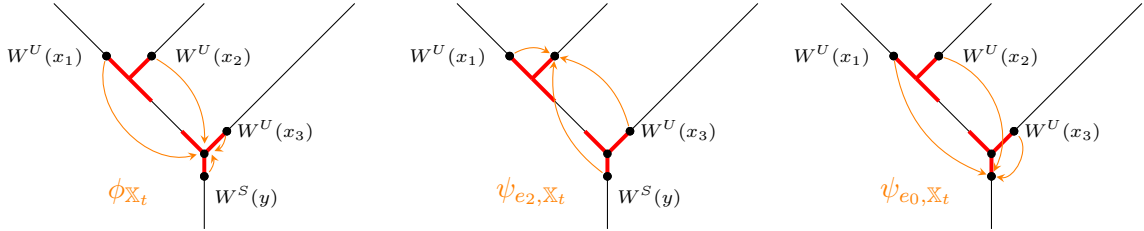
 
    \centering
    \begin{subfigure}{0.3\textwidth}
    \centering
       \examplemappsiun
    \end{subfigure} ~
    \begin{subfigure}{0.3\textwidth}
    \centering
        \examplemappsideux
    \end{subfigure} 
    \begin{subfigure}{0.3\textwidth}
    \centering
        \examplemappsitrois
    \end{subfigure} 
    \caption{Representations of a map $\phi_{\mathbb{X}_t}$, a map $\psi_{e_2,\mathbb{X}_t}$ and a map $\psi_{e_0,\mathbb{X}_t}$} \label{geo:fig:different-flow-maps}
\end{figure}

\subsection{The twisted $\Omega B As$-algebra structure on the Morse cochains} \label{geo:ss:twisted-ainf-alg-Morse}

\subsubsection{Summary of the proof of Theorem~\ref{geo:th:ombas-alg}} \label{geo:sss:result-ainf-alg-morse}

\begin{definition}
\begin{enumerate}[label=(\roman*)]
\item We define $\widetilde{\mathcal{T}}_{t}^\mathbb{X}(y ; x_1,\dots,x_n)$ to be the oriented manifold $\mathcal{T}_{t}^\mathbb{X}(y ; x_1,\dots,x_n)$ whose natural orientation has been twisted by a sign of parity
\[ \sigma (t ; y ; x_1 , \dots , x_n) := dn ( 1 + |y| + |t| ) + |t| |y| + d \sum_{i=1}^n |x_i| (n-i) \ . \]
\item Similarly, we define $\widetilde{\mathcal{T}}(y;x)$ to be the oriented manifold $\mathcal{T}(y;x)$ whose natural orientation has been twisted by a sign of parity
\[ \sigma (y ; x) := 1 \ . \]
\end{enumerate}
\end{definition}

The operations $m_t$ and the differential on $C^*(f)$ are then defined as
\begin{align*}
m_{t} (x_1 , \dots , x_n) &=  \sum_{|y|= \sum_{i=1}^n|x_i| + |t| } \# \widetilde{\mathcal{T}}_{t}^\mathbb{X}( y ; x_1,\dots,x_n ) \cdot y \ , \\
\partial_{Morse} (x) &=  \sum_{|y|= |x|+1 } \# \widetilde{\mathcal{T}}(y;x) \cdot y \ .
\end{align*} 

\begin{proposition}
If $\widetilde{\mathcal{T}}_{t}( y ; x_1,\dots,x_n )$ is 1-dimensional, its boundary decomposes as the disjoint union of the following components
\begin{enumerate}[label=(\roman*)]
\item $(-1)^{|y| + \dagger_{\Omega B As} + |t_2| \sum_{i=1}^{i_1} |x_i|} \widetilde{\mathcal{T}}_{t_1}(y ; x_1,\dots,x_{i_1 } , z , x_{i_1+i_2+1} , \dots , x_n ) \times \widetilde{\mathcal{T}}_{t_2}(z ; x_{i_1 +1 },\dots,x_{i_1 + i_2})$ ;
\item $(-1)^{|y| + \dagger_{\Omega B As}} \widetilde{\mathcal{T}}_{t'}(y ; x_1,\dots,x_n)$ for $t' \in coll(t)$ ;
\item $(-1)^{|y| + \dagger_{Koszul}+(d+1)|x_i|} \widetilde{\mathcal{T}}_{t}( y ; x_1, \dots , z , \dots , x_n ) \times \widetilde{\mathcal{T}}(z ; x_i)$ where $\dagger_{Koszul} = |t| + \sum_{j=1}^{i-1}|x_j|$ ;
\item $(-1)^{|y|+1} \widetilde{\mathcal{T}}(y ; z) \times  \widetilde{\mathcal{T}}_{t}( z ; x_1, \dots ,  x_n )$.
\end{enumerate} 
\end{proposition}

Applying the method of subsection~\ref{geo:sss:counting-boundary-points} finally proves that :
\theoremstyle{theorem}
\newtheorem*{geo:th:ombas-alg}{Theorem~\ref{geo:th:ombas-alg}}
\begin{geo:th:ombas-alg}
The operations $m_t$ define a twisted $\Omega B As$-algebra structure on $(C^*(f), \partial_{Morse}^{Tw},\partial_{Morse})$, where
\[ (\partial_{Morse}^{Tw})^k = (-1)^{(d+1)k}\partial_{Morse}^k \ . \]
\end{geo:th:ombas-alg}

\subsubsection{Signs for the (int-break) boundary} \label{geo:sss:signs-int-break-Morse}

We resort to the formalism of short exact sequences of vector bundles to handle orientations in this section. For the sake of readability, we will write $N$ rather than $TN$ for the tangent bundle of a manifold $N$ in the upcoming computations.

The moduli space $\mathcal{T}_{t}(y ; x_1 , \dots , x_n)$ is defined as the inverse image of the diagonal $\Delta \subset M^{\times n+1}$ under the map
\[ \phi_{\mathbb{X}_t} : \mathcal{T}_n(t) \times W^S(y) \times W^U(x_1) \times \cdots \times W^U(x_n) \longrightarrow M^{\times n+1} \ , \]
where the factors of $M^{\times n+1}$ are labeled in the order $M_y \times M_{x_1} \times \cdots \times M_{x_n}$. Orienting the domain and codomain of $\phi_{\mathbb{X}_t}$ by taking the product orientations, and orienting $\Delta$ as $M$, defines the natural orientation on $\mathcal{T}_{t}(y ; x_1 , \dots , x_n)$ as in subsection~\ref{geo:sss:orientation-transversality}. Choose $M^{\times n}$ labeled by $x_1 , \dots , x_n$ as complementary to $\Delta$. Then the orientation induced on $M^{\times n}$ by the orientations on $M^{\times n+1}$ and on $\Delta$, differs by a $(-1)^{d^2 n}$ sign from the product orientation of $M^{\times n}$. In the language of short exact sequences, $\mathcal{T}_{t}(y ; x_1 , \dots , x_n)$ is oriented by the short exact sequence
\[ 0 \longrightarrow \mathcal{T}_{t}(y ; x_1 , \dots , x_n) \longrightarrow  \mathcal{T}_n(t) \times W^S(y) \times \prod_{i=1}^n W^U(x_i) \longrightarrow M^{\times n} \longrightarrow 0 \ , \]
which has a sign of parity
\begin{align*}
dn \tag{A} \ .
\end{align*}
 
In the case of $\mathcal{T}^{Morse}_{t_1} := \mathcal{T}_{t_1}(y ; x_1 , \dots , x_{i_1} , z , x_{i_1+i_2+1}, \dots , x_n)$, we choose $M^{\times i_1+1+i_3}$ labeled by $y , x_1 , \dots , x_{i_1} , x_{i_1+i_2+1}, \dots , x_n$ as complementary to $\Delta$. The orientation induced on $M^{\times i_1 + 1+i_3}$, by the orientations on $M^{\times i_1 + 2+i_3}$ and on $\Delta$, differs by a $(-1)^{d^2 i_3}$ sign from the product orientation of $M^{\times i_1 + 1+i_3}$. Hence the short exact sequence
\[ \resizebox{\hsize}{!}{$\displaystyle{
0 \longrightarrow \mathcal{T}^{Morse}_{t_1} \longrightarrow \mathcal{T}_{i_1+1+i_3}(t_1) \times W^S(y) \times \prod_{i=1}^{i_1} W^U(x_i) \times W^U(z) \times \prod_{i=i_1+i_2+1}^{n} W^U(x_i) 
 \longrightarrow M^{\times i_1+1+i_3} \rightarrow 0  \ ,}$} \] 
has a sign of parity
\begin{align*}
di_3 \tag{B} \ .
\end{align*}
In the case of $\mathcal{T}^{Morse}_{t_2} := \mathcal{T}_{t_2}(z ; x_{i_1+1} , \dots , x_{i_1+i_2})$, we choose $M^{\times i_2}$ labeled by $x_{i_1+1} , \dots , x_{i_1+i_2}$ as complementary to $\Delta$. The orientation induced on $M^{\times i_2}$ differs this time by a $(-1)^{d^2 i_2}$ sign from the product orientation. The short exact sequence
\begin{align*}
0 \longrightarrow \mathcal{T}^{Morse}_{t_2} \longrightarrow \mathcal{T}_{i_2}(t_2) \times W^S(z) \times \prod_{i=i_1+1}^{i_1+i_2} W^U(x_i) \longrightarrow M^{\times i_2} \rightarrow 0  \ , 
\end{align*} 
has now a sign given by the parity of
\begin{align*}
di_2 \tag{C} \ .
\end{align*}

Following the convention of subsection~\ref{geo:sss:additional-tools}, taking the product 
\[ \resizebox{\hsize}{!}{\begin{math} \begin{aligned}
0 \longrightarrow \mathcal{T}^{Morse}_{t_1} \times \mathcal{T}^{Morse}_{t_2} \longrightarrow \mathcal{T}_{i_1+1+i_3}(t_1) \times W^S(y) \times \prod_{i=1}^{i_1} W^U(x_i) \times &W^U(z) \times \prod_{i=i_1+i_2+1}^{n} W^U(x_i) \times \mathcal{T}_{i_2}(t_2) \times W^S(z) \times \prod_{i=i_1+1}^{i_1+i_2} W^U(x_i) \\
 &\longrightarrow M^{\times i_1+1+i_3} \times M^{\times i_2} \rightarrow 0  
\end{aligned}
\end{math}} \]
doesn't introduce a sign, as $\mathcal{T}^{Morse}_{t_1}$ and $\mathcal{T}^{Morse}_{t_2}$ are 0-dimensional.

In the previous short exact sequence, $M^{\times i_1+1+i_3} \times M^{\times i_2}$ is labeled by
\[ y , x_1 , \dots , x_{i_1} , x_{i_1+i_2+1}, \dots , x_n , x_{i_1+1}, \dots , x_{i_1+i_2} \ .\]
We rearrange this labeling into 
\[ y , x_1 , \dots , x_n \ , \]
which induces a sign given by the parity of
\begin{align*}
d i_2 i_3 \tag{D} \ .
\end{align*}

We also rearrange the expression  
\[ \resizebox{\hsize}{!}{$\displaystyle{ \mathcal{T}_{i_1+1+i_3}(t_1) \times W^S(y) \times \prod_{i=1}^{i_1} W^U(x_i) \times W^U(z) \times \prod_{i=i_1+i_2+1}^{n} W^U(x_i) \times \mathcal{T}_{i_2}(t_2) \times W^S(z) \times \prod_{i=i_1+1}^{i_1+i_2} W^U(x_i) \ ,}$} \]
into
\[ W^U(z) \times W^S(z) \times \mathcal{T}_{i_1+1+i_3}(t_1) \times \mathcal{T}_{i_2}(t_2) \times W^S(y) \times \prod_{i=1}^{n} W^U(x_i) \ . \]
The parity of the produced sign is that of
\begin{align*}
&|z| \left( |t_2| + \sum_{i=i_1+i_2+1}^n (d - |x_i|) \right) + m \left( |t_1| + |y| + \sum_{i=1}^{i_1} (d-|x_i|) \right) \tag{E} \\
+ \ &|t_2| \left( |y| + \sum_{i=1}^{i_1} (d-|x_i|) + \sum_{i=i_1+i_2+1}^n (d - |x_i|) \right) + \left( \sum_{i=i_1+1}^{i_1+i_2} (d-|x_i|) \right) \left( \sum_{i=i_1+i_2+1}^{n} (d-|x_i|)  \right)  \ .
\end{align*} 

Introduce now the factor $[L,+ \infty [$, corresponding to the length $l_e$ increasing towards $+\infty$, where $e$ is the edge of $t$ whose breaking produces $t_1$ and $t_2$. Following convention~\ref{geo:sss:orientation-transversality}, the short exact sequence 
\[ \resizebox{\hsize}{!}{$\displaystyle{ 0 \longrightarrow [L,+ \infty [ \times \mathcal{T}^{Morse}_{t_1} \times \mathcal{T}^{Morse}_{t_2} \longrightarrow [L,+ \infty [ \times W^U(z) \times W^S(z)  \times \mathcal{T}(t_1) \times \mathcal{T}(t_2) \times  \times W^S(y) \times \prod_{i=1}^{n} W^U(x_i) \longrightarrow M^{\times n+1} \longrightarrow 0  \ ,}$} \] 
induces a sign change whose parity is given by
\begin{align*}
d(n+1) \tag{F} \ .
\end{align*}

Define the map
\[ \psi : M \times \mathcal{T}_n(t) \times W^S(y) \times \prod_{i=1}^n W^U(x_i) \longrightarrow M \times M^{\times n+1} \ , \]
which is defined on the factors $\mathcal{T}_n(t) \times W^S(y) \times \prod_{i=1}^n W^U(x_i)$ as $\phi$ and is defined on $M \times \mathcal{T}_n(t)$ by seeing $M$ as the point lying in the middle of the edge $e$ in $t$. This map is depicted on figure~\ref{geo:fig:map-psi}.
The inverse image of the diagonal of $M \times M^{\times n+1}$ is exactly $\mathcal{T}_{t}(y ; x_1 , \dots , x_n)$. Fix now a sufficiently great $L > 0$. We prove in subsection~\ref{geo:sss:gluing-and-orientations} that orienting $[L,+ \infty [ \times \mathcal{T}^{Morse}_{t_1} \times \mathcal{T}^{Morse}_{t_2}$ with the previous short exact sequence, the orientation induced on $\mathcal{T}_{t}^{Morse}$ by gluing is the exactly the one given by the short exact sequence
\[ \begin{tikzcd}[scale cd = 0.9]
0 \arrow[r] & \mathcal{T}^{Morse}_{t} \arrow[r] & \left[L,+ \infty \right[ \times M \times \mathcal{T}(t_1) \times \mathcal{T}(t_2) \times W^S(y) \times \prod_{i=1}^{n} W^U(x_i) \arrow{r}[below]{d\psi} &  M^{\times n+1} \arrow[r] & 0 
\end{tikzcd} \ , \]
where our convention on orientations for the unstable and stable manifolds of $z$ implies that $W^U(z) \times W^S(z)$ yields indeed the orientation of $M$, and $M^{\times n+1}$ is labeled by $y , x_1 , \dots , x_n$.

\begin{figure}[h]
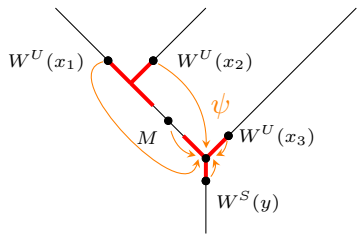
 
\centering
\examplemappsi
\caption{Representation of the map $\psi$} \label{geo:fig:map-psi}
\end{figure}

Transform the coorientation labeled by $y , x_1 , \dots , x_n$ into the coorientation labeled by $M , x_1 , \dots , x_n$ and rearrange the factors $\left[L,+ \infty \right[ \times M \times \mathcal{T}(t_1) \times \mathcal{T}(t_2) \times \cdots$ into $ M \times \left[L,+ \infty \right[ \times \mathcal{T}(t_1) \times \mathcal{T}(t_2) \times \cdots$
This produces a sign change of parity
\begin{align*}
d + d \equiv 0 \tag{G} \ .
\end{align*}
We can moreover now delete the two $M$ factors associated to the label $M$ to obtain the short exact sequence
\[ 0 \longrightarrow \mathcal{T}_{t}(y ; x_1 , \dots , x_n) \longrightarrow [L,+ \infty [ \times \mathcal{T}(t_1) \times \mathcal{T}(t_2) \times W^S(y) \times \prod_{i=1}^{n} W^U(x_i) \longrightarrow M^{\times n} \longrightarrow 0  \ , \]
where $M^{\times n} = M_{x_1} \times \cdots \times M_{x_n}$. 

Transforming finally $[L,+ \infty [ \times \mathcal{T}(t_1) \times \mathcal{T}(t_2)$ into $\mathcal{T}_{n}(t)$ gives a sign of parity 
\begin{align*}
\dagger_{\Omega B As} \tag{H} \ .
\end{align*}
In closing, the short exact sequence 
\[ 0 \longrightarrow \mathcal{T}_{t}(y ; x_1 , \dots , x_n) \longrightarrow  \mathcal{T}_n(t) \times W^S(y) \times \prod_{i=1}^n W^U(x_i) \longrightarrow M^{\times n} \longrightarrow 0 \ , \]
has sign given by the parity of $A$ when $\mathcal{T}_{t}^{Morse}$ is endowed with its natural orientation. It has sign given by the parity of $B + C + D + E + F + G + H$ when $\mathcal{T}_{t}^{Morse}$ is endowed with the orientation induced by $[L , + \infty [ \times \mathcal{T}_{t_1}^{Morse} \times \mathcal{T}_{t_2}^{Morse}$, where the first factor is the length $l_e$ and determines the outward-pointing direction $\nu_e$ to the boundary component $\mathcal{T}_{t_1}^{Morse} \times \mathcal{T}_{t_2}^{Morse}$.

We thus obtain that with our choice of orientation on the moduli spaces $\mathcal{T}_t^{Morse}$, the sign of $\mathcal{T}_{t_1}(y ; x_1 , \dots , x_{i_1} , z , x_{i_1+i_2+1}, \dots , x_n) \times \mathcal{T}_{t_2}(z ; x_{i_1+1} , \dots , x_{i_1+i_2})$ in the boundary of the 1-dimensional moduli space $\mathcal{T}_{t}(y ; x_1 , \dots , x_n)$ is given by the parity of 
\[ \resizebox{\hsize}{!}{\begin{math} \begin{aligned}
(*) \ \ \ &A + B + C + D + E + F + G + H \\
= \ & |z| |t_2| + d |y| + d |t_1 | + (n+1)d + \sum_{i=1}^{i_1}d|x_i| + |t_2| |y| + d i_1 |t_2| + d i_2 \sum_{i=i_1 + i_2 + 1}^n |x_i| + \dagger_{\Omega B As} + |t_2| \sum_{i=1}^{i_1} |x_i| \ .
\end{aligned}
\end{math}} \]
Hence the sign of $\widetilde{\mathcal{T}}_{t_1}(y ; x_1 , \dots , x_{i_1} , z , x_{i_1+i_2+1}, \dots , x_n) \times \widetilde{\mathcal{T}}_{t_2}(z ; x_{i_1+1} , \dots , x_{i_1+i_2})$ in the boundary of the 1-dimensional moduli space $\widetilde{\mathcal{T}}_{t}(y ; x_1 , \dots , x_n)$ is given by the parity of
\begin{align*}
&\sigma (t ; y ; x_1 , \dots , x_n) + \sigma (t_1 ; y ; x_1 , \dots , x_{i_1} , z , x_{i_1+i_2+1} , \dots , x_n) + \sigma (t_2 ; z ; x_{i_1+1} , \dots , x_{i_1+i_2}) + (*)  \\
= \ &|y| + \dagger_{\Omega B As} + |t_2| \sum_{i=1}^{i_1} |x_i| \ .
\end{align*}

\subsubsection{Gluing and orientations} \label{geo:sss:gluing-and-orientations}

We prove in this subsection that after orienting $[L , + \infty [ \times \mathcal{T}_{t_1}^{Morse} \times \mathcal{T}_{t_2}^{Morse}$ with the short exact sequence
\[ \begin{tikzcd}[scale cd = 0.63]
0 \arrow[r] & \left[ L,+ \infty \right[ \times \mathcal{T}^{Morse}_{t_1} \times \mathcal{T}^{Morse}_{t_2} \arrow[r] & \left[L,+ \infty \right[ \times W^U(z) \times W^S(z) \times \mathcal{T}(t_1) \times \mathcal{T}(t_2) \times W^S(y) \times \prod_{i=1}^{n} W^U(x_i) \arrow[r] &  M^{\times n+1} \arrow[r] & 0
\end{tikzcd} \ , \]
the orientation induced on $\mathcal{T}_{t}^{Morse}$ by gluing is the one given by the short exact sequence
\[ \begin{tikzcd}[scale cd = 0.85]
0 \arrow[r] & \mathcal{T}^{Morse}_{t} \arrow[r] &  \left[L,+ \infty \right[ \times M \times  \mathcal{T}(t_1) \times \mathcal{T}(t_2) \times W^S(y) \times \prod_{i=1}^{n} W^U(x_i) \arrow{r}[below]{d\psi} &  M^{\times n+1} \arrow[r] & 0 
\end{tikzcd} \ . \]
The proof boils down to the following lemma.

\begin{lemma} \label{geo:lemma:lemma-gluing}
Let $M$ and $N$ be manifolds and $S \subset N$ a submanifold of $N$. Suppose that $M$, $N$ and $S$ are orientable and oriented. Let $ f : [ 0 , 1] \times M \rightarrow N$ be a smooth map such that $f_1 := f(1,\cdot) : M \rightarrow N$ is transverse to $S$. Let $x \in f_1^{-1}(S)$. Then there exist an open subset $V$ of $M$ containing $x$ and $0 \leqslant t_1 <1$ such that
\begin{enumerate}[label=(\roman*)]
\item The map $f|_{[t_1 , 1] \times V} : [t_1 , 1] \times V \rightarrow N$ is transverse to $S$. In particular the inverse image $f|_{[t_1 , 1] \times V}^{-1}(S)$ is then a submanifold of $[ t_1 , 1] \times V$.
\item There exists an orientation-preserving embedding
\[ f|_{[t_1 , 1] \times V}^{-1}(S) \longrightarrow [t_1 , 1] \times f_1^{-1}(S)  \]
equal to the identity on $f_1|_V^{-1}(S)$ and preserving the $t$ coordinate, where we orient $[t_1 , 1] \times f_1^{-1}(S)$ with the short exact sequence 
\[ 0 \longrightarrow [t_1 , 1] \times f_1^{-1}(S) \longrightarrow [ 0 , 1] \times M \longrightarrow \nu_S \longrightarrow 0 \]
and we orient $f|_{[t_1 , 1] \times V}^{-1}(S)$ with the short exact sequence 
\[ 0 \longrightarrow f|_{[t_1 , 1] \times V}^{-1}(S) \longrightarrow [ 0 , 1] \times M \longrightarrow \nu_S \longrightarrow 0 \ . \]
\end{enumerate}
\end{lemma}

\begin{proof}
Choose an adapted chart for $S$ around $f_1(x)$, i.e. a chart $\phi : U' \subset N \rightarrow \R^{n}$ such that
\[ \phi ( U' \cap S ) = \{ (y_1 , \dots , y_{n-s},x_1,\dots,x_s) \in \R^n , \ y_1 = \cdots = y_{n-s} = 0 \} \ , \]
where $n$ and $s$ respectively denote the dimensions of $N$ and $S$.
Using the local normal form theorem for submersions, there exists a local chart $\psi : U \subset M \rightarrow \R^m$ around $x$ such that the map $f_1$ reads as 
\[ (y_1 , \dots , y_{n-s},x_1,\dots,x_{m+s-n}) \longmapsto \left( y_1 , \dots , y_{n-s},F_1(\vec{y},\vec{x}),\dots,F_s( \vec{y},\vec{x} ) \right) \] 
in the local charts $\psi$ and $\phi$, where the $F_i$ are smooth maps and $\vec{y} := y_1 , \dots , y_{n-s}$, $\vec{x} := x_1,\dots,x_{m+s-n}$ and $m := \mathrm{dim} (M)$. In these local charts, 
\[ U \cap f_1^{-1}(U' \cap S ) = \{ (y_1 , \dots , y_{n-s},x_1,\dots,x_{m+s-n}) \in \R^m , \ y_1 = \cdots = y_{n-s} = 0 \} \ . \]
The property "being transverse to $S$" being open, there exists a neighborhood $W$ of $x$ in $M$ and $t_0 \in [0,1[$ such that the map $f|_{[t_0 , 1] \times W} : [t_0 , 1] \times W \rightarrow N$ is transverse to $S$. Suppose $W \subset U$ and consider now the projection $\pi : \R^m \rightarrow \R^{m+s-n}$ given by
\[ (y_1 , \dots , y_{n-s},x_1,\dots,x_{m+s-n}) \longmapsto (x_1,\dots,x_{m+s-n}) \]
and define the smooth map 
\[ \iota := \ide_t \times \pi : f|_{[t_0 , 1] \times W}^{-1}(S) \longrightarrow [0,1] \times f_1^{-1}(S) \]
in the local charts $\phi$ and $\psi$. The differential of this map is invertible at $(1,x)$. The inverse function theorem then ensures that there exits $t_1 \in [t_0 , 1[$ and a neighborhood $V \subset W$ of $x$ such that the map 
\[ \iota : f|_{[t_1 , 1] \times V}^{-1}(S) \longrightarrow [0,1] \times f_1^{-1}(S) \]
is a diffeomorphism on its image.

Orient now $[0,1] \times f_1^{-1}(S)$ and $f|_{[t_1 , 1] \times V}^{-1}(S)$ with the previous short exact sequences. It remains to show that the map $\iota$ is orientation-preserving. The proof of this result can be reduced to a proof in linear algebra, i.e. by considering a smooth family of linear maps $f : [0,1] \times \R^m \rightarrow \R^n$ such that $f_1$ reads as
\[ (y_1 , \dots , y_{n-s},x_1,\dots,x_{m+s-n}) \longmapsto \left( y_1 , \dots , y_{n-s},F_1(\vec{y},\vec{x}),\dots,F_s( \vec{y},\vec{x} ) \right) \ , \]
and the linear subspace $S = \{ 0 \} \times \R^s \subset \R^n$.
Then there exists $t_0 \in [0,1]$ such that $f|_{[t_0 , 1] \times \R^m}$ is transverse to $S$, and we can consider the smooth map
\[ \iota := \ide_t \times \pi : f|_{[t_0 , 1] \times \R^m}^{-1}(S) \longrightarrow [0,1] \times f_1^{-1}(S)  \]
which is a diffeomorphism on its image.
Basic computations finally show that the map $\iota$ is indeed orientation-preserving.
\end{proof}

We now go back to our initial problem. Let $T_1^{Morse} \in \mathcal{T}_1^{Morse}$ and $T_2^{Morse} \in \mathcal{T}_2^{Morse}$, where we refer to subsection~\ref{geo:sss:signs-int-break-Morse} for notations. Consider a local Euclidean chart $\phi_z : U_z \rightarrow \R^d$ for the critical point $z$ as in subsection~\ref{geo:sss:morse-theory}. Introduce the map $ev : [0, + \infty] \times U_z \rightarrow U_z \times U_z$ reading as 
\[ (\delta,x+y) \longmapsto ( e^{-2\delta}x + y , x + e^{-2\delta} y) \]
in the chart $\phi_z$. The pair $ev(\delta,x+y)$ corresponds to the two endpoints of the unique finite Morse trajectory parametrized by $[-\delta,\delta]$ and meeting $e^{-\delta}x + e^{-\delta} y$ at time $0$. 

Consider the trajectory $\gamma_{e,1} : \ ] - \infty , 0 ] \rightarrow M$ and the trajectory $\gamma_{e,2} : [ 0 , + \infty [ \rightarrow M$, respectively associated to the incoming edge of  $T_1^{Morse}$ and to the outgoing edge of $T_2^{Morse}$ which result from the breaking of the edge $e$ in $t$. Choose $L$ large enough such that $\gamma_{e,1}(-L)$ and $\gamma_{e,2}(L)$ belong to $U_z$. Introduce the map $f := ev \times (\phi^{-(L-1)})^{\times i_1 + 1 + i_3} \circ \psi_{e,\mathbb{X}_{t_1}} \times (\phi^{L-1})^{\times i_2} \circ \psi_{e,\mathbb{X}_{t_2}}$ acting as
\[ \begin{tikzcd}[scale cd = 0.9, row sep = 4 pt]
\left[ 0 , + \infty \right] \times U_z \times \mathcal{T}_{i_1+1+i_3}(t_1) \times W^S(y) \times \prod_{i=1}^{i_1} W^U(x_i) \times \prod_{i=i_1+i_2+1}^{n} W^U(x_i) \times \mathcal{T}_{i_2}(t_2) \times \prod_{i=i_1+1}^{i_1+i_2} W^U(x_i)  \\
 \hspace{25 pt} \longrightarrow M^{\times 2} \times M^{\times i_1 + 1 + i_3} \times M^{\times i_2} \ , 
\end{tikzcd} \] 
where $\phi^{L-1}$ stands for the time $L-1$ Morse flow and the maps $\psi_{e,\mathbb{X}_{t_2}}$ and $\psi_{e,\mathbb{X}_{t_1}}$ have been introduced in subsection~\ref{geo:sss:morse-theory}.
This map is depicted in figure~\ref{geo:f:the-map-f}.

\begin{figure}[h]
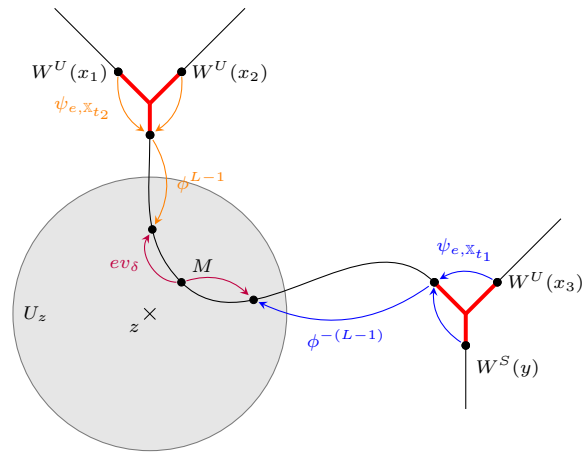

\centering
\examplemapf
\caption{Representation of the map $f$. The label $M$ corresponds to the point $e^{-\delta}x + e^{-\delta}y$ and not to the point $x+y$.} \label{geo:f:the-map-f}
\end{figure}

Define the $2d$-dimensional submanifold $\Lambda \subset M^{\times 2} \times M^{\times i_1 + 1 + i_3} \times M^{\times i_2}$ to be
\[ \Lambda := \left\{ \begin{array}{l}
         (m_z^1 , m_z^2 , m_y , m_{1} , \dots , m_{i_1}, m_{i_1 + 1 +i_2} , \dots , m_n , m_{i_1+1} , \dots , m_{i_1+i_2}) \\
         \text{such that }  m_z^1 = m_{i_1+1} = \cdots = m_{i_1+i_2} \ \ \text{ and } \\
          m_z^2 = m_y = m_{1} = \cdots = m_{i_1} = m_{i_1 + 1 +i_2} = \cdots = m_n   
  \end{array}
  			\right\} \ . \]
The pair $(T_1^{Morse} , T_2^{Morse})$ then belongs to the inverse image $f_{+ \infty}^{-1} ( \Lambda)$. By assumption on the choice of perturbation data $(\mathbb{X}_n)_{n \geqslant 2}$, the map $f_{+ \infty}$ is moreover transverse to $\Lambda$. Applying Lemma~\ref{geo:lemma:lemma-gluing} to the map $f$ at the point $(T_1^{Morse} , T_2^{Morse})$, there exists $R > 0$ and an embedding
\[ \# _{ T_1^{Morse} , T_2^{Morse} } : [ R , + \infty ] \longrightarrow \overline{\mathcal{T}}_{t}(y ; x_1,\dots ,x_n) \ . \] 
Note that the parameter $\delta$ corresponds to an edge of length $2L + 2\delta$ in the resulting glued tree. Upon reordering the factors of the domain of $f$, it is finally easy to check that this lemma also implies the result on orientations stated at the beginning of this subsection. 

\subsubsection{Signs for the (int-collapse) and (Morse) boundary} \label{geo:sss:int-collapse-Morse-signs}

Repeating the beginning of the previous section, for the moduli spaces $\mathcal{T}_{t'}(y ; x_1 , \dots , x_{n})$, where $t' \in coll(t)$, and $\mathcal{T}_{t}(y ; x_1 , \dots , x_n)$, we choose $M^{\times n}$ labeled by $x_1,\dots,x_n$ as complementary to the diagonal $\Delta \subset M^{\times n+1}$. The parity of the total sign change coming from these coorientation choices is
\begin{align*}
dn + dn = 0 \tag{A} \ .
\end{align*}

Introduce the factor $]0,L]$, corresponding to the length $l_e$ going towards $0$, where $e$ is the edge of $t$ whose collapsing produces $t'$. Applying again Lemma~\ref{geo:lemma:lemma-gluing} and following convention~\ref{geo:sss:additional-tools}, the short exact sequence
\[ \resizebox{\hsize}{!}{$\displaystyle{0 \longrightarrow \mathcal{T}_{t}(y ; x_1 , \dots , x_n) = ]0,L] \times \mathcal{T}_{t'}(y ; x_1 , \dots , x_n) \longrightarrow ]0,L] \times \mathcal{T}_{n}(t')  \times W^S(y) \times \prod_{i=1}^{n} W^U(x_i) \longrightarrow M^{\times n} \longrightarrow 0  \ ,}$} \]
introduces a sign change whose parity is given by
\begin{align*}
dn \tag{B} \ .
\end{align*}
Transforming finally $]0,L] \times \mathcal{T}_{n}(t')$ into $\mathcal{T}_{n}(t)$ gives a sign of parity 
\begin{align*}
\dagger_{\Omega B As} \tag{C} \ .
\end{align*}

Adding these contributions, we obtain that the sign of $\mathcal{T}_{t'}(y ; x_1 , \dots , x_n)$ in the boundary of the 1-dimensional moduli space $\mathcal{T}_{t}(y ; x_1 , \dots , x_n)$ is given by the parity of
\begin{align*}
A + B + C = dn + \dagger_{\Omega B As} \ . \tag{*}
\end{align*} 
The sign of $\widetilde{\mathcal{T}}_{t'}(y ; x_1 , \dots , x_n)$ in the boundary of the 1-dimensional moduli space $\widetilde{\mathcal{T}}_{t}(y ; x_1 , \dots , x_n)$ is hence given by the parity of
\[ \sigma (t ; y ; x_1 , \dots , x_n) + \sigma (t' ; y ; x_1 , \dots , x_n) + (*) = |y| + \dagger_{\Omega B As} \ . \]
Finally, the signs for the (Morse) boundary can be computed following the exact same lines of the two previous proofs. 

\subsection{The twisted $\Omega B As$-morphism between the Morse cochains} \label{geo:ss:twisted-ombas-morph-Morse}

\subsubsection{Reformulating the $\Omega B As$-equations} \label{geo:sss:reformulating-eq-morph}

We set again for the rest of this section an orientation $\omega$ for each $t_g \in SCRT_n$, which endows each moduli space $\mathcal{CT}_n(t_g)$ with an orientation, and write moreover $\mu_{t_g}$ for the operations $(t_g, \omega)$ of $\Omega B As - \mathrm{Morph}$. The $\Omega B As$-equations for an $\Omega B As$-morphism then read as
\begin{align*}
[ \partial , \mu_{t_g} ] = &\sum_{t'_g \in coll(t_g)} (-1)^{\dagger_{\Omega B As}} \mu_{t'_g} + \sum_{t'_g \in g-vert(t_g)} (-1)^{\dagger_{\Omega B As}} \mu_{t'_g} + \sum_{t^1_g \#_i t^2 = t_g} (-1)^{\dagger_{\Omega B As}} \mu_{t^1_g} \circ_i m_{t^2} \\
&+ \sum_{t^0 \# (t_g^1 , \dots , t_g^s) = t_g} (-1)^{\dagger_{\Omega B As}} m_{t^0} \circ (\mu_{t^1_g} \otimes \cdots \otimes \mu_{t^s_g} ) \ ,
\end{align*} 
where the notations for trees are transparent. The signs $(-1)^{\dagger_{\Omega B As}}$ are obtained as in subsection~\ref{geo:sss:reformulating-equations}.

\subsubsection{Twisted \Ainf -morphisms and twisted $\Omega B As$-morphisms} \label{geo:sss:twisted-ainf-morph-def}

Again, it is clear using the counting method of~\ref{geo:sss:counting-boundary-points} that if we work over $\Z / 2$, the operations $\mu_{t_g}$ of~\ref{geo:ss:ombas-morph-Morse} define an $\Omega B As$-morphism. We will prove a weaker result in the case of integers, introducing for this matter the notion of twisted \Ainf -morphisms and twisted $\Omega B As$-morphisms.

\begin{definition}
Let $(A, \partial_1 , \partial_2 , m_n)$ and $(B, \partial_1 , \partial_2 , m_n)$ be two twisted \Ainf -algebras. A \emph{twisted \Ainf -morphism} from $A$ to $B$ is defined to be a sequence of degree $1-n$ operations $f_n : A^{\otimes n} \rightarrow B$ such that
\[ \left[ \partial , f_n \right] = \sum_{\substack{i_1+i_2+i_3=n \\ i_2 \geqslant 2}} (-1)^{i_1 + i_2i_3} f_{i_1+1+i_3} (\ide^{\otimes i_1} \otimes m_{i_2} \otimes \ide^{\otimes i_3})  - \sum_{\substack{i_1 + \cdots + i_s = n \\ s \geqslant 2 }} (-1)^{\epsilon_B} m_s ( f_{i_1} \otimes \cdots \otimes f_{i_s}) \ , \]
where $[ \partial , \cdot ]$ denotes the bracket for the maps $ (A^{\otimes n} , \partial_1) \rightarrow (B , \partial_2)$.
A \emph{twisted $\Omega B As$-morphism} between twisted $\Omega B As$-algebras is defined similarly.
\end{definition}

The formulae obtained by evaluating the $\Omega B As$-equations on $A^{\otimes n}$ then become
\begin{align*}
&- \partial_2 \mu_{t_g} (a_1 , \dots , a_n) + (-1)^{|t_g| + \sum_{j=1}^{i-1}|a_j|} \mu_{t_g} ( a_1 , \dots , a_{i-1} , \partial_1 a_i , a_{i+1} , \dots , a_n) \\
&+ \sum_{t^1_g \# t^2 = t} (-1)^{\dagger_{\Omega B As} + |t^2| \sum_{j=1}^{i_1} |a_j|} \mu_{t^1_g} (a_1 , \dots , a_{i_1} , m_{t^2} (a_{i_1 + 1} , \dots , a_{i_1 + i_2} ) , a_{i_1 + i_2 + 1} , \dots , a_n) \\
&+ \sum_{t^1 \# (t_g^1 , \dots , t_g^s) = t_g} (-1)^{\dagger_{\Omega B As} + \dagger_{Koszul} } m_{t^0} ( \mu_{t_g^1} (a_1 , \dots ,a_{i_1} ) , \dots , \mu_{t_g^s} (a_{i_1 + \cdots + i_{s-1} +1} , \dots , a_n)) \\
&+ \sum_{t'_g \in coll(t_g)} (-1)^{\dagger_{\Omega B As}} \mu_{t'_g} (a_1 , \dots , a_n) + \sum_{t'_g \in g-vert(t_g)} (-1)^{\dagger_{\Omega B As}} \mu_{t'_g} (a_1 , \dots , a_n) \\
&= 0 \ ,
\end{align*}
where
\[ \dagger_{Koszul} = \sum_{r=1}^s|t_g^r| \left( \sum_{t=1}^{r-1} \sum_{j=1}^{i_t} |a_{i_1 + \cdots + a_{i_{t-1}} + j}| \right) \ . \]

Again these two definitions cannot be phrased using an operadic viewpoint. However, a twisted $\Omega B As$-morphism between twisted $\Omega B As$-algebras always descends to a twisted \Ainf -morphism between twisted \Ainf -algebras, for the same reason as in subsection~\ref{geo:sss:twisted-structure}. 

\subsubsection{Summary of the proof of Theorem~\ref{geo:th:ombas-morph}} \label{geo:sss:result-summary-ainf-morph}

Let $\mathbb{X}^f$ and $\mathbb{X}^g$ be admissible choices of perturbation data on the moduli spaces $\mathcal{T}_n$ for the Morse functions $f$ and $g$, and $\mathbb{Y}$ be a choice of perturbation data on the moduli spaces $\mathcal{CT}_n$ that is admissible w.r.t. $\mathbb{X}^f$ and $\mathbb{X}^g$.

\begin{definition}
We define $\widetilde{\mathcal{CT}}_{t_g}^\mathbb{Y}( y ; x_1,\dots,x_n )$ to be the oriented manifold $\mathcal{CT}_{t_g}^\mathbb{Y}( y ; x_1,\dots,x_n )$ whose natural orientation has been twisted by a sign of parity
\[ \sigma (t_g ; y ; x_1 , \dots , x_n) := dn ( 1 + |y| + |t_g| ) + |t_g| |y| + d \sum_{i=1}^n |x_i| (n-i) \ . \]
\end{definition}
The moduli spaces $\widetilde{\mathcal{T}}(y;x)$ and $\widetilde{\mathcal{T}}_{t}( y ; x_1,\dots,x_n )$ are moreover defined as in section~\ref{geo:ss:twisted-ainf-alg-Morse}. 
We define the operations $ \mu_{t_g} : C^*(f)^{\otimes n} \rightarrow C^*(g)$ as
\begin{align*}
\mu_{t_g} (x_1 , \dots , x_n) =  \sum_{|y|= \sum_{i=1}^n|x_i| + |t_g| } \# \widetilde{\mathcal{CT}}_{t_g}^\mathbb{Y}(y ; x_1,\dots,x_n) \cdot y \ .
\end{align*} 

\begin{proposition}
If $\widetilde{\mathcal{CT}}_{t_g}( y ; x_1,\dots,x_n )$ is 1-dimensional, its boundary decomposes as the disjoint union of the following components
\begin{enumerate}[label=(\roman*)]
\item $(-1)^{|y| + \dagger_{\Omega B As} + |t^2| \sum_{i=1}^{i_1} |x_i|} \widetilde{\mathcal{CT}}_{t^1_g}(y ; x_1,\dots,x_{i_1 } , z , x_{i_1+i_2+1} , \dots , x_n ) \times \widetilde{\mathcal{T}}_{t^2}(z ; x_{i_1 +1 },\dots,x_{i_1 + i_2})$ ;
\item $(-1)^{|y| + \dagger_{\Omega B As} + \dagger_{Koszul}} \widetilde{\mathcal{T}}_{t^1}(y ; y_1,\dots,y_s) \times \widetilde{\mathcal{CT}}_{t^1_g}(y_1 ; x_{1},\dots ) \times \cdots \times \widetilde{\mathcal{CT}}_{t^s_g}(y_s ; \dots,x_n)$ ;
\item $(-1)^{|y| + \dagger_{\Omega B As}} \widetilde{\mathcal{CT}}_{t'_g}(y ; x_1,\dots,x_n)$ for $t' \in coll(t)$ ;
\item $(-1)^{|y| + \dagger_{\Omega B As}} \widetilde{\mathcal{CT}}_{t'_g}(y ; x_1,\dots,x_n)$ for $t' \in g-vert(t)$ ;
\item $(-1)^{|y| + \dagger_{Koszul}+(m+1)|x_i|} \widetilde{\mathcal{CT}}_{t_g}( y ; x_1, \dots , z , \dots , x_n ) \times \widetilde{\mathcal{T}}(z ; x_i)$ where $\displaystyle{\dagger_{Koszul} = |t_g| + \sum_{j=1}^{i-1}|x_j|}$~;
\item $(-1)^{|y|+1} \widetilde{\mathcal{T}}(y ; z) \times  \widetilde{\mathcal{CT}}_{t_g}( z ; x_1, \dots ,  x_n )$.
\end{enumerate} 
\end{proposition}

Applying the method of subsection~\ref{geo:sss:counting-boundary-points} again finally proves that~:
\theoremstyle{theorem}
\newtheorem*{geo:th:ombas-morph}{Theorem~\ref{geo:th:ombas-morph}}
\begin{geo:th:ombas-morph}
The operations $\mu_{t_g}$ define a twisted $\Omega B As$-morphism between the Morse cochains $(C^*(f), \partial_{Morse}^{Tw}, \partial_{Morse})$ and $(C^*(g), \partial_{Morse}^{Tw}, \partial_{Morse})$.
\end{geo:th:ombas-morph}

\subsubsection{Gluing} \label{geo:sss:gluing-two-colored}

We construct explicit gluing maps in the two-colored framework using Lemma~\ref{geo:lemma:lemma-gluing}. Gluing maps for the (above-break) boundary components are built as in subsection~\ref{geo:sss:gluing-and-orientations}. In the (below-break) case, consider critical points $y,y_1,\dots,y_s \in \mathrm{Crit}(g)$ and $x_1,\dots,x_n \in \mathrm{Crit}(f)$ such that the moduli spaces $\mathcal{T}_{t^0}(y ; y_{1},\dots,y_{s})$ and $\mathcal{CT}_{t^r_g}(y_r ; x_{i_1 + \cdots + i_{r-1}+1},\dots,x_{i_1 + \cdots + i_{r}})$ are 0-dimensional. Let $T^{0,Morse} \in \mathcal{T}_{t^0}^{Morse}$ and $T^{r,Morse}_g \in \mathcal{CT}_{t^r_g}^{Morse}$. Fix moreover an Euclidean neighborhood $U_{z_r}$ of each critical point $z_r$ and choose $L$ large enough such that for $r=1, \dots ,s$, $\gamma_{e_r,T^{0,Morse}}(-L)$ and $\gamma_{e_0,T^{r,Morse}_g}(L)$ belong to $U_{z_r}$. Define finally the map $\sigma_{e_0,\mathbb{X}_{t^0}} : M \rightarrow M^{\times s}$ in a similar fashion to the maps $\psi_{e_i,\mathbb{X}_t}$, as depicted for instance in figure~\ref{geo:f:example-map-sigma}. Gluing maps for the perturbed Morse trees $T^{0,Morse}$ and $T^{r,Morse}_g$ can then be defined by applying Lemma~\ref{geo:lemma:lemma-gluing} to the map
\[ \resizebox{\hsize}{!}{$\displaystyle{
\left[ 0 , + \infty \right] \times \prod_{r=1}^s U_{z_r} \times \mathcal{T}_{s}(t^0) \times W^S(y) \times \prod_{r=1}^{s} \left( \mathcal{CT}_{i_r}(t^r_g) \times \prod_{i=i_1 + \cdots + i_{r-1} + 1}^{i_1 + \cdots + i_r} W^U(x_i) \right) \longrightarrow M^{\times 2s} \times M^{\times s} \times \prod_{r=1}^s M^{\times i_r} \ .  }$} \]
defined as follows :
\begin{enumerate}[label=(\roman*)]
\item the factor $\mathcal{T}_{s}(t^0) \times W^S(y)$ is sent to $M^{\times s}$ under the map $(\phi^{-(L-1)})^{\times s} \circ \sigma_{e_0,t^0}$ ;
\item the factor $\mathcal{CT}_{i_r}(t^r_g) \times \prod W^U(x_i)$ is sent to $M^{\times i_r}$ under the map $(\phi^{(L-1)})^{\times i_r} \circ \sigma_{e_0,t^r_g}$ ;
\item the factor $\left[ 0 , + \infty \right] \times \prod_{r=1}^s U_{z_r} $ is sent to $M^{\times 2s}$ under the map $ev^{U_{z_1}}_{l_\delta^1} \times \cdots \times ev^{U_{z_s}}_{l_\delta^s}$ where $\delta$ denotes the parameter in $\left[ 0 , + \infty \right]$ and the lengths $l_\delta^r$ are defined as in subsection~\ref{alg:sss:CTm-below-break-bound} of part~\ref{p:algebra} in order for them to define a two-colored metric ribbon tree. In particular, we have explicit formulae depending on $\delta$ for the resulting edges in the glued tree.
\end{enumerate}

\begin{figure}[h]
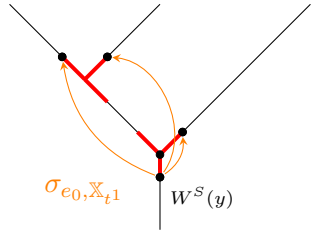

\centering
\examplemapsigma
\caption{Representation of the map $\sigma_{e_0,\mathbb{X}_{t^1}}$.} \label{geo:f:example-map-sigma}
\end{figure}

\subsection{On these twisted structures} \label{geo:ss:on-twisted-str}

Note first that if we work with coefficients in $\Z / 2$, the operations $m_t$ define of course an $\Omega  B As$-algebra structure on the Morse cochains. The operations $\mu_{t_g}$ then define an $\Omega  B As$-morphism between two $\Omega B As$-algebras. We will say that the structure we defined are \emph{untwisted}. We hence work now over the integers \Z . 
It appears from the definition of $\partial_{Morse}^{Tw}$ that when $M$ is odd-dimensional, the structures we define are untwisted. In the even-dimensional case, the structures are twisted, and it remains to be proven that all the operations $m_t$ could be twisted in order to get an untwisted structure.

We also point out that the twisted structures arise from the two uncompatible orientation conventions on an intersection $R \cap S$ and $S \cap R$ detailed in~\ref{geo:sss:orientation-transversality}. Indeed, we decided to orient $\mathcal{T}(y;x)$ inside the intersection $W^S(y) \cap W^U(x)$. The signs then compute nicely for the boundary component $\widetilde{\mathcal{T}}(y ; z) \times  \widetilde{\mathcal{CT}}_{t_g}( z ; x_1, \dots ,  x_n )$, and the twist in $\partial_{Morse}^{Tw}$ arises in $\widetilde{\mathcal{CT}}_{t_g}( y ; x_1, \dots , z , \dots , x_n ) \times \widetilde{\mathcal{T}}(z ; x_i)$. 

Orienting $\mathcal{T}(y;x)$ inside the intersection $ W^U(x) \cap W^S(y)$ makes these two boundary components switch roles. In that case, redefining the twist on the orientation of the moduli space $\mathcal{T}(y;x)$ as given by the parity of
\[ \sigma (y ; x) := 1 + |x| \ , \]
we check that the operations $m_t$ define a twisted $\Omega B As$-algebra structure on $(C^*(f),\partial_{Morse} , \partial_{Morse}^{Tw})$. The operations $\mu_{t_g}$ on their side define a twisted $\Omega B As$-morphism between $(C^*(f),\partial_{Morse} , \partial_{Morse}^{Tw})$ and $(C^*(g),\partial_{Morse} , \partial_{Morse}^{Tw})$.

\newpage

\begin{leftbar}
\part{Further developments} \label{p:further-developments}
\end{leftbar}

\setcounter{section}{0}

\section{The map $\mu^\mathbb{Y}$ is a quasi-isomorphism} \label{fd:s:q-iso}

The goal of this section is to prove the following proposition :
\begin{proposition} \label{fd:prop:quasi-iso-Morse}
The twisted $\Omega B As$-morphism $\mu^{\mathbb{Y}} : (C^*(f),m_t^{\mathbb{X}^f}) \longrightarrow (C^*(g),m_t^{\mathbb{X}^g})$ constructed in Theorem~\ref{geo:th:ombas-morph} is a quasi-isomorphism.
\end{proposition}
\noindent In other words we want to prove that the arity 1 component $\mu^{\mathbb{Y}}_{\arbreopunmorpha} : C^*(f) \rightarrow C^*(g)$ is a map which induces an isomorphism in cohomology. The map $\mu^{\mathbb{Y}}_{\arbreopunmorpha}$ is a dg-map $(C^*(f) , \partial_{Morse}^{Tw}) \rightarrow (C^*(g),\partial_{Morse})$, but the cohomologies defined by the differentials $\partial_{Morse}^{Tw}$ and $\partial_{Morse}$ are equal.

In this regard, we will prove that given three perturbation data on $\mathcal{CT}_1 := \{ \arbreopunmorph \}$, $\mathbb{Y}^{fg}_{\arbreopunmorphfg}$ , $\mathbb{Y}^{gf}_{\arbreopunmorphgf}$ and $\mathbb{Y}^{ff}_{\arbreopunmorphff}$, defining dg-maps
\[ \mu^{\mathbb{Y}^{ij}_{\arbreopunmorpha}} : (C^*(i),\partial_{Morse}^{Tw}) \longrightarrow (C^*(j),\partial_{Morse}) \ , \]
we can construct a homotopy $h : C^*(f) \rightarrow C^*(f)$ such that \[ (-1)^d\mu^{\mathbb{Y}^{gf}}_{\arbreopunmorpha} \circ \mu^{\mathbb{Y}^{fg}}_{\arbreopunmorpha} - \mu^{\mathbb{Y}^{ff}}_{\arbreopunmorpha} = \partial_{Morse} h + h \partial^{Tw}_{Morse}  \ . \]
Specializing to the case where $\mathbb{Y}^{ff}_{\arbreopunmorphff}$ is null, $\mu^{\mathbb{Y}^{ff}_{\arbreopunmorphff}} = \ide$ and this yields the desired result.
For the sake of readability, we will write $\mathbb{Y}^{ij} := \mathbb{Y}^{ij}_{\arbreopunmorpha}$ in the rest of this section. Note also that the choice of perturbation data $\mathbb{X}^f$ and $\mathbb{X}^g$ are not necessary for this construction.

In the last paragraph of subsection \ref{geo:ss:ombas-str-Morse} of part~\ref{p:geometry}, we explained that given any Morse function $f$ together with an admissible choice of perturbation data $\mathbb{X}^f$, the Morse cochains $C^*(f)$ and the singular cochains $C^*_{sing}(M)$ are quasi-isomorphic as twisted \ombas  -algebras. In particular, given another Morse function $g$ together with an admissible choice of perturbation data $\mathbb{X}^g$, the Morse cochains $C^*(f)$ and $C^*(g)$ are quasi-isomorphic as twisted \ombas -algebras. Proposition~\ref{fd:prop:quasi-iso-Morse} show that the twisted \ombas -morphism  $\mu^\mathbb{Y}$ realizes such a quasi-isomorphism explicitly.

\subsection{The moduli space $\mathcal{H}(y;x)$} \label{fd:ss:mod-space-H}

Begin by considering the moduli space of metric trees $\mathcal{H}$, represented in two equivalent ways in figure~\ref{fd:fig:mod-space-H}. Adapting the discussions of section~\ref{geo:ss:pert-Morse-tree}, we infer without difficulty the notion of \emph{smooth choice of perturbation data on $\mathcal{H}$.} Given such a choice of perturbation data $\mathbb{W}$, we then say that it is consistent with the $\mathbb{Y}^{ij}$ if it is such that, when $l \rightarrow 0$, $ \mathrm{lim} (\mathbb{W}) = \mathbb{Y}^{ff}$, and when $ l \rightarrow + \infty$, the limit $ \mathrm{lim} (\mathbb{W})$ on the above part of the broken tree is $\mathbb{Y}^{fg}$ and the limit $ \mathrm{lim} (\mathbb{W})$ on the bottom part of the broken tree is $\mathbb{Y}^{gf}$.

\begin{figure}[h] 
    \centering
    \begin{subfigure}{0.4\textwidth}
    \centering
       \hmodspaceA 
    \end{subfigure} ~
    \begin{subfigure}{0.4\textwidth}
    \centering
        \hmodspaceB
    \end{subfigure}
    \caption{} \label{fd:fig:mod-space-H}
\end{figure}

For $x$ and $y$ critical points of the function $f$, introduce now the moduli space $\mathcal{H}^\mathbb{W}(y;x)$ consisting of perturbed Morse gradient trees modeled on \ophomotopie , and such that the two external edges correspond to perturbed Morse equations for $f$, and the internal edge corresponds to a perturbed Morse equation for $g$. We then check that a generic choice of perturbation data $\mathbb{W}$ makes them into orientable manifolds of dimension
\[ \mathrm{dim}(\mathcal{H}^\mathbb{W}(y;x)) = |y| - |x| + 1 \ . \]
The 1-dimensional moduli spaces $\mathcal{H}(y;x)$ can be compactified into compact manifolds with boundary $\overline{\mathcal{H}}(y;x)$, whose boundary is given by the three following phenomena :
\begin{enumerate}[label=(\roman*)]
\item an external edge breaks at a critical point of $f$ (Morse) ;
\item the length of the internal edge tends towards 0 : this yields the moduli spaces $$\mathcal{CT}^{\mathbb{Y}^{ff}}(y;x) \ ;$$
\item the internal edge breaks at a critical point of $g$ : this yields the moduli spaces $$\bigcup_{z \in \mathrm{Crit}(g)}\mathcal{CT}^{\mathbb{Y}^{gf}}(y;z)  \times \mathcal{CT}^{\mathbb{Y}^{fg}}(z;x) \ . $$
\end{enumerate}

Defining the map $h : C^*(f) \rightarrow C^*(f)$ as $h(x) := \sum_{|y| = |x| - 1} \# \mathcal{H}^\mathbb{W}(y;x) \cdot y$, a signed count of the boundary points of the 1-dimensional compactified moduli spaces $\overline{\mathcal{H}}^\mathbb{W}(y;x)$ then proves that : 
\begin{proposition} \label{fd:prop:homotopy-Morse}
The map $h$ defines an homotopy between $(-1)^d\mu^{\mathbb{Y}^{gf}} \circ \mu^{\mathbb{Y}^{fg}}$ and $\mu^{\mathbb{Y}^{ff}}$ i.e. is such that 
\[ (-1)^d\mu^{\mathbb{Y}^{gf}} \circ \mu^{\mathbb{Y}^{fg}} - \mu^{\mathbb{Y}^{ff}} = \partial_{Morse} h + h \partial^{Tw}_{Morse}  \ . \]
\end{proposition}
\noindent Proposition~\ref{fd:prop:quasi-iso-Morse} is then a simple corollary to this proposition.

\subsection{Proof of Propositions~\ref{fd:prop:quasi-iso-Morse}~and~\ref{fd:prop:homotopy-Morse}} \label{fd:ss:signs-or}

We define the moduli space $\mathcal{H}(y;x)$ as before, by introducing the map
\[ \phi_\mathbb{W} : \mathcal{H} \times W^S(y) \times W^U(x) \longrightarrow M \times M \ , \]
and setting $\mathcal{H}(y;x) := \phi^{-1} (\Delta)$ where $\Delta$ is the diagonal of $M \times M$.
We recall moreover that $\sigma (\arbreopunmorpha ; y ; x) = d(1+|y|)$, $\sigma (y ; x ) = 1$ and that
\begin{align*}
\mu^{\mathbb{Y}^{ij}} (x) = \sum_{|y| = |x|} \# \widetilde{\mathcal{CT}}_{\arbreopunmorpha}^{\mathbb{Y}^{ij}} (y ; x) \cdot y    &&   \partial_{Morse} (x) =  \sum_{|y|= |x|+1 } \# \widetilde{\mathcal{T}}(y;x) \cdot y \ .
\end{align*}
We then set 
\[ \sigma (\ophomotopie ; y ; x) = (d+1)|y| \ , \]
and write $\widetilde{\mathcal{H}}(y;x)$ for the moduli space $\mathcal{H}(y;x)$ endowed with the orientation obtained by twisting its natural orientation by a sign of parity $\sigma (\ophomotopie ; y ; x)$. We can now define the map $h : C^*(f) \rightarrow C^*(f)$ by
\[ h (x) := \sum_{|y| = |x| - 1} \# \widetilde{\mathcal{H}} (y ; x) \cdot y \ .  \]

If $\widetilde{\mathcal{H}}( y ; x )$ is 1-dimensional, its boundary decomposes as the disjoint union of the following four types of components
\begin{align*} 
&(-1)^{|y|+d} \widetilde{\mathcal{CT}}^{\mathbb{Y}^{gf}}(y;z)  \times \widetilde{\mathcal{CT}}^{\mathbb{Y}^{fg}}(z;x) &&
(-1)^{|y|+1} \widetilde{\mathcal{CT}}^{\mathbb{Y}^{ff}}(y;x) \\
&(-1)^{|y| +1} \widetilde{\mathcal{T}}(y ; z) \times  \widetilde{\mathcal{H}}( z ; x ) && (-1)^{|y|+1+(d+1)|x|} \widetilde{\mathcal{H}}( y ; z ) \times \widetilde{\mathcal{T}}( z ; x )  \ . 
\end{align*} 
Counting the boundary points of these 1-dimensional moduli spaces implies that
\[ (-1)^d\mu^{\mathbb{Y}^{gf}} \circ \mu^{\mathbb{Y}^{fg}} - \mu^{\mathbb{Y}^{ff}} = \partial_{Morse} h + h \partial^{Tw}_{Morse}  \ . \]
To prove Proposition~\ref{fd:prop:quasi-iso-Morse}, it remains to note that this relation descends in cohomology to the relation
\[ (-1)^d[\mu^{\mathbb{Y}^{gf}}] \circ [\mu^{\mathbb{Y}^{fg}}] = [\mu^{\mathbb{Y}^{ff}}] \ . \]

\section{More on the $\Omega B As$ viewpoint} \label{fd:s:more-ombas}

We stated in section~\ref{alg:ss:homotopy} that because the two-colored operad $A_\infty^2$ is a fibrant-cofibrant replacement of $As^2$ in the model category of two-colored operads, the category of \Ainf -algebras with \Ainf -morphisms between them yields a nice homotopic framework to study the notion of "dg-algebras which are associative up to homotopy". In fact, most classical theorems for \Ainf -algebras can be proven using  the machinery of model categories, on the model category of two-colored operads in dg-\Z -modules.
We can thus similarly introduce the two-colored operad $\Omega B As^2$, which is again a fibrant-cofibrant replacement of $As^2$ in the model category of two-colored operads. The category of $\Omega B As$-algebras with $\Omega B As$-morphisms between them yields another satisfactory homotopic framework to study "dg-algebras which are associative up to homotopy", in which most classical theorems for \Ainf -algebras still hold.

We also point out that while there exists a morphism of operads $\Ainf \rightarrow \Omega B As$ which is canonically given by refining the cell decompositions on the associahedra, Markl and Shnider constructed in~\cite{markl-assoc} an explicit non-canonical morphism of operads $\Omega B As \rightarrow \Ainf$.
The operads $\Omega B As$ and $\Ainf$ being fibrant-cofibrant replacements of $As$, model category theory tells us that there necessarily exist two morphisms $\Ainf \rightarrow \Omega B As$ and $\Omega B As \rightarrow \Ainf$. Hence the noteworthy property of these two morphisms is \emph{not that they exist, but that they are explicit and computable}.

Switching to the two-colored operadic viewpoint, model category theory tells us again that there necessarily exist two morphisms $A_\infty^2 \rightarrow \Omega B As^2$ and $\Omega B As^2 \rightarrow A_\infty^2$. We have already introduced the necessary material to define an explicit and computable morphism of two-colored operads $A_\infty^2 \rightarrow \Omega B As^2$. To render explicit a morphism $ \Omega B As^2 \rightarrow A_\infty^2$ it would be enough to construct a morphism of operadic bimodules 
$\Omega B As - \mathrm{Morph} \rightarrow \infmor $. To our knowledge, this has not yet been done, but we conjecture that the construction of Markl-Shnider should adapt nicely to the multiplihedra to define such a morphism.

\section{\Ainf -structures in symplectic topology} \label{fd:s:quilted-disks}

We explained in this article how the associahedra can be realized as compactified moduli spaces of stable metric ribbon trees. In fact, writing $\mathcal{D}_{n,1}$ for the moduli space of stable disks with $n+1$ marked points on their boundary, where $n$ points are seen as incoming, and 1 as outgoing, the moduli space $\mathcal{D}_{n,1}$ can be compactified and topologized in such a way that it is isomorphic as a CW-complex to the associahedron $K_n$. See~\cite{seidel-fukaya} for instance. Mau-Woodward also prove in~\cite{mau-woodward} that the multiplihedra $J_n$ can be realized as the compactified moduli spaces of stable quilted disks $\overline{\mathcal{QD}}_{n,1}$. The objects of $\mathcal{QD}_{n,1}$ are disks with $n+1$ points $z_0,z_1,\cdots,z_n$ marked on the boundary, with an additional interior disk passing through the point $z_0$. An instance is depicted in figure~\ref{fd:fig:quilted-disk}.
These families of moduli spaces however only contain the \Ainf -cell decompositions of the associahedra resp. multiplihedra, and do not contain their $\Omega B As$-cell decompositions. 

A \textit{symplectic manifold} corresponds to the data of a smooth manifold $M$ together with a closed non-degenerate 2-form $\omega$ on $M$. The purpose of \emph{symplectic topology} is the study of the geometrical properties of symplectic manifolds $(M,\omega)$, and of the way they are preserved under smooth transformations preserving the symplectic structure. As algebraic topology seeks to associate algebraic invariants to topological spaces, in the hope of distinguishing them and understanding some of their topological properties, the same \textit{modus operandi} can be applied to the study of symplectic manifolds. This point view was prompted by the seminal work of Gromov~\cite{gromov-pseudo} on \emph{moduli spaces of pseudo-holomorphic curves}. By counting the points of 0-dimensional moduli spaces of pseudo-holomorphic curves, one will be able to define algebraic operations stemming from the geometry of the underlying symplectic manifolds.

The most famous example is that of the \emph{Fukaya category} $\mathrm{Fuk}(M)$ of a symplectic manifold $M$ (with additional technical assumptions). It is an \Ainf -category whose higher multiplications are defined by counting moduli spaces of pseudo-holomorphic disks with Lagrangian boundary conditions and $n+1$ marked points on their boundary, in other words by realizing the moduli spaces $\mathcal{D}_{n,1}$ in symplectic topology. We refer for instance to~\cite{smith-prolegomenon}~and~\cite{auroux-fukaya} for introductions to the subject. 
The moduli spaces of quilted disks can be similarly realized as pseudo-holomorphic curves in symplectic topology as in~\cite{mau-wehrheim-woodward}, to construct \Ainf -functors between Fukaya categories.

\begin{figure}[h]
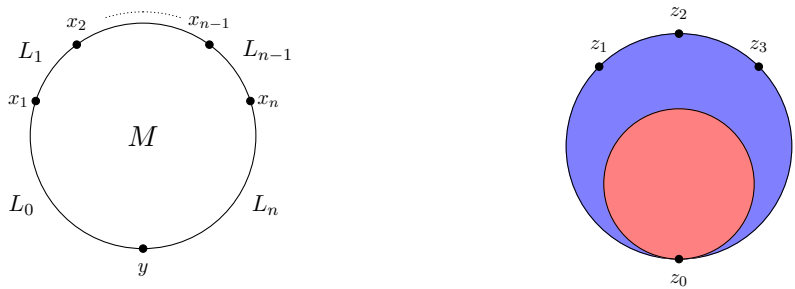
 
    \centering
    \begin{subfigure}{0.4\textwidth}
    \centering
       \disquebordlagrangienassoc
    \end{subfigure} ~
    \begin{subfigure}{0.4\textwidth}
    \centering
       \quilteddiskexample
    \end{subfigure}
    \caption{On the left, an example of a pseudo-holomorphic disk with Lagrangian boundary conditions on the Lagrangian submanifolds $L_0,\dots,L_n$ whose $n+1$ marked points are sent to the points $y,x_1,\dots,x_n$ in $M$. On the right, an example of a quilted disk in $\mathcal{QD}_{3,1}$.} \label{fd:fig:quilted-disk}
\end{figure}

\noindent It is also worth mentioning the work of Bottman on that matter. He is currently developing an algebraic model for the notion of $(\Ainf , 2)$-categories, using moduli spaces of witch curves. The goal is to prove that one can then define an $(\Ainf , 2)$-category $\mathtt{Symp}$ whose objects would be symplectic manifolds (with suitable technical assumptions), and such that the space of morphisms between two symplectic manifolds $M$ and $N$ would be the Fukaya category $\mathrm{Fuk}(M^-\times N)$. We refer to his recent papers~\cite{bottman-assoc}~and~\cite{bottman-witch} for more details. 

\section{Towards higher algebra} \label{fd:s:higher-alg}

In closing, two questions naturally arise from this construction. They will respectively represent the starting points to the parts II and III to this article.

\paragraph{\textbf{Problem 1.}} Given two Morse functions $f,g$, choices of perturbation data $\mathbb{X}^f$ and $\mathbb{X}^g$, and choices of perturbation data $\mathbb{Y}$ and $\mathbb{Y}'$, is $\mu^{\mathbb{Y}}$ always \Ainf -homotopic (resp. $\Omega B As$-homotopic) to $\mu^{\mathbb{Y}'}$ ? I.e., when can the following diagram be filled in the \Ainf\ (resp. $\Omega B As$) world
\[ \begin{tikzcd}[row sep=large, column sep = large]
C^*(f) \arrow[bend left=40]{r}[above]{\mu^{\mathbb{Y}}}[name=U,below,pos=0.5]{}
\arrow[bend right=40]{r}[below]{\mu^{\mathbb{Y}'}}[name=D,pos=0.5]{}
& C^*(g)
\arrow[Rightarrow, from=U, to=D]
\end{tikzcd} \ ? \]
In which sense, with which notion of homotopy can it be filled ? And in general, which notion of higher operadic algebra naturally encodes this type of problem ?

\paragraph{\textbf{Problem 2.}} Given three Morse functions $f_0,f_1,f_2$, choices of perturbation data $\mathbb{X}^i$, and choices of perturbation data $\mathbb{Y}^{ij}$ defining morphisms 
\begin{align*}
\mu^{\mathbb{Y}^{01}} : (C^*(f_0),m_t^{\mathbb{X}^0}) \longrightarrow (C^*(f_1),m_t^{\mathbb{X}^1}) \ , \\
\mu^{\mathbb{Y}^{12}} : (C^*(f_1),m_t^{\mathbb{X}^1}) \longrightarrow (C^*(f_2),m_t^{\mathbb{X}^2}) \ , \\
\mu^{\mathbb{Y}^{02}} : (C^*(f_0),m_t^{\mathbb{X}^0}) \longrightarrow (C^*(f_2),m_t^{\mathbb{X}^2}) \ ,
\end{align*}
can we construct an \Ainf -homotopy (or an $\Omega B As$-homotopy), such that $\mu^{\mathbb{Y}^{12}} \circ \mu^{\mathbb{Y}^{01}} \simeq \mu^{\mathbb{Y}^{02}}$ through this homotopy ? That is, can the following cone be filled in the \Ainf\ (resp. $\Omega B As$) world
\[ \begin{tikzcd}[row sep=large, column sep = large]
C^*(f_0) \arrow{dr}[below left]{\mu^{\mathbb{Y}^{02}}}[name=U,above right,pos=0.6]{} \arrow{r}[above]{\mu^{\mathbb{Y}^{01}}}
& C^*(f_1) \arrow{d}[right]{\mu^{\mathbb{Y}^{12}}} \arrow[Rightarrow, to=U] \\
& C^*(f_2) 
\end{tikzcd} \ ? \]
Which higher operadic algebra naturally arises from this basic question ? Note that the construction of section~\ref{fd:s:q-iso} solves the arity 1 step of this problem.

Problem 1 is solved in~\cite{mazuir-II} by introducing the notions of $n- \Ainf$-morphisms and $n- \Omega B As$-morphisms. Problem 2 will be adressed in an upcoming paper, in which it will appear that the higher algebra of $n-\Ainf$-morphisms provides a natural framework to solve this problem.

\newpage

\bibliographystyle{alpha}

\bibliography{ha-morse-I-biblio}

\begin{thebibliography}{MTTV19}

\bibitem[Abo11]{abouzaid-plumbings}
Mohammed Abouzaid.
\newblock A topological model for the {F}ukaya categories of plumbings.
\newblock {\em J. Differential Geom.}, 87(1):1--80, 2011.

\bibitem[AR67]{abraham-transversal}
Ralph Abraham and Joel Robbin.
\newblock {\em Transversal mappings and flows}.
\newblock An appendix by Al Kelley. W. A. Benjamin, Inc., New York-Amsterdam,
  1967.

\bibitem[Aur14]{auroux-fukaya}
Denis Auroux.
\newblock A beginner's introduction to {F}ukaya categories.
\newblock In {\em Contact and symplectic topology}, volume~26 of {\em Bolyai
  Soc. Math. Stud.}, pages 85--136. J\'{a}nos Bolyai Math. Soc., Budapest,
  2014.

\bibitem[Bot19a]{bottman-assoc}
Nathaniel Bottman.
\newblock 2-associahedra.
\newblock {\em Algebr. Geom. Topol.}, 19(2):743--806, 2019.

\bibitem[Bot19b]{bottman-witch}
Nathaniel Bottman.
\newblock Moduli spaces of witch curves topologically realize the
  2-associahedra.
\newblock {\em J. Symplectic Geom.}, 17(6):1649--1682, 2019.

\bibitem[BV73]{boardman-vogt}
J.~M. Boardman and R.~M. Vogt.
\newblock {\em Homotopy invariant algebraic structures on topological spaces}.
\newblock Lecture Notes in Mathematics, Vol. 347. Springer-Verlag, Berlin-New
  York, 1973.

\bibitem[For08]{forcey-multipl}
Stefan Forcey.
\newblock Convex hull realizations of the multiplihedra.
\newblock {\em Topology Appl.}, 156(2):326--347, 2008.

\bibitem[Fuk97]{fukaya-morse-homotopy}
Kenji Fukaya.
\newblock Morse homotopy and its quantization.
\newblock In {\em Geometric topology ({A}thens, {GA}, 1993)}, volume~2 of {\em
  AMS/IP Stud. Adv. Math.}, pages 409--440. Amer. Math. Soc., Providence, RI,
  1997.

\bibitem[GJ94]{getzler-jones-operads}
Ezra Getzler and John~DS Jones.
\newblock Operads, homotopy algebra and iterated integrals for double loop
  spaces.
\newblock {\em hep-th/9403055}, 1994.

\bibitem[Gro85]{gromov-pseudo}
M.~Gromov.
\newblock Pseudo holomorphic curves in symplectic manifolds.
\newblock {\em Invent. Math.}, 82(2):307--347, 1985.

\bibitem[Hai84]{haiman-assoc}
M.~Haiman.
\newblock Constructing the associahedron, 1984.
\newblock available for download at
  http://math.berkeley.edu/mhaiman/ftp/assoc/manuscript.pdf.

\bibitem[Hut08]{hutchings-floer}
Michael Hutchings.
\newblock Floer homology of families. {I}.
\newblock {\em Algebr. Geom. Topol.}, 8(1):435--492, 2008.

\bibitem[IM89]{iwase-mimura}
Norio Iwase and Mamoru Mimura.
\newblock Higher homotopy associativity.
\newblock In {\em Algebraic topology ({A}rcata, {CA}, 1986)}, volume 1370 of
  {\em Lecture Notes in Math.}, pages 193--220. Springer, Berlin, 1989.

\bibitem[Kad80]{kadeishvili-theory}
T.~V. Kadei\v{s}vili.
\newblock On the theory of homology of fiber spaces.
\newblock {\em Uspekhi Mat. Nauk}, 35(3(213)):183--188, 1980.
\newblock International Topology Conference (Moscow State Univ., Moscow, 1979).

\bibitem[LAM]{masuda-diagonal-multipl}
Guillaume Laplante-Anfossi and Thibaut Mazuir.
\newblock The diagonal of the multiplihedra and the product of
  {$A_\infty$}-categories.
\newblock In preparation.

\bibitem[Lee89]{lee-assoc}
Carl~W. Lee.
\newblock The associahedron and triangulations of the {$n$}-gon.
\newblock {\em European J. Combin.}, 10(6):551--560, 1989.

\bibitem[LH02]{lefevre-hasegawa}
Kenji Lefevre-Hasegawa.
\newblock {\em Sur les $A_\infty$-cat{\'e}gories}.
\newblock PhD thesis, Ph. D. thesis, Universit{\'e} Paris 7, UFR de
  Math{\'e}matiques, 2003, math. CT/0310337, 2002.

\bibitem[Lod04]{loday-assoc}
Jean-Louis Loday.
\newblock Realization of the {S}tasheff polytope.
\newblock {\em Arch. Math. (Basel)}, 83(3):267--278, 2004.

\bibitem[LV12]{loday-vallette-algebraic-operads}
Jean-Louis Loday and Bruno Vallette.
\newblock {\em Algebraic operads}, volume 346 of {\em Grundlehren der
  Mathematischen Wissenschaften [Fundamental Principles of Mathematical
  Sciences]}.
\newblock Springer, Heidelberg, 2012.

\bibitem[Mar02]{markl-homotopy-diagram}
Martin Markl.
\newblock Homotopy diagrams of algebras.
\newblock In {\em Proceedings of the 21st {W}inter {S}chool ``{G}eometry and
  {P}hysics'' ({S}rn\'{\i}, 2001)}, number~69, pages 161--180, 2002.

\bibitem[Mar06]{markl-transferring}
Martin Markl.
\newblock Transferring {$A_\infty$} (strongly homotopy associative) structures.
\newblock {\em Rend. Circ. Mat. Palermo (2) Suppl.}, (79):139--151, 2006.

\bibitem[Maz21]{mazuir-II}
Thibaut Mazuir.
\newblock Higher algebra of {$A_\infty$} and {$\Omega B As$}-algebras in
  {Morse} theory {II}.
\newblock {arXiv:2102.08996}, 2021.

\bibitem[Mes18]{mescher-morse}
Stephan Mescher.
\newblock {\em Perturbed gradient flow trees and {$A_\infty$}-algebra
  structures in {M}orse cohomology}, volume~6 of {\em Atlantis Studies in
  Dynamical Systems}.
\newblock Atlantis Press, [Paris]; Springer, Cham, 2018.

\bibitem[MS06]{markl-assoc}
Martin Markl and Steve Shnider.
\newblock Associahedra, cellular {$W$}-construction and products of
  {$A_\infty$}-algebras.
\newblock {\em Trans. Amer. Math. Soc.}, 358(6):2353--2372, 2006.

\bibitem[MS12]{mcduff-salamon}
Dusa McDuff and Dietmar Salamon.
\newblock {\em {$J$}-holomorphic curves and symplectic topology}, volume~52 of
  {\em American Mathematical Society Colloquium Publications}.
\newblock American Mathematical Society, Providence, RI, second edition, 2012.

\bibitem[MTTV19]{masuda-diagonal-assoc}
Naruki Masuda, Hugh Thomas, Andy Tonks, and Bruno Vallette.
\newblock The diagonal of the associahedra.
\newblock arXiv:1902.08059, 2019.

\bibitem[MW10]{mau-woodward}
S.~Ma'u and C.~Woodward.
\newblock Geometric realizations of the multiplihedra.
\newblock {\em Compos. Math.}, 146(4):1002--1028, 2010.

\bibitem[MWW18]{mau-wehrheim-woodward}
S.~Ma'u, K.~Wehrheim, and C.~Woodward.
\newblock {$A_\infty$} functors for {L}agrangian correspondences.
\newblock {\em Selecta Math. (N.S.)}, 24(3):1913--2002, 2018.

\bibitem[Sei08]{seidel-fukaya}
Paul Seidel.
\newblock {\em Fukaya categories and {P}icard-{L}efschetz theory}.
\newblock Zurich Lectures in Advanced Mathematics. European Mathematical
  Society (EMS), Z\"{u}rich, 2008.

\bibitem[Sma65]{smale-sard}
S.~Smale.
\newblock An infinite dimensional version of {S}ard's theorem.
\newblock {\em Amer. J. Math.}, 87:861--866, 1965.

\bibitem[Smi15]{smith-prolegomenon}
Ivan Smith.
\newblock A symplectic prolegomenon.
\newblock {\em Bull. Amer. Math. Soc. (N.S.)}, 52(3):415--464, 2015.

\bibitem[Sta63]{stasheff-homotopy}
James~Dillon Stasheff.
\newblock Homotopy associativity of {$H$}-spaces. {I}, {II}.
\newblock {\em Trans. Amer. Math. Soc. 108 (1963), 275-292; ibid.},
  108:293--312, 1963.

\bibitem[Tam54]{tamari-monoides}
Dov Tamari.
\newblock Mono{\"\i}des pr{\'e}ordonn{\'e}s et cha{\^\i}nes de {Malcev}.
\newblock {\em Bulletin de la Soci{\'e}t{\'e} math{\'e}matique de France},
  82:53--96, 1954.

\bibitem[Val14]{vallette-algebra}
Bruno Vallette.
\newblock Algebra + homotopy = operad.
\newblock In {\em Symplectic, {P}oisson, and noncommutative geometry},
  volume~62 of {\em Math. Sci. Res. Inst. Publ.}, pages 229--290. Cambridge
  Univ. Press, New York, 2014.

\bibitem[Weh12]{wehrheim-morse}
Katrin Wehrheim.
\newblock Smooth structures on {M}orse trajectory spaces, featuring finite ends
  and associative gluing.
\newblock In {\em Proceedings of the {F}reedman {F}est}, volume~18 of {\em
  Geom. Topol. Monogr.}, pages 369--450. Geom. Topol. Publ., Coventry, 2012.

\bibitem[Yau16]{yau-colored}
Donald Yau.
\newblock {\em Colored operads}, volume 170 of {\em Graduate Studies in
  Mathematics}.
\newblock American Mathematical Society, Providence, RI, 2016.

\end{thebibliography}
\end{document}